\documentclass[11pt]{amsart}
\usepackage[latin1]{inputenc}
\usepackage{amsmath}
\usepackage{amssymb}
\usepackage[all]{xy}
\usepackage{mathrsfs}
\usepackage{amsthm}

\theoremstyle{plain}
\newtheorem{thm}[equation]{Theorem}
\newtheorem{lem}[equation]{Lemma}
\newtheorem{prop}[equation]{Proposition}
\newtheorem{cor}[equation]{Corollary}

\theoremstyle{definition}
\newtheorem{defin}[equation]{Definition}

\theoremstyle{remark}
\newtheorem{remark}[equation]{Remark}

\numberwithin{equation}{subsection}

\def\sheafEnd{\mathcal{E} \hspace{-1pt} \mathit{nd}}
\def\sheafHom{\mathcal{H} \hspace{-1pt} \mathit{om}}
\def\Av{{\rm A} \hspace{-1pt} {\rm v}}

\newcommand{\bk}{\Bbbk}
\newcommand{\bbY}{\mathbb Y}
\newcommand{\bbX}{\mathbb X}
\newcommand{\bbZ}{\mathbb Z}
\newcommand{\calA}{\mathcal{A}}
\newcommand{\calB}{\mathcal{B}}
\newcommand{\calC}{\mathcal{C}}
\newcommand{\calD}{\mathcal{D}}
\newcommand{\calE}{\mathcal{E}}
\newcommand{\calF}{\mathcal{F}}
\newcommand{\calG}{\mathcal{G}}
\newcommand{\calH}{\mathcal{H}}
\newcommand{\calI}{\mathcal{I}}
\newcommand{\calJ}{\mathcal{J}}
\newcommand{\calK}{\mathcal{K}}
\newcommand{\calL}{\mathcal{L}}
\newcommand{\calM}{\mathcal{M}}
\newcommand{\calN}{\mathcal{N}}
\newcommand{\calO}{\mathcal{O}}
\newcommand{\calP}{\mathcal{P}}
\newcommand{\calQ}{\mathcal{Q}}
\newcommand{\calR}{\mathcal{R}}
\newcommand{\calS}{\mathcal{S}}
\newcommand{\calT}{\mathcal{T}}
\newcommand{\calU}{\mathcal{U}}

\newcommand{\calX}{\mathcal{X}}
\newcommand{\calY}{\mathcal{Y}}
\newcommand{\calZ}{\mathcal{Z}}
\newcommand{\wcalN}{\widetilde{\mathcal{N}}}
\newcommand{\wfrakg}{\widetilde{\mathfrak{g}}}
\newcommand{\wcalD}{\widetilde{\mathcal{D}}}

\newcommand{\frakn}{\mathfrak{n}}
\newcommand{\frakg}{\mathfrak{g}}
\newcommand{\frakt}{\mathfrak{t}}
\newcommand{\frakb}{\mathfrak{b}}

\newcommand{\frakh}{\mathfrak{h}}
\newcommand{\frakZ}{\mathfrak{Z}}

\newcommand{\frakp}{\mathfrak{p}}
\newcommand{\frakR}{\mathfrak{R}}
\newcommand{\frakS}{\mathfrak{S}}

\newcommand{\mbfC}{\mathbf{C}}

\newcommand{\mbfE}{\mathbf{E}}
\newcommand{\mbfF}{\mathbf{F}}
\newcommand{\mbfG}{\mathbf{G}}
\newcommand{\lotimes}{{\stackrel{_L}{\otimes}}}
\newcommand{\rcap}{{\stackrel{_R}{\cap}}}
\newcommand{\Gm}{{\mathbb{G}}_{\mathbf{m}}}

\newcommand{\Hom}{{\rm Hom}}
\newcommand{\Ext}{{\rm Ext}}

\newcommand{\Ind}{{\rm Ind}}
\newcommand{\Coind}{{\rm Coind}}
\newcommand{\For}{{\rm For}}
\newcommand{\Id}{{\rm Id}}
\newcommand{\Ker}{{\rm Ker}}
\newcommand{\Ima}{{\rm Im}}

\newcommand{\Coh}{{\rm Coh}}
\newcommand{\QCoh}{{\rm QCoh}}

\newcommand{\Mod}{{\rm Mod}}

\newcommand{\qis}{{\rm qis}}
\newcommand{\rad}{{\rm rad}}
\newcommand{\soc}{{\rm soc}}
\newcommand{\pro}{{\rm prop}}
\newcommand{\rk}{{\rm rk}}
\newcommand{\qc}{{\rm qc}}
\newcommand{\fg}{{\rm fg}}

\newcommand{\gr}{{\rm gr}}
\newcommand{\dg}{{\rm dg}}
\newcommand{\fd}{{\rm fd}}
\newcommand{\SL}{{\rm SL}}
\newcommand{\op}{{\rm op}}
\newcommand{\Shift}{{\rm Shift}}
\newcommand{\scra}{\mathscr{A}}

\newcommand{\scrb}{\mathscr{B}}

\newcommand{\rmi}{\rm{(i)}}
\newcommand{\rmii}{\rm{(ii)}}
\newcommand{\rmiii}{\rm{(iii)}}
\newcommand{\rmiv}{\rm{(iv)}}
\newcommand{\rmv}{\rm{(v)}}

\newcommand{\aff}{{\rm aff}}
\newcommand{\DGCoh}{{\rm DGCoh}}
\newcommand{\DGSh}{\rm{DGSh}}
\newcommand{\hdot}{^{\:\raisebox{3pt}{\text{\circle*{2}}}}}

\newcommand{\Fr}{{\rm Fr}}

\author{Simon Riche}
\address{Universit{\'e} Pierre et Marie Curie, Institut de Math{\'e}matiques de
  Jussieu (UMR 7586 du CNRS), {\'E}quipe d'Analyse Alg{\'e}brique, Paris, France.}
\email{riche@math.jussieu.fr}


\thanks{\emph{Author's current address}: Laboratoire de Mathématiques de l'Université Blaise Pascal (UMR 6620 du CNRS), Campus Universitaire des Cézeaux, 63177 Aubière cedex, France.}

\title[Koszul duality and representations of Lie algebras]{Koszul
  duality and modular representations of semi-simple Lie algebras}

\begin{document}

\begin{abstract}

In this paper we prove that if $G$ is a connected, simply-connected, semi-simple algebraic group over an algebraically closed field of sufficiently large characteristic, then all the blocks of the restricted enveloping algebra $(\calU \frakg)_0$ of the Lie algebra $\frakg$ of $G$ can be endowed with a Koszul grading (extending results of Andersen, Jantzen, Soergel). We also give information about the Koszul dual rings. In the case of the block associated to a regular character $\lambda$ of the Harish-Chandra center, the dual ring is related to modules over the specialized algebra $(\calU \frakg)^{\lambda}$ with generalized trivial Frobenius character. Our main tool is the localization theory developed by Bezru\-kavnikov, Mirkovi{\'c}, Rumynin.

\end{abstract}

\maketitle

\tableofcontents

\section*{Introduction}

\setcounter{subsection}{0}

\subsection{}

Since \cite{BGS}, Koszul duality has proved to be a very useful and
powerful tool in Lie theory. In \cite{BGS}, Beilinson,
Ginzburg and Soergel prove that every block of the category $\calO$ of
a complex semi-simple Lie algebra is governed by a Koszul ring, whose
dual ring governs another subcategory of the category $\calO$. In this
paper we obtain, using completely different methods, counterparts of these results for modules over the Lie
algebra $\frakg$ of a connected, simply-connected, semi-simple
algebraic group $G$ over an algebraically closed field $\bk$ of sufficiently
large positive characteristic. In particular we prove that every
block of the category of finitely generated modules over the
restricted enveloping algebra $(\calU \frakg)_0$ is governed by a
Koszul ring, whose dual ring is also related to the representation
theory of $\frakg$.  The Koszulity of the regular blocks was already
proved (under the same assumptions) using different methods by Andersen, Jantzen and
Soergel in \cite{AJS}. The Koszulity for singular blocks, as well as
the information on the dual ring (in all cases) are new, however.

As in \cite{BGS} we use a geometric picture to prove Koszulity. Over
$\mathbb{C}$, the authors of \cite{BGS} use perverse sheaves and Beilinson-Bernstein localization; over $\bk$ we rather use coherent sheaves, and the localization theory in positive characteristic developed by
Bezrukavnikov, Mirkovi{\'c} and Rumynin.

\subsection{} \label{ss:introlkd}

The base of our arguments is a geometric interpretation, due to
Mirko\-vi{\'c}, of Koszul duality between symmetric and
exterior algebras. For simplicity, consider first the case of a finite dimensional
vector space $V$. Usual Koszul duality (see \cite{BGG, BGS, GKM}) relates modules (or dg-modules) over the symmetric algebra
${\rm S}(V)$ of $V$ and modules (or dg-modules) over the exterior
algebra $\Lambda(V^*)$ of the dual vector space. Geometrically, ${\rm
  S}(V)$ is the ring of functions on the variety $V^*$. As for
$\Lambda(V^*)$, there exists a quasi-isomorphism of
dg-algebras $\Lambda(V^*) \cong \bk \, \lotimes_{{\rm S}(V^*)} \, \bk$,
where $\Lambda(V^*)$ is equipped with the trivial differential, and
the grading such that $V^*$ is in degre $-1$. Hence $\Lambda(V^*)$ is
the ring of functions on the ``derived intersection'' $\{0\} \,
\rcap_{V} \, \{0\},$ considered as a dg-scheme. An
extension of the constructions of \cite{GKM} yields similarly, if $E$ is a
vector bundle over a non-singular variety $X$ and $F \subset E$ is a
sub-bundle, a Koszul duality
between a certain category of (dg)-sheaves on $F$ and a certain
category of (dg)-sheaves on the derived intersection $F^{\bot} \,
\rcap_{E^*} \, X,$ where $E^*$ is the dual vector bundle, $F^{\bot}
\subset E^*$ the orthogonal of $F$, and $X$ the zero
section of $E^*$ (see Theorem \ref{thm:thmlkd}).

We have proved a result of the same flavor with Mirkovi{\'c}, in a more general context, in \cite{MRlkd}. Let us point out, however, that Theorem \ref{thm:thmlkd} is \emph{not} a particular case of the main result of \cite{MRlkd}. In particular, our equivalence here is covariant, while the equivalence of \cite{MRlkd} is contravariant.

\subsection{}

Return to our connected, simply-connected, semi-simple group
$G$ over the field $\bk$ of (sufficiently large)
positive characteristic $p$. Let $\frakg$ be its Lie algebra, and $\calU
\frakg$ the enveloping algebra of $\frakg$. Fix a maximal torus $T$,
with Lie algebra $\frakt$, and a Borel subgroup $B \supset T$,
with Lie algebra $\frakb$. Let $U$ be the unipotent radical of $B$,
and $\frakn$ its Lie algebra. Let $W$ be the Weyl group. The center $\frakZ$ of $\calU \frakg$
has two parts: the Frobenius center $\frakZ_{{\rm Fr}}$,
isomorphic to ${\rm S}(\frakg^{(1)})$ (the supscript denotes
Frobenius twist), and the Harish-Chandra center
$\frakZ_{{\rm HC}} \cong {\rm S}(\frakt)^{(W,\bullet)}$. Hence a character of $\frakZ$ is given by a
``compatible pair'' $(\lambda,\chi)$ where $\lambda \in \frakt^*$, $\chi \in \frakg^* {}^{(1)}$. We only consider the case when
$\chi=0$, and $\lambda$ is integral. Let
$\Mod^{\fg}_{(\lambda,\chi)}(\calU \frakg)$ be the category of
finitely generated $\calU \frakg$-modules on which $\frakZ$ acts with
generalized character $(\lambda,\chi)$. Let $(\calU \frakg)^{\lambda}:=\calU \frakg \otimes_{\frakZ_{{\rm HC}}} \bk_{\lambda}$ be the specialization, and $\Mod^{\fg}_0((\calU \frakg)^{\lambda})$ the category of finitely generated modules on which $\frakZ_{{\rm HC}}$ acts via $\lambda$ and $\frakZ_{{\rm Fr}}$ acts with generalized character $\chi=0$. Similarly, let $(\calU \frakg)_0:=\calU
\frakg \otimes_{\frakZ_{{\rm Fr}}} \bk_0$ be the restricted enveloping algebra, and $\Mod^{\fg}_{\lambda}((\calU \frakg)_0)$ the category of finitely generated restricted modules on which $\frakZ_{{\rm HC}}$ acts with generalized character $\lambda$.

Fix a regular weight $\lambda \in X^*(T)$; we denote similarly the induced element of $\frakt^*$. Important results of
Bezru\-kavnikov, Mirkovi{\'c} and Rumynin (\cite{BMR, BMR2}) give geometric pictures for the derived categories $\calD^b
\Mod^{\fg}_{(\lambda,0)}(\calU \frakg)$ and $\calD^b
\Mod^{\fg}_{0}((\calU \frakg)^{\lambda})$, as follows. Denote by $\calB$ the flag variety, $\wcalN$ the Springer resolution, and $\wfrakg$ the Grothendieck resolution. Then \begin{equation}
  \label{eq:equivintro1} \left\{ \begin{array}{ccc} \calD^b
  \Coh_{\calB^{(1)}}(\wcalN^{(1)}) & \cong & \calD^b
  \Mod^{\fg}_0((\calU \frakg)^{\lambda}), \\ \calD^b
  \Coh_{\calB^{(1)}}(\wfrakg^{(1)}) & \cong & \calD^b
  \Mod^{\fg}_{(\lambda,0)}(\calU \frakg), \end{array} \right. \end{equation} As a first step we
derive from \eqref{eq:equivintro1} a localization theorem for restricted $\calU
\frakg$-modules with generalized character $\lambda$ (see Theorem \ref{thm:localizationfixedFr}): \[ \DGCoh((\wfrakg
  \, \rcap_{\frakg^* \times \calB} \, \calB)^{(1)}) \ \cong \ \calD^b
  \Mod^{\fg}_{\lambda}((\calU \frakg)_0), \] where $\wfrakg \, \rcap_{\frakg^* \times \calB} \, \calB$ is the derived intersection
of $\wfrakg$ and the zero section $\calB$ inside the trivial vector
bundle $\frakg^* \times
\calB$, and $\DGCoh((\wfrakg \, \rcap_{\frakg^* \times \calB} \,
\calB)^{(1)})$ is the deri\-ved category of coherent dg-sheaves on the
Frobenius twist of this intersection.

\subsection{} \label{ss:introkeypoint}

Under our assumptions there is an isomorphism of
$G$-equiva\-riant vector bundles $(\frakg^* \times \calB)^* \cong
\frakg^* \times \calB$, such that $\wfrakg$ identifies
with the orthogonal of $\wcalN \subset \frakg^* \times \calB$. Hence
the Koszul duality of \S \ref{ss:introlkd} yields a duality between certain
dg-sheaves on $\wcalN^{(1)}$ and on the derived intersection $(\wfrakg \,
\rcap_{\frakg^* \times \calB} \, \calB)^{(1)}$. Now observe that there is an inclusion $\calD^b
\Mod^{\fg}_0((\calU \frakg)^{\lambda}) \hookrightarrow \calD^b
\Coh(\wcalN^{(1)})$, induced by
\eqref{eq:equivintro1}. Using the results of \S \ref{ss:introlkd}, we construct categories $\calC^{\gr}$,
$\calD^{\gr}$ of $\Gm$-equivariant dg-sheaves, endowed with auto-equivalences $\langle 1
\rangle$ (the \emph{internal shift}), an equivalence $\kappa :
\calC^{\gr} \xrightarrow{\sim} \calD^{\gr}$ (Koszul duality), and a diagram \[
\xymatrix@R=16pt{ & \calC^{\gr} \ar[d]_-{\For} \ar[r]^-{\kappa}_-{\sim} &
  \calD^{\gr} \ar[d]^-{\For} \\ \calD^b \Mod^{\fg}_0((\calU
  \frakg)^{\lambda}) \ar@{^{(}->}[r] & \calD^b \Coh(\wcalN^{(1)}) &
  \calD^b \Mod^{\fg}_{\lambda}((\calU \frakg)_0). } \] In other words, we construct a ``Koszul duality'' which relates two categories of $\calU \frakg$-modules with central character determined by the pair $(\lambda,0)$: one in which the \emph{Frobenius character} is generalized, and one in which the \emph{Harish-Chandra character} is generalized (the other one being fixed).

Let $W_{\aff}':=W \ltimes X^*(T)$ be
the extended affine Weyl group. By a celebrated theorem of Curtis (\cite{CUR}) and by the decription of
the Harish-Chandra center $\frakZ_{{\rm HC}}$, the simple
objects in the categories $\Mod^{\fg}_0((\calU \frakg)^{\lambda})$ and
$\Mod^{\fg}_{\lambda}((\calU \frakg)_0)$ are the (restrictions of the) simple
$G$-modules $L(\mu)$ for $\mu \in \bbX$ dominant restricted, in the orbit of
$\lambda$ under the dot-action of the extended affine Weyl group
$W_{\aff}'$. The category $\Mod^{\fg}_{\lambda}((\calU \frakg)_0)$ is
the category of finitely generated modules over the finite dimensional
algebra $(\calU \frakg)_0^{\hat{\lambda}}$ (the block of $(\calU
\frakg)_0$ associated to $\lambda$). We denote by $P(\mu)$ the
projective cover of $L(\mu)$ in this category. The objects $L(\mu)$
can be lifted to the category $\calC^{\gr}$, uniquely up to a shift. The same is true for the objects
$P(\mu)$ and the category $\calD^{\gr}$.

Assume $\lambda$ is in the fundamental alcove. Consider $\tau_0:=t_{\rho}
\cdot w_0 \in W_{\aff}'$, where $t_{\rho}$ is the translation by the half sum of positive roots
$\rho$, and $w_0 \in W$ is the longest
element. Our key-point is the following
(see Theorem \ref{thm:mainthm}): \[ \begin{array}{c} \text{Assume } p \gg 0. \text{ There exists a unique choice of the lifts} \\[2pt]
  L^{\gr}(\mu) \in \calC^{\gr}, \ P^{\gr}(\mu) \in \calD^{\gr} \ \text{ such that if } w \in W_{\aff}' \text{ and } \ w \bullet \lambda \ \text{ is} \\[2pt] \text{dominant restricted, then } \kappa(L^{\gr}(w \bullet \lambda)) \cong   P^{\gr}(\tau_0 w \bullet \lambda). \end{array} \] In other words,
our ``geometric'' Koszul duality exchanges semi-simple and projective
modules. This result was supported by calculations in small ranks obtained with Bezrukavnikov and published as an appendix to \cite{BMR}\footnote{See \cite{RPhD} for a more detailed version of these computations, and the case $G=\SL(2)$.}.

\subsection{} \label{ss:introlusztig}

Our proof of this key-point relies on the study of ``geometric
counterparts'' of the reflection functors $\frakR_{\delta}^{\gr} :
\calD^{\gr} \to \calD^{\gr}$ (here $\delta$ is an affine simple root),
which send (lifts of) projectives to (lifts of) projectives. We identify the ``Koszul dual'' (i.e.~the conjugate by
$\kappa$) of these functors, which are related to some functors
$\frakS_{\delta}^{\gr}$ which send (lifts of) some semi-simple modules to
(lifts of) semi-simple modules (see Theorem
\ref{thm:dualreflection}). Then we only have to check our key-point when
$\ell(w)=0$, which can be done explicitly.
 
To prove the ``semi-simplicity'' of the functors $\frakS_{\delta}^{\gr}$ we
use Lusztig's conjecture on the characters of simple
$G$-modules, see \cite{LUSPro} (or rather an equivalent formulation due to Andersen, see \cite{ANDInv}). By the work of Andersen-Jantzen-Soergel (\cite{AJS}), combined with works of
Kazhdan-Lusztig (\cite{KL, LUSMon}) and Kashiwara-Tanisaki (\cite{KT}), (see \cite{ABG, FIEShe} for other approaches), this conjecture is true for $p$ sufficiently large
(with no explicit bound). Recently Fiebig has given a proof of this conjecture for $p$ bigger than a (very large) explicit bound (see \cite{FIELus}). This explains our restriction on $p$.

Let us remark that related ideas (in particular, a contruction of graded versions of translation functors) were considered by Stroppel in \cite{STRCat} for the category $\calO$ in characteristic $0$. However, our techniques are different.

\subsection{}

We derive from the key-point of \S \ref{ss:introkeypoint} the Koszulity of regular blocks of $(\calU \frakg)_0$. For this we use a general criterion
for a graded ring to be Morita equivalent to a Koszul ring, proved in Theorem \ref{thm:thmkoszulity}. More precisely we
obtain the following result (see Theorem \ref{thm:thmkoszul}): \[ \begin{array}{c} \text{There exists a
    (unique) grading on the block } (\calU 
\frakg)_0^{\hat{\lambda}} \text{ which makes it a} \\[2pt] \text{Koszul ring. The Koszul dual ring controls the category } \Mod^{\fg}_0((\calU \frakg)^{\lambda}). \end{array}\] Hence, from a ``geometric'' Koszul duality between dg-sheaves on the dg-schemes $\wcalN^{(1)}$ and $(\wfrakg \,
\rcap_{\frakg^* \times \calB} \, \calB)^{(1)}$ we derive an
``algebraic'' Koszul duality between the abelian categories $\Mod^{\fg}_0((\calU
\frakg)^{\lambda})$ and $\Mod^{\fg}_{\lambda}((\calU \frakg)_0)$.

\subsection{}

Finally we consider ``parabolic analogues'' of our geometric
duality, where $\calB$ is replaced by a partial flag variety
$\calP$. We prove a version of our restricted localization theorem for
singular weights (Theorem
\ref{thm:localizationfixedFrparabolic}). Then we derive from our key-point
(see \S \ref{ss:introkeypoint}) a version of it for this ``parabolic''
duality, and we deduce Koszulity of singular blocks of $(\calU
\frakg)_0$ (Theorem \ref{thm:thmkoszulsingular}). In this case the
Koszul dual ring is related to a quotient of $\calU \frakg$ introduced
in \cite[{\S}1.10]{BMR2}. In particular it follows that, for $p \gg 0$, $(\calU \frakg)_0$ can be endowed with a (unique) Koszul grading, i.e.~a grading which makes it a Koszul ring (Corollary \ref{cor:KoszulUg0}). This fact was conjectured (for $p>h$) by Soergel in \cite{SOEICM}.

\subsection{}

Our results also give information on the complexes of coherent sheaves corresponding, under equivalences \eqref{eq:equivintro1}, to simple and projective $\calU \frakg$-modules. (The question of computing these objects was raised in \cite[1.5.1]{BMR2}.) Namely, the objects corresponding to indecomposable projectives and to simples are related by the simple geometric construction of \S \ref{ss:introlkd}. We also provide a way to
``generate'' these objects: namely, to compute them it suffices to apply explicit functors to explicit sheaves, and to take direct factors. In practice
these computations are very difficult, however.

\subsection{Organization of the paper} In section \ref{sec:sectiondgalg} we develop the necessary background on derived categories of sheaves of dg-modules over sheaves of dg-algebras, extending results of \cite{BLEqu, SPARes}. We also introduce some notions related to
dg-schemes in the sense of \cite{CK}. In section \ref{sec:sectionKoszulduality} we give a geometric
version of Koszul duality, and study how this duality behaves under proper flat base change, and with respect to sub-bundles. In section \ref{sec:sectionRT} we prove a localization theorem for
restricted $\calU \frakg$-modules, completing
\cite{BMR, BMR2}.

In section \ref{sec:statement} we state a version of our key-point. Sections \ref{sec:sectionbraidgpaction} to \ref{sec:sectionmainthm} are devoted to the proof of this theorem. In section \ref{sec:sectionbraidgpaction} we introduce useful tools
for our study, in particular braid group actions, using the main
result of \cite{RAct}. In section \ref{sec:sectionreflectionfunctors} we study the projective
$(\calU \frakg)_0^{\hat{\lambda}}$-modules and their geometric counterparts,
and their behaviour under reflection functors. In section \ref{sec:sectionsimplemodules} we study the simple
restricted $(\calU \frakg)^{\lambda}$-modules and their geometric
counterparts, and their behaviour under the ``semi-simple'' functors
$\frakS_{\delta}$. In section \ref{sec:sectionmainthm} we finally prove that
the ``geometric'' Koszul duality exchanges the indecomposable projective
$(\calU \frakg)_0^{\hat{\lambda}}$-modules and the simple restricted $(\calU
\frakg)^{\lambda}$-modules.

In section \ref{sec:sectionapplications} we prove that there is an ``algebraic'' Koszul duality relating
$(\calU \frakg)^{\hat{\lambda}}_0$-modules and $(\calU
\frakg)^{\lambda}$-modules with generalized trivial Frobenius
character. Finally, in section \ref{sec:sectionparabolicanalogues} we extend some of our
results to singular characters. In particular we prove Koszulity of
singular blocks of $(\calU \frakg)_0$.

\subsection{Acknowledgments}

The author deeply thanks R. Bezrukav\-nikov for suggesting lots of
these results to him, and for his useful help, remarks and support. He
also thanks P. Polo for his very careful reading of several
earlier versions of this paper and for his encouragement, I. Mirkovi{\'c} for allowing him to use his ideas on Koszul duality, and J.~C. Jantzen for pointing out an inaccuracy in section \ref{sec:sectionparabolicanalogues}.

This work is part of the author's PhD thesis at Paris VI University,
under the joint supervision of R. Bezrukavnikov and P. Polo. Part of
it was done while the author was a Visiting Student at the
M.I.T., supported by the {\'E}.N.S. Paris. He thanks both institutions for their support and
hospitality.

\section{Sheaves of dg-algebras and dg-modules} \label{sec:sectiondgalg}

In this section we extend classical results on dg-algebras and ringed spaces
(see \cite{BLEqu,SPARes}) to the case of a sheaf of dg-algebras on a
ringed space\footnote{In this section we have tried to use only elementary methods, and to construct derived functors ``concretely'', using resolutions. Using more ``modern'' tools of homological algebra, it would certainly be possible to prove similar results under weaker assumptions.}.

We fix a commutative ringed space
$(X,\calO_{X})$, and write simply $\otimes$ for $\otimes_{\calO_{X}}$.

\subsection{Definitions}\label{ss:paragraphdefinitions}

Let $\calA=\bigoplus_{p \in \mathbb{Z}} \calA^p$ be a sheaf of $\mathbb{Z}$-graded $\calO_X$-algebras on $X$, and denote by $\mu_{\calA} : \calA \otimes \calA \to \calA$ the multiplication map.

\begin{defin}

$\calA$ is a \emph{sheaf of dg-algebras} if it is provided with an
endomorphism of
$\calO_{X}$-modules $d_{\calA} : \calA \to \calA$, of degree $1$,
such that $d_{\calA} \circ d_{\calA}=0$, and satisfying the
following formula on $\calA^{p} \otimes \calA$, for any $p \in
\mathbb{Z}$: \[d_{\calA} \circ \mu_{\calA} = \mu_{\calA} \circ
(d_{\calA} \otimes \Id_{\calA}) +
(-1)^{p} \mu_{\calA} \circ (\Id_{\calA^p} \otimes d_{\calA}).\] A \emph{morphism of sheaves of dg-algebras} is a
morphism of sheaves of graded algebras commuting with the
differentials.

A \emph{sheaf of dg-modules} over $\calA$ (in short: $\calA$-dg-module) is a
sheaf of graded left $\calA$-modules $\calF$ on $X$, provided with an
endomorphism of $\calO_{X}$-modules $d_{\calF} : \calF \to \calF$,
of degree $1$, such that $d_{\calF} \circ d_{\calF}=0$, and
satisfying the following formula on $\calA^{p} \otimes \calF$ for
$p \in \mathbb{Z}$, where $\alpha_{\calF} : \calA \otimes \calF
\to \calF$ is the action map:
$$d_{\calF} \circ \alpha_{\calF} = \alpha_{\calF} \circ (d_{\calA}
\otimes \Id_{\calF}) +
(-1)^{p} \alpha_{\calF} \circ (\Id_{\calA^p} \otimes d_{\calF}).$$ A
\emph{morphism of sheaves of dg-modules} is a
morphism of sheaves of graded $\calA$-modules commuting with the
differentials.

\end{defin}

We will always consider $\calO_{X}$ as a sheaf of dg-algebras
concentrated in degree $0$, provided with the zero differential.
In the rest of this section we fix a sheaf of dg-algebras
$\calA$.

We denote by $\calC(X,\calA)$ (or sometimes $\calC(\calA)$) the
category of sheaves of dg-modules
over $\calA$. There is a natural \emph{translation functor} $[1] : \calC(X,\calA) \to
\calC(X,\calA)$, defined as in \cite[10.3]{BLEqu}. One defines as usual the
homotopy category $\calH(X,\calA)$ (see \cite[10.3.1]{BLEqu}). It is triangulated (\cite[10.3.5]{BLEqu}).

If $\calF$ is an object of $\calC(X,\calA)$ or
$\calH(X,\calA)$, we define its cohomology to be
the graded sheaf of $\calO_X$-modules
$H(\calF)=\Ker(d_{\calF})/\Ima(d_{\calF})$. A dg-module $\calF$ is
said to be \emph{acyclic} if $H(\calF)=0$. A morphism $\phi :
\calF \to \calG$ in $\calC(X,\calA)$ or $\calH(X,\calA)$ is said to be
a \emph{quasi-isomorphism} if it induces an isomorphism $H(\phi) :
H(\calF) \xrightarrow{\sim} H(\calG)$. Finally, one defines the
derived category $\calD(X,\calA)$, which is also triangulated, as in \cite[10.4.1]{BLEqu}.

One defines similarly the category $\calC^{\mathrm{r}}(X,\calA)$ of sheaves of
right $\calA$-dg-mo\-dules, its homotopy category $\calH^{\mathrm{r}}(X,\calA)$
and its derived category $\calD^{\mathrm{r}}(X,\calA)$. One defines the
opposite sheaf of dg-algebras $\calA^{\op}$ as in \cite[10.6.2]{BLEqu}. There is a natural equivalence of
categories $\calC^{\mathrm{r}}(X,\calA)
  \xrightarrow{\sim} \calC(X,\calA^{\op})$, defined as in \cite[10.6.3]{BLEqu}. A sheaf of dg-algebras $\calA$ is said to be
\emph{graded-commutative} if the identity map $\Id : \calA \to
\calA^{\op}$ is an isomorphism of sheaves of dg-algebras. In this case
we have an equivalence $\calC(X,\calA) \cong \calC^{\mathrm{r}}(X,\calA)$.

\subsection{Hom, Tens and (co)induction}\label{ss:paragraphHomTens}

Let $\calF$ and $\calG$ be objects of $\calC(X,\calA)$. We define
the sheaf of $\calO_{X}$-dg-modules $\sheafHom_{\calA}(\calF,\calG)$ having, as degree $p$ component,
the $\calO_{X}$-module of local homomorphisms of graded
$\calA$-modules $\calF \to \calG[p]$ (not necessarily commuting
with the differentials), and provided with the differential given
by \begin{equation} \label{eq:diffHom}d(\phi)=d_{\calG} \circ
  \phi - (-1)^{p}\phi \circ d_{\calF}\end{equation} if $\phi \in
(\sheafHom_{\calA}(\calF,\calG))^{p}$. This construction defines a
bifunctor $$\sheafHom_{\calA}(-,-) : \calC(X,\calA)^{\op} \times
\calC(X,\calA) \to \calC(X,\calO_{X}).$$ One checks that
$\sheafHom_{\calA}(-,-)$
preserves homotopy, hence defines a bifunctor between homotopy categories. If $\calA$ is graded-commutative,
this construction even defines a bifunctor with values in $\calH(X,\calA)$.

We define the functor $\Hom_{\calA}(-,-)$, from
$\calC(X,\calA)^{\op} \times \calC(X,\calA)$ to the
complexes of abelian groups, by $(\Hom_{\calA}(\calF,\calG))^i
:= \Gamma(X, (\sheafHom_{\calA}(\calF,\calG))^i),$ the differential
being that of \eqref{eq:diffHom}. Then
$\Hom_{\calC(X,\calA)}(\calF,\calG)$ is the kernel of the
differential $d^{0}$, and
$\Hom_{\calH(X,\calA)}(\calF,\calG) \cong
H^{0}(\Hom_{\calA}(\calF,\calG))$.

Let $\calF$ be in $\calC^{\mathrm{r}}(X,\calA)$, and $\calG$ in $\calC(X,\calA)$. We define the sheaf of
$\calO_{X}$-dg-modules $\calF \otimes_{\calA} \calG$, graded in
the natural way, and having the differential given on local sections
of $\calF^{p} \otimes_{\calA} \calG$ by
\begin{equation*} d(f
\otimes g)=d(f) \otimes g + (-1)^p f \otimes d(g). \end{equation*} This
construction defines a bifunctor \[ (-
\otimes_{\calA} -) : \calC^{\mathrm{r}}(X,\calA)
\times \calC(X,\calA) \to \calC(X,\calO_{X}).\] One checks that
$(- \otimes_{\calA} -)$ preserves homotopy, hence
defines a bifunctor between homotopy categories. As usual the tensor product is
associative.\medskip

Let us define the induction functor in the usual way: \[ \Ind : \left\{
\begin{array}{ccc}
  \calC(X,\calO_{X}) & \to & \calC(X,\calA) \\
  \calF & \mapsto & \calA \otimes_{\calO_{X}} \calF. \\
\end{array} \right. \] This functor is a left adjoint to the
forgetful functor $\For: \calC(X,\calA) \to \calC(X,\calO_{X})$. More
precisely, for $\calF$ in $\calC(X,\calO_X)$ and $\calG$ in $\calC(X,\calA)$, we have a functorial isomorphism
\begin{equation} \label{eq:adjunctionInd}
  \sheafHom_{\calA}(\Ind(\calF),\calG) \ \cong \
  \sheafHom_{\calO_{X}}(\calF, \For(\calG)). \end{equation} The functor $\Ind$ also induces a
functor $\calH(X,\calO_X) \to \calH(X,\calA)$, which is left adjoint
to the forgetful functor.

Now we define the coinduction functor
\[ \Coind : \left\{ \begin{array}{ccc}
  \calC(X,\calO_{X}) & \to & \calC(X,\calA) \\
  \calG & \mapsto & \sheafHom_{\calO_{X}}(\calA,\calG) \\
\end{array} \right.\] (and similarly for the homotopy categories)
where the grading and differential are defined as in \eqref{eq:diffHom}, and
the action of $\calA$ is given by
\[ (\alpha \cdot \phi)(\gamma) = (-1)^{\deg(\alpha) \deg(\phi) +
\deg(\alpha) \deg(\gamma)}\phi(\gamma \alpha). \] One easily checks that the functor $\Coind$ is a right adjoint to
the forgetful functor $\calC(X,\calA) \to \calC(X,\calO_{X})$. More precisely, for $\calF$ in $\calC(X,\calA)$ and $\calG$ in $\calC(X,\calO_{X})$ there is a functorial isomorphism \[ \sheafHom_{\calO_{X}}(\calF,\calG) \ \cong \ \sheafHom_{\calA}(\calF,
\Coind(\calG)).\]

For later use, let us remark that the adjunction morphism $\Ind(\calF) \to \calF$, resp. $\calF \to \Coind(\calF)$, is surjective, resp. injective, for $\calF \in \calC(X,\calA)$.

\subsection{Existence of K-flat and K-injective resolutions}
\label{ss:existenceresolutions}

To ensure the existence of the usual derived functors, we have to show that there are enough ``nice'' objects in $\calC(\calA)$. We follow Spaltenstein's approach (\cite{SPARes}).

\begin{defin}\label{def:defKinjective}

Let $\calF$ be an object of $\calC(\calA)$. 

(a) $\calF$ is said to be \emph{K-injective} if the following equivalent\footnote{See \cite[10.12.2.2]{BLEqu} for details on the equivalence of these conditions.} properties hold:

$\quad$ $\rmi$ For every $\calG \in \calC(\calA)$,
$\Hom_{\calH(\calA)}(\calG,\calF)=\Hom_{\calD(\calA)}(\calG,\calF)$;

$\quad$ $\rmii$ For every $\calG \in \calC(\calA)$ such that
$H(\calG)=0$, $H(\Hom_{\calA}(\calG,\calF))=0$.

(b) $\calF$ is said to be \emph{K-flat} if for
every $\calG \in \calC^{\mathrm{r}}(\calA)$ such that $H(\calG)=0$, $H(\calG \otimes_{\calA} \calF)=0$.

\end{defin}

Easy applications of the basic properties of induction and coinduction
functors give the following lemma:

\begin{lem}\label{lem:coinductionKinjective}

$\rmi$ If $\calF$ is a K-flat $\calO_X$-dg-module, then $\Ind(\calF)$ is a
K-flat $\calA$-dg-module. If $\calG$ is a K-injective
$\calO_X$-dg-module, then $\Coind(\calG)$ is a K-injective $\calA$-dg-module.

$\rmii$ Assume $\calA$ is K-flat as an $\calO_X$-dg-module. Then every
K-injective $\calA$-dg-module is also K-injective as an
$\calO_X$-dg-module. Similarly, every K-flat $\calA$-dg-module is also
K-flat as an $\calO_X$-dg-module.

\end{lem}

Now we prove that
there exist enough K-flat modules in
$\calC(X,\calA)$. The case $\calA=\calO_X$ is treated
in \cite{SPARes}, and will be the base of our
proofs. If $\calM$ is a complex of
sheaves, we denote by $Z(\calM)$ the graded sheaf $\Ker(d_{\calM})$.

\begin{thm}\label{thm:Kflatresolution}

For every sheaf of $\calA$-dg-modules $\calF$, there exists a
K-flat sheaf of $\calA$-dg-modules $\calP$ and a quasi-isomorphism
$\calP \xrightarrow{\text{qis}} \calF$.

\end{thm}

\begin{proof} First, let us consider $\calF$ as an
$\calO_X$-dg-module. By \cite[5.6]{SPARes}, there exists a K-flat
$\calO_{X}$-dg-module
$\calQ_{0}$ and a surjective $\calO_X$-quasi-isomorphism
$\calQ_{0} \twoheadrightarrow \calF$. Thus there
exists a surjective morphism of $\calA$-dg-modules
$$\calP_{0}:=\Ind(\calQ_{0}) \twoheadrightarrow \Ind(\calF)
\twoheadrightarrow \calF,$$
and the $\calA$-dg-module $\calP_{0}$ is K-flat, by Lemma
\ref{lem:coinductionKinjective}$\rmi$. One can check that the induced morphism
$Z(\calP_0) \to Z(\calF)$ is also surjective.

Doing the same construction for the kernel of the morphism $\calP_0
\to \calF$, and repeating, we obtain an exact sequence of
$\calA$-dg-modules $$\cdots \to \calP_{1} \to \calP_{0} \to \calF
\to 0$$ where each $\calP_{p}$ is K-flat, and such that the
induced sequence $$ \cdots \to Z(\calP_{1}) \to Z(\calP_{0}) \to Z(\calF)
\to 0$$ is also exact. Now we take
the $\calA$-dg-module $\calP:={\rm Tot}^{\oplus}(\cdots \to \calP_{1} \to
\calP_{0} \to 0 \to \cdots)$, where $\calP_p$ is in horizontal degree
$-p$. It is K-flat, as the direct limit of the K-flat
$\calA$-dg-modules $\calP_{\leq p}:={\rm Tot}^{\oplus}(\cdots \to 0
\to \calP_p \to \cdots \to \calP_{0} \to 0 \to \cdots)$ (see
\cite[5.4.(c)]{SPARes}). Now we prove that the natural morphism $\calP
\to \calF$ is a quasi-isomorphism, i.e.~that the complex $\calX:={\rm
  Tot}^{\oplus}(\cdots \to \calP_{1} \to \calP_{0} \to \calF \to 0 \to
\cdots)$, where $\calF$ is in horizontal degree $1$, is acyclic. 

The argument for this is borrowed from \cite[3.3]{KELDer},
\cite{KELDerCor}. We put $\calP_{-1}:=\calF$, and $\calP_p=0$ if $p <
-1$. Consider, for $m \geq 1$, the double complex of
$\calO_X$-modules $\calX_{m}$ defined by $(\calX_{m})^{i,j}=0$
if $j \notin [-m, m]$, $(\calX_{m})^{i,j}=(\calP_{-i})^j$ if $j \in
[-m, m-1]$, and $(\calX_{m})^{i,m}=Z(\calP_{-i})^m$. Then $\calX$ is
the direct limit of the complexes ${\rm Tot}^{\oplus}(\calX_{m})$,
which are acyclic because they admit a finite filtration with acyclic
subquotients. Hence $\calX$ is acyclic. \end{proof}

We will also need the following result, which is borrowed from
\cite[5.7]{SPARes}:

\begin{lem}\label{lem:Kflatsplit}

If $\calP$ in $\calC(\calA)$ is K-flat and acyclic, then for any
$\calF$ in $\calC^{\mathrm{r}}(\calA)$ the $\calO_{X}$-dg-module $\calF
\otimes_{\calA} \calP$ is acyclic.

\end{lem}

{F}rom now on in this section
we make the following assumptions: $$\begin{array}{cc} (\dag) &
  \text{All our topological spaces are noetherian of finite
  dimension}. \\ (\dag \dag) & \text{All our dg-algebras are
  concentrated in non-positive degrees}. \end{array}$$ These
assumptions are needed for our proofs and sufficient
for our applications, but we hope they are not essential. In order to
construct resolutions by K-injective $\calA$-dg-modules, we begin with
bounded below dg-modules.

\begin{lem}\label{lem:lemKinjective}

For $\calF$ a bounded-below $\calA$-dg-module, there exists a quasi-isomorphism of
$\calA$-dg-modules $\calF \xrightarrow{\qis} \calI$, where $\calI$ is
K-injective, bounded below with the same
bound as $\calF$ and such that $\calI^p$ is a flabby sheaf for
$p \in \mathbb{Z}$.

\end{lem}

\begin{proof} First, there exists
a bounded below $\calO_{X}$-dg-module $\calJ_{0}$ (with the same
bound as $\calF$), all of whose components are injective $\calO_{X}$-modules,
and an injective morphism $\phi : \calF \hookrightarrow
\calJ_{0}$. Then $\calJ_{0}$ is a K-injective
$\calO_{X}$-dg-module by \cite[1.2, 2.2.(c), 2.5]{SPARes}.
By Lemma \ref{lem:coinductionKinjective}$\rmi$, $\calI_{0}:=\Coind(\calJ_{0})$ is a
K-injective $\calA$-dg-module, and one obtains an injective
morphism of $\calA$-dg-modules $$\calF \hookrightarrow \Coind(\calF)
\hookrightarrow \calI_{0}.$$ This module is bounded below, again
with the same bound, and its graded components
are flabby (by \cite[II.2.4.6.(vii)]{KSShe}, and the fact that a
product of flabby sheaves is flabby). Repeating the same construction for
the cokernel, and then again and again, we obtain an exact sequence of
$\calA$-dg-modules (bounded below with the same bound for all the modules)
\[0 \to \calF \to \calI_{0} \to \calI_{1} \to \calI_{2} \to
\cdots\] where each $\calI_{p}$ is K-injective and has flabby components.

Consider the $\calA$-dg-module
$\calI:=\text{Tot}^{\oplus}(\cdots \to 0 \to \calI_0 \to \calI_1 \to \cdots)$. This
module is the inverse limit of the $\calA$-dg-modules
$\calK_p:=\text{Tot}^{\oplus}( \cdots \to 0 \to \calI_{0} \to \cdots
\to \calI_{p}
\to 0 \to \cdots)$ for $p \geq 0$. Each
$\calK_p$ is K-injective (because it has a finite
filtration with K-injective subquotients). Moreover, the
morphisms $\calK_{p+1} \to \calK_p$ are surjective, and split as
morphisms of graded $\calA$-modules. Hence this inverse system is
``special'' in the sense of \cite[2.1]{SPARes}. We deduce that $\calI$ is
K-injective (\cite[2.3, 2.4]{SPARes}). This module has
flabby components, and one can checks that the morphism $\calF \to \calI$ is a
quasi-isomorphism. \end{proof}

Now we can treat the general case. Recall the
definition of the truncation functors $\tau_{\geq n}$ (\cite[(1.3.11)]{KSShe}). Because of our assumption $(\dag \dag)$, this definition still
makes sense (and has the usual properties) for $\calA$-dg-modules.

\begin{thm}\label{thm:Kinjresolution}

For every $\calA$-dg-module $\calF$, there exists a quasi-isomor\-phism of
$\calA$-dg-modules $\calF \xrightarrow{\text{qis}}
\calI$ where $\calI$ is K-injective.

\end{thm}

\begin{proof} Using the preceding lemma, the construction of
\cite[3.7]{SPARes} generalizes: there
exists an inverse system of morphisms $f_n :
\tau_{\geq -n}\calF \xrightarrow{\qis} \calI_{n}$, where $f_n$ is a
quasi-isomorphism, $\calI_n$ is a K-injective $\calA$-dg-module with
$\calI_n^p=0$ for $p < -n$ and $\calI_n^p$ flabby for $p \geq -n$,
and, furthermore, the morphisms $\calI_{n+1} \to \calI_n$ are
surjective and split as morphisms of graded $\calA$-modules. Then, as
in the previous lemma, $\varprojlim \calI_n$ is
K-injective. It remains only to prove that
$f:=\varprojlim f_n$ is a quasi-isomorphism. For this we can follow
\cite[3.13]{SPARes}. Indeed, using Grothendieck's
vanishing theorem (\cite[III.2.7]{HARAG}), condition $3.12.(1)$ of
\cite{SPARes} is satisfied with $\mathfrak{B}=\Mod(\calO_X)$, and
$d_x={\rm dim}(X)$ for any $x \in X$. Moreover, in the proof of
\cite[3.13]{SPARes}, the fact that the $\calI_n$ are K-injective over
$\calO_X$ is not really needed. In fact, we only need to know that,
for every $n$, the kernel $\calK_n$ of the morphism $\calI_n \to
\calI_{n-1}$ is a resolution of $\calH^{-n}(\calF)[n]$ which is
acyclic for the functors $\Gamma(U,-)$ for every open $U \subset
X$. In our case, $\calK_n$ is a flabby resolution of
$\calH^{-n}(\calF)[n]$. Hence $f$ is a
quasi-isomorphism. \end{proof}

\subsection{Derived functors}\label{ss:paragraphderivedfunctors}

In this subsection we construct the derived functors (see \cite[1.2]{DELCoh}, \cite[sections
13-15]{KELUse}) of
$\Hom_{\calA}(-,-)$ and $(- \otimes_{\calA} -)$.

Let $(X,\calO_X)$ and $(Y,\calO_Y)$ be commutative ringed spaces, and let $\calA$, resp. $\calB$, be a dg-algebra on $X$, resp. $Y$. Let $F: \calH(\calA) \to
\calH(\calB)$ be a triangulated functor. Following Deligne, one says that the right derived functor $RF$ is defined at $\calF \in \calH(\calA)$ if $\calF$ has a right
resolution $\calX$ which is $F$-split on the right, i.e.~every right
resolution $\calY$ of $\calX$ has itself a right resolution $\calZ$
such that $F$ induces a quasi-isomorphism between $F(\calX)$ and
$F(\calZ)$. Similarly, left derived functors are
defined at objects which are $F$-split on the left. Remark that a K-injective $\calA$-dg-module
is $F$-split on the right for any such functor (see
Definition \ref{def:defKinjective}$\rmi$). Hence, under
assumptions $(\dag)$, $(\dag \dag)$, right derived functors are
defined on the whole category $\calD(\calA)$, by Theorem
\ref{thm:Kinjresolution}.

Let ${\rm Ab}$ be the category of abelian groups, $\calH({\rm
  Ab})$ its homotopy category,
and $\calD({\rm Ab})$ its derived category. Consider the
bifunctor $\Hom_{\calA}(-,-) : \calH(\calA)^{\op} \times \calH(\calA)
\to \calH({\rm Ab}).$ Fix $\calF \in \calH(\calA)^{\op}$. Then we define the
functor $R\Hom_{\calA}(\calF,-) : \calD(\calA) \to
\calD({\rm Ab})$ as the right derived functor of
$\Hom_{\calA}(\calF,-)$. Now for each $\calG \in
\calD(\calA)$, the functor $R\Hom_{\calA}(-,\calG) :
\calH(\calA)^{\op} \to \calD({\rm Ab})$ sends quasi-isomorphisms to
isomorphisms, hence factorizes through a functor
$\calD(\calA)^{\op} \to \calD({\rm Ab})$, denoted
similarly. This defines the bifunctor
\[R\Hom_{\calA}(-,-) : \calD(\calA)^{\op} \times \calD(\calA)
\to \calD({\rm Ab}).\]

Now consider the bifunctor $(- \otimes_{\calA} -) :
\calH^{r}(\calA) \times \calH(\calA) \to \calH(\calO_{X}).$ As
above, for each $\calF$ in $\calH^{r}(\calA)$, by Theorem
\ref{thm:Kflatresolution} and Lemma \ref{lem:Kflatsplit} there are enough
objects split on the left (e.g.~K-flat dg-modules) for the functor
$(\calF \otimes_{\calA}
-) : \calH(\calA) \to \calH(\calO_{X})$. Hence, its left derived
functor $(\calF \lotimes_{\calA} -) : \calD(\calA) \to
\calD(\calO_{X})$ is well defined. Hence we have the derived bifunctor
\[(- \lotimes_{\calA} -) : \calD^{r}(\calA) \times \calD(\calA)
\to \calD(\calO_{X}).\]

\subsection{Direct and inverse image
functors}\label{ss:directinverseimage}

Use the same notation as in \S \ref{ss:paragraphderivedfunctors}. A pair such as $(X,\calA)$ is called a \emph{dg-ringed
space}. A \emph{morphism of dg-ringed spaces} $f : (X,\calA) \to
(Y,\calB)$ is a morphism $f_{0} : (X,\calO_{X}) \to (Y,\calO_{Y})$
of ringed spaces, together with a morphism of sheaves of
dg-algebras $f_{0}^{*}\calB \to \calA$.

We have a natural direct image functor $f_{*} : \calC(X,\calA) \to
\calC(Y,\calB)$ and its right derived functor
$$Rf_{*} : \calD(X,\calA) \to \calD(Y,\calB),$$ which can be computed
by means of right K-injective resolutions (see \S \ref{ss:paragraphderivedfunctors}). Similarly, there is a
natural inverse image functor \[f^{*} : \left\{
\begin{array}{ccc}
  \calC(Y,\calB) & \to & \calC(X,\calA) \\
  \calF & \mapsto & \calA \otimes_{f_{0}^{*}\calB} f_{0}^{*}\calF \\
\end{array} \right. .\] Its left derived functor $Lf^{*} : \calD(Y,\calB) \to
\calD(X,\calA)$ is defined on the whole of $\calD(\calA)$, and can be
computed by means
of left K-flat resolutions.

The following definition is adapted from \cite[5.11]{SPARes}:

\begin{defin}

The $\calA$-dg-module $\calF$ is said to be \emph{weakly K-injective}
if $\Hom_{\calA}(\calG,\calF)$ is acyclic for any acyclic K-flat
$\calA$-dg-module $\calG$.

\end{defin}

It is clear that K-injective implies
weakly K-injective. The following lemma is a more general (but weaker)
version of Lemma \ref{lem:coinductionKinjective}$\rmii$.

\begin{lem}\label{lem:directimageweaklyK-injective}

Let $\calF$ be a weakly K-injective $\calA$-dg-module. Then $f_* \calF$
is a weakly K-injective $\calB$-dg-module. In particular, a weakly
K-injective $\calA$-dg-module is also weakly K-injective when
considered as an $\calO_X$-dg-module.

\end{lem}

\begin{proof} The first statement can be proved as in \cite[5.15(b)]{SPARes}. The second one follows, using the natural morphism $(X,\calA) \to (X,\calO_X)$. \end{proof}

Let $\For_X :
\calD(X,\calA) \to \calD(X,\calO_{X})$ and $\For_Y : \calD(Y,\calB) \to
\calD(Y,\calO_{Y})$ denote the forgetful functors. Let
$R(f_{0})_{*} : \calD(X,\calO_{X}) \to \calD(Y,\calO_{Y})$ be the
right derived functor of the morphism of dg-ringed spaces $f_{0}$.

\begin{cor}\label{cor:directimagediagram}

$\rmi$ There exists an isomorphism of functors $\For_Y \circ Rf_* \cong R(f_0)_* \circ \For_X$. In other words, $Rf_*$ is compatible with $R(f_0)_*$.

$\rmii$ If $(Z,\calC)$ is a dg-ringed space and $g : (Y,\calB)
\to (Z,\calC)$ a morphism, the natural
morphism of functors $R(g \circ f)_{*} \to Rg_{*} \circ Rf_{*}$ is
an isomorphism.

\end{cor}

\begin{proof} $\rmi$ The isomorphism follows from Lemma \ref{lem:directimageweaklyK-injective} and
\cite[6.7]{SPARes} (which says that $R(f_0)_*$ can be
computed using a weakly K-injective resolution).

$\rmii$ If $\calF$ is a weakly K-injective, acyclic $\calA$-dg-module, then $\calF$ is also acyclic and weakly K-injective as an
$\calO_X$-dg-module (by Lemma
\ref{lem:directimageweaklyK-injective}). Hence $f_* \calF=(f_0)_* \calF$
is acyclic
(see \cite[5.16]{SPARes}). It follows that weakly K-injective
dg-modules are split for direct image functors. Then the result
follows from classical facts on derived
functors (see \cite[14.2]{KELUse}). \end{proof}

Similarly, one can prove:

\begin{prop}\label{prop:compositioninverseimages}

If $g : (Y,\calB) \to (Z,\calC)$ is a second morphism of dg-ringed
spaces, then there exists an isomorphism of functors $L(g
\circ f)^{*} \cong Lf^{*} \circ Lg^{*}$.

\end{prop}

\begin{defin} 

The morphism $f: (X,\calA) \to (Y,\calB)$ is a
\emph{quasi-isomor\-phism} if $X=Y$, $f_{0}=\Id$,
and the associated morphism $\phi : \calB \to \calA$ induces an isomorphism
on cohomology. 

\end{defin}

The following result is an immediate extension of
\cite[Theorem 10.12.5.1]{BLEqu}, and can be proved similarly. It says that the category $\calD(X,\calA)$
depends on $\calA$ only up to quasi-isomor\-phism.

\begin{prop} \label{prop:qisequivalence}

Let $f : (X,\calA) \to (X,\calB)$ be a quasi-isomorphism. Then
$$Rf_{*} : \calD(X,\calA) \to \calD(X,\calB) \quad \text{and} \quad
Lf^{*} : \calD(X,\calB) \to \calD(X,\calA)$$ are
equivalences of categories, quasi-inverse to each other.

\end{prop}

\subsection{Adjunction}\label{ss:paragraphadjunction}

Let $f : (X,\calA) \to (Y,\calB)$ be a morphism of dg-ringed
spaces. In this subsection we show that $Rf_*$ and $Lf^*$ are adjoint
functors.

Following \cite[5.0]{SPARes}, we denote by $\mathfrak{P}(X)$ the class of
dg-modules $\calF$ in $\calC(X,\calO_X)$ which are bounded above, and
such that for each $i \in \mathbb{Z}$, $\calF^i$ is a direct sum of
sheaves of the form $\calO_{U \subset X}$ (the extension by zero of
$\calO_X |_{U}$ to $X$) for $U$ open in $X$. We denote\footnote{This
  subcategory is a priori
smaller than the one considered in \cite[2.9]{SPARes}, which allows more
general direct limits, but it will be sufficient for us.} by
$\underrightarrow{\mathfrak{P}}(X)$ the smallest full subcategory of
$\calC(X,\calO_X)$ containing $\mathfrak{P}(X)$ and such that for any
direct system $(\calF_n)_{n \geq 0}$ of objects of
$\underrightarrow{\mathfrak{P}}(X)$ such that the morphisms $\calF_n
\to \calF_{n+1}$ are injective and split as morphisms of graded
$\calA$-modules, the object $\varinjlim \calF_n$ is in
$\underrightarrow{\mathfrak{P}}(X)$. The objects in
$\underrightarrow{\mathfrak{P}}(X)$ are K-flat (as in \cite[5.5]{SPARes}).

\begin{lem}\label{lem:resolutionsRHom}

Let $\calF$ be a K-flat $\calA$-dg-module, and $\calG$ a weakly
K-injective, acyclic $\calA$-dg-module. Then $\Hom_{\calA}(\calF,\calG)$ is acyclic.

\end{lem}

\begin{proof} By Lemma \ref{lem:directimageweaklyK-injective}, $\calG$ is
also weakly K-injective as an $\calO_X$-dg-module. Consider the class
$\mathfrak{Q}$ of objects $\calE$ of $\calC(X,\calA)$ such that
$\Hom_{\calA}(\calE,\calG)$ is acyclic. By \cite[5.20]{SPARes} and
\eqref{eq:adjunctionInd}, $\mathfrak{Q}$ contains the class
$\mathfrak{C}$ of objects of the form $\Ind(\calM)$
for $\calM \in \underrightarrow{\mathfrak{P}}(X)$. Now, using the same
proof as that of Theorem \ref{thm:Kflatresolution}, there exists a direct
system $(\calP_{\leq n})_{n \geq 0}$ of $\calA$-dg-modules such that
each $\calP_{\leq n}$ has a finite filtration which subquotients in
$\mathfrak{C}$ and such that the morphisms $\calP_{\leq n} \to
\calP_{\leq n+1}$ are injective and split as morphisms of graded
$\calA$-modules, and a
quasi-isomorphism $\calP:=\varinjlim \calP_{\leq n} \to \calF$. Using again
\cite[2.3, 2.4]{SPARes}, $\calP$ is in
$\mathfrak{Q}$. As $\calG$ is weakly K-injective, and $\calF$ and
$\calP$ are K-flat, the morphism
$\Hom_{\calA}(\calF,\calG) \to \Hom_{\calA}(\calP,\calG)$ is a
quasi-isomorphism. \end{proof}

Using Lemma \ref{lem:resolutionsRHom}, the following result can be proved as in \cite[6.7(c)]{SPARes}. Alternatively, the adjunction statement also follows from \cite[13.6]{KELUse}.

\begin{thm}

The functors $Lf^*$ and $Rf_*$ are adjoint. More precisely, for $\calF \in \calD(Y,\calB)$ and $\calG \in \calD(X,\calA)$ there
exists a functorial isomorphism \[R\Hom_{\calA}(Lf^* \calF,\calG) \cong
R\Hom_{\calB}(\calF, Rf_* \calG).\]

\end{thm}

\subsection{The $\Gm$-equivariant case}\label{ss:gradedcase}

In this subsection we explain how one can adapt the preceding
constructions to the case when $\calA$ is equipped with an additionnal
grading, called the ``internal grading''. More precisely, in
addition to the assumptions of \S \ref{ss:paragraphdefinitions},
we assume we are given a decomposition $\calA \cong \oplus_{n \in
  \mathbb{Z}} \calA_n$ as an $\calO_X$-dg-module such that, for every
$n,m$ in $\mathbb{Z}$, $\mu_{\calA}(\calA_n \otimes
\calA_m) \subset \calA_{n+m}$. We call such a data a
$\Gm$-\emph{equivariant dg-algebra} (in short:
$\Gm$-dg-algebra). Geometrically, if we equip the
space $X$ with a trivial $\Gm$-action, such a grading indeed
corresponds to a $\Gm$-equivariant structure. In what follows,
$\calO_X$ will be considered as a $\Gm$-equivariant dg-algebra
concentrated in degree $0$ for both gradings.

To avoid confusion, the first grading of $\calA$ will be called the
``cohomological grading''. When a homogeneous element of $\calA$ has
cohomological degree $i$ and internal degree $j$, we also say
that it has bidegree $(i,j)$.

We keep the assumptions $(\dag)$ and $(\dag \dag)$ of
\S \ref{ss:existenceresolutions}. We define as above the notion of $\Gm$-\emph{equivariant}
$\calA$-\emph{dg-module} (in short: $\Gm$-$\calA$-dg-module). This is
a sheaf of bigraded $\calA$-modules
$\calF=\bigoplus_{n,m \in \mathbb{Z}} \calF^n_m$, equipped with a
differential $d_{\calF}$ of bidegree $(1,0)$ satisfying the natural
compatibility condition. In a similar way we define morphisms between
dg-modules, and the categories $\calC_{\Gm}(X,\calA)$,
$\calH_{\Gm}(X,\calA)$, $\calD_{\Gm}(X,\calA)$. We also have natural
bifunctors $\Hom_{\calA,\Gm}(-,-)$ and $(- \otimes_{\calA,\Gm}
-)$. More precisely if $\calF$ and $\calG$ are $\Gm$-$\calA$-dg-modules, then $\Hom_{\calA,\Gm}(\calF,\calG)$ is the
complex of $\mathbb{Z}$-graded abelian groups whose $(p,q)$ term
consists of morphisms of $\calA$-modules mapping $\calF^i_j$ in
$\calG^{i+p}_{j+q}$. 

We also define the notions of $\Gm$-\emph{equivariant K-injective}
and $\Gm$-\emph{equiva\-riant K-flat} $\calA$-dg-modules, replacing
the bifunctors $\Hom_{\calA}(-,-)$, $(- \otimes_{\calA} -)$ by $\Hom_{\calA,\Gm}(-,-)$, $(- \otimes_{\calA,\Gm}
-)$. If $\calA=\calO_X$, a $\Gm$-equivariant dg-module is just a
direct sum of $\calO_X$-dg-modules indexed by $\mathbb{Z}$.

The proof of the following lemma is easy, and left to the reader.

\begin{lem} \label{lem:GmO_X}

A $\Gm$-equivariant $\calO_X$-dg-module $\calG$ is $\Gm$-equivariant
K-injective (resp. K-flat) if and only if each of its internal
graded components $\calG_m$ is K-injective (resp. K-flat).

\end{lem}


It follows from this lemma that there are enough K-injective and
K-flat objects in $\calC_{\Gm}(X,\calO_X)$. Then the proofs of
Theorems \ref{thm:Kflatresolution} and \ref{thm:Kinjresolution} generalize,
thus there are enough K-injective and K-flat objects in
$\calC_{\Gm}(X,\calA)$. Then one constructs the
derived bifunctors $R\Hom_{\calA,\Gm}(-,-)$ and $(-
\lotimes_{\calA,\Gm} -)$.

Let $\For : \calC_{\Gm}(X,\calA) \to \calC(X,\calA)$ denote the
natural forgetful functor. The following lemma is clear.

\begin{lem}\label{lem:lemKflatsGm}

For every $\Gm$-$\calA$-dg-module $\calF$, there exists a
$\Gm$-equivariant K-flat $\calA$-dg-module $\calP$ and a
$\Gm$-equivariant quasi-isomor\-phism $\calP \to \calF$ such that the image
$\For(\calP) \to \For(\calF)$ is a K-flat resolution in
$\calC(X,\calA)$.

\end{lem}

By Lemma \ref{lem:lemKflatsGm}, $(-
\lotimes_{\calA,\Gm} -)$ and $(- \lotimes_{\calA} -)$ correspond under
the forgetful functors. Hence we denote both
bifunctors by $(- \lotimes_{\calA} -)$.

Now we consider direct and inverse image functors. Let $f:
(X,\calA) \to (Y,\calB)$ a $\Gm$-equivariant morphism of dg-ringed
spaces. There are functors $(f_{\Gm})_*$, $(f_{\Gm})^*$, and
their derived functors {\small \[ R(f_{\Gm})_* : \calD_{\Gm}(X,\calA)
  \to \calD_{\Gm}(Y,\calB) \ \text{ and } \ L(f_{\Gm})^* :
  \calD_{\Gm}(Y,\calB) \to \calD_{\Gm}(X,\calA).\]}As above, these functors are
adjoint. It follows from Lemma \ref{lem:lemKflatsGm} that the following diagram
is commutative: \[ \xymatrix@R=16pt{ \calD_{\Gm}(Y,\calB)
  \ar[rr]^-{L(f_{\Gm})^*} \ar[d]_-{\For} & & \calD_{\Gm}(X,\calA)
  \ar[d]^{\For} \\ \calD(Y,\calB) \ar[rr]^-{Lf^*} & & \calD(X,\calA). }
\] 

In order to prove the similar result for $R(f_{\Gm})_*$, we need
some preparation. First, consider the case of $\calO_X$. Recall the
notation of \S \ref{ss:paragraphadjunction}.

\begin{defin}

$\calF \in \calC(X,\calO_X)$ is said to be \emph{K-limp} if
$\Hom_{\calO_X}(\calG,\calF)$ is acyclic for every acyclic complex
$\calS$ in $\mathfrak{P}(X)$.

\end{defin}

This notion (defined in \cite[5.11]{SPARes}) is weaker than
weak K-injecti\-vity.

As $X$ is assumed to be noetherian, a direct sum of flabby sheaves on
$X$ is flabby (\cite[III.2.8]{HARAG}). Moreover, for every open $U \subset X$ the functor
$\Gamma(U,-)$ commutes with infinite direct sums (\cite[III.2.9]{HARAG}). Hence
the functor $R\Gamma(U,-)$ commutes with infinite direct sums in the
case of a family of $\calO_X$-dg-modules which are uniformly bounded
below. Let us generalize this fact.

\begin{lem}\label{lem:directsumK-limp}

A direct sum of K-limp $\calO_X$-dg-modules is K-limp.

\end{lem}

\begin{proof} Let $(\calF_j)_{j \in J}$ be K-limp
$\calO_X$-dg-modules. Let $\bigoplus_{j \in J} \calF_j \to \calI$ be a
K-injective resolution, constructed as in \cite[3.7, 3.13]{SPARes}. Using
\cite[5.17]{SPARes}, it will be sufficient to
prove that for every open $U \subset X$, the morphism
$\Gamma(U,\bigoplus_{j \in J} \calF_j)=\bigoplus_{j \in J}
\Gamma(U,\calF_j) \to \Gamma(U,\calI)$ is a quasi-isomorphism. We fix
an open $U$, and $m \in \mathbb{Z}$. We have $\calI \cong
\varprojlim_n \calI_n$ where $\calI_n$ is a K-injective resolution of
$\tau_{\geq -n} (\bigoplus_{j \in J} \calF_j) \cong
\bigoplus_{j \in J} \tau_{\geq -n} \calF_j$. Then for $N$ sufficiently
large, we have an isomorphism $H^m(\Gamma(U,\calI)) \cong
H^m(\Gamma(U,\calI_N))$ (see the proof of \cite[3.13]{SPARes}). But
$H^m(\Gamma(U,\calI_N)) \cong R^m \Gamma(U, \bigoplus_{j \in J}
\tau_{\geq -N} \calF_j)$. Using the remark before the lemma, the
latter is isomorphic to $\bigoplus_{j \in J} R^m \Gamma(U,\tau_{\geq -N}
\calF_j)$. For the same reason, for $N$ sufficiently large
(uniformly in $j$) we have $R^m \Gamma(U,\tau_{\geq -N} \calF_j)
\cong R^m \Gamma(U,\calF_j)$. We conclude using the fact that, as $\calF_j$ is
K-limp, by \cite[6.4]{SPARes} we have $R^m \Gamma(U,\calF_j) \cong
H^m(\Gamma(U,\calF_j))$. \end{proof}

Let $f:(X,\calO_X) \to (Y,\calO_Y)$ be a morphism of ringed spaces, also considered as a morphism of $\Gm$-equivariant dg-ringed
spaces.

\begin{cor}\label{cor:cordirectimageGm}

For every family of objects $(\calF_i)_{i \in I}$ of
$\calC(X,\calO_X)$ we have $Rf_*( \bigoplus_{i \in I} \calF_i) \cong
\bigoplus_{i \in I} Rf_*
(\calF_i)$. Hence the following diagram commutes: \[
\xymatrix@R=16pt{ \calD_{\Gm}(X,\calO_X) \ar[rr]^-{R(f_{\Gm})_*}
  \ar[d]_-{\For} & & \calD_{\Gm}(Y,\calO_Y) \ar[d]^{\For} \\
  \calD(X,\calO_X) \ar[rr]^-{Rf_*} & & \calD(Y,\calO_Y). } \]

\end{cor}

\begin{proof} The isomorphism follows from the facts that $f_*$
commutes with direct sums, that $Rf_*$ can
be computed by means of K-limp resolutions (\cite[6.7]{SPARes}), and Lemma
\ref{lem:directsumK-limp}. Then the commutativity of the diagram follows from the obvious isomorphism $\For \circ R(f_{\Gm})_* (\calF) \cong
\bigoplus_{n \in \mathbb{Z}} Rf_* (\calF_n)$.\end{proof}

Let $f: (X,\calA) \to (Y,\calB)$ be a morphism of $\Gm$-dg-ringed spaces.

\begin{cor}\label{cor:cordiagramRf_*Gm}

The following diagrams are commutative: \[
\xymatrix@R=16pt@C=35pt{\calD_{\Gm}(X,\calA) \ar[r]^{R(f_{\Gm})_*} \ar[d]_{\For}
 & \calD_{\Gm}(Y,\calB) \ar[d]^{\For} & \calD_{\Gm}(X,\calA) \ar[r]^{R(f_{\Gm})_*} \ar[d]_{\For} & \calD_{\Gm}(Y,\calB) \ar[d]^{\For} \\ \calD_{\Gm}(X,\calO_X)
  \ar[r]^{R(f_{0,\Gm})_*} & \calD_{\Gm}(Y,\calO_Y), & \calD(X,\calA)
  \ar[r]^{Rf_*} & \calD(Y,\calB). } \]

\end{cor}

\begin{proof} The commutativity of the second diagram follows from
the commutativity of the first one and Corollaries
\ref{cor:directimagediagram} and \ref{cor:cordirectimageGm}. Now consider a $\Gm$-equivariant
K-injective $\calA$-dg-module $\calF$. By an analogue of Lemma
\ref{lem:directimageweaklyK-injective}, $\calF$ is weakly K-injective as a
$\Gm$-$\calO_X$-dg-module. Hence its graded
components are weakly K-injective as $\calO_X$-dg-modules. The result follows, since one can compute
$R(f_{0,\Gm})_*$ using K-limp resolutions of each components. \end{proof}

Proofs similar to those of \S \ref{ss:directinverseimage} show
that if $g : (Y,\calB) \to (Z,\calC)$
is a second morphism of $\Gm$-equivariant dg-algebras, one has
isomorphisms {\small \begin{equation} \label{eq:compositiondirectimageGm} R((g
  \circ f)_{\Gm})_* \cong
  R(g_{\Gm})_* \circ R(f_{\Gm})_*, \ L((g \circ f)_{\Gm})^* \cong
  L(f_{\Gm})^* \circ L(g_{\Gm})^*. \end{equation}}\begin{remark}\label{rk:rkregrading} One motivation for introducing
$\Gm$-dg-modules comes from the following situation,
that will be encountered in section \ref{sec:sectionKoszulduality}. Consider the
dg-algebra $\calA=S_{\calO_X}(\calF)$, the symmetric algebra of an $\calO_X$-module $\calF$, with trivial differential and the grading such
that $\deg(\calF)=2$. It is not concentrated in
non-positive degrees, hence we cannot apply the constructions of
\S\S \ref{ss:existenceresolutions}--\ref{ss:paragraphadjunction}. Now,
consider $\calA$ as a $\Gm$-dg-algebra, with
$\calF$ in bidegree $(2,-2)$. Let $\calB$ denote the $\Gm$-dg-algebra which is also isomorphic to $S_{\calO_X}(\calF)$ as a sheaf
of algebras, with trivial differential, and with $\calF$ in bidegree
$(0,-2)$. Then the ``regrading'' functor $\xi: \calD_{\Gm}(\calA) \to
\calD_{\Gm}(\calB)$ defined by $\xi (\calM)^i_j:=\calM^{i-j}_j$ is an
equivalence of categories. Using this and the fact that
$\calB$ is concentrated in non-positive degrees, all the constructions
and results obtained in \S \ref{ss:gradedcase} can be transfered to $\calA$. \end{remark}

\subsection{Dg-schemes and dg-sheaves}\label{ss:dgschemes}

In this subsection we define dg-sche\-mes, following \cite{CK} (with some modifications according to our purposes).

\begin{defin}

A \emph{dg-scheme} is a dg-ringed space $X=(X^{0},
\calO_{X}\hdot)$ where $X^{0}$ is
a scheme and $\calO_{X}\hdot$ is a sheaf of
non-positively graded, graded-commuta\-tive dg-algebras on $X^{0}$,
such that each $\calO^{i}_{X}$ is a quasi-coherent
$\calO_{X^{0}}$-module.

A \emph{morphism of dg-schemes} $f : X \to Y$ is a morphism of dg-ringed
spaces $f : (X,\calO_{X}\hdot) \to (Y,\calO_{Y}\hdot)$ (see
\S \ref{ss:directinverseimage}).

We denote by $\DGSh(X)$ the full subcategory\footnote{It is not clear from this definition that $\DGSh(X)$ is a \emph{triangulated} subcategory. In fact it is the case under reasonable conditions. In this paper we essentially consider \emph{coherent} dg-sheaves over bounded dg-algebras, hence this point will not be a problem.} of
$\calD(X^0,\calO_{X}\hdot)$ whose objects are the dg-modules $\calF$ such that each
$H^{i}(\calF)$ is a quasi-coherent $\calO_{X^{0}}$-module, and by $\DGCoh(X)$ the full subcategory of $\DGSh(X)$ whose objects are
the dg-modules $\calF$ whose cohomology $H(\calF)$ is locally finitely generated over $H(\calO_{X}\hdot)$.

\end{defin}

\begin{remark} $\rmi$ If $X$ is an ordinary scheme (i.e.~$\calO_{X}^{0}=\calO_{X^{0}}$ and $\calO_{X}^{i}=0$ for $i \neq
0$) which is quasi-compact and separated,
$\DGSh(X)$ is equivalent to the derived category of $\QCoh(X)$
(\cite[5.5]{BNHom}). If moreover $X$
is noetherian, then $\DGCoh(X)$ is equivalent to the
bounded derived category of $\Coh(X)$ (\cite[II.2.2.2.1]{SGA6}\footnote{See also \cite[VI.2.B]{BDmod}
for a more elementary proof, following Bernstein and
Deligne.}).

$\rmii$ If $f:X \to Y$ is a morphism of dg-schemes, then it induces
functors $Rf_* : \calD(X^0, \calO_X\hdot) \to \calD(Y^0,\calO_Y\hdot)$
and $Lf^* : \calD(Y^0,\calO_Y\hdot) \to \calD(X^0, \calO_X\hdot)$. It
is not clear in general if these functors restrict to functors
between $\DGSh(X)$ and $\DGSh(Y)$, or between $\DGCoh(X)$ and
$\DGCoh(Y)$. It will always be the case in this paper; we
will prove it in each particular case\footnote{See \cite{MRHec} for other remarks on this question.}. \end{remark}

Recall that if $f: X \to Y$ is a quasi-isomorphism, then $Rf_*$ and $Lf^*$ are equivalences (Proposition \ref{prop:qisequivalence}). Moreover, if $g: Y \to Z$ is a morphism of dg-schemes, by Corollary
\ref{cor:directimagediagram} and Proposition \ref{prop:compositioninverseimages} we have isomorphisms $R(g \circ f)_* \cong Rg_* \circ Rf_*$ and $L(g \circ f)^* \cong Lf^* \circ Lg^*$. Hence the functors
$Rg_*$ and $Lg^*$ restrict to functors between
$\DGSh(Y)$ and $\DGSh(Z)$ (or
between $\DGCoh(Y)$ and $\DGCoh(Z)$) iff the functors $R(g \circ
f)_*$ and $L(g \circ f)^*$ restrict to functors between $\DGSh(X)$ and
$\DGSh(Z)$ (or $\DGCoh(X)$ and $\DGCoh(Z)$). These properties allows one to replace a given dg-scheme by a
quasi-isomorphic one when convenient. Hence we
will consider dg-schemes only ``up to quasi-isomorphism''.

As a typical example, we define the derived intersection as follows. Consider a scheme $X$, and two closed subschemes
$Y$ and $Z$. Let $i : Y \hookrightarrow X$ and $j : Z \hookrightarrow X$ be the
closed embeddings. Consider the sheaf of dg-algebras $i_* \calO_Y
\, \lotimes_{\calO_X} \, j_* \calO_Z$ on $X$. It is defined up to
quasi-isomorphism: if $\calA_Y \to i_* \calO_Y$ and $\calA_Z
\to j_* \calO_Z$ are quasi-isomorphisms of non-positively graded,
graded-commutative sheaves of dg-algebras on $X$, with $\calA_Y$ and
$\calA_Z$ quasi-coherent and K-flat over $\calO_X$, then $i_* \calO_Y
\, \lotimes_{\calO_X} \, j_* \calO_Z$ is quasi-isomorphic to $\calA_Y
\otimes_{\calO_X} j_* \calO_Z$, or to $i_* \calO_Y \otimes_{\calO_X}
\calA_Z$, or to $\calA_Y \otimes_{\calO_X} \calA_Z$.

\begin{defin}

The \emph{right derived intersection} of $Y$ and $Z$ in $X$ ``is'' \[ Y \, \rcap_X \, Z := (X, \, i_* \calO_Y \, \lotimes_{\calO_X} \, j_* \calO_Z),\]
a dg-scheme considered up to quasi-isomorphism.

\end{defin}

\begin{remark}\label{rk:rkderivedintersection} Keep the notation
  as above. The sheaf of dg-algebras
$\calA_Y \otimes_{\calO_X} j_*\calO_Z$ is isomorphic to $j_*(j^* \calA_Y)$, hence the functor $j_* :
\calC(Z, j^* \calA_Y) \to \calC(X,\calA_Y \otimes_{\calO_X}
j_*\calO_Z)$ is an equivalence of categories. As a consequence, we have an equivalence $\DGCoh(Z, j^*\calA_Y) \cong \DGCoh(Y \, \rcap_X \, Z)$. \end{remark}

\section{Linear Koszul duality} \label{sec:sectionKoszulduality}

Usual Koszul duality (\cite{BGG, BGS, GKM}) relates
modules over the symmetric algebra $S(V)$ of a
vector space $V$ and modules over the exterior algebra
$\Lambda(V^*)$ of the dual vector space. In this section we give
a relative version of this duality, and a geometric interpretation in
terms of derived intersections (due to I. Mirkovi{\'c}). Our approach is similar to that of \cite{GKM}.

\subsection{Reminder on Koszul duality}\label{ss:koszulduality}

fix a scheme $(X, \calO_X)$. Let $\calF$ be a locally
free sheaf of finite rank over $X$. We denote by
\[\calS:=S_{\calO_X}(\calF^{\vee})\] the symmetric algebra of
$\calF^{\vee}:=\sheafHom_{\calO_X}(\calF, \calO_X)$ over $\calO_X$,
considered as a sheaf of dg-algebras with trivial
differential, and with the grading such that $\calF^{\vee}$ is in
degree $2$. Similarly, we denote by
\[\calT:=\Lambda_{\calO_X}(\calF)\] the exterior algebra of $\calF$, a sheaf of dg-algebras with trivial
differential and the grading such that $\calF$ is in degree
$-1$. For the categories of dg-modules over these dg-algebras, we
use the notation of section \ref{sec:sectiondgalg}. Let $\calC^+(\calS)$ be the category of bounded below
$\calS$-dg-modules. We define similarly $\calC^+(\calT)$, $\calH^+(\calS)$, $\calH^+(\calT)$,
$\calD^+(\calS)$ and $\calD^+(\calT)$ using the usual procedures.

Following \cite{GKM}, we define the functor \[\mathscr{A} :
\calC^+(\calS) \to \calC^+(\calT)\] by setting
$\mathscr{A}(\calM):=\sheafHom_{\calO_X}(\calT,\calM) \cong
\calT^{\vee} \otimes_{\calO_X} \calM$,
where the $\calT$-module structure is given by the formula $(t
\cdot \phi)(s)=(-1)^{\deg(t)(\deg(t)+1)/2} \phi(ts)$ and the
differential is the sum of $d_1$ and $d_2$, where
$d_1(\phi)(t)=(-1)^{\deg(t)}
  d_M(\phi(t))$, and $d_2$ is defined as
follows. Consider the canonical morphism $\calO_X \to
\sheafHom_{\calO_X}(\calF,\calF) \cong \calF \otimes_{\calO_X}
\calF^{\vee}$. Then $d_2$ is the opposite of the composition
$$\calT^{\vee} \otimes_{\calO_X} \calM \to \calT^{\vee}
\otimes_{\calO_X} \calF \otimes_{\calO_X} \calF^{\vee}
\otimes_{\calO_X} \calM \xrightarrow{\beta \otimes \alpha_{\calF}}
\calT^{\vee} \otimes_{\calO_X} \calM$$ where $\alpha_{\calF}$ is the
action $\calF^{\vee} \otimes_{\calO_X} \calM \to \calM$ and
$\beta$ is the (right) action of $\calF$ on $\calT^{\vee}$ which is the
transpose of left multiplication. If $t$ is a local section of $\calT$
in a neighborhood of $x$, with $\{y_i\}$ a basis of $\calF_x$
as $\calO_{X,x}$-module and $\{y_i^*\}$ the dual basis of $(\calF^{\vee})_x$, we have
$d_2(\phi)(t)=- \sum_i y_i^* \phi(y_i
  t)$. One
easily checks that $d_1 + d_2$ is a differential, and that
$\mathscr{A}(\calM)$ is a $\calT$-dg-module.

We also define the functor \[\mathscr{B} : \calC^+(\calT) \to
\calC^+(\calS)\] by setting $\mathscr{B}(\calN):=\calS
\otimes_{\calO_X} \calN$,
where the $\calS$-module structure is by left multiplication and the differential is the sum $d_3 + d_4$, where
$d_3(s \otimes n)=s \otimes d_{\calN}(n)$ and $d_4$ is the composition
$\calS \otimes_{\calO_X} \calN \to \calS \otimes_{\calO_X}
\calF^{\vee} \otimes_{\calO_X} \calF \otimes_{\calO_X} \calN \to
\calS \otimes_{\calO_X} \calN$. With the
same notation as above, we have $d_4(s
  \otimes m)=\sum_i s y_i^* \otimes y_i n$. One again easily checks that $d_3
+ d_4$ is a differential, and that $\mathscr{B}(\calN)$ is a
$\calS$-dg-module.

Taking the stalks at a point and using spectral sequence arguments
(see \cite[9.1]{GKM}), one proves
that $\mathscr{A}$ and $\mathscr{B}$ send quasi-isomorphisms to
quasi-isomorphisms, and hence
define functors $$\mathscr{A} : \calD^+(\calS) \to \calD^+(\calT) \quad
\text{and} \quad \mathscr{B} : \calD^+(\calT) \to \calD^+(\calS).$$

\begin{thm}

The functors $\mathscr{A}$ and $\mathscr{B}$ are equivalences of categories
between $\calD^+(\calS)$ and $\calD^+(\calT)$, quasi-inverse to
each other.

\end{thm}

To prove this theorem, one constructs morphisms of functors $\Id \to
\mathscr{A}
\circ \mathscr{B}$, $\mathscr{B} \circ \mathscr{A} \to \Id$ as in
\cite[section 16]{GKM}. To prove that they are isomorphisms, it
suffices to look at the stalks. Then the proof of \cite{GKM} works similarly.

\subsection{Restriction to certain subcategories}
\label{ss:paragraphrestriction}

Now we assume that $X$ is a non-singular algebraic variety over an
algebraically closed field $\bk$. If $\calA$ is a dg-algebra on $X$,
we denote by $\calD^{\qc}(\calA)$, resp. $\calD^{\qc,\fg}(\calA)$ the full
subcategory of $\calD(\calA)$ consisting of dg-modules whose
cohomology is $\calO_X$-quasi-coherent, resp. whose
cohomology is $\calO_X$-quasi-coherent and locally finitely
generated over the sheaf of algebras $H(\calA)$. Similarly we define
$\calD^{+,\qc}(\calA)$, $\calD^{+,\qc,\fg}(\calA)$, and bigraded
analogues. Let $\calF$, $\calS$, $\calT$ be as in
\S \ref{ss:koszulduality}.

\begin{lem}\label{lem:restrictionKoszul}

The equivalences $\mathscr{A}$ and $\mathscr{B}$ restrict to
equivalences between the categories $\calD^{+,\qc,\fg}(\calS)$ and
$\calD^{+,\qc,\fg}(\calT)$.

\end{lem}

\begin{proof} First, $\mathscr{A}$ and $\mathscr{B}$ restrict to equivalences
$\calD^{+,\qc}(\calS) \cong \calD^{+,\qc}(\calT)$. Indeed, we only have to prove that $\mathscr{A}$ and $\mathscr{B}$ map
these subcategories one into each other; but this is
clear from the existence of the spectral sequences (of sheaves)
analogous to the ones of \cite[9.1]{GKM}. 

Now we have to prove that $\mathscr{A}$ maps
$\calD^{+,\qc,\fg}(\calS)$ into $\calD^{+,\qc,\fg}(\calT)$, and that
$\mathscr{B}$ maps $\calD^{+,\qc,\fg}(\calT)$ into
$\calD^{+,\qc,\fg}(\calS)$. Let us consider
$\mathscr{B}$. Let $\calM$ be an object of
$\calD^{+,\qc,\fg}(\calT)$. Then $\mathscr{B}(\calM)
\in \calD^{+,\qc}(\calS)$. Moreover, for any $x \in
X$, the $\calS_x$-dg-module $\mathscr{B}(\calM)_x$ has finitely
generated cohomology. Indeed, $H(\calM_x)$ is finitely generated
over $\calO_{X,x}$ (because it is finitely generated over $\calT_x$,
which is a finitely generated $\calO_{X,x}$-module). Thus,
the $E_1$-term of the spectral sequence analogous to
\cite[9.1.4]{GKM} is finitely generated over $\calS_x$. The result
follows since $\calS_x$ is a noetherian ring.

Concerning $\mathscr{A}$, again taking stalks, one can use the
arguments of \cite[16.7]{GKM}. (Since $X$ is non-singular, $\calO_{X,x}$ has
finite homological dimension.) \end{proof}

The inclusion $\calC^+(\calT) \subset \calC(\calT)$ induces a
functor $\calD^{+,\qc,\fg}(\calT) \to \calD^{\qc,\fg}(\calT)$. Recall the definition of the functors $\tau_{\geq n}$ given just before Theorem \ref{thm:Kinjresolution}.

\begin{lem}\label{lem:equivalenceT}

The functor $\calD^{+,\qc,\fg}(\calT) \to
\calD^{\qc,\fg}(\calT)$ is an equivalence.

\end{lem}

\begin{proof} We only have to prove that for every
$\calT$-dg-module $\calN$ whose cohomology is locally finitely
generated, there exists a bounded below $\calT$-dg-module $\calN'$ and a
quasi-isomorphism $\calN \xrightarrow{\qis} \calN'$. Now the
cohomology of $\calN$ is bounded. If $H^i(\calN)=0$ for $i < n$, we
may take $\calN'=\tau_{\geq n} \calN$. \end{proof}

\begin{remark} \label{rk:rkequivalenceS} We cannot use such an argument
  for $\calS$, and we do
  not know if the natural functor $\calD^{+,\qc,\fg}(\calS) \to
\calD^{\qc,\fg}(\calS)$ is an equivalence.\end{remark}

Combining Lemmas \ref{lem:restrictionKoszul} and \ref{lem:equivalenceT}, one
gets an equivalence of categories \[ \calD^{+,\qc,\fg}(X,\calS) \cong
  \calD^{\qc,\fg}(X,\calT).\] Now we
give a geometric interpretation of this equivalence.

\subsection{Linear Koszul Duality} \label{ss:lkd}

Let $E$ be a vector bundle
over $X$ (of finite rank), $F \subset E$ a sub-bundle, and $p : E \to X$ the natural projection. Let $\calE$ and
$\calF$ be the sheaves of sections of $E$ and $F$. (These are locally
free $\calO_X$-modules of finite rank.) Let $E^*$ be the vector bundle
dual to $E$, $F^{\bot} \subset E^*$ the orthogonal of $F$
(a sub-bundle of $E^*$), and $q : E^* \to X$ the projection. We define
an action of $\Gm$ on $E$ and
$F$, letting $t \in \bk^{\times}$ act by multiplication by $t^2$ on
the fibers. This induces a dual action on $E^*$ and $F^{\bot}$. We denote by $\calS$ and $\calT$ the
following $\Gm$-dg-algebras with trivial differential:
\[\begin{array}{cl}
\calS:=S_{\calO_X}(\calF^{\vee}) & \quad \text{with} \ \calF^{\vee} \
\text{in bidegree} \ (2,-2) \\ \calT:=\Lambda_{\calO_X}(\calF) & \quad
\text{with} \ \calF \ \text{in bidegree} \ (-1,2). \end{array}\] 

Bigraded analogues of the
constructions of \S\S \ref{ss:koszulduality}, \ref{ss:paragraphrestriction} yield
an equivalence
\begin{equation}\label{eq:KoszuldualityGm}
  \calD_{\Gm}^{+,\qc,\fg}(X,\calS) \cong
  \calD_{\Gm}^{\qc,\fg}(X,\calT),\end{equation} where $\calD_{\Gm}^+
(X,\calS)$ is the derived category of $\Gm$-$\calS$-dg-modules
which are bounded below for the cohomological degree (uniformly in
the internal degree), and $\calD_{\Gm}^{+,\qc,\fg}(X,\calS)$,
$\calD_{\Gm}^{\qc,\fg}(X,\calT)$ are defined as in \S \ref{ss:paragraphrestriction}.

\begin{lem}\label{lem:lemdirectimage}

There exists a natural equivalence of categories \begin{equation}
  \label{eq:equivalenceCoh(E)nonGm}
  \calD^b \Coh(E) \ \cong \ \calD^{\qc,\fg}(X,S_{\calO_X}(\calE^{\vee})),
\end{equation} where $S_{\calO_X}(\calE^{\vee})$ is a dg-algebra in degree $0$, with trivial
differential. Similarly, if $S_{\calO_X}(\calE^{\vee})$ is considered as a $\Gm$-dg-algebra with $\calE^{\vee}$ in bidegree
$(0,-2)$,
\begin{equation}\label{eq:equivalenceCoh(E)} \calD^b \Coh^{\Gm}(E) \ \cong \
  \calD^{\qc,\fg}_{\Gm}(X,S_{\calO_X}(\calE^{\vee})). \end{equation}

\end{lem}

\begin{proof} We only give the proof of
\eqref{eq:equivalenceCoh(E)nonGm}. Let $\QCoh(X,S_{\calO_X}(\calE^{\vee}))$ be the
category of modules
over $S_{\calO_X}(\calE^{\vee})$ which are
$\calO_X$-quasi-coherent, and let
$\Coh(X,S_{\calO_X}(\calE^{\vee}))$ be the full subcategory whose objects are locally finitely
generated over $S_{\calO_X}(\calE^{\vee})$. First, $p_*$ induces equivalences (\cite[1.4.3]{EGA2}): \[ \QCoh(E) \xrightarrow{\sim} 
\QCoh(X,S_{\calO_X}(\calE^{\vee})), \quad \Coh(E)
\xrightarrow{\sim} \Coh(X,S_{\calO_X}(\calE^{\vee})). \] By arguments similar to those of \cite[VI.2.11]{BDmod},
$\calD^b \Coh(X,S_{\calO_X}(\calE^{\vee}))$ identifies with the full
subcategory of $\calD^b \QCoh(X,S_{\calO_X}(\calE^{\vee}))$ whose
objects have their cohomology sheaves in
$\Coh(X,S_{\calO_X}(\calE^{\vee}))$. Now,
by a theorem of Bernstein (\cite[VI.2.10]{BDmod}), $\calD^b
\QCoh(X,S_{\calO_X}(\calE^{\vee}))$ is equivalent to the full subcategory of
$\calD^b \Mod(X,S_{\calO_X}(\calE^{\vee}))$ (the bounded derived
category of \emph{all} $S_{\calO_X}(\calE^{\vee})$-modules) whose
objects have quasi-coherent cohomology. Hence $\calD^b \Coh(E)$ is
equivalent to the full subcategory of $\calD^b
\Mod(X,S_{\calO_X}(\calE^{\vee}))$ whose objects have their cohomology
in $\Coh(X,S_{\calO_X}(\calE^{\vee}))$. Then \eqref{eq:equivalenceCoh(E)nonGm} is clear. \end{proof}

Let us now introduce the
following $\Gm$-dg-algebra with trivial differential:
\[ \calR:=S_{\calO_X}(\calF^{\vee}) \qquad \text{with} \ \calF^{\vee} \
\text{in bidegree} \ (0,-2).\] We have
equivalences of categories (``regrading''): \[\xi :
\calC_{\Gm}(X,\calS) \xrightarrow{\sim} \calC_{\Gm}(X,\calR), \quad \xi :
\calD_{\Gm}(X,\calS) \xrightarrow{\sim} \calD_{\Gm}(X,\calR)\] sending the
$\calS$-dg-module $M$ to the $\calR$-dg-module defined by
$\xi(M)^i_j:=M^{i-j}_j$ (with the same action of
$S_{\calO_X}(\calF^{\vee})$, and the same differential). Composing $\xi$ with the inclusion $\calD_{\Gm}^{+,\qc,\fg}(X,\calS) \hookrightarrow
\calD_{\Gm}^{\qc,\fg}(X,\calS)$ and using \eqref{eq:equivalenceCoh(E)} applied to $F$, we obtain a functor
\begin{equation}\label{eq:GmF} \calD_{\Gm}^{+,\qc,\fg}(X,\calS) \ \to \ \calD^b
  \Coh^{\Gm}(F).\end{equation} Hence we consider
$\calD_{\Gm}^{+,\qc,\fg}(X,\calS)$ as a graded version of $\calD^b \Coh(F)$, and put
\begin{equation}\label{eq:DGCoh1} \DGCoh^{\gr}(F) \ := \
\calD_{\Gm}^{+,\qc,\fg}(X,\calS).\end{equation} Note that there exists
a natural forgetful functor \begin{equation} \label{eq:forgetDGCoh(F)}
  \For: \DGCoh^{\gr}(F) \ \to \ \calD^b \Coh(F), \end{equation} the
composition of \eqref{eq:GmF} with the forgetful functor from $\calD^b
\Coh^{\Gm}(F)$ to $\calD^b \Coh(F)$ or, equivalently, the composition
\[\calD_{\Gm}^{+,\qc,\fg}(X,\calS) \to \calD_{\Gm}^{\qc,\fg}(X,\calS)
\cong \calD_{\Gm}^{\qc,\fg}(X,\calR) \to \calD^{\qc,\fg}(X,\calR)
\cong \calD^b \Coh(F).\]

Consider the dg-scheme $F^{\bot} \, \rcap_{E^*} \, X$. As modules over
$q_*\calO_{E^*} \cong S_{\calO_X}(\calE)$, we have $q_* \calO_{F^{\bot}} \cong S_{\calO_X}(\calE) / (\calF \cdot
S_{\calO_X}(\calE))$. Hence there is a Koszul resolution
\[ S_{\calO_X}(\calE) \otimes_{\calO_X} \Lambda_{\calO_X}(\calF)
\xrightarrow{\qis} S_{\calO_X}(\calE) / (\calF \cdot
S_{\calO_X}(\calE)),\] where the generators of
$\Lambda_{\calO_X}(\calF)$ are in degree $-1$. Using Remark
\ref{rk:rkderivedintersection}, we deduce an equivalence of categories
$\DGCoh(F^{\bot} \, \rcap_{E^*} \, X) \cong
\calD^{\qc,\fg}(X,\calT).$ We are also interested
in the ``graded version'' \begin{equation}\label{eq:DGCoh2} \DGCoh^{\gr}(F^{\bot}
\, \rcap_{E^*} \, X) \ := \ \calD_{\Gm}^{\qc,\fg}(X,\calT).\end{equation} By
definition we have a natural forgetful functor \begin{equation}
  \label{eq:ForDGCohderivedintersection} \For:  \DGCoh^{\gr}(F^{\bot}
\, \rcap_{E^*} \, X) \ \to \ \DGCoh(F^{\bot} \, \rcap_{E^*} \, X). \end{equation}

Finally, with notations \eqref{eq:DGCoh1} and \eqref{eq:DGCoh2}, equivalence \eqref{eq:KoszuldualityGm} becomes:

\begin{thm}\label{thm:thmlkd}

The functors $\scra$ and $\scrb$ of \S \ref{ss:koszulduality} induce an equivalence of categories, called \emph{linear Koszul
  duality}, \[ \DGCoh^{\gr}(F) \ \cong \ \DGCoh^{\gr}(F^{\bot} \, \rcap_{E^*} \, X). \]

\end{thm}

Finally we have constructed the following
  diagram, which allows to relate (dg-)sheaves on $F$ and on $F^{\bot} \, \rcap_{E^*} \, X$: \[ \xymatrix@R=16pt{ \DGCoh^{\gr}(F)
  \ar@{<->}[rr]^-{\sim}_-{\ref{thm:thmlkd}}
  \ar[d]_-{\eqref{eq:forgetDGCoh(F)}}^-{\For} & & \DGCoh^{\gr}(F^{\bot}
  \, \rcap_{E^*} \, X) \ar[d]^-{\eqref{eq:ForDGCohderivedintersection}}_-{\For}
  \\ \calD^b \Coh(F) & & \DGCoh(F^{\bot} \, \rcap_{E^*} \, X). } \]

\subsection{Linear Koszul duality and base change}
\label{ss:koszulbasechange}

Let $X$, $Y$ be non-singular varieties, and $\pi : X \to Y$ a
\emph{flat} and \emph{proper} morphism. Let $E$ be a
vector bundle over $Y$, $F \subset E$ a sub-bundle, and $\calE$, $\calF$ their sheaves of sections. Consider the vector bundles $E_X := E \times_{Y} X$, $F_X := F \times_Y X$
over $X$: their sheaves of sections are $\pi^* \calE$
and $\pi^* \calF$ (\cite[1.7.11]{EGA2}). Let $\pi_F : F_X
\to F$ be the morphism induced by $\pi$. We consider the following
$\Gm$-dg-algebras with trivial differential: \[\begin{array}{ccc}
  \calS_Y:=S_{\calO_Y}(\calF^{\vee}), & \calS_X:=S_{\calO_X}(\pi^*
  \calF^{\vee}), & \text{with} \ \calF^{\vee} \ \text{in bidegree} \
  (2,-2); \\ \calR_Y:=S_{\calO_Y}(\calF^{\vee}), &
  \calR_X:=S_{\calO_X}(\pi^* \calF^{\vee}), & \text{with} \
  \calF^{\vee} \ \text{in bidegree} \ (0,-2); \\
  \calT_Y:=\Lambda_{\calO_Y}(\calF), &
  \calT_X:=\Lambda_{\calO_X}(\pi^* \calF), & \text{with} \ \calF \
  \text{in bidegree} \ (-1,2).\end{array}\]

In this situation we have two Koszul dualities (see Theorem \ref{thm:thmlkd}): $\kappa_X$ on $X$ and $\kappa_Y$ on $Y$. In this
subsection we construct functors fitting in the following diagram, and
prove some compatibility results: \begin{equation} \label{eq:compatibilitybasechange}
\xymatrix@R=18pt{\DGCoh^{\gr}(F_X)
  \ar@<0.5ex>[rr]^-{R(\tilde{\pi}_{\Gm})_*}
  \ar[d]_-{\kappa_X}^-{\wr} & & \DGCoh^{\gr}(F) 
  \ar@<0.5ex>[ll]^-{L(\tilde{\pi}_{\Gm})^*}
  \ar[d]^-{\kappa_Y}_-{\wr} \\ \DGCoh^{\gr}(F_X^{\bot} \, \rcap_{E_X^*} \,
  X) \ar@<0.5ex>[rr]^-{R(\hat{\pi}_{\Gm})_*} & &
  \DGCoh^{\gr}(F^{\bot}
  \, \rcap_{E^*} \, Y) \ar@<0.5ex>[ll]^-{L(\hat{\pi}_{\Gm})^*}. } \end{equation}

Recall that, by definition, {\small \[
  \DGCoh^{\gr}(F^{\bot} \, \rcap_{E^*} \, Y) \cong
  \calD_{\Gm}^{\qc,\fg}(Y,\calT_Y), \quad \DGCoh^{\gr}(F_X^{\bot}
  \, \rcap_{E_X^*} \, X) \cong \calD^{\qc,\fg}_{\Gm}(X,\calT_X).
\]}The morphism $\pi$ induces a morphism of
$\Gm$-equivariant dg-ringed spaces $\hat{\pi} :
(X,\calT_X) \to (Y,\calT_Y).$
In \S \ref{ss:gradedcase} we have constructed functors
{\small \[ R(\hat{\pi}_{\Gm})_* :
  \calD_{\Gm}(X,\calT_X) \to \calD_{\Gm}(Y,\calT_Y), \ \
  L(\hat{\pi}_{\Gm})^* : \calD_{\Gm}(Y,\calT_Y) \to
  \calD_{\Gm}(X,\calT_X). \]}As $\pi^*(\calT_Y) \cong \calT_X$, $\hat{\pi}^*$ is simply
$\pi^*$, and similarly for the $\Gm$-analogues, i.e.~the following
diagram is commutative: \[ \xymatrix@R=15pt{ \calC_{\Gm}(Y,\calT_Y) \ar[d]_{\For}
  \ar[rr]^-{(\hat{\pi}_{\Gm})^*} & & \calC_{\Gm}(X, \calT_X)
  \ar[d]^{\For} \\ \calC_{\Gm}(Y,\calO_Y)
  \ar[rr]^-{(\pi_{\Gm})^*} & & \calC_{\Gm}(X,\calO_X). } \] As
$\pi$ is flat,
$(\hat{\pi}_{\Gm})^*$ is exact, and the
similar diagram of derived categories
also commutes. As
$\Lambda_{\calO_Y}(\calF)$ is locally finitely generated over
$\calO_Y$, a
$\Lambda_{\calO_Y}(\calF)$-module is locally finitely generated iff it is locally finitely generated over $\calO_Y$. The same is true
for $\Lambda_{\calO_X}(\pi^* \calF)$. Hence $L(\hat{\pi})^*$ and $L(\hat{\pi}_{\Gm})^*$ restrict to functors making the following diagram commute: \[
  \xymatrix@R=14pt{\DGCoh^{\gr}(F^{\bot} \, \rcap_{E^*} \, Y) \ar[d]_{\For}
    \ar[rr]^-{L(\hat{\pi}_{\Gm})^*} & & \DGCoh^{\gr}(F_X^{\bot}
    \, \rcap_{E^*_X} \, X) \ar[d]^{\For} \\ \DGCoh(F^{\bot}
  \, \rcap_{E^*} \, Y) \ar[rr]^-{L(\hat{\pi})^*} & & \DGCoh(F_X^{\bot}
  \, \rcap_{E^*_X} \, X). } \]

We have seen in \S \ref{ss:gradedcase} that the following diagrams
are commutative: \[ \xymatrix@R=16pt@C=34pt{\calD_{\Gm}(X, \calT_X) \ar[d]_{\For}
  \ar[r]^-{R(\hat{\pi}_{\Gm})_*} & \calD_{\Gm}(Y,\calT_Y)
  \ar[d]^{\For} & \calD(X,\calT_X) \ar[d]_{\For}
  \ar[r]^-{R(\hat{\pi})_*} & \calD(Y,\calT_Y) \ar[d]^{\For} \\ \calD(X,\calT_X)
  \ar[r]^-{R(\hat{\pi})_*} & \calD(Y,\calT_Y), &
  \calD(X,\calO_X) \ar[r]^-{R\pi_*} &
  \calD(Y,\calO_Y). } \] As $\pi$ is proper, as above
the functors $R(\hat{\pi})_*$ and $R(\hat{\pi}_{\Gm})_*$ restrict to functors
between the full subcategories whose objects have quasi-coherent,
locally finitely generated cohomology. Moreover, the
following diagram commutes: \[
  \xymatrix@R=15pt{ \DGCoh^{\gr}(F_X^{\bot} \, \rcap_{E^*_X} \, X) \ar[d]_{\For}
    \ar[rr]^-{R(\hat{\pi}_{\Gm})_*} & & \DGCoh^{\gr}(F^{\bot}
    \, \rcap_{E^*} \, Y) \ar[d]^{\For} \\ \DGCoh(F_X^{\bot} \, \rcap_{E^*_X} \, X)
    \ar[rr]^-{R(\hat{\pi})_*} & & \DGCoh(F^{\bot} \, \rcap_{E^*} \, Y). } \]

As a step towards the categories $\DGCoh^{\gr}(F)$ and
$\DGCoh^{\gr}(F_X)$, we now study
$\calD^{\qc,\fg}_{\Gm}(X,\calS_X)$ and
$\calD^{\qc,\fg}_{\Gm}(Y,\calS_Y)$. The morphism $\pi$ induces a
morphism of $\Gm$-equivariant dg-ringed spaces $\tilde{\pi}:
(X,\calS_X) \to (Y,\calS_Y).$ By Remark \ref{rk:rkregrading}, the following derived
functors are well defined: {\small \[ R(\tilde{\pi}_{\Gm})_*
  : \calD_{\Gm}(X,\calS_X) \to \calD_{\Gm}(Y,\calS_Y), \
  L(\tilde{\pi}_{\Gm})^* : \calD_{\Gm}(Y,\calS_Y) \to
  \calD_{\Gm}(X,\calS_X). \]}As above, we show that these functors restrict to appropriate subcategories, and that the natural diagrams
commute.

As $\pi^* \calS_Y \cong \calS_X$, the functor $(\tilde{\pi}_{\Gm})^*$
is exact, and corresponds to
$\pi^*: \calD(Y,\calO_Y) \to \calD(X,\calO_X)$ under the forgetful
functor. Hence it restricts to a functor
$\calD_{\Gm}^{\qc,\fg}(Y,\calS_Y) \to
\calD_{\Gm}^{\qc,\fg}(X,\calS_X)$. Moreover,
the following diagram is clearly commutative (see
\eqref{eq:equivalenceCoh(E)} for the vertical arrows):
\begin{equation}\label{eq:pi^*} \xymatrix@R=16pt{\calD_{\Gm}^{\qc,\fg}(Y,\calR_Y) \ar@{<->}[d]_{\wr} &
  \calD_{\Gm}^{\qc,\fg}(Y,\calS_Y)
\ar[r]^-{L(\tilde{\pi}_{\Gm})^*} \ar[l]_{\xi_Y}^{\sim} &
\calD_{\Gm}^{\qc,\fg}(X,\calS_X) \ar[r]^{\xi_X}_{\sim} & \calD_{\Gm}^{\qc,\fg}(X,\calR_X) \ar@{<->}[d]^{\wr} \\ \calD^{b}\Coh^{\Gm}(F) \ar[r]_{\For} &
\calD^b \Coh(F) \ar[r]^-{L(\pi_F)^*} &
\calD^b \Coh(F_X) & \calD^{b}\Coh^{\Gm}(F_X) \ar[l]^{\For} } \end{equation}

Now, consider the functor $R(\tilde{\pi}_{\Gm})_*$. If $\calF$ is in $\calD^{\qc,\fg}_{\Gm}(X,\calS_X)$, then $\xi_X(\calF)$ is in $\calD^{\qc,\fg}_{\Gm}(X,\calR_X)$, and $\For \circ \xi_X(\calF)$
in $\calD^{\qc,\fg}(X,\calR_X) \cong \calD^b \Coh(F_X)$ (see \eqref{eq:equivalenceCoh(E)nonGm}). Hence,
as $\pi_F$ is proper, $R(\pi_F)_* \circ \For \circ \xi_X(\calF)$ is
in $\calD^b \Coh(F)$. But this object coincides by construction with
the object $\For \circ \xi_Y \circ R(\tilde{\pi}_{\Gm})_* \calF$ of
$\calD(Y,\calR_Y)$. Hence $R(\tilde{\pi}_{\Gm})_* \calF$ belongs to
the subcategory $\calD^{\qc,\fg}_{\Gm}(Y,\calS_Y)$ of
$\calD_{\Gm}(Y,\calS_Y)$. This proves that $R(\tilde{\pi}_{\Gm})_*$
restricts to a functor between
$\calD_{\Gm}^{\qc,\fg}(X,\calS_X)$ and
$\calD_{\Gm}^{\qc,\fg}(Y,\calS_Y)$, and also that the analogue of
diagram \eqref{eq:pi^*} for $R(\tilde{\pi}_{\Gm})_*$ and $R(\pi_F)_*$
commutes.\smallskip

Let us now extend these results to \emph{bounded below}
$\Gm$-dg-modules.

\begin{lem} \label{lem:lemmafunctors+1}

The functors \[ (\tilde{\pi}_{\Gm}^+)_* :
  \calC_{\Gm}^+(X,\calS_X) \to \calC_{\Gm}^+(Y,\calS_Y), \quad
  (\tilde{\pi}_{\Gm}^+)^* : \calC_{\Gm}^+(Y,\calS_Y) \to
  \calC_{\Gm}^+(X,\calS_X) \]
admit a right, resp. left, derived functor. The
following diagrams commute: \[ \xymatrix@R=14pt@C=30pt{
  \calD_{\Gm}^+(X,\calS_X) \ar[d]
  \ar[r]^-{R(\tilde{\pi}_{\Gm}^+)_*} &
  \calD_{\Gm}^+(Y,\calS_Y) \ar[d] & \calD_{\Gm}^+(Y,\calS_Y) \ar[d]
  \ar[r]^-{L(\tilde{\pi}_{\Gm}^+)^*} &
  \calD_{\Gm}^+(X,\calS_X) \ar[d] \\
  \calD_{\Gm}(X,\calS_X) \ar[r]^-{R(\tilde{\pi}_{\Gm})_*} &
  \calD_{\Gm}(Y,\calS_Y), & \calD_{\Gm}(Y,\calS_Y)
  \ar[r]^-{L(\tilde{\pi}_{\Gm})^*} &
  \calD_{\Gm}(X,\calS_X). } \]

\end{lem}

\begin{proof} The case of the inverse image is easy, and left
to the reader. We have to show that the functor $(\tilde{\pi}_{\Gm}^+)_*$ admits a right derived functor. But
each object $\calM \in \calC^+_{\Gm}(S_{\calO_X}(\pi^* \calF^{\vee}))$
admits a right resolution $\calI \in \calC^+_{\Gm}(S_{\calO_X}(\pi^*
\calF^{\vee}))$ all of whose components $\calI^i_j$ are flabby (which can be constructed e.g.~using the Godement resolution, see \cite[II.4.3]{GODTop}). This
dg-module $\calI$ is $(\tilde{\pi}_{\Gm}^+)_*$-split, hence the
derived functor is defined at $\calM$. The commutation of the diagram follows from this construction and Corollary
\ref{cor:cordiagramRf_*Gm}. \end{proof}

\begin{cor}\label{cor:functors+}

The functors $R(\tilde{\pi}_{\Gm}^+)_*$ and
$L(\tilde{\pi}_{\Gm}^+)^*$ restrict to the subcategories whose
objects have quasi-coherent, locally finitely generated
cohomology. Moreover, recalling {\rm \eqref{eq:DGCoh1}}, {\rm \eqref{eq:forgetDGCoh(F)}}, the
following diagrams commute: \[
\xymatrix@R=16pt{ \DGCoh^{\gr}(F_X) \ar[d]_-{\For}
  \ar[r]^-{R(\tilde{\pi}_{\Gm}^+)_*} &
  \DGCoh^{\gr}(F) \ar[d]^-{\For} & \DGCoh^{\gr}(F) \ar[d]_-{\For}
  \ar[r]^-{L(\tilde{\pi}_{\Gm}^+)^*} &
  \DGCoh^{\gr}(F_X) \ar[d]^-{\For} \\ \calD^b \Coh(F_X) \ar[r]^-{R(\pi_F)_*}
  & \calD^b \Coh(F), & \calD^b
  \Coh(F) \ar[r]^-{L(\pi_F)^*} & \calD^b \Coh(F_X). } \]

\end{cor}

Because of these results, we will not write the superscript ``$+$'' anymore.

\begin{prop}\label{prop:koszulbasechangeprop}

Consider diagram \eqref{eq:compatibilitybasechange}. There are
isomorphisms \[\left\{
  \begin{array}{ccc} R(\hat{\pi}_{\Gm})_* \circ \kappa_X & \cong &
    \kappa_Y \circ R(\tilde{\pi}_{\Gm})_*, \\[2pt]
    L(\hat{\pi}_{\Gm})^* \circ \kappa_Y & \cong & \kappa_X \circ
    L(\tilde{\pi}_{\Gm})^*. \\ \end{array} \right.\]

\end{prop}

\begin{proof} The second isomorphism is easy, and left to the
reader. The first one can be proved just like \cite[II.5.6]{HARRD}. More
precisely, let $\calM$ be an object of
$\DGCoh^{\gr}(F_X)$, with flabby components. Then $\kappa_Y \circ
R(\tilde{\pi}_{\Gm})_* (\calM) \cong \calT_Y^{\vee}
\otimes_{\calO_Y} \pi_* \calM$. Next, by the projection formula (\cite[ex.~II.5.1]{HARAG}), $(\calT_Y)^{\vee}
\otimes_{\calO_Y} \pi_* \calM \cong \pi_* (\calT_X^{\vee}
\otimes_{\calO_X} \calM)$. Finally, as $R\pi_*=\For \circ
R(\hat{\pi}_{\Gm})_*$, one has a natural morphism $\pi_*
(\calT_X^{\vee} \otimes_{\calO_X} \calM) \to R(\hat{\pi}_{\Gm})_*
(\calT_X^{\vee} \otimes_{\calO_X} \calM)$. This defines a morphism of
functors $\kappa_Y \circ R(\tilde{\pi}_{\Gm})_* \to
R(\hat{\pi}_{\Gm})_* \circ \kappa_X$. To show that it is an
isomorphism, as the question is local over $Y$, we can assume $\calF$
is free; then it is clear. \end{proof}

\subsection{Linear Koszul duality and sub-bundles}
\label{ss:koszulinclusion}

Consider the following situation: $F_1 \subset F_2
\subset E$ are vector bundles over the non-singular variety $X$. Let
$\calF_1$ and $\calF_2$ be the sheaves of sections of $F_1$, $F_2$. We
define as above the $\Gm$-equivariant dg-algebras with trivial
differential: \[\begin{array}{ccc}
  \calS_1:=S_{\calO_X}(\calF_1^{\vee}), &
  \calS_2:=S_{\calO_X}(\calF_2^{\vee}), & \text{with} \ \calF_i^{\vee} \
  \text{in bidegree} \ (2,-2), \\
  \calR_1:=S_{\calO_X}(\calF_1^{\vee}), &
  \calR_2:=S_{\calO_X}(\calF_2^{\vee}), & \text{with} \ \calF_i^{\vee} \
  \text{in bidegree} \ (0,-2), \\ \calT_1:=\Lambda_{\calO_X}(\calF_1),
  & \calT_2:=\Lambda_{\calO_X}(\calF_2), & \text{with} \ \calF_i \
  \text{in bidegree} \ (-1,2). \end{array}\] We
have two Koszul dualities (see Theorem \ref{thm:thmlkd}): 
$\kappa_1$ for the choice $F=F_1$, and $\kappa_2$ for $F=F_2$. We will study the compatibility as in \S \ref{ss:koszulbasechange}. 

The inclusion
$f: F_1 \to F_2$ induces morphisms $\calF_1 \hookrightarrow
\calF_2$ and $\calF_2^{\vee} \twoheadrightarrow
\calF_1^{\vee}$. Let $g: (X,\calT_2) \to (X,\calT_1)$ be the natural
morphism of $\Gm$-dg-ringed spaces. Let us consider the categories $\DGCoh^{\gr}(F_i^{\bot}
\, \rcap_{E^*} \, X)$. We have functors {\small \[ R(g_{\Gm})_* :
  \calD_{\Gm}(X,\calT_2) \to \calD_{\Gm}(X,\calT_1), \ \
  L(g_{\Gm})^* : \calD_{\Gm}(X,\calT_1) \to \calD_{\Gm}(X,\calT_2).
\]}The functor
$R(g_{\Gm})_*$ is the restriction of scalars, and
$L(g_{\Gm})^*$ is the functor $\calM \mapsto
\Lambda_{\calO_X}(\calF_2) \otimes_{\Lambda_{\calO_X}(\calF_1)}
\calM$. Both are induced by exact functors on abelian
categories. They clearly preserve the conditions
``$\qc,\fg$'', hence induce functors making the following diagram
commute: \[ \xymatrix@R=16pt{\DGCoh^{\gr}(F_2^{\bot} \, \rcap_{E^*} \, X)
  \ar[d]_{\For}
\ar@<0.5ex>[rr]^-{R(g_{\Gm})_*} & & \DGCoh^{\gr}(F_1^{\bot} \, \rcap_{E^*} \, X)
\ar[d]^{\For} \ar@<0.5ex>[ll]^-{L(g_{\Gm})^*} \\
\DGCoh(F_2^{\bot} \, \rcap_{E^*} \, X) \ar@<0.5ex>[rr]^-{Rg_*}  & & \DGCoh(F_1^{\bot}
\, \rcap_{E^*} \, X). \ar@<0.5ex>[ll]^-{Lg^*} } \]

Now, consider the categories $\calD^{\qc,\fg}_{\Gm}(X,\calS_i)$. We have a morphism
of $\Gm$-dg-ringed spaces $\tilde{f} :
(X,\calS_1) \to (X,\calS_2)$
and functors $R(\tilde{f}_{\Gm})_*$,
$L(\tilde{f}_{\Gm})^*$ (see again Remark \ref{rk:rkregrading}). Easy arguments show that $R(\tilde{f}_{\Gm})_*$
restricts to the subcategories whose objects have quasi-coherent,
locally finitely generated cohomology, and that the
following diagram commutes:
\begin{equation} \label{eq:f_*}
\xymatrix@R=14pt{ \calD_{\Gm}^{\qc,\fg}(X,\calR_1) \ar@{<->}[d]_{\wr} & \calD_{\Gm}^{\qc,\fg}(X,\calS_1)
  \ar[r]^-{R(\tilde{f}_{\Gm})_*} \ar[l]^{\sim}_{\xi_1} &
  \calD_{\Gm}^{\qc,\fg}(X,\calS_2) \ar[r]^{\xi_2}_{\sim} & \calD_{\Gm}^{\qc,\fg}(X,\calS_2) \ar@{<->}[d]^{\wr} \\ \calD^b
  \Coh^{\Gm}(F_1) \ar[r]_{\For} &
  \calD^b \Coh(F_1) \ar[r]^{Rf_*} & \calD^b \Coh(F_2) & \calD^b \Coh^{\Gm}(F_2).
  \ar[l]_{\For} }
\end{equation} The functor $L(\tilde{f}_{\Gm})^*$
is given by $\calM \mapsto \calS_1 \lotimes_{\calS_2}
\calM$. Arguments similar to those for
$R(\tilde{\pi}_{\Gm})_*$ in \S \ref{ss:koszulbasechange} show that
$L(\tilde{f}_{\Gm})^*$ induces a functor from
$\calD_{\Gm}^{\qc,\fg}(X,\calS_2)$ to
$\calD_{\Gm}^{\qc,\fg}(X,\calS_1)$ and that the diagram analogous to
\eqref{eq:f_*} commutes.


\begin{lem} \label{lem:lemmafunctors+}

The functors \[ (\tilde{f}_{\Gm}^+)_* :
  \calC^+_{\Gm}(X,\calS_1) \to \calC^+_{\Gm}(X,\calS_2), \quad
  (\tilde{f}_{\Gm}^+)^* : \calC^+_{\Gm}(X,\calS_2)
  \to \calC^+_{\Gm}(X,\calS_1) \]
admit a right, resp. left, derived functor. The
following diagrams commute: \[
\xymatrix@R=14pt@C=34pt{\calD^+_{\Gm}(X,\calS_1)
  \ar[r]^-{R(\tilde{f}_{\Gm}^+)_*} \ar[d] &
  \calD^+_{\Gm}(X,\calS_2) \ar[d] & \calD^+_{\Gm}(X,\calS_2) \ar[d]
  \ar[r]^-{L(\tilde{f}_{\Gm}^+)^*} &
  \calD^+_{\Gm}(X,\calS_1) \ar[d] \\
  \calD_{\Gm}(X,\calS_1) \ar[r]^-{R(\tilde{f}_{\Gm})_*} &
  \calD_{\Gm}(X,\calS_2), & \calD_{\Gm}(X,\calS_2)
  \ar[r]^-{L(\tilde{f}_{\Gm})^*} &
  \calD_{\Gm}(X,\calS_1). } \]

\end{lem}

\begin{proof} The case of the direct image is easy, and left
to the reader. We define $\calF:=\calF_1 \oplus \calF_2$, and denote
by $\calS$ the $\Gm$-dg-algebra
$\calS:=S_{\calO_X}(\calF^{\vee})$, with trivial differential and
$\calF^{\vee}$ in bidegree $(2,-2)$. Recall that $(\tilde{f}_{\Gm}^+)^*$ is
the tensor product $\calM \mapsto \calS_1 \otimes_{\calS_2}
\calM$. Here $\calS_1$ is considered as a
  $\calS_1$-$\calS_2$-bimodule; we can also consider it as a module over $\calS_1 \otimes_{\calO_X}
\calS_2 \cong \calS$. Now the natural morphism $\calS \to \calS_1$ is
induced by the 
transpose of the diagonal embedding $\calF_1 \hookrightarrow \calF$. Hence, if we denote by $\calG$ the orthogonal of the image of
$\calF_1$ in this embedding, we have a (bounded below) Koszul resolution
$\calS \otimes_{\calO_X} \Lambda_{\calO_X}(\calG) \xrightarrow{\qis}
\calS_1.$ The first dg-module is K-flat over
$\calS_2$. Thus the tensor product with this
dg-module defines a functor $L(\tilde{f}_{\Gm}^+)^* :
\calD^+_{\Gm}(X,\calS_2) \to \calD^+_{\Gm}(X,\calS_1).$ The commutativity of
the diagram is obvious. \end{proof}

\begin{cor} \label{cor:functors+2}

The functors $R(\tilde{f}_{\Gm}^+)_*$ and $L(\tilde{f}_{\Gm}^+)^*$
restrict to the subcategories whose objects have quasi-coherent, locally
finitely generated cohomology. Moreover, the following diagrams are
commutative: \[ \xymatrix@R=16pt{\DGCoh^{\gr}(F_1) \ar[d]_-{\For}
  \ar[r]^-{R(\tilde{f}_{\Gm}^+)_*} & \DGCoh^{\gr}(F_2) \ar[d]^-{\For} & \DGCoh^{\gr}(F_2) \ar[d]_-{\For} \ar[r]^-{L(\tilde{f}_{\Gm}^+)^*} & \DGCoh^{\gr}(F_1) \ar[d]^-{\For} \\
  \calD^b \Coh(F_1) \ar[r]^-{Rf_*} & \calD^b \Coh(F_2) & \calD^b \Coh(F_2)
\ar[r]^-{Lf^*} & \calD^b \Coh(F_1). } \]

\end{cor}

As in \S \ref{ss:koszulbasechange}, from now on we will not write the superscript
``$+$'' anymore. Now we study
compatibility. We
assume that $\calF_i$ is of constant rank
$n_i$ ($i=1,2$). Consider the line bundles $\calL_i:=\Lambda_{\calO_X}^{n_i}(\calF_i)$. One has isomorphisms $\psi_i :
\calT_i \to \calT_i^{\vee} \otimes_{\calO_X} \calL_i [n_i]$, induced
by \[\left\{ \begin{array}{ccc} \Lambda^j_{\calO_X}(\calF_i)
  \otimes_{\calO_X} \Lambda^{n_i - j}_{\calO_X}(\calF_i) & \to &
  \calL_i \\ t \otimes u & \mapsto & (-1)^{j(j+1) / 2} t \wedge u \\
\end{array} \right. . \] Under $\psi_i$, the action of
  $\calT_i$ by left multiplication
  corresponds to the action on the dual defined as in
  \S \ref{ss:koszulduality}:
  $\psi_i(st)(u)=(-1)^{\deg(s)(\deg(s)+1)/2} \psi_i(t)(su)$. We denote
  by $\langle 1 \rangle$ the shift in the $\Gm$-grading defined by $(M
  \langle 1 \rangle)_n=M_{n-1}$, and by $\langle j \rangle$ its $j$-th 
  power. This functor corresponds to the tensor product with the
  one-dimensional $\Gm$-module corresponding to $\Id_{\Gm}$. Taking
  the $\Gm$-structure into account, $\psi_i$ becomes an isomorphism
  $\calT_i \cong \calT_i^{\vee} \otimes_{\calO_X}  \calL_i [n_i]
  \langle 2n_i \rangle$.

\begin{prop}\label{prop:koszulinclusionprop}

Consider the diagram \[ \xymatrix@R=18pt{ \DGCoh^{\gr}(F_1)
\ar@<0.5ex>[rr]^-{R(\tilde{f}_{\Gm})_*} \ar[d]^-{\kappa_1}_-{\wr} & &
\DGCoh^{\gr}(F_2)
\ar@<0.5ex>[ll]^-{L(\tilde{f}_{\Gm})^*}
\ar[d]^-{\kappa_2}_-{\wr} \\
\DGCoh^{\gr}(F_1^{\bot} \, \rcap_{E^*} \, X) \ar@<0.5ex>[rr]^-{L(g_{\Gm})^*} & &
\DGCoh^{\gr}(F_2^{\bot} \, \rcap_{E^*} \, X) \ar@<0.5ex>[ll]^-{R(g_{\Gm})_*}. } \] There are natural isomorphisms of functors {\small \[\left\{ \begin{array}{ccl}
      \kappa_1 \circ
      L(\tilde{f}_{\Gm})^* & \cong
    & R(g_{\Gm})_* \circ \kappa_2, \\[2pt] \kappa_2 \circ
    R(\tilde{f}_{\Gm})_* & \cong & (L(g_{\Gm})^* \circ
    \kappa_1) \otimes_{\calO_X} \calL_1 \otimes_{\calO_X} \calL_2^{-1} [n_1
    - n_2] \langle 2 n_1 - 2 n_2 \rangle. \end{array} \right.\]}

\end{prop}

\begin{proof} Let us first prove the first isomorphism, or rather an isomorphism $L(\tilde{f}_{\Gm})^* \circ
(\kappa_2)^{-1}
\cong (\kappa_1)^{-1} \circ R(g_{\Gm})_*$. Recall the notation
$\calF$, $\calS$ and
$\calG$ introduced in the proof of Lemma
\ref{lem:lemmafunctors+}. Let $\calN \in \DGCoh^{\gr}(F_2^{\bot} \, \rcap_{E^*} \, X)$, assumed to be
bounded below (Lemma \ref{lem:equivalenceT}). Then $(\kappa_1)^{-1} \circ
R(g_{\Gm})_* (\calN) \cong \calS_1 \otimes_{\calO_X} \calN$,
where $\calN$ is considered as a $\calT_1$-dg-module. On the
other hand, \begin{multline*} L(\tilde{f}_{\Gm})^* \circ
  (\kappa_2)^{-1} (\calN) \ \cong \ L(\tilde{f}_{\Gm})^* (\calS_2
  \otimes_{\calO_X} \calN) \\ \cong \ (\calS \otimes_{\calO_X} \Lambda(\calG))
  \otimes_{\calS_2} (\calS_2 \otimes_{\calO_X} \calN) \ \cong \
  (\calS \otimes_{\calO_X} \Lambda(\calG)) \otimes_{\calO_X}
  \calN. \end{multline*} Hence there is a natural morphism
$L(\tilde{f}_{\Gm})^* \circ (\kappa_2)^{-1} \to (\kappa_1)^{-1} \circ
R(\tilde{g}_{\Gm})_*,$ induced by $\calS
\otimes_{\calO_X} \Lambda(\calG) \to \calS_1$. We want to prove that
it is an
isomorphism. Using the exact sequence of dg-modules $\Ima(d_{\calN}) \hookrightarrow \calN \twoheadrightarrow \calN / \Ima(d_{\calN})$ we
can assume, in addition,
that $\calN$ has trivial differential. (The dg-modules
$\Ima(d_{\calN})$, $\calN / \Ima(d_{\calN})$ may not have
quasi-coherent, locally finitely generated cohomology, but we will not need this assumption.) 

Set $\calP:=\calS \otimes_{\calO_X} \Lambda(\calG)$. It is a K-flat
$\calO_X$-dg-module, as well as $\calS_1$, and $\calP \to \calS_1$ is
a quasi-isomorphism. We want to prove that the morphism $\calP
\otimes_{\calO_X} \calN \to \calS_1 \otimes_{\calO_X} \calN$ is a
quasi-isomorphism, too. One can consider 
$\calP \otimes_{\calO_X}
\calN$ as the total complex of the double complex with $(p,q)$-term $\calP^{q+2p} \otimes_{\calO_X} \calN^{-p}$, with first
differential the Koszul differential, and second differential $d_{\calP}
\otimes \Id$. The first grading of this double complex is bounded
above, hence the associated first
spectral sequence (\cite{GODTop}) converges. The same is true for $\calS_1 \otimes_{\calO_X} \calN$ (in
this case the second differential is trivial). Hence we
can forget about the Koszul differentials. Then
the result follows from Lemma \ref{lem:Kflatsplit}.

Let us now prove the second isomorphism. For $\calM$ in
$\DGCoh^{\gr}(F_1)$, we have $\kappa_2 \circ R(\tilde{f}_{\Gm})_*
(\calM) \cong \calT_2^{\vee} \otimes_{\calO_X} \calM \cong (\calT_2 \otimes_{\calO_X} \calM) \otimes_{\calO_X}
\calL_2^{-1}[-n_2] \langle -2n_2 \rangle.$ On the other
hand, we have $L(g_{\Gm})^* \circ \kappa_1 (\calM) \cong \calT_2
\otimes_{\calT_1} (\calT_1^{\vee} \otimes_{\calO_X} \calM)$, which,
using the same remarks,
is isomorphic to the dg-module $\calT_2 \otimes_{\calO_X} \calM
\otimes_{\calO_X} \calL_1^{-1} [-n_1] \langle -2n_1 \rangle$. This
concludes the proof. \end{proof}

\section{Localization for restricted $\frakg$-modules}
\label{sec:sectionRT}

In this section we prove localization theorems for restricted $\calU \frakg$-modules.

\subsection{Notation and introduction}\label{ss:intro3}

Let $\bk$ be an algebraically closed field of characteristic
$p>0$. Let $R$ be a root system, and $G$ the corresponding connected,
semi-simple, simply-connected
algebraic group over $\bk$. Let $h$ be
the Coxeter number of $G$. In the whole paper we assume that \[ p>h. \] Let $B$ be a
Borel subgroup of $G$, $T \subset B$ a maximal torus, $U$ the unipotent radical
of $B$, $B^+$ the Borel subgroup opposite to $B$, and $U^+$ its
unipotent radical. Let $\frakg$, $\frakb$, $\frakt$, $\frakn$, $\frakb^+$,
$\frakn^+$ be their respective Lie algebras. Let $R^+ \subset R$
be the positive roots, chosen as the roots in $\frakn^+$, and $\Phi$ be
the corresponding set of simple roots. We denote by $U_{\alpha}
\subset G$ the one-parameter subgroup corresponding to the root
$\alpha$. Let $\calB:=G/B$ be the flag variety of $G$, and
$\wcalN:=T^*\calB$ its cotangent bundle. We have the geometric description
$$\wcalN=\{(X,gB) \in \frakg^* \times \calB \mid X_{|g \cdot
  \frakb}=0\}.$$ We also introduce the ``extended cotangent bundle''
  $$\wfrakg:=\{(X,gB) \in \frakg^* \times \calB \mid X_{|g \cdot
    \frakn}=0\}.$$ Let $\frakh$ denote the ``abstract'' Cartan
  subalgebra of $\frakg$, isomorphic to
  $\frakb_0/[\frakb_0,\frakb_0]$ for any Borel subalgebra $\frakb_0$
  of $\frakg$. The Lie algebras $\frakt$ and $\frakh$ are naturally isomorphic, via the morphism $\frakt \xrightarrow{\sim} \frakb / \frakn \cong \frakh$.

For each positive root $\alpha$, we choose isomorphisms of algebraic groups
$u_{\alpha} : \bk \xrightarrow{\sim} U_{\alpha}$ and $u_{-\alpha}
: \bk \xrightarrow{\sim} U_{-\alpha}$ such that for $t \in T$, $t
\cdot u_{\alpha}(x) \cdot t^{-1}=u_{\alpha}(\alpha(t) x)$ and $t
\cdot u_{-\alpha}(x) \cdot t^{-1}=u_{-\alpha}(\alpha(t)^{-1} x)$,
and such that these morphisms extend to a morphism of algebraic groups
$\psi_{\alpha} :
\SL(2,\bk) \to G$ such that \[ \psi_{\alpha} \left(%
\begin{array}{cc}
  1 & x \\
  0 & 1 \\
\end{array}%
\right) = u_{\alpha}(x), \quad \psi_{\alpha} \left(%
\begin{array}{cc}
  1 & 0 \\
  x & 1 \\
\end{array}%
\right) = u_{-\alpha}(x). \] We define the elements
$e_{\alpha}:=d(u_{\alpha})_0(1)$, $e_{-\alpha}:=d(u_{-\alpha})_0(1)$, $h_{\alpha}:=[e_{\alpha},e_{-\alpha}]$.

Let $\bbY:=\mathbb{Z} R$ be the root latice of $R$, and
$\bbX:=X^*(T)$ the weight lattice. Let $W$ be the Weyl group of $(G,T)$,
$W_{\aff}:=W \ltimes \bbY$ the affine Weyl group, and $W_{\aff}':=W
\ltimes \bbX$ the extended affine Weyl group. We let $\ell$ denote the
length function of $W$ and
$W_{\aff}$ (for our choice of $R^+$); it extends canonically to
$W_{\aff}'$ (\cite[\S
1]{IM}). Let $\rho \in \bbX$ be the half sum of the positive roots, and \[ C_0:=\{ \lambda \in
\bbX \mid \forall \alpha \in
R^+, \ 0 < \langle \lambda + \rho, \alpha^{\vee} \rangle < p \},\]
the set of integral weights in the fundamental
alcove (which contains $0$). We consider the action of $W_{\aff}'$ on $\bbX$ defined by $w \bullet \lambda = w(\lambda+\rho) - \rho$. Denote by $L(\lambda)$, $\Ind_B^G(\lambda)$ the simple and induced $G$-modules with highest weight $\lambda$.

Let us apply the results of section
\ref{sec:sectionKoszulduality} in the following
situation: $X=\calB^{(1)}$, the Frobenius twist\footnote{See
\cite[1.1.1]{BMR} for Frobenius twists.} of $\calB$, $E=(\frakg^* \times \calB)^{(1)}$, and  $F=\wcalN^{(1)} \subset (\frakg^* \times
\calB)^{(1)}$. Let $\calT_{\calB^{(1)}}$ denote the tangent bundle to
$\calB^{(1)}$; $\calT_{\calB^{(1)}}^{\vee}$ is the sheaf of
sections of $\wcalN^{(1)}$. Under our hypothesis $p>h$, there exists a
$G$-equivariant isomorphism $\frakg^* \cong \frakg$, which induces an
isomorphism $E \cong E^*$. Under
this isomorphism, $(\wcalN^{(1)})^{\bot}$ identifies with
$\wfrakg^{(1)}$. By Theorem \ref{thm:thmlkd}, we obtain a Koszul duality
\begin{equation}\label{eq:defkappa} \kappa_{\calB} : \DGCoh^{\gr}(\wcalN^{(1)})
  \xrightarrow{\sim} \DGCoh^{\gr}((\wfrakg \rcap_{\frakg^*
\times \calB} \calB)^{(1)}). \end{equation} This equivalence is given
by the following formula, for $\calM$ in $\DGCoh^{\gr}(\wcalN^{(1)})$:
$\kappa_{\calB}(\calM) =
  (\Lambda(\calT_{\calB^{(1)}}^{\vee}))^{\vee}
  \otimes_{\calO_{\calB^{(1)}}} \calM.$ We have an
isomorphism $\Lambda^{{\rm top}}(\calT_{\calB^{(1)}}^{\vee}) \cong
\calO_{\calB^{(1)}}(-2\rho).$ Hence we have
$(\Lambda(\calT_{\calB^{(1)}}^{\vee}))^{\vee} \cong
\Lambda(\calT_{\calB^{(1)}}^{\vee}) \otimes \calO_{\calB^{(1)}}(2\rho)
[-N] \langle -2N \rangle,$ where
$N=\rk(\calT_{\calB^{(1)}}^{\vee})=\# R^+$. It follows that, for $\calM$ in
$\DGCoh^{\gr}(\wcalN^{(1)})$, \begin{equation}
  \label{eq:formulakappa} \kappa_{\calB}(\calM) =
  \Lambda(\calT_{\calB^{(1)}}^{\vee}) \otimes \calM \otimes
  \calO_{\calB^{(1)}}(2\rho) [-N] \langle -2N \rangle. \end{equation}

In section \ref{sec:sectionKoszulduality} (see e.g.~equation \eqref{eq:DGCoh2}) we have used the realization
\begin{equation}\label{eq:DGCohkoszul1} \DGCoh((\wfrakg
  \rcap_{\frakg^* \times \calB} \calB)^{(1)})
\ \cong \ \calD^{\qc,\fg}
(\calB^{(1)},\Lambda_{\calO_{\calB^{(1)}}}(\calT^{\vee}_{\calB^{(1)}})),
\end{equation} where
$\Lambda_{\calO_{\calB^{(1)}}}(\calT^{\vee}_{\calB^{(1)}})$ has trivial differential, and
$\calT^{\vee}_{\calB^{(1)}}$ is in degree $-1$. Consider the embeddings $i: \wfrakg^{(1)}
\hookrightarrow (\frakg^* \times \calB)^{(1)}$, $j: \calB^{(1)}
\hookrightarrow (\frakg^* \times \calB)^{(1)}$. The realization \eqref{eq:DGCohkoszul1} was constructed using
a resolution of
$i_* \calO_{\wfrakg^{(1)}}$ over $\calO_{(\frakg^* \times
  \calB)^{(1)}}$. We can obtain another realization using the Koszul resolution $\calO_{(\frakg^* \times
  \calB)^{(1)}} \otimes_{\bk} \Lambda(\frakg^{(1)})
\xrightarrow{\qis} j_* \calO_{\calB^{(1)}}$, which we will rather use from now on. More precisely, using Remark
\ref{rk:rkderivedintersection} we have:

\begin{prop}\label{prop:DGCohkoszul2nongr}

There exists an equivalence of triangulated categories
\[ \DGCoh((\wfrakg
  \rcap_{\frakg^* \times \calB}
\calB)^{(1)}) \ \cong \ \calD^{\qc,\fg}(\wfrakg^{(1)},
\calO_{\wfrakg^{(1)}} \otimes_{\bk}
\Lambda(\frakg^{(1)})) \] where $\calO_{\wfrakg^{(1)}}
\otimes_{\bk} \Lambda(\frakg^{(1)})$ is a dg-algebra with the
generators of $\Lambda(\frakg^{(1)})$ in degree $-1$, equipped with a
Koszul differential.

\end{prop}

\subsection{Review of \cite{BMR, BMR2}}\label{ss:reviewlocalization}

If $X$ is a variety, and $Y \subset X$ a closed subscheme, one says that an
$\calO_X$-module $\calF$ is \emph{supported on} $Y$ if $\calF_{|X - Y}=0$. If $\calF$ is coherent, this is equivalent to requiring that the ideal
of definition $\calI_Y \subset \calO_X$ of $Y$ acts nilpotently. We write
$\Coh_Y(X)$ for the full subcategory of $\Coh(X)$ whose objects are supported on $Y$.

If $P \subseteq G$ is a parabolic subgroup containing $B$, $\frakp$
its Lie algebra, $\frakp^{{\rm u}}$ the nilpotent radical of $\frakp$, and
$\calP=G/P$ the corresponding flag variety, we consider the
following analogue of the variety $\wfrakg$:
$$\wfrakg_{\calP}:=\{(X,gP) \in \frakg^* \times
\calP \mid X_{|g \cdot \frakp^{{\rm u}}}=0 \}.$$ In particular,
$\wfrakg_{\calB}=\wfrakg$. The morphism $\pi_{\calP} : \calB
\to \calP$ induces a morphism \begin{equation} \label{eq:defpiP}
  \widetilde{\pi}_{\calP}: \wfrakg \to
\wfrakg_{\calP}.\end{equation} In this situation, we also denote by
$W_{P} \subseteq W$ the Weyl group of $P$. If $\alpha \in \Phi$, and $P_{\alpha}$ is the
minimal parabolic subgroup containing $B$ associated to $\alpha$, we
simplify the notation by setting
$\wfrakg_{\alpha}:=\wfrakg_{G/P_{\alpha}}$,
$\widetilde{\pi}_{\alpha}:=\widetilde{\pi}_{G/P_{\alpha}}$.\smallskip

Let $\frakZ$ be the center of $\calU \frakg$, the enveloping
algebra of $\frakg$. The subalgebra of $G$-invariants
$\frakZ_{{\rm HC}}:=(\calU \frakg)^G$ is central in $\calU \frakg$. This
is the ``Harish-Chandra part'' of the center, isomorphic
to $S(\frakt)^{(W,\bullet)}$, the algebra of $W$-invariants in the
symmetric algebra of $\frakt$, for the dot-action. The center
$\frakZ$ also has an other part, the ``Frobenius part''
$\frakZ_{{\rm Fr}}$, which is generated as an algebra by the elements
$X^p - X^{[p]}$ for $X \in \frakg$. It is isomorphic to
$S(\frakg^{(1)})$, the functions on
$\frakg^* {}^{(1)}$. Under our assumption $p>h$, there is an isomorphism (see
\cite{MRCen}):
$$\frakZ_{{\rm HC}} \otimes_{\frakZ_{{\rm Fr}} \cap \frakZ_{{\rm HC}}}
\frakZ_{{\rm Fr}}
\xrightarrow{\sim} \frakZ.$$ Hence a character of $\frakZ$ is
given by a ``compatible pair'' $(\nu,\chi) \in \frakt^* \times
\frakg^{*}{}^{(1)}$. In this paper we only consider the case when
$\chi=0$, and $\nu$ is \emph{integral}, i.e.~in the image
of the natural map $\bbX \to \frakt^*$ (such a pair is always ``compatible''). If $\lambda \in
\bbX$, we still denote by $\lambda$ its image in $\frakt^*$. We
consider the specializations \[ (\calU
\frakg)^{\lambda} := (\calU \frakg) \otimes_{\frakZ_{{\rm HC}}}
\bk_{\lambda}, \ \ (\calU \frakg)_{0}:=(\calU \frakg)
\otimes_{\frakZ_{{\rm Fr}}} \bk_{0}, \ \ (\calU
\frakg)^{\lambda}_{0}:=(\calU \frakg) \otimes_{\frakZ}
\bk_{(\lambda,0)}. \]

Let $\Mod^{\fg}(\calU \frakg)$ be the category of finitely
generated $\calU \frakg$-modules. If $\lambda \in \bbX$, we denote by
$\Mod^{\fg}_{(\lambda,0)}(\calU
\frakg)$ the subcategory of $\calU \frakg$-modules on
which $\frakZ$ acts with generalized character $(\lambda,0)$.
We define similarly the categories $\Mod^{\fg}_{0}((\calU
\frakg)^{\lambda})$, $\Mod^{\fg}_{\lambda}((\calU \frakg)_{0})$,
$\Mod^{\fg}((\calU \frakg)^{\lambda}_{0})$. We also denote by
$\Mod^{\fg}_{\lambda}(\calU \frakg)$ the subcategory of $\calU \frakg$-modules on
which $\frakZ_{{\rm HC}}$ acts with generalized character
$\lambda$. Hence we have inclusions \[
\xymatrix@R=13pt{ \Mod^{\fg}((\calU \frakg)^{\lambda}_{0}) \ar@{^{(}->}[r] \ar@{_{(}->}[rd] & \Mod^{\fg}_{0}((\calU \frakg)^{\lambda}) \ar@{^{(}->}[r]
  & \Mod^{\fg}_{(\lambda,0)}(\calU
\frakg) \ar@{^{(}->}[r] \ar@{^{(}->}[rd] & \Mod^{\fg}_{\lambda}(\calU \frakg) \ar@{^{(}->}[d] \\ &
\Mod^{\fg}_{\lambda}((\calU \frakg)_{0}) \ar@{^{(}->}[ru] & & \Mod^{\fg}(\calU \frakg) } \]

Recall that a weight $\lambda \in \bbX$ is called
\emph{regular} if, for any root $\alpha$, $\langle \lambda +
\rho, \alpha^{\vee} \rangle \notin p \mathbb{Z}$, i.e.~if $\lambda$ is not on any reflection hyperplane of $W_{\aff}$ (for the dot-action). If $\mu \in \bbX$, we denote by ${\rm Stab}_{(W_{\aff},\bullet)}(\mu)$ the stabilizer of $\mu$ for the dot-action of $W_{\aff}$ on $\bbX$. Under our hypothesis $p>h$, we have $(p \bbX) \cap \bbY = p \bbY$. It follows that ${\rm Stab}_{(W_{\aff},\bullet)}(\mu)$ is also the stabilizer of $\mu$ for the action of $W_{\aff}'$ on $\bbX$.

We have (see \cite[5.3.1]{BMR} for $\rmi$, and
\cite[1.5.1.c, 1.5.2.b]{BMR2} for $\rmii$):

\begin{thm}\label{thm:thmBMR}

$\rmi$ Let $\lambda \in \bbX$ be regular. There exist equivalences \begin{align} \label{eq:equivBMR} \calD^b
  \Coh_{\calB^{(1)}}(\wfrakg^{(1)})
  \ & \cong \ \calD^b\Mod^{\fg}_{(\lambda,0)}(\calU \frakg), \\ \label{eq:equivBMR2}
  \calD^b\Coh_{\calB^{(1)}}(\wcalN^{(1)}) \ & \cong \
  \calD^b\Mod^{\fg}_{0}((\calU \frakg)^{\lambda}). \end{align}

$\rmii$ More generally, let $\mu \in \bbX$, and let $P$ be a parabolic
subgroup of $G$ containing $B$ such that\footnote{Equivalently, this means that $\mu$ is on the reflection hyperplane corresponding to any simple root of $W_{P}$, but not on any hyperplane of a reflection (simple or not) in $W_{\aff} - W_{P}$.} ${\rm
  Stab}_{(W_{\aff},\bullet)}(\mu)=W_P$. Let $\calP=G/P$ be the corresponding
flag variety. Then there exists an equivalence of categories
\begin{equation} \calD^b\Coh_{\calP^{(1)}}(\wfrakg_{\calP}^{(1)}) \
  \cong \ \calD^b\Mod^{\fg}_{(\mu,0)}(\calU \frakg). \end{equation}

\end{thm}

Let us recall briefly how equivalence \eqref{eq:equivBMR} can be
constructed. We use
the notation of \cite{BMR}. Consider the sheaf of
algebras $\wcalD$ on $\calB$; it can also be
considered as a sheaf of algebras on $\wfrakg^{(1)}
\times_{\frakh^*{}^{(1)}} \frakh^*$, and it is an
Azumaya algebra on this space (see \cite[3.1.3]{BMR}). We denote by
$\Mod^{{\rm c}}(\wcalD)$ the
category of quasi-coherent, locally finitely generated
$\wcalD$-modules (equivalently, either on $\calB$ or on $\wfrakg^{(1)}
\times_{\frakh^*{}^{(1)}} \frakh^*$). For $\nu \in \frakt^* \cong \frakh^*$ we denote by $\Mod^{{\rm
    c}}_{\nu}(\wcalD)$, resp. $\Mod^{{\rm c}}_{(\nu,0)}(\wcalD)$, the
full subcategory of $\Mod^{{\rm c}}(\wcalD)$ whose objects are
supported on $\wcalN^{(1)} \times \{ \nu \} \subset
\wfrakg^{(1)} \times_{\frakh^*{}^{(1)}} \frakh^*$, respectively on
$\calB^{(1)} \times \{ \nu \} \subset \wfrakg^{(1)}
\times_{\frakh^*{}^{(1)}} \frakh^*$. If $\lambda \in \bbX$ is
regular, the global sections functor $R\Gamma : \calD^b \Mod^{{\rm c}}_{\lambda}(\wcalD)
\to \calD^b \Mod^{\fg}_{\lambda}(\calU \frakg)$ is an equivalence of
categories. Its inverse is the localization functor
$\calL^{\hat{\lambda}}$. These functors restrict to equivalences
between $\calD^b \Mod^c_{(\lambda,0)}(\wcalD)$ and
$\calD^b\Mod^{\fg}_{(\lambda,0)}(\calU \frakg)$.

Next, the Azumaya algebra $\wcalD$ splits on the formal neighborhood
of $\calB^{(1)} \times \{ \lambda \}$ in $\wfrakg^{(1)}
\times_{\frakh^*{}^{(1)}} \frakh^*$. Hence, the choice of a splitting
bundle on this formal neighborhood yields an equivalence of
categories $\Coh_{\calB^{(1)} \times \{ \lambda \} }(\wfrakg^{(1)}
\times_{\frakh^*{}^{(1)}} \frakh^*) \cong
\Mod^c_{(\lambda,0)}(\wcalD)$. Finally, the projection $\wfrakg^{(1)}
\times_{\frakh^{*}{}^{(1)}} \frakh^* \to
\wfrakg^{(1)}$ induces an isomorphism between the formal
neighborhood of $\calB^{(1)} \times \{ \lambda \}$
and the formal neighborhood of $\calB^{(1)}$ (see \cite[1.5.3.c]{BMR2}). This isomorphism induces
an equivalence $\Coh_{\calB^{(1)} \times
  \{ \lambda \} }(\wfrakg^{(1)} \times_{\frakh^*{}^{(1)}} \frakh^*) \cong
\Coh_{\calB^{(1)}}(\wfrakg^{(1)})$, and gives \eqref{eq:equivBMR}.

We choose the normalization of the splitting bundles as in
\cite[1.3.5]{BMR2}. We denote by $\calM^{\lambda}$
the splitting bundle associated to $\lambda$ (it is denoted
$\calM^{\calB}_{0,\lambda}$ in \cite{BMR2}), and the associated equivalence by \[\gamma_{\lambda}^{\calB} :
\calD^b\Coh_{\calB^{(1)}}(\wfrakg^{(1)}) \xrightarrow{\sim}
\calD^b\Mod^{\fg}_{(\lambda,0)}(\calU \frakg).\] Similarly, for $\lambda, \mu, \calP$ as in Theorem \ref{thm:thmBMR}, we
denote by \begin{align*} \epsilon_{\lambda}^{\calB} :
\calD^b\Coh_{\calB^{(1)}}(\wcalN^{(1)}) \ & \xrightarrow{\sim} \
\calD^b\Mod^{\fg}_{0}((\calU \frakg)^{\lambda}),\\ \gamma^{\calP}_{\mu}
: \calD^b\Coh_{\calP^{(1)}}(\wfrakg_{\calP}^{(1)})
\ & \xrightarrow{\sim} \ \calD^b\Mod^{\fg}_{(\mu,0)}(\calU \frakg) \end{align*} the
equivalences obtained with the normalization of \cite[1.3.5]{BMR2}.

If $\lambda \in \bbX$ is regular and $\nu \in \bbX$, then 
$\Mod^{\fg}_{(\lambda,0)}(\calU \frakg)$ and
$\Mod^{\fg}_{(\lambda + p \nu,0)}(\calU \frakg)$ coincide. But the
equivalences $\gamma^{\calB}_{\lambda}$ and $\gamma^{\calB}_{\lambda +
  p \nu}$ differ by a shift: $\gamma^{\calB}_{\lambda + p
  \nu}(\calF)=\gamma^{\calB}_{\lambda}(\calO_{\wfrakg^{(1)}}(\nu)
\otimes_{\calO_{\wfrakg^{(1)}}} \calF)$ for $\calF$ in $\calD^b
\Coh_{\calB^{(1)}}(\wfrakg^{(1)})$.

\subsection{An equivalence of derived categories}

In this subsection we prove an equivalence of derived categories that
will be needed later. Recall the notation $\qc$ and $\fg$
introduced in \S \ref{ss:paragraphrestriction}.

Let $X$ be a variety, and $\calY$ be a sheaf of
dg-algebras on $X$, non-positively graded and
$\calO_X$-quasi-coherent. We also consider the sheaf of
algebras $\calA=\calY^0$. Let $Z \subset X$ be a closed subscheme. We denote by
$\calD^{\qc}_Z(X, \, \calY)$ the full subcategory of
$\calD^{\qc}(X, \, \calY)$ whose objects
have their cohomology supported on $Z$ (and similarly with $\qc$
replaced by $\qc,\fg$).

\begin{lem}\label{lem:Kinjresolutionsupport}

Let $\calF$ be a $\calY$-dg-module which is $\calO_X$-quasi-coherent, supported
on $Z$, and bounded below. There exists a K-injective
$\calY$-dg-module $\calI$, $\calO_X$-quasi-coherent and supported on
$Z$, and a quasi-isomorphism $\calF \xrightarrow{\qis} \calI$.

\end{lem}

\begin{proof} The case $\calY=\calO_X$ can be derived from \cite[II.7.18 and its
proof]{HARRD}. The general case follows, as in the proof of Lemma \ref{lem:lemKinjective}.\end{proof}

\begin{lem}\label{lem:Kinjresolutionsupport2}

Let $\calF$ be an object of $\calD^{\qc}_Z(X, \, \calY)$, whose cohomology is
bounded. There exists a K-injective $\calY$-dg-module $\calG$, which is
$\calO_X$-quasi-coherent and supported on $Z$, and a quasi-isomorphism
$\calF \xrightarrow{\qis} \calG$.

\end{lem}

\begin{proof} We prove the lemma by induction on $l(\calF)$, where $l(\calF):={\rm max}\{i \in \mathbb{Z} \mid
H^i(\calF) \neq 0\} - {\rm min}\{i \in \mathbb{Z} \mid
H^i(\calF) \neq 0\}$ if $H(\calF) \neq 0$, and $l(\calF)=-1$
otherwise. If $l(\calF)=-1$, the result is obvious. Now assume $l(\calF)=n \geq 0$, and
the result is true for any dg-module $\calG$ with $l(\calG) <
n$. Let $j$ be
the lowest integer such that $H^j(\calF) \neq 0$. Using a truncation
functor, we can assume that
$\calF^k=0$ for $k < j$. Then $\ker(d_{\calF}^j)=H^j(\calF)$ is quasi-coherent and supported on $Z$. Let $\calK$ denote the sub-$\calY$-dg-module of $\calF$ concentrated in degree $j$, with $\calK^j=\ker(d_{\calF}^j)$. By Lemma
\ref{lem:Kinjresolutionsupport}, there exists a K-injective
$\calY$-dg-module $\calI_1$, quasi-coherent and supported on $Z$, and a
quasi-isomorphism $\calK \xrightarrow{\qis} \calI_1$. Consider $\calG:=\mathrm{Coker}(\calK \hookrightarrow \calF)$. Then
$l(\calG) < l(\calF)$. Hence, by induction, there exists a K-injective
$\calY$-dg-module $\calI_2$, quasi-coherent and supported on $Z$, and a quasi-isomorphism $\calG \xrightarrow{\qis} \calI_2$.

There exists a natural morphism $\calG[-1] \to \calK$ in
$\calD(X,\calY)$, hence also a morphism $\calI_2[-1] \to \calI_1$. By K-injectivity of $\calI_1$ (see Definition
\ref{def:defKinjective}), one can represent this morphism
by an actual morphism of $\calY$-dg-modules $f: \calI_2[-1] \to
\calI_1$ (unique up to homotopy). Let
$\calI_3$ be the cone of $f$. Then $\calI_3$ is K-injective,
quasi-coherent and supported on $Z$. Moreover, one easily checks that there exists a quasi-isomorphism $\calF \to \calI_3$. \end{proof}

{F}rom now on we assume in addition that $\calY$ is coherent as an $\calO_X$-module. In particular, as $\calA$ is coherent over
$\calO_X$, an $\calA$-module quasi-coherent over $\calO_X$ is locally
finitely generated over $\calA$ if and only if it is coherent over
$\calO_X$. The same applies for $\calA$ replaced by $H(\calY)$.

\begin{lem}\label{lem:dgmodulecoherent}

Every $\calY$-dg-module $\calF$ which is bounded, $\calO_X$-quasi-coherent, and whose cohomology is $\calO_X$-coherent is the
inductive limit of $\calO_X$-coherent sub-$\calY$-dg-modules which are quasi-isomorphic
to $\calF$ under the inclusion map.

\end{lem}

\begin{proof} Our proof is similar to that of
\cite[VI.2.11.(a)]{BDmod}. First, $\calF$ is the inductive limit of coherent
sub-dg-modules, hence it is sufficient to show that given a coherent
sub-dg-module $\calK$ of $\calF$, there exists a coherent
sub-dg-module $\calG$ of $\calF$ containing $\calK$ and
quasi-isomorphic to $\calF$ under inclusion.

This is proved by a simple (descending) induction. Let $j \in \mathbb{Z}$, and
assume that we have found a subcomplex $\calG_j$ of $\bigoplus_{i \geq
  j} \calF^i$, coherent over $\calO_X$, containing $\bigoplus_{i \geq
  j} \calK^i$, stable under $\calY$ (i.e.~if $g
\in \calG_j^i$ and $y \in \calY^k$, and if $i+k \geq j$, then $y \cdot
g \in \calG_j^{i+k}$), such that $\calG_j \to \calF$ is a quasi-isomorphism
in degrees greater than $j$ and that $\calG^j_j \cap \ker(d_{\calF}^j)
\to H^j(\calF)$ is surjective. Then we choose a sub-$\calA$-module
$\calN^{j-1}$ of $\calF^{j-1}$ containing $\calK^{j-1}$, coherent over
$\calO_X$, whose image
under $d_{\calF}^{j-1}$ is $\calG^j_j \cap
\Ima(d_{\calF}^{j-1})$. Without altering these conditions, we can add
a coherent sub-module of cocycles so that the new sub-module
$\calN^{j-1}$ contains representatives of all the elements of
$H^{j-1}(\calF)$. We can also assume that $\calN^{j-1}$ contains all the
sections of the form $y \cdot g$ for $y \in \calY^i$ and $g \in
\calG_j^k$ with $i+k=j-1$. Then we define $\calG_{j-1}$ by
$$\calG_{j-1}^k=\left\{ \begin{array}{cl} \calG_j^k & \text{if } k
    \geq j, \\ \calN^{j-1} & \text{if } k=j-1. \end{array} \right.$$
For $j$ small enough, $\calG_j$ is the desired sub-dg-module. \end{proof}

We denote by $\calC^{\qc,\fg}_Z(X, \calY)$ the category of
$\calY$-dg-modules which are $\calO_X$-quasi-coherent and locally finitely generated over $\calY$ (i.e.~$\calO_X$-coherent), and supported on $Z$. Let $\calD \bigl(
\calC^{\qc,\fg}_Z(X, \calY) \bigr)$ be the corresponding derived category.

\begin{prop}\label{prop:dgmodulessupport}

The functor $\iota:
\calD \bigl( \calC^{\qc,\fg}_Z(X, \, \calY) \bigr) \to
\calD^{\qc,\fg}_Z(X, \, \calY)$ induced by the inclusion
$\calC^{\qc,\fg}_Z(X, \, \calY) \hookrightarrow \calC(X, \, \calY)$ is an
equivalence.

\end{prop}

\begin{proof} This proof is similar to that of
\cite[VI.2.11]{BDmod}. It follows from Lemmas
\ref{lem:Kinjresolutionsupport2} and \ref{lem:dgmodulecoherent}, using
truncation functors, that $\iota$ is essentially surjective.

Now, let us prove that it is full. Let $\calF$, $\calG$ be objects of
$\calC^{\qc,\fg}_Z(X, \, \calY)$. In particular, $\calF$ and $\calG$ are
bounded. A morphism $\phi: \iota(\calF) \to \iota(\calG)$ in
$\calD^{\qc,\fg}_Z(X, \, \calY)$ is represented by a diagram
$\iota(\calF) \xrightarrow{\alpha} \calN \xleftarrow{\beta}
\iota(\calG)$, where $\beta$ is a quasi-isomorphism. Using Lemma
\ref{lem:Kinjresolutionsupport2} and truncation functors, one can assume
that $\calN$ is bounded,
quasi-coherent, and supported on $Z$. By Lemma \ref{lem:dgmodulecoherent}, there
exists a coherent sub-dg-module $\calN '$ of $\calN$ (supported on
$Z$), containing $\alpha(\calF)$ and $\beta(\calG)$, and
quasi-isomorphic to $\calN$ under the inclusion map. Then $\phi$ is
represented by $\iota(\calF) \xrightarrow{\alpha} \calN'
\xleftarrow{\beta} \iota(\calG),$ which is the image of a morphism in $\calD
\bigl( \calC^{\qc,\fg}_Z(X, \, \calY) \bigr)$. Hence $\iota$ is full.

Finally we prove that $\iota$ is faithful. If a morphism $f:\calF \to
\calG$ in $\calC^{\qc,\fg}_Z(X, \, \calY)$ is such that $\iota(f)=0$, then
there exists $\calN$ in $\calD^{\qc,\fg}_Z(X, \, \calY)$, which can again
be assumed to be bounded, quasi-coherent and supported on $Z$, and a
quasi-isomorphism of $\calY$-dg-modules $g: \calG \to \calN$ such that
$g \circ f$ is homotopic to zero. This homotopy is given by a morphism
$h: \calF \to \calN[-1]$. By Lemma \ref{lem:dgmodulecoherent}, there
exists a coherent
sub-dg-module $\calN '$ of $\calN$ containing $g(\calG)$ and
$h(\calF)[1]$, and quasi-isomorphic to $\calN$ under the
inclusion. Replacing $\calN$ by $\calN'$, this
proves that $f=0$ in $\calD \bigl( \calC^{\qc,\fg}_Z(X, \, \calY)
\bigr)$. \end{proof}

\subsection{Localization with a fixed Frobenius central character}
\label{ss:paragraphfixedFrobenius}

In \cite{BMR, BMR2} the authors give geometric counterparts
for the derived categories of $\calU \frakg$-modules with a
\emph{generalized} Frobenius central character, and a fixed or
generalized Harish-Chandra central character. The relation between the Koszul duality \eqref{eq:defkappa}
and representation theory is based on Theorem
\ref{thm:localizationfixedFr}, which gives a geometric picture for the
derived category of $\calU \frakg$-modules with a generalized
(integral, regular) Harish-Chandra central character and a \emph{fixed}
trivial Frobenius central character. More precisely, we prove: 

\begin{thm}\label{thm:localizationfixedFr}

Let $\lambda \in \bbX$ be regular. There exists an equivalence of
triangulated categories \[ \DGCoh((\wfrakg
\, \rcap_{\frakg^* \times \calB} \, \calB)^{(1)}) \ \xrightarrow{\sim} \
\calD^{b} \Mod^{\fg}_{\lambda}((\calU \frakg)_{0}).\]

\end{thm}

The proof will occupy the whole subsection. Consider the derived intersection $(\wfrakg \,
\rcap_{\frakg^* \times \calB} \, \calB)^{(1)}.$ By
Proposition \ref{prop:DGCohkoszul2nongr}, we have an
equivalence \begin{equation} \label{eq:DGC} \DGCoh((\wfrakg
  \, \rcap_{\frakg^* \times \calB} \, \calB)^{(1)}) \ \cong \
\calD^{\qc,\fg}(\wfrakg^{(1)}, \, \calO_{\wfrakg^{(1)}} \otimes_{\bk}
\Lambda(\frakg^{(1)})).\end{equation} 

We have seen in the remarks following Theorem \ref{thm:thmBMR} that the
projection $\wfrakg^{(1)} \times_{\frakh^*
  {}^{(1)}} \frakh^* \to \wfrakg^{(1)}$ induces an isomorphism between
the formal neighborhoods of $\calB^{(1)} \times \{ \lambda \}$ and of $\calB^{(1)}$. We denote these formal
neighborhoods by $\widehat{\calB^{(1)}}$. In
this subsection, for simplicity we put $X:=\wfrakg^{(1)} \times_{\frakh^*
  {}^{(1)}} \frakh^*$.

\begin{lem}\label{lem:lemmaX}

The following natural functor is an equivalence of categories: $\calD^{\qc,\fg}(\wfrakg^{(1)}, \,
\calO_{\wfrakg^{(1)}} \otimes_{\bk}
\Lambda(\frakg^{(1)})) \ \to \ \calD^{\qc,\fg}_{\calB^{(1)} \times
  \{ \lambda \} }(X, \, \calO_{X} \otimes_{\bk}
\Lambda(\frakg^{(1)}))$.

\end{lem}

\begin{proof} Any object of
$\calD^{\qc,\fg}(\wfrakg^{(1)}, \, \calO_{\wfrakg^{(1)}}
\otimes_{\bk} \Lambda(\frakg^{(1)}))$ has its cohomology supported on
$\calB^{(1)}$ (because $H^0(\calO_{\wfrakg^{(1)}} \otimes_{\bk}
\Lambda(\frakg^{(1)})) = \calO_{\calB^{(1)}}$). Hence, by Proposition
  \ref{prop:dgmodulessupport}, $\calD^{\qc,\fg}(\wfrakg^{(1)},
  \, \calO_{\wfrakg^{(1)}} \otimes_{\bk} 
\Lambda(\frakg^{(1)})) \cong \calD \bigl( \calC^{\qc,\fg}_{\calB^{(1)}}(\wfrakg^{(1)}, \,
\calO_{\wfrakg^{(1)}} \otimes_{\bk} \Lambda(\frakg^{(1)})) \bigr)$. Now, as the formal neighborhoods of $\calB^{(1)}$ in $\wfrakg^{(1)}$
and of $\calB^{(1)} \times \{ \lambda \}$ in $X$ are isomorphic, the
category $\calC^{\qc,\fg}_{\calB^{(1)}}(\wfrakg^{(1)}, \,
\calO_{\wfrakg^{(1)}} \otimes_{\bk} \Lambda(\frakg^{(1)}))$ is
equivalent to $\calC^{\qc,\fg}_{\calB^{(1)} \times \{ \lambda \}}(X, \,
\calO_{X} \otimes_{\bk} \Lambda(\frakg^{(1)}))$. We conclude using Proposition \ref{prop:dgmodulessupport} again. \end{proof}

Let now $K_{\frakg}$ denote the Koszul complex $S(\frakg^{(1)})
\otimes_{\bk} \Lambda(\frakg^{(1)})$, which is quasi-isomorphic to the
trivial $S(\frakg^{(1)})$-module $\bk_0$. Here $S(\frakg^{(1)})$ is in
degree $0$, and the generators of $\Lambda(\frakg^{(1)})$ are in
degree $-1$. By Poincar{\'e}-Birkhoff-Witt theorem, the enveloping algebra $\calU \frakg$ is free (hence flat) over
$\frakZ_{{\rm Fr}} \cong S(\frakg^{(1)})$. Hence, if we consider
$\calU \frakg$ as a (trivial) sheaf of dg-algebras on ${\rm Spec}(\bk)$, there is a
quasi-isomorphism of dg-algebras $\calU \frakg
\otimes_{\frakZ_{{\rm Fr}}} K_{\frakg} \xrightarrow{\sim} \calU \frakg
\otimes_{\frakZ_{{\rm Fr}}} \bk_{0},$ and hence an equivalence
of categories (see Proposition \ref{prop:qisequivalence}):
$\calD \Mod((\calU \frakg)_0) \ \cong \
  \calD({\rm Spec}(\bk), \, \calU \frakg
\otimes_{\frakZ_{{\rm Fr}}} K_{\frakg}).$ Restricting to the
subcategories of objects with finitely generated cohomology, we obtain
an equivalence: \begin{equation} \label{eq:Ug0} \calD^b \Mod^{\fg}((\calU
  \frakg)_0) \ \cong \ \calD^{\fg}({\rm Spec}(\bk), \, \calU \frakg
  \otimes_{\frakZ_{{\rm Fr}}} K_{\frakg}). \end{equation} In the sequel, we write $\calU \frakg \otimes_{\bk} \Lambda(\frakg^{(1)})$ for the dg-algebra $\calU \frakg
  \otimes_{\frakZ_{{\rm Fr}}} K_{\frakg}$.

We can consider $\calU \frakg$ as a sheaf of algebras either on ${\rm Spec}(\bk)$, or on ${\rm Spec}(\frakZ) \cong \frakg^*{}^{(1)}
\times_{\frakh^*{}^{(1)}/W} \frakh^*/(W,\bullet)$. It follows easily from
Proposition \ref{prop:dgmodulessupport} that the category
$\calD^{\qc,\fg}(\frakg^*{}^{(1)} \times_{\frakh^*{}^{(1)}/W}
\frakh^*/(W,\bullet), \, \calU \frakg \otimes_{\bk}
\Lambda(\frakg^{(1)}))$ is equivalent to
$\calD^{\fg}({\rm Spec}(\bk), \ \calU \frakg \otimes_{\bk}
\Lambda(\frakg^{(1)}))$. We denote this category simply by
$\calD^{\fg}(\calU \frakg \otimes_{\bk} \Lambda(\frakg^{(1)}))$. We
also denote by $\calD^{\fg}_{\lambda}(\calU \frakg \otimes_{\bk}
\Lambda(\frakg^{(1)}))$ the full subcategory whose objects are the
dg-modules $M$ such that $\calU \frakg$ acts on $H(M)$ with
generalized character $(\lambda,0)$. It also follows from Proposition
\ref{prop:dgmodulessupport} that this category is equivalent to the
localization of the homotopy category of finitely generated $\calU
\frakg \otimes_{\bk} \Lambda(\frakg^{(1)})$-dg-modules on which $\calU
\frakg$ acts with generalized character $(\lambda,0)$. We use the
same notation for $\calU \frakg$ instead of $\calU \frakg
\otimes_{\bk} \Lambda(\frakg^{(1)})$.

The following result follows easily from the definitions and
\cite[1.3.7]{BMR}.

\begin{lem} \label{lem:Ug0centralcharacter}

Equivalence {\rm \eqref{eq:Ug0}} restricts to an equivalence of
categories \[ \calD^b \Mod^{\fg}_{\lambda}((\calU
  \frakg)_0) \ \cong \ \calD^{\fg}_{\lambda}(\calU \frakg
  \otimes_{\bk} \Lambda(\frakg^{(1)})). \]

\end{lem}

Next, let us recall some results on dg-algebras. Let $A$ be a
dg-algebra (i.e.~a
sheaf of dg-algebras on ${\rm Spec}(\bk)$). We use the same
notation as in section \ref{sec:sectiondgalg}, except that we omit
``${\rm Spec}(\bk)$''. An $A$-dg-module $M$ is \emph{K-projective} if for any acyclic $A$-dg-module $N$,
the complex $\Hom_A(M,N)$ is acyclic. By the results
of \cite[\S 10.12.2]{BLEqu}, every $A$-dg-module has a left K-projective
resolution. As in \S \ref{ss:paragraphderivedfunctors}, we deduce that any triangulated functor from $\calC(A)$ to a triangulated category
has a left derived functor, which can be
computed by means of K-projective resolutions.

\begin{proof}[Proof of Theorem \ref{thm:localizationfixedFr}] We will show that the equivalences constructed in \cite{BMR} are
``compatible with the tensor product with $K_{\frakg}$''.

\emph{First step}: Let us prove the following equivalence of
categories: \begin{equation}\label{eq:equivstep1}
  \calD^{\qc,\fg}_{\calB^{(1)} \times \{ \lambda \} }(X, \, \calO_{X}
  \otimes_{\bk} \Lambda(\frakg^{(1)})) \ \cong \
  \calD^{\qc,\fg}_{\calB^{(1)} \times \{ \lambda \}}(X, \, \wcalD
  \otimes_{\bk} \Lambda(\frakg^{(1)})). \end{equation} As in \cite{BMR} we define the functors
\[F : \left\{ \begin{array}{ccc}
    \calC^{\qc,\fg}_{\calB^{(1)} \times \{ \lambda \}}(X, \, \calO_{X}
    \otimes_{\bk} \Lambda(\frakg^{(1)})) & \to &
    \calC^{\qc,\fg}_{\calB^{(1)} \times \{ \lambda \} }(X, \, \wcalD \otimes_{\bk}
    \Lambda(\frakg^{(1)})) \\ \calF & \mapsto & \calM^{\lambda}
    \otimes_{\calO_{\widehat{\calB^{(1)}}}} \calF \\
\end{array} \right. ,\] \[G : \left\{ \begin{array}{ccc}
  \calC^{\qc,\fg}_{\calB^{(1)} \times \{ \lambda \}}(X, \, \wcalD \otimes_{\bk}
  \Lambda(\frakg^{(1)})) & \to & \calC^{\qc,\fg}_{\calB^{(1)} \times
  \{ \lambda \} }(X, \, \calO_{X} \otimes_{\bk} \Lambda(\frakg^{(1)})) \\
  \calG & \mapsto & \sheafHom_{\wcalD}(\calM^{\lambda},\calG)
\end{array} \right. .\] These functors are exact. There are natural
morphisms of functors $F \circ G \to \Id$ and
$\Id \to G \circ F$. These functors and morphisms
of functors coincide with the ones considered in \cite[5.1.1]{BMR}
under the forgetful functors \begin{align*}
  \calC^{\qc,\fg}_{\calB^{(1)} \times \{ \lambda \}}(X, \, \calO_{X}
  \otimes_{\bk} \Lambda(\frakg^{(1)})) \ & \to \
  \calC^{\qc,\fg}_{\calB^{(1)} \times \{ \lambda \}}(X, \, \calO_{X})
  \cong \calC^b \Coh_{\calB^{(1)} \times \{ \lambda \}}(X) \\
    \calC^{\qc,\fg}_{\calB^{(1)} \times \{ \lambda \} }(X, \, \wcalD
  \otimes_{\bk} \Lambda(\frakg^{(1)})) \ & \to \
  \calC^{\qc,\fg}_{\calB^{(1)} \times
   \{ \lambda \}}(X, \, \wcalD) \cong \calC^b \Mod^{{\rm
     c}}_{(\lambda,0)}(\wcalD). \end{align*} Hence, by
\cite[5.1.1]{BMR}, the morphisms $F \circ G \to \Id$ and $\Id \to G \circ F$ are isomorphisms on each object, hence $F$ and $G$ are equivalences of
categories. They induce equivalence \eqref{eq:equivstep1}, using Proposition \ref{prop:dgmodulessupport}.

Thus, combining \eqref{eq:DGC}, Lemma \ref{lem:lemmaX} and
\eqref{eq:equivstep1}, we have obtained: \begin{equation}
  \label{eq:equivstep1big} \DGCoh((\wfrakg \, \rcap_{\frakg^* \times \calB} \,
  \calB)^{(1)}) \ \cong \ \calD^{\qc,\fg}_{\calB^{(1)} \times \{
    \lambda \}}(X, \, \wcalD \otimes_{\bk}
  \Lambda(\frakg^{(1)})). \end{equation}

\emph{Second step}: Now we construct an equivalence of categories
\begin{equation} \label{eq:equivstep2}
  \calD^{\qc,\fg}_{\calB^{(1)} \times \{ \lambda \}}(X, \, \wcalD
\otimes_{\bk} \Lambda(\frakg^{(1)})) \ \cong \
\calD^{\fg}_{\lambda}(\calU \frakg \otimes_{\bk}
\Lambda(\frakg^{(1)})). \end{equation} This will conclude the proof, due to the following chain of equivalences: \begin{multline*} \DGCoh((\wfrakg \, \rcap_{\frakg^*
    \times \calB} \, \calB)^{(1)}) \
  \overset{\eqref{eq:equivstep1big}}{\cong} \
  \calD^{\qc,\fg}_{\calB^{(1)} \times \{ \lambda \}}(X, \, \wcalD
  \otimes_{\bk} \Lambda(\frakg^{(1)})) \\
  \overset{\eqref{eq:equivstep2}}{\cong} \ \calD^{\fg}_{\lambda}(\calU
  \frakg \otimes_{\bk} \Lambda(\frakg^{(1)})) \
  \overset{\ref{lem:Ug0centralcharacter}}{\cong} \ \calD^b
  \Mod^{\fg}_{\lambda}((\calU \frakg)_0). \end{multline*}

By the projection formula (\cite[II.Ex.5.1]{HARAG}) and \cite[3.4.1]{BMR}, we have \[ \Gamma(\wcalD
\otimes_{\bk} \Lambda(\frakg^{(1)})) \ \cong \ \Gamma(\wcalD)
\otimes_{\bk} \Lambda(\frakg^{(1)}) \ \cong \ \widetilde{U}
\otimes_{\bk} \Lambda(\frakg^{(1)}) \] where $\widetilde{U}:=\calU
\frakg \otimes_{\frakZ_{{\rm HC}}} S(\frakh)$. The dg-algebra $\widetilde{U}
\otimes_{\bk} \Lambda(\frakg^{(1)})$ contains $\calU \frakg \otimes_{\bk}
\Lambda(\frakg^{(1)})$ as a sub-dg-algebra. Hence (see \S \ref{ss:directinverseimage}) there exists a functor $ R\Gamma :
\calD(X, \, \wcalD \otimes_{\bk} \Lambda(\frakg^{(1)})) \to
\calD({\rm Spec}(\bk), \, \calU \frakg \otimes_{\bk}
\Lambda(\frakg^{(1)})).$ By Corollary \ref{cor:directimagediagram}, the following diagram commutes:
\begin{equation}\label{eq:diagram-RGamma-For}
\xymatrix@R=16pt{ \calD(X, \, \wcalD \otimes_{\bk} \Lambda(\frakg^{(1)}))
\ar[r]^-{R\Gamma} \ar[d]_-{\For} & \calD({\rm Spec}(\bk), \, \calU \frakg
\otimes_{\bk} \Lambda(\frakg^{(1)})) \ar[d]^-{\For} \\
\calD(X, \, \wcalD) \ar[r]^-{R\Gamma} & \calD({\rm Spec}(\bk), \, \calU \frakg). } \end{equation}

By Proposition \ref{prop:dgmodulessupport}, the
functor $\calD^b \Mod^{{\rm c}}_{(\lambda,0)}(\wcalD) \to
\calD^{\qc,\fg}_{\calB^{(1)} \times \{ \lambda \}}(X, \, \wcalD)$ is an
equivalence of categories. If $\calF$ is an object of the category
$\calD^{\qc,\fg}_{\calB^{(1)} \times \{ \lambda \}}(X, \, \wcalD
\otimes_{\bk} \Lambda(\frakg^{(1)}))$, then $\For(\calF)$ is in
$\calD^{\qc,\fg}_{\calB^{(1)} \times \{ \lambda \}}(X, \, \wcalD) \cong
\calD^b \Mod^{{\rm c}}_{(\lambda,0)}(\wcalD)$. Hence, by \cite[3.1.9]{BMR},
$R\Gamma(\For(\calF))$ is in
$\calD^{\fg}_{\lambda}(\calU \frakg)$. Using diagram
\eqref{eq:diagram-RGamma-For}, we deduce that $R\Gamma(\calF)$ is in
$\calD^{\fg}_{\lambda}(\calU \frakg \otimes_{\bk}
\Lambda(\frakg^{(1)}))$. Hence we have proved that $R\Gamma$ induces a
functor $R\Gamma : \calD^{\qc,\fg}_{\calB^{(1)} \times \{ \lambda
  \}}(X, \, \wcalD
\otimes_{\bk} \Lambda(\frakg^{(1)})) \to \calD^{\fg}_{\lambda}(\calU
\frakg \otimes_{\bk} \Lambda(\frakg^{(1)}))$, such that the following
diagram commutes: \begin{equation} \label{eq:diagramRgamma}
\xymatrix@R=16pt{ \calD^{\qc,\fg}_{\calB^{(1)} \times \{ \lambda \}}(X, \, \wcalD
  \otimes_{\bk} \Lambda(\frakg^{(1)})) \ar[r]^-{R\Gamma}
  \ar[d]_-{\For} & \calD^{\fg}_{\lambda}(\calU \frakg
  \otimes_{\bk} \Lambda(\frakg^{(1)})) \ar[d]^-{\For} \\ \calD^b
  \Mod^{c}_{(\lambda,0)}(\wcalD) \ar[r]^-{R\Gamma} & \calD^b
  \Mod^{\fg}_{(\lambda,0)}(\calU \frakg). } \end{equation}

Now we construct a left adjoint for this functor. First,
consider \[ {\rm Loc}_K : \left\{ \begin{array}{ccc}
  \calC({\rm Spec}(\bk), \, \calU \frakg \otimes_{\bk}
  \Lambda(\frakg^{(1)})) & \to & \calC(X, \, \wcalD \otimes_{\bk}
  \Lambda(\frakg^{(1)})) \\ M & \mapsto & \wcalD \otimes_{\calU
    \frakg} M \end{array} \right. .\] Using the remarks after Lemma \ref{lem:Ug0centralcharacter}, ${\rm Loc}_K$ admits a left derived functor $\calL_K$. Moreover, the following diagram is
commutative: \begin{equation}\label{eq:diagram-Loc-For} \xymatrix@R=16pt{
\calD({\rm Spec}(\bk), \, \calU \frakg \otimes_{\bk}
\Lambda(\frakg^{(1)})) \ar[r]^-{\calL_K} \ar[d]_-{\For} & \calD(X,
\, \wcalD \otimes_{\bk} \Lambda(\frakg^{(1)})) \ar[d]^-{\For} \\
\calD({\rm Spec}(\bk), \, \calU \frakg) \ar[r]^-{\wcalD \lotimes_{\calU
    \frakg} -} & \calD(X, \, \wcalD). } \end{equation} Indeed, both derived
functors can be computed using K-projective resolutions, and one easily checks, using coinduction, that every
K-projective $\calU \frakg \otimes_{\bk} \Lambda(\frakg^{(1)})$-dg-module
restricts to a K-projective complex of $\calU \frakg$-modules.

Using diagram \eqref{eq:diagram-Loc-For}, $\calL_K$ induces a functor $\calL_K : \calD^{\fg}(\calU \frakg
\otimes_{\bk} \Lambda(\frakg^{(1)})) \ \to \ \calD^{\qc,\fg}(X, \, \wcalD
\otimes_{\bk} \Lambda(\frakg^{(1)})).$ Moreover, for
any $M$ in $\calD^{\fg}_{\lambda}(\calU \frakg
\otimes_{\bk} \Lambda(\frakg^{(1)}))$ there is a canonical
decomposition $\calL_K(M) \cong \bigoplus_{\mu \in W \bullet \lambda}
\calL_K^{\lambda \to \mu}(M)$ with $\calL_K^{\lambda \to \mu}(M)$ in
$\calD^{\qc,\fg}_{\calB^{(1)} \times \{ \mu \} }(X, \wcalD \otimes_{\bk}
\Lambda(\frakg^{(1)}))$. Indeed, using Proposition
\ref{prop:dgmodulessupport}, we have such a decomposition as a
complex of $\wcalD$-modules (as in \cite[3.3.1]{BMR}). As the actions
of $\Lambda(\frakg^{(1)})$ and ${\rm S}(\frakh) \subset \wcalD$ commute, each summand
is a sub-$\wcalD \otimes_{\bk} \Lambda(\frakg^{(1)})$-dg-module. 

Now we define
$\calL^{\widehat{\lambda}}_{K}:=\calL_K^{\lambda \to \lambda}$. The following diagram commutes, where $\calL^{\widehat{\lambda}}$ is the functor defined in
\cite[3.3.1]{BMR}:
\begin{equation}\label{eq:diagram-LK-For} \xymatrix@R=16pt{
\calD^{\fg}_{\lambda}(\calU \frakg \otimes_{\bk}
\Lambda(\frakg^{(1)})) \ar[r]^-{\calL^{\widehat{\lambda}}_K}
\ar[d]_-{\For} & \calD^{\qc,\fg}_{\calB^{(1)} \times \{ \lambda \}}(X, \,
\wcalD \otimes_{\bk} \Lambda(\frakg^{(1)})) \ar[d]^-{\For}
\\ \calD^b \Mod^{\fg}_{(\lambda,0)}(\calU \frakg)
\ar[r]^-{\calL^{\widehat{\lambda}}} & \calD^{b}
\Mod^{{\rm c}}_{(\lambda,0)}(\wcalD). }
\end{equation}

As in \cite[3.3.2]{BMR}, the functors
$\calL_K^{\widehat{\lambda}}$, $R\Gamma$ are adjoint.
Hence there are morphisms $\Id \to R\Gamma \circ
\calL_K^{\widehat{\lambda}}$, $\calL_K^{\widehat{\lambda}}
\circ R\Gamma \to \Id$ which coincide, under the forgetful
functors, with the morphisms $\Id \to
R\Gamma \circ \calL^{\widehat{\lambda}}$, $\calL^{\widehat{\lambda}}
\circ R\Gamma \to \Id$ of \cite{BMR}. In \cite[3.6]{BMR} it is
proved that the latter morphisms are isomorphisms. Hence the former
morphisms also are isomorphisms, which proves \eqref{eq:equivstep2}.\end{proof}

We denote by $\widehat{\gamma}^{\calB}_{\lambda}$ the equivalence constructed in the proof of Theorem \ref{thm:localizationfixedFr}. Let $p: (\wfrakg \, \rcap_{\frakg^* \times \calB} \, \calB)^{(1)} \to
\wfrakg^{(1)}$ be the natural morphism of dg-schemes, which can be
realized as the morphism of dg-ringed spaces $(\wfrakg^{(1)}, \,
\calO_{\wfrakg^{(1)}} \otimes_{\bk} \Lambda(\frakg^{(1)})) \ \to \ (\wfrakg^{(1)}, \,
\calO_{\wfrakg^{(1)}}).$ The following proposition is clear from our
constructions (see diagrams \eqref{eq:diagramRgamma} and
\eqref{eq:diagram-LK-For}):

\begin{prop} \label{prop:loccompatibility}

The following diagram is commutative, where the functor ${\rm Incl}$
is induced by the inclusion $\Mod^{\fg}_{\lambda}((\calU \frakg)_0)
\hookrightarrow \Mod^{\fg}_{(\lambda,0)}(\calU \frakg)$: \[ \xymatrix@R=16pt{
  \DGCoh((\wfrakg \, \rcap_{\frakg^* \times
    \calB} \, \calB)^{(1)}) \ar[d]_-{Rp_*}
  \ar[rrr]_-{\sim}^-{\widehat{\gamma}^{\calB}_{\lambda} \ \mathrm{(Th. \ \ref{thm:localizationfixedFr})}} & & & \calD^b
  \Mod^{\fg}_{\lambda}((\calU \frakg)_0) \ar[d]^-{{\rm Incl}}
   \\ \calD^b
  \Coh_{\calB^{(1)}}(\wfrakg^{(1)}) \ar[rrr]^-{\gamma^{\calB}_{\lambda} \ \mathrm{(Th. \ \ref{thm:thmBMR})}}_-{\sim} & & & \calD^b
  \Mod^{\fg}_{(\lambda,0)}(\calU \frakg). \\ } \] 

\end{prop}

To finish this subsection, let us remark that similar
arguments give the following more general theorem:

\begin{thm}\label{thm:localizationfixedFrparabolic}

Let $\mu, \calP$ be as in Theorem
{\rm \ref{thm:thmBMR}}$\rmii$. There exists an equivalence of triangulated categories
\[ \widehat{\gamma}^{\calP}_{\mu}: \DGCoh((\wfrakg_{\calP}
\, \rcap_{\frakg^* \times \calP} \, \calP)^{(1)}) \ \xrightarrow{\sim} \ \calD^b
\Mod^{\fg}_{\mu}((\calU \frakg)_0) \] making the following diagram
commute, where ${\rm Incl}$ is induced by the inclusion
$\Mod^{\fg}_{\mu}((\calU \frakg)_0) \hookrightarrow \Mod^{\fg}_{(\mu,0)}(\calU
\frakg)$, and $p_{\calP}: (\wfrakg_{\calP} \, \rcap_{\frakg^* \times \calP} \,
\calP)^{(1)} \to \wfrakg_{\calP}^{(1)}$ is the natural morphism of
dg-schemes: \[ \xymatrix@R=16pt{ \DGCoh((\wfrakg_{\calP} \, \rcap_{\frakg^* \times
\calP} \, \calP)^{(1)}) \ar[d]_-{R(p_{\calP})_*}
\ar[rr]_-{\sim}^-{\widehat{\gamma}^{\calP}_{\mu}} & & \calD^b
\Mod^{\fg}_{\mu}((\calU \frakg)_0)
\ar[d]^-{{\rm Incl}} \\ \calD^b
\Coh_{\calP^{(1)}}(\wfrakg_{\calP}^{(1)}) \ar[rr]^-{\gamma^{\calP}_{\mu}}_-{\sim} & & \calD^b
\Mod^{\fg}_{(\mu,0)}(\calU \frakg). \\ } \] 

\end{thm}

\begin{remark}\label{rk:endsection3}

$\rmi$ Theorems \ref{thm:localizationfixedFr} and \ref{thm:localizationfixedFrparabolic} easily generalize to the case where $0 \in \frakg^*$ is relaced by a nilpotent $\chi$ (with the same proof). One obtains the following result, which will not be used in this paper.

\begin{thm}

Let $\mu, \calP$ be as in Theorem
{\rm \ref{thm:thmBMR}}$\rmii$, and $\chi \in \frakg^*$ nilpotent. There is an equivalence $\DGCoh \bigl((\wfrakg_{\calP}
\, \rcap_{\frakg^* \times \calP} \, (\chi \times \calP))^{(1)}\bigr) \cong \calD^b
\Mod^{\fg}_{\mu}((\calU \frakg)_{\chi}).$ 

\end{thm}

$\rmii$ One cannot use similar methods to describe categories of $(\calU \frakg)^{\lambda}_{\chi}$-modules, because $(\calU \frakg)^{\lambda}$ is not flat over $\frakZ_{\mathrm{Fr}}$.

\end{remark}

\section{Simples correspond to projective covers under $\kappa_{\calB}$} \label{sec:statement}

Recall the equivalence $\kappa_{\calB}$ of \eqref{eq:defkappa}. Our
situation is the following: \[
(*) \xymatrix@R=16pt{ & \DGCoh^{\gr}(\wcalN^{(1)})
  \ar[r]^-{\sim}_-{\kappa_{\calB}}
  \ar[d]^-{\eqref{eq:forgetDGCoh(F)}}_-{\For}
& \DGCoh^{\gr}((\wfrakg \, \rcap_{\frakg^* \times \calB} \, \calB)^{(1)})
\ar[d]^-{\For}_-{\eqref{eq:ForDGCohderivedintersection}} \\ \calD^b
\Coh_{\calB^{(1)}}(\wcalN^{(1)}) \ar@{^{(}->}[r]
\ar@{<->}[d]_-{\wr}^-{\eqref{eq:equivBMR2}} & \calD^b \Coh(\wcalN^{(1)}) &
\DGCoh((\wfrakg \, \rcap_{\frakg^* \times \calB} \, \calB)^{(1)})
\ar@{<->}[d]^-{\wr}_-{\ref{thm:localizationfixedFr}} \\ \calD^b
\Mod^{\fg}_{0}((\calU \frakg)^{\lambda}) & & \calD^b
\Mod^{\fg}_{\lambda}((\calU \frakg)_0) } \]

Hence we have constructed some ``correspondence'' between
$\calU \frakg$-modules with fixed trivial Frobenius character
and generalized Harish-Chandra character $\lambda$ (on the RHS), and $\calU \frakg$-modules with generalized trivial Frobenius
character and fixed Harish-Chandra character $\lambda$ (on the LHS). The key of our approach to Koszulity is that, if $p$ is large enough, ``indecomposable
projective modules correspond to simple modules'' under this
correspondence. In this section we state precisely this result (see Theorem \ref{thm:mainthm}).

\subsection{Restricted dominant weights}
\label{ss:paragraphrestrictedweights}

Consider $\tau_0:=t_{\rho} \cdot w_0 \in W_{\aff}'.$
Recall the formula
for the length in $W_{\aff}'$: for $w
\in W$ and $x \in \bbX$,
\begin{equation}\label{eq:formulaIM'} \ell(w
  \cdot t_x)=\sum_{\alpha \in R^+ \cap w^{-1} (R^+)} |\langle x,
  \alpha^{\vee} \rangle | + \sum_{\alpha \in
      R^+ \cap w^{-1}(R^-)}
  |1 + \langle x, \alpha^{\vee} \rangle | \end{equation} (see \cite[1.23]{IM}). In
particular, we obtain $\ell(\tau_0)=\sum_{\alpha \in R^+} (\langle
\rho, \alpha^{\vee} \rangle - 1)$.

Recall the definition of $C_0$ in \S \ref{ss:intro3}. Let us define \[ W^0:=\{ w \in W_{\aff}' \mid w \bullet C_0 \ \text{contains a restricted dominant weight} \}.\] If $\lambda \in C_0$, $W^0$ is also the set of $w \in W_{\aff}'$ such that $w \bullet \lambda$ is restricted dominant. It is a finite set, in bijection with $W$ (see \eqref{eq:restricteddominantweights} below).

\begin{prop} \label{prop:tau_0}

The map $w \mapsto \tau_0 w$ is an involution of
$W^0$. Moreover, if $w \in W^0$ we have $\ell(\tau_0 w)=\ell(\tau_0) -
\ell(w)$.

\end{prop}

\begin{proof} By definition, $(\tau_0)^2=1$. Hence for the first assertion it is
sufficient to prove that $w \in W^0$ implies $\tau_0 w \in W^0$. We have $w \in W^0$ iff $w \bullet 0$ is restricted dominant. Write $w=t_{\lambda} \cdot v$ with $\lambda \in \bbX$ and $v \in W$. Then $w \bullet 0=v(\rho) + p\lambda
- \rho$. Hence if $\alpha \in \Phi$, $\langle w \bullet 0,
\alpha^{\vee} \rangle = \langle \rho, (v^{-1} \alpha)^{\vee} \rangle +
p \langle \lambda, \alpha^{\vee} \rangle - 1$. As $p > h$, we have $|
\langle \rho, (v^{-1} \alpha)^{\vee} \rangle | < p$. Hence, $w \bullet
0$ dominant restricted implies:
\begin{equation}\label{eq:restricteddominantweights} \langle \lambda,
  \alpha^{\vee} \rangle = \left\{ \begin{array}{cl} 0 & \text{if} \
      v^{-1} \alpha \in R^+; \\ 1 & \text{if} \ v^{-1} \alpha \in
      R^-. \end{array} \right. \end{equation} In both cases, $\langle
w \bullet 0, \alpha^{\vee} \rangle \in \{0, 1, \cdots, p-2 \}$. Now $\tau_0 w \bullet 0=w_0(w \bullet 0 + \rho) + (p-1)\rho = w_0( w
\bullet 0) + (p-2)\rho$. Hence if $\alpha \in \Phi$, $\langle \tau_0 w
\bullet 0, \alpha^{\vee} \rangle = \langle w \bullet 0, (w_0
\alpha)^{\vee} \rangle + (p-2)$. We have $w_0 \alpha \in -\Phi$,
hence $\langle w \bullet 0, (w_0
\alpha)^{\vee} \rangle \in \{-p+2, \cdots, 0 \}$. Thus $\tau_0 w \in W^0$, and the first assertion of the proposition follows.

Let us compute $\ell(\tau_0 w)$. We have $\tau_0 w=w_0 v \cdot
t_{v^{-1}(\lambda - \rho)}$. Hence, by \eqref{eq:formulaIM'},
\[ \ell(\tau_0 w) \ = \ 
  \sum_{\alpha \in R^+ \cap v^{-1} R^-} |
  \langle \lambda - \rho, (v \alpha)^{\vee} \rangle | +
  \sum_{\alpha \in R^+ \cap v^{-1} R^+} |1 +
  \langle \lambda - \rho, (v \alpha)^{\vee} \rangle |. \]
By \eqref{eq:restricteddominantweights}, for
$\alpha \in \Phi$ we have $0 \leq \langle \lambda, \alpha^{\vee}
\rangle \leq \langle \rho, \alpha^{\vee} \rangle$. Hence the same is
true for $\alpha \in R^+$. Moreover, if $v^{-1} \alpha \in R^+$
then the second inequality is strict, and if $v^{-1} \alpha \in R^-$
the first one is strict. Hence $\ell(\tau_0 w)$ equals
{\small \begin{multline*}
  \sum_{\alpha \in R^+ \cap v^{-1} R^-}
  \langle \lambda - \rho, (v \alpha)^{\vee} \rangle +
  \sum_{\alpha \in R^+ \cap v^{-1} R^+} (-1 +
  \langle \rho - \lambda, (v \alpha)^{\vee} \rangle) \\ =
  \sum_{\beta \in R^+} \langle \rho, \beta^{\vee} \rangle +
  \sum_{\alpha \in R^+ \cap v^{-1} R^-}
  \langle \lambda, v \alpha^{\vee} \rangle -
  \sum_{\alpha \in R^+ \cap v^{-1} R^+}
  \langle \lambda, v \alpha^{\vee} \rangle - \# ( R^+ \cap v^{-1} R^+ ). \end{multline*}}We deduce that {\small \begin{align*} \ell(\tau_0 w) & = \ell(\tau_0) +
  \sum_{\alpha \in R^+ \cap v^{-1} R^-}
  \langle \lambda, v \alpha^{\vee} \rangle -
  \sum_{\alpha \in R^+ \cap v^{-1} R^+}
  \langle \lambda, v \alpha^{\vee} \rangle + \# (R^+ \cap v^{-1} R^-) \\ & = \ell(\tau_0) -  \sum_{\alpha \in R^+ \cap v^{-1} R^-} | 1 +
  \langle \lambda, v \alpha^{\vee} \rangle| -
  \sum_{\alpha \in R^+ \cap v^{-1} R^+} |
  \langle \lambda, v \alpha^{\vee} \rangle |. \end{align*}}Here the second equality uses the
fact that if $\alpha \in R^+$ and $v\alpha \in R^-$ then $\langle
\lambda, (v \alpha)^{\vee}
\rangle < 0$ (see above). We conclude using the equality $w=v \cdot
t_{v^{-1} \lambda}$. \end{proof}

\subsection{Coherent sheaves and dg-sheaves on $\wcalN^{(1)}$} \label{ss:paragraphsimplemodules}

As in \S\S \ref{ss:lkd} and \ref{ss:intro3}, let us consider the following
$\Gm$-dg-algebras on $\calB^{(1)}$:
\[ \begin{array}{cl} \calS:={\rm
    S}_{\calO_{\calB^{(1)}}}(\calT_{\calB^{(1)}}) & \text{with }
  \calT_{\calB^{(1)}} \text{ in bidegree } (2,-2), \\ \calR:={\rm
    S}_{\calO_{\calB^{(1)}}}(\calT_{\calB^{(1)}}) & \text{with }
  \calT_{\calB^{(1)}} \text{ in bidegree } (0,-2). \end{array} \] We
have a ``regrading'' functor $\xi :
\calD_{\Gm}(\calB^{(1)}, \, \calS) \ \xrightarrow{\sim} \
\calD_{\Gm}(\calB^{(1)}, \, \calR)$, defined by
$\xi(M)^i_j=M^{i-j}_j$. We also have an equivalence of categories (see
\eqref{eq:equivalenceCoh(E)}): $\phi :
\calD^{\qc,\fg}_{\Gm}(\calB^{(1)}, \, \calR) \ \xrightarrow{\sim} \ \calD^b
\Coh^{\Gm}(\wcalN^{(1)}).$ As in \eqref{eq:GmF} we consider the functor $\eta :
\DGCoh^{\gr}(\wcalN^{(1)}) \ \to \ \calD^b \Coh^{\Gm}(\wcalN^{(1)})$
defined as the composition \begin{multline*} \DGCoh^{\gr}(\wcalN^{(1)}):=
\calD^{+,\qc,\fg}_{\Gm}(\calB^{(1)}, \, \calS) \ \to \
\calD^{\qc,\fg}_{\Gm}(\calB^{(1)}, \, \calS) \\ \xrightarrow{\xi} \
\calD^{\qc,\fg}_{\Gm}(\calB^{(1)}, \, \calR) \ \xrightarrow{\phi} \ \calD^b
\Coh^{\Gm}(\wcalN^{(1)}). \end{multline*}

\begin{lem}\label{lem:lemmagradedmodules}

There exists a fully faithful triangulated functor \[\zeta: \calD^b \Coh^{\Gm}_{\calB^{(1)}}(\wcalN^{(1)}) \ \to \ \DGCoh^{\gr}(\wcalN^{(1)})\] such that $\eta \circ \zeta$ is the inclusion\footnote{See \cite[3.1.7]{BMR} for the fact that this functor is an inclusion.} $\calD^b \Coh^{\Gm}_{\calB^{(1)}}(\wcalN^{(1)}) \hookrightarrow \calD^b \Coh^{\Gm}(\wcalN^{(1)})$.

\end{lem}

\begin{proof} The objects of $\calD^b
\Coh^{\Gm}_{\calB^{(1)}}(\wcalN^{(1)})$ are bounded complexes of
$\Gm$-equi\-variant coherent sheaves on $\wcalN^{(1)}$, supported on $\calB^{(1)}$. In particular, they are
bounded for both gradings (cohomological and internal). Consider the functor $\zeta: \calC^b
\Coh^{\Gm}_{\calB^{(1)}}(\wcalN^{(1)}) \ \to \ \DGCoh^{\gr}(\wcalN^{(1)})$
sending $M$ to the dg-module defined by
$\zeta(M)^i_j:=M^{i+j}_j$. This functor induces $\zeta: \calD^b
\Coh^{\Gm}_{\calB^{(1)}}(\wcalN^{(1)}) \to
\DGCoh^{\gr}(\wcalN^{(1)})$, such that $\eta \circ
\zeta$ is the inclusion of the full subcategory
$\calD^b \Coh^{\Gm}_{\calB^{(1)}}(\wcalN^{(1)}) \subset \calD^b
\Coh^{\Gm}(\wcalN^{(1)})$. Hence $\zeta$ is
faithful. 

Now we show that $\zeta$ is full. Let $M$, $N$ be in $\calD^b
\Coh_{\calB^{(1)}}^{\Gm}(\wcalN^{(1)})$. A morphism $f : \zeta(M)
\to \zeta(N)$ in $\DGCoh^{\gr}(\wcalN^{(1)})$ can be
represented by a diagram \[ \zeta(M) \xleftarrow{\qis} P
\longrightarrow \zeta(N) \] with $P$ in
$\DGCoh^{\gr}(\wcalN^{(1)})$. Fix a positive integer
$a$ such that $M_j=N_j=0$ for $|j| \geq a$. We define
the sub-dg-module $P^{[1]}$ of $P$ by $(P^{[1]})_j=P_j$ if
$j < a$, $(P^{[1]})_j=0$ if $j \geq a$. The inclusion $P^{[1]} \hookrightarrow
P$ is a quasi-isomorphism. Next, define the sub-dg-module $P^{[2]}$ of
$P^{[1]}$ by $(P^{[2]})_j=(P^{[1]})_j$ if $j \leq - a$,
$(P^{[2]})_j=0$ if $j > -a$, and denote by $P^{[3]}$ the quotient
$P^{[1]}/P^{[2]}$. Then $P^{[1]} \to
P^{[3]}$ is again a quasi-isomorphism. Moreover, we have the diagram
\begin{equation} \label{eq:diagrammorphism} \xymatrix@R=12pt{ & P \ar[ld]_-{\qis} \ar[rd] & \\
\zeta(M) & P^{[1]} \ar[l]_-{\qis} \ar[r] \ar[d]^-{\qis}
\ar[u]_-{\qis} & \zeta(N). \\ & P^{[3]} \ar[ul]^-{\qis} \ar[ur]
& \\ } \end{equation} Hence we can assume that
$P$ is bounded for the internal grading. Going to $\calC_{\Gm}(\calB^{(1)}, \, \calR)$ (via $\xi$), using a truncation functor and then going back to $\calC_{\Gm}(\calB^{(1)}, \, \calS)$, one can even assume that $P$ is bounded for both gradings. 

Consider now the
morphism $\phi^{-1} \eta(f): \phi^{-1} M \to \phi^{-1} N$ in
$\calD^{\qc,\fg}_{\Gm}(\calB^{(1)}, \, \calR)$. As $\calD^b
\Coh^{\Gm}_{\calB^{(1)}}(\wcalN^{(1)})$ is a full subcategory of
$\calD^b \Coh^{\Gm}(\wcalN^{(1)})$, there exists a diagram similar to \eqref{eq:diagrammorphism} in
$\calC_{\Gm}(\calB^{(1)}, \, \calR)$, with $\zeta(M)$ and $\zeta(N)$ replaced by $\phi^{-1} M$ and $\phi^{-1} N$, $P$ by $\phi^{-1} \eta(P)$, $P^{[1]}$ by an object $Q^{[1]}$ of
$\calD^{\qc,\fg}_{\Gm}(\calB^{(1)}, \, \calR)$, and $P^{[3]}$ by $\phi^{-1} Q^{[2]}$, where $Q^{[2]}$ is a
bounded complex of $\Gm$-equivariant coherent sheaves on
$\wcalN^{(1)}$, supported on the zero section. Now,
as above, we can assume $Q^{[1]}$
is bounded for the internal grading, and bounded below for the
cohomological one. It follows that $f$ is
the image under $\zeta$ of the morphism defined by the diagram $M
\xleftarrow{\qis} Q^{[2]} \rightarrow N$. \end{proof}

\subsection{Translation functors}\label{ss:paragraphtranslationfunctors}

The translation functors for $\calU \frakg$-modules are defined in
\cite[6.1]{BMR}. In this subsection we prove, in particular cases
sufficient for our purposes, that these
translation functors (for $\calU \frakg$-modules) coincide (on
$G$-modules) with the usual translation functors defined e.g.~in
\cite[II.7]{JANAlg}. We denote by $T_{\lambda}^{\mu}$ the
translation functors defined in \cite{BMR}, and by
$\hat{T}_{\lambda}^{\mu}$ the ones defined in \cite{JANAlg}. We
also denote by $\Mod^{\fd}_{\lambda}(G)$ the category of finite
dimensional $G$-modules in the block of $\lambda$, for $\lambda
\in \bbX$. Consider the ``closure'' $\overline{C}_0:=\{ \nu \in \bbX \mid \forall \alpha \in R^+, \ 0 \leq
\langle \nu + \rho, \alpha^{\vee} \rangle \leq p\}.$

\begin{lem}\label{lem:translationfunctors}

Let $\lambda, \mu \in \overline{C}_0$. Consider the
following diagram: \[ \xymatrix@R=16pt{\Mod^{\fd}_{\lambda}(G)
\ar@<0.5ex>[r]^-{\hat{T}_{\lambda}^{\mu}} \ar[d]_{\For} &
\Mod^{\fd}_{\mu}(G)
\ar@<0.5ex>[l]^-{\hat{T}^{\lambda}_{\mu}} \ar[d]^{\For} \\
\Mod^{\fg}_{(\lambda,0)}(\calU \frakg)
\ar@<0.5ex>[r]^-{T_{\lambda}^{\mu}} &
\Mod^{\fg}_{(\mu,0)}(\calU \frakg).
\ar@<0.5ex>[l]^-{T^{\lambda}_{\mu}} } \] If $\mu$ is in the closure of the facet of $\lambda$, then $\For \circ
\hat{T}_{\lambda}^{\mu} \cong T_{\lambda}^{\mu} \circ \For$. If
$\lambda$ is regular, and $\mu$ is on exactly one wall of $\overline{C}_0$,
then $\For \circ \hat{T}^{\lambda}_{\mu} \cong
T^{\lambda}_{\mu} \circ \For$.

\end{lem}

\begin{proof} We only give the proof of the first isomorphism. Both translation functors are constructed by
tensoring with a certain module, and then taking
a direct summand. A priori the direct summand corresponding to
$\hat{T}_{\lambda}^{\mu}$ is smaller than the one corresponding to
$T_{\lambda}^{\mu}$. Hence there exists a morphism of functors
$\For \circ \hat{T}_{\lambda}^{\mu} \to T_{\lambda}^{\mu} \circ
\For$. As these functors
are exact, and as $\Mod^{\fd}_{\lambda}(G)$ is generated
by the modules $\Ind_B^G(w \bullet \lambda)$ for $w \in W_{\aff}$ and
$w \bullet \lambda$ dominant we only have to
prove the result for these modules. But the images under our functors
of these modules are explicitly known (see \cite[II.7.11 and
II.7.12]{JANAlg} and \cite[2.2.3]{BMR2}), and they indeed coincide. \end{proof}

{F}rom now on, for simplicity we omit the functors ``$\For$''. Using this lemma, the usual rules to compute translates
of simple or induced modules (\cite[II.7]{JANAlg}) generalize. If $\mu$ is in the closure
of the facet of $\lambda$ (in $\overline{C}_0$),
$T_{\lambda}^{\mu} \Ind_B^G(w \bullet \lambda)=\Ind_B^G(w \bullet
\mu)$ for $w \in W_{\aff}'$. If $w \bullet \lambda$ is
dominant restricted, \begin{equation}
  \label{eq:translationsimples} T_{\lambda}^{\mu} L(w \bullet \lambda)
= \left\{ \begin{array}{ll} L(w \bullet \mu) & \text{if } w \bullet \mu
    \text{ is in the upper closure} \\ & \text{of the facet of } w \bullet
    \lambda; \\ 0
    & \text{otherwise}. \end{array} \right. \end{equation}

To finish this subsection, let us remark that, as the tensor product of
two restricted $\calU \frakg$-modules is again restricted, for $\lambda,\mu$
in $\bbX$ the functor $T_{\lambda}^{\mu} :
\Mod^{\fg}_{(\lambda,0)}(\calU \frakg)
\to \Mod^{\fg}_{(\mu,0)}(\calU \frakg)$ induces a functor denoted
similarly: $$T_{\lambda}^{\mu} : \Mod^{\fg}_{\lambda}((\calU
\frakg)_0) \ \to \ \Mod^{\fg}_{\mu}((\calU \frakg)_0).$$

\subsection{Objects corresponding to simple and projective modules} \label{ss:objectsLwPw}

Let $\lambda \in C_0$. By a theorem of Curtis (see \cite{CUR}), a complete system
of simple $(\calU \frakg)_0$-modules is given by the
restriction to $(\calU \frakg)_0$ of the simple $G$-modules
$L(\nu)$ for $\nu \in \bbX$ restricted and dominant. The simple objects in the category $\Mod^{\fg}_{0}((\calU \frakg)^{\lambda})$ (or similarly in $\Mod^{\fg}_{\lambda}((\calU \frakg)_0)$), i.e.~the simple $(\calU \frakg)_0$-modules with Harish-Chandra central character $\lambda$ are the $L(w \bullet \lambda)$, for $w \in W_{\aff}'$ such that $w \bullet \lambda$ is restricted dominant, i.e.~for $w \in W^0$ (see \S \ref{ss:paragraphrestrictedweights}).

Recall the equivalence $\epsilon^{\calB}_{\lambda}$ of \eqref{eq:equivBMR2}. For $w \in W^0$ we define \begin{equation} \label{eq:defLw} \calL_w:=(\epsilon^{\calB}_{\lambda})^{-1} L(w \bullet \lambda) \quad \in \calD^b \Coh_{\calB^{(1)}}(\wcalN^{(1)}) \end{equation} This object does not depend on the choice of $\lambda \in C_0$. Indeed, let $\mu \in C_0$. By \cite[6.1.2.(a)]{BMR},
for any $\calF \in \calD^b \Coh_{\calB^{(1)}}(\wfrakg^{(1)})$ we have
\[ T_{\lambda}^{\mu} \gamma^{\calB}_{\lambda}(\calF) \cong R\Gamma \bigl( \calO_{\calB}(\mu - \lambda) \otimes_{\calO_{\calB}} (\calM^{\lambda}
\otimes_{\calO_{\wfrakg^{(1)}}} \calF) \bigr) \] (here $\calM^{\lambda} \otimes \calF$ is considered as a $\wcalD$-module on $\calB$). By our choice of
splitting bundles (see \cite[1.3.5]{BMR2}), $\calM^{\mu}=\calO_{\calB}(\mu - \lambda)
\otimes_{\calO_{\calB}} \calM^{\lambda}$, hence \[ T_{\lambda}^{\mu} \circ
\gamma^{\calB}_{\lambda}(\calF) \cong \gamma^{\calB}_{\mu}(\calF). \] Similarly, for $\calF \in \calD^b \Coh_{\calB^{(1)}}(\wcalN^{(1)})$ we have $T_{\lambda}^{\mu} \circ \epsilon^{\calB}_{\lambda}(\calF) \cong \epsilon^{\calB}_{\mu}(\calF).$ Hence if $\calL_w$ is defined using $\lambda$, we have $\epsilon^{\calB}_{\mu}(\calL_w) \cong T_{\lambda}^{\mu} \circ \epsilon^{\calB}_{\lambda}(\calL_w) \cong T_{\lambda}^{\mu} L(w \bullet \lambda) \cong L(w \bullet \mu)$, which proves the claim. Here the last isomorphism follows from \eqref{eq:translationsimples}.

Consider now $\Mod^{\fg}_{\lambda}((\calU \frakg)_0)$. The
algebra $(\calU \frakg)_0$ is finite dimensional. Hence, if
$\frakZ_{{\rm HC}}^{\lambda}$ is the image in $(\calU \frakg)_0$ of the maximal ideal of
$\frakZ_{{\rm HC}} \cong {\rm S}(\frakh)^{(W,\bullet)}$ defined by $\lambda$, the following sequence of ideals of $(\calU \frakg)_0$ stabilizes:
\[ \langle \frakZ_{{\rm HC}}^{\lambda} \rangle \ \supset \ \langle
\frakZ_{{\rm HC}}^{\lambda} \rangle^2 \ \supset \ \langle \frakZ_{{\rm HC}}^{\lambda} \rangle^3 \ \supset \ \ldots\] Thus, for $n \gg 0$, $\Mod^{\fg}_{\lambda}((\calU \frakg)_0)$ is
equivalent to the category of finitely generated modules over
$(\calU
\frakg)_0^{\hat{\lambda}} := (\calU \frakg)_0 / \langle \frakZ_{{\rm HC}}^{\lambda} \rangle^n$. As seen above, the simple $(\calU \frakg)_0^{\hat{\lambda}}$-modules are the
$L(w \bullet \lambda)$ for $w \in W^0$. We denote by $P(w \bullet \lambda)$ the projective cover of $L(w \bullet \lambda)$ as a $(\calU
\frakg)_0^{\hat{\lambda}}$-module. For $w \in W^0$ we define \begin{equation} \label{eq:defPw} \calP_w := (\widehat{\gamma}^{\calB}_{\lambda})^{-1} P(w \bullet \lambda) \quad \in \DGCoh((\wfrakg \, \rcap_{\frakg^* \times \calB} \, \calB)^{(1)}).
\end{equation} As above, this object does not depend on the choice of $\lambda \in C_0$.

Our key-result states that the objects $\calL_w$, $\calP_w$ correspond under the linear Koszul duality $\kappa_{\calB}$ of \eqref{eq:defkappa}. If $\calG \in \calD^b \Coh_{\calB^{(1)}}(\wcalN^{(1)}),$ we say that $\calF \in \calD^b \Coh^{\Gm}_{\calB^{(1)}}(\wcalN^{(1)})$ is a \emph{lift} of $\calG$ if $\For(\calF) \cong \calG$, for the forgetful functor $\For: \calD^b \Coh^{\Gm}_{\calB^{(1)}}(\wcalN^{(1)}) \to \calD^b \Coh_{\calB^{(1)}}(\wcalN^{(1)})$. We use the same terminology for $\DGCoh^{\gr}((\wfrakg \, \rcap_{\frakg^* \times \calB} \, \calB)^{(1)})$ and $\DGCoh((\wfrakg \, \rcap_{\frakg^* \times \calB} \, \calB)^{(1)})$. 

\begin{thm} \label{thm:mainthm}

Assume $p>h$ is such that Lusztig's conjecture is true\footnote{See \S \ref{ss:introlusztig} for comments.}. 

There is a unique choice of lifts $\calP^{\gr}_v \in \DGCoh^{\gr}((\wfrakg \, \rcap_{\frakg^* \times \calB} \, \calB)^{(1)})$ of $\calP_v$, resp. $\calL^{\gr}_v \in \calD^b \Coh^{\Gm}_{\calB^{(1)}}(\wcalN^{(1)})$ of $\calL_v$ ($v \in W^0$), such that for all $w \in W^0$, \begin{equation*} \kappa_{\calB}^{-1} \calP^{\gr}_{\tau_0 w} \ \cong \ \zeta(\calL^{\gr}_w)
\otimes_{\calO_{\calB^{(1)}}} \calO_{\calB^{(1)}}(-\rho) \qquad \text{in }\DGCoh^{\gr}(\wcalN^{(1)}). \end{equation*}

\end{thm}

The theorem will be proved in section \ref{sec:sectionmainthm}. The unicity statement is not difficult to check (see \S \ref{ss:paragraphmainthm}). The existence is much more complex. To prove it we will need several tools, which we introduce in sections \ref{sec:sectionbraidgpaction}, \ref{sec:sectionreflectionfunctors} and \ref{sec:sectionsimplemodules}.

As explained above, this statement does not depend on the choice of $\lambda \in C_0$. From now on, for simplicity we mainly restrict to the case $\lambda=0$.

\begin{remark} One could also study the linear Koszul duality where the roles of $\wcalN$ and $\wfrakg$ are exchanged. On the geometric side, results could be deduced from Theorem \ref{thm:mainthm} using \S \ref{ss:koszulinclusion}. However, the representation-theoretic interpretation of $\wcalN \, \rcap_{\frakg^* \times \calB} \, \calB$ is not very interesting: it is related to the ``derived specialization'' $(\calU \frakg)^{\lambda} \lotimes_{\frakZ_{\mathrm{HC}}} \bk_0$ (see Remark \ref{rk:endsection3}$\rmii$). In particular, it is not clear what ``projective'' means in this context.

\end{remark}

\section{Braid group actions and translation functors} \label{sec:sectionbraidgpaction}

In this section we introduce important tools for our study: the braid group actions and the geometric counterparts of the translation functors.

\subsection{Braid group actions}\label{ss:paragraphbraidgpaction}

In this subsection we recall the main result of \cite{RAct}. Recall the notation of \S \ref{ss:intro3}. Let $S:=\{s_{\alpha},
\ \alpha \in \Phi\}$ be the
Coxeter generators of $W$. Let also $S_{\aff} \subset W_{\aff}$ be
the Coxeter generators, i.e. $S_{\aff}$ contains $S$
together with one additional (affine) reflection for each component of $R$. Let $\Phi_{\aff}$ be the set which contains $\Phi$ (the \emph{finite} simple roots) and
additional symbols for each element of $S_{\aff} - S$, called
\emph{affine simple roots}. If $\alpha_0
\in \Phi_{\aff} - \Phi$, we denote by $s_{\alpha_0} \in S_{\aff} - S$ the corresponding
element.

We define the extended affine braid group as follows\footnote{See the appendix of \cite{RAct} for the equivalence with the usual definition.}. For $\alpha, \beta \in \Phi$, we denote by $n_{\alpha,\beta}$ the order of $s_{\alpha} s_{\beta}$
in $W$.

\begin{defin}\label{def:defBaff'}

Let $B_{\aff}'$ be the group defined by the presentation with
generators $\{T_{\alpha}, \
\alpha \in \Phi\} \cup \{\theta_x, \ x \in \bbX\}$ and relations:

\begin{enumerate}
    \item $T_{\alpha} T_{\beta} \cdots = T_{\beta}
      T_{\alpha} \cdots \quad$($n_{\alpha,\beta}$ elements on each side);
    \item $\theta_{x} \theta_{y} = \theta_{x+y}$;
    \item $T_{\alpha} \theta_{x} = \theta_{x}
    T_{\alpha} \quad$ if $\langle x, \alpha^{\vee} \rangle = 0$,
    i.e.~if $s_{\alpha}(x)=x$;
    \item $\theta_{x} = T_{\alpha} \theta_{x-\alpha} T_{\alpha} \quad$ if
    $\langle x, \alpha^{\vee} \rangle =1$, i.e.~if $s_{\alpha}(x)=x - \alpha$.
\end{enumerate}

\end{defin}

We denote by $C: W_{\aff}' \to B_{\aff}'$ the canonical lift (see
\cite[\S 1.1]{RAct}).

Let $X,Y$ be varieties. Let $p_X : X \times Y \to X$, $p_Y
: X \times Y \to Y$ be the projections. We
define the full subcategory $\calD^b_{\pro}\Coh(X \times Y)$ of
$\calD^b \Coh(X \times Y)$ as follows: an object of $\calD^b
\Coh(X \times Y)$ belongs to $\calD^b_{\pro} \Coh(X \times Y)$ if its
cohomology sheaves are
supported on a closed subscheme $Z \subset X \times Y$ such that $(p_X)_{|Z}$ and $(p_Y)_{|Z}$ are proper. Any
$\calF \in \calD^{b}_{\pro} \Coh(X \times Y)$ gives rise to a
convolution functor \[ F^{\calF}_{X \to Y} : \left\{
\begin{array}{ccc}
  \calD^{b}\Coh(X) & \to & \calD^{b}\Coh(Y) \\
  \calM & \mapsto & R(p_{Y})_*(\calF \; \lotimes_{X \times Y} \;
  p_{X}^*\calM) \\ \end{array} \right. .\] One defines
similarly convolution functors for equivariant
coherent sheaves.

For $\alpha \in \Phi$ we define the following subvariety of $\wfrakg
\times \wfrakg$: \[ S_{\alpha} = \left\{ (X,g B, h B) \in \frakg^*
  \times \calB \times_{\calP_{\alpha}} \calB \ \left|
\begin{array}{l}
  X_{|g \cdot \frakn + h \cdot \frakn}=0  \\
  \text{and} \ X(g \cdot h_{\alpha})=0 \ \text{if} \ g B = h B \\
\end{array} \right\} , \right. \] a vector bundle
over $\calB \times_{\calP_{\alpha}} \calB$ of rank $\dim(\frakg/\frakn)
- 1$. We also define $S_{\alpha}':=S_{\alpha} \cap (\wcalN \times
\wcalN)$, a closed subvariety of $\wcalN \times \wcalN$ with two irreducible components (see
\cite[\S 4]{RAct}). If
$p : X \to \calB$ (resp. $p : X \to \calB^2$) is a variety
over $\calB$ (resp. $\calB^2$) and $\lambda,\mu \in \bbX$, we
denote by $\calO_X(\lambda)$ (resp. $\calO_X(\lambda,\mu)$) the
line bundle $p^* \calO_{\calB}(\lambda)$ (resp. $p^*
(\calO_{\calB}(\lambda) \boxtimes \calO_{\calB}(\mu))$). If $\calF \in
\calD^b \Coh(X)$, we denote by $\calF(\lambda)$ (resp.
$\calF(\lambda,\mu)$) the tensor product of $\calF$ and
$\calO_X(\lambda)$ (resp. $\calO_X(\lambda,\mu)$).

By an action of a group on a category we mean a \emph{weak} action,
i.e.~a morphism from the group to the isomorphism classes of
auto-equivalences of the category. The following theorem was announced in \cite[2.1]{BEZICM}. It has been proved in \cite[1.4.1]{RAct} in the case $G$ has no factor of type $\mbfG_2$, and in \cite[\S II.8]{RPhD} (as a joint work with R. Bezrukavnikov) in the general case.

\begin{thm}\label{thm:actionBaff'}

There exists an
action of $B_{\aff}'$ on $\calD^{b} \Coh(\wfrakg^{(1)})$ $($respectively
$\calD^{b} \Coh(\wcalN^{(1)}))$ for which:

$\rmi$ the action of $\theta_x$ is given by the convolution with
kernel $\Delta_*(\calO_{\wfrakg^{(1)}}(x))$ $($resp.
$\Delta_*(\calO_{\wcalN^{(1)}}(x)))$  for $x \in \bbX$, where $\Delta$ is
the diagonal embedding;

$\rmii$ the action of $T_{\alpha}$ is given by the convolution
with kernel $\calO_{S_{\alpha}^{(1)}}$ $($resp.
$\calO_{S_{\alpha}'{}^{(1)}})$ for $\alpha \in \Phi$. The action of
$(T_{\alpha})^{-1}$ is the convolution with kernel
$\calO_{S_{\alpha}^{(1)}}(-\rho,\rho-\alpha)$
$($resp. $\calO_{S_{\alpha}'{}^{(1)}}(-\rho,\rho-\alpha))$.

The actions on $\calD^{b} \Coh(\wcalN^{(1)})$ and $\calD^{b}
\Coh(\wfrakg^{(1)})$ correspond under the functor $i_* :
\calD^{b} \Coh(\wcalN^{(1)}) \to \calD^{b} \Coh(\wfrakg^{(1)})$ where $i: \wcalN^{(1)} \hookrightarrow \wfrakg^{(1)}$ is
the embedding.

\end{thm}

Consider $\Theta:=W_{\aff}'
\bullet 0$. In \cite{BMR2} the authors construct an action of
an incarnation ${}^{\Theta}B_{\aff}'$ of $B_{\aff}'$ (see
\cite[2.1.2]{BMR2}) on $\calD^b
\Mod^{\fg}_{0}(\calU \frakg)$, which restricts to an action on
$\calD^b \Mod^{\fg}_{(0,0)}(\calU \frakg)$. There exist isomorphisms
$B_{\aff}' \cong {}^{\Theta} B_{\aff}'$, associated to any
element in $\Theta$. We normalize this isomorphism by choosing $0 \in \Theta$. This way we obtain an action of $B_{\aff}'$ on
$\calD^b \Mod^{\fg}_{(0,0)}(\calU \frakg)$; for $b \in B_{\aff}'$, we denote by
\[ \mathbf{I}_b : \calD^b \Mod^{\fg}_{(0,0)}(\calU \frakg) \to \calD^b
\Mod^{\fg}_{(0,0)}(\calU \frakg) \] the corresponding action. On the other hand, let us denote by \[ \mathbf{J}_b :
  \calD^{b} \Coh(\wfrakg^{(1)}) \to \calD^{b} \Coh(\wfrakg^{(1)}), \ 
\text{resp.} \ \mathbf{K}_b : \calD^{b} \Coh(\wcalN^{(1)}) \to \calD^{b} \Coh(\wcalN^{(1)}), \] the actions of $b$ given by Theorem \ref{thm:actionBaff'}. Then for $b \in B_{\aff}'$
the following diagram is commutative (see \cite[5.4.1]{RAct}): \begin{equation} \label{eq:diagramactionUg} \xymatrix@R=15pt{\calD^b
  \Coh_{\calB^{(1)}}(\wfrakg^{(1)}) \ar[rr]^-{\mathbf{J}_b}
  \ar[d]_{\gamma^{\calB}_0}^{\wr} & & \calD^b
  \Coh_{\calB^{(1)}}(\wfrakg^{(1)}) \ar[d]^{\gamma^{\calB}_0}_{\wr} \\
  \calD^b \Mod^{\fg}_{(0,0)}(\calU \frakg) \ar[rr]^-{\mathbf{I}_b} & &
  \calD^b \Mod^{\fg}_{(0,0)}(\calU \frakg). } \end{equation}

\subsection{Graded versions of the actions}
\label{ss:paragraphgradedversionsaction}

Let us define actions of $\Gm \cong \bk^{\times}$ on $\wfrakg^{(1)}$ and $\wcalN^{(1)}$, by setting for $t \in \bk^{\times}$ and $(X,gB)$ in $\wfrakg^{(1)}$, resp. $\wcalN^{(1)}$: \begin{equation} \label{eq:defactionGm} t \cdot (X,gB) = (t^{-2} \cdot X,gB),
\quad \text{resp.} \quad t \cdot (X,gB) = (t^2 \cdot X, gB). \end{equation} Note that the action on $\wcalN^{(1)}$ is \emph{not} the restriction of the action on $\wfrakg^{(1)}$, but the \emph{dual} action\footnote{Recall also that the action of $\bk$
on $\frakg^*{}^{(1)}$ is twisted: if ${\rm Fr}: \frakg^* \to
\frakg^*{}^{(1)}$ denotes the Frobenius morphism, and if $t \in \bk$,
then we have $t \cdot {\rm Fr}(X)={\rm Fr}(t^{1/p}X)$.}, which is consistent with \S \ref{ss:intro3}. As in \S \ref{ss:koszulinclusion}, we denote by $\langle 1 \rangle$ the
shift in the grading given by the tensor product with the
$\Gm$-module given by $\Id_{\Gm}$. An easy extension of Theorem \ref{thm:actionBaff'} yields:

\begin{prop}\label{prop:actionBaff'Gm}

There exists an action of
$B_{\aff}'$ on $\calD^{b} \Coh^{\Gm}(\wfrakg^{(1)})$ $($resp.
$\calD^{b} \Coh^{\Gm}(\wcalN^{(1)}))$ for which:

$\rmi$ For $x \in \bbX$, the action of $\theta_x$ is given by the convolution with kernel
$\Delta_* \calO_{\wfrakg^{(1)}}(x)$ $($resp. $\Delta_*
\calO_{\wcalN^{(1)}}(x))$, where $\Delta$ is the diagonal
embedding;

$\rmii$ For $\alpha \in
\Phi$, the action of $T_{\alpha}$ is given by the convolution
with kernel $\calO_{S_{\alpha}^{(1)}}\langle -1\rangle$ $($resp.
$\calO_{S_{\alpha}'{}^{(1)}}\langle 1\rangle)$. Moreover, the action of $(T_{\alpha})^{-1}$ is the convolution
with kernel $\calO_{S_{\alpha}^{(1)}}(-\rho,\rho-\alpha)\langle -1\rangle$
$($resp. $\calO_{S_{\alpha}'{}^{(1)}}(-\rho,\rho-\alpha)\langle
1\rangle)$.

\end{prop}

\begin{proof} We only consider $\wfrakg^{(1)}$ (the proof for
$\wcalN^{(1)}$ is similar). It suffices to observe that
$S_{\alpha}$ is a $\Gm$-stable subvariety of $\wfrakg \times
\wfrakg$, and that all the constructions of \cite{RAct, RPhD} respect the $\Gm$-structures. The only subtlety is in \cite[1.5.4]{RAct}. In this proof, the $\Gm$-equivariant version
of the exact sequence $\calO_{V_{\alpha}^1} \hookrightarrow
\calO_{V_{\alpha}}(\rho - \alpha, -\rho, 0) \twoheadrightarrow
\calO_{V_{\alpha}^2}(\rho - \alpha, -\rho, 0)$ is
$\calO_{V_{\alpha}^1}\langle 2\rangle \hookrightarrow
\calO_{V_{\alpha}}(\rho -
\alpha, -\rho, 0) \twoheadrightarrow \calO_{V_{\alpha}^2}(\rho -
\alpha, -\rho, 0).$ The rest of the proof is similar. \end{proof}

Now we consider the dg-scheme $(\wfrakg \, \rcap_{\frakg^* \times \calB} \, \calB)^{(1)}$. Recall the notation for categories of dg-modules in section \ref{sec:sectiondgalg}. By definition (see equation \eqref{eq:DGCoh2}), 
\[ \DGCoh^{\gr}((\wfrakg \, \rcap_{\frakg^* \times \calB} \, \calB)^{(1)})
\ \cong \ \calD^{\qc,\fg}_{\Gm}(\calB^{(1)}, \, 
\Lambda_{\calO_{\calB^{(1)}}}(\calT_{\calB^{(1)}}^{\vee})). \] Using arguments similar to those for Proposition \ref{prop:DGCohkoszul2nongr}, one obtains:

\begin{lem} \label{lem:equDGCoh}

There exist equivalences of categories \begin{align*}
  \DGCoh^{\gr}((\wfrakg \, \rcap_{\frakg^* \times \calB} \, \calB)^{(1)}) \ &
  \cong \ \calD^{\qc,\fg}_{\Gm}(\calB^{(1)}, \, (\pi_*
  \calO_{\wfrakg^{(1)}}) \otimes_{\bk} \Lambda(\frakg^{(1)})), \\
  \DGCoh((\wfrakg \, \rcap_{\frakg^* \times \calB} \, \calB)^{(1)}) \ & \cong
  \ \calD^{\qc,\fg}(\calB^{(1)}, \, (\pi_* \calO_{\wfrakg^{(1)}})
  \otimes_{\bk} \Lambda(\frakg^{(1)})), \end{align*} where $(\pi_*
\calO_{\wfrakg^{(1)}}) \otimes_{\bk} \Lambda(\frakg^{(1)})$ is
considered as a dg-algebra equipped with a Koszul differential, with
$\pi_* \calO_{\wfrakg^{(1)}}$ in cohomological degree $0$ and
$\frakg^{(1)}$ in cohomological degree $-1$. In the first equivalence,
the internal grading on $\pi_* \calO_{\wfrakg^{(1)}}$ is induced by
the $\Gm$-action {\rm \eqref{eq:defactionGm}}, and $\frakg^{(1)}$ is in
bidegree $(-1,2)$.

\end{lem}

Recall that $p: (\wfrakg \, \rcap_{\frakg^* \times \calB} \, \calB)^{(1)} \to
\wfrakg^{(1)}$ is the natural projection.

\begin{prop}\label{prop:actionBaff'DGCoh}

There exist actions of $B_{\aff}'$ on $\DGCoh((\wfrakg \, \rcap_{\frakg^* \times \calB} \, \calB)^{(1)})$ and $\DGCoh^{\gr}((\wfrakg
\, \rcap_{\frakg^* \times \calB} \, \calB)^{(1)})$ such that the functors \[
\xymatrix@R=18pt{ \DGCoh^{\gr}((\wfrakg \, \rcap_{\frakg^* \times \calB} \,
  \calB)^{(1)}) \ar[rr]^-{\For} \ar[d]_-{R(p_{\Gm})_*} & &
  \DGCoh((\wfrakg \, \rcap_{\frakg^* \times \calB} \, \calB)^{(1)})
  \ar[d]^-{Rp_*} \\ \calD^b \Coh^{\Gm}(\wfrakg^{(1)}) \ar[rr]^-{\For}
  & & \calD^b \Coh(\wfrakg^{(1)}) } \] commute with the action of
$B_{\aff}'$.

\end{prop}

\begin{proof} We give the proof only for $\DGCoh((\wfrakg
\, \rcap_{\frakg^* \times \calB} \, \calB)^{(1)})$. The only difficulty is to define carefully the action of $T_{\alpha}$; for this we first derive a description of $\mathbf{J}_{T_{\alpha}}$ which only involves sheaves on $\calB^{(1)}$.

As above, let $\pi: \wfrakg^{(1)} \to \calB^{(1)}$ be the natural
morphism. Denote by $p_i: \wfrakg^{(1)} \times \wfrakg^{(1)} \to
\wfrakg^{(1)}$, $q_i: \calB^{(1)} \times \calB^{(1)} \to \calB^{(1)}$
the projections ($i=1,2$). Recall that $\pi$ is affine, hence
$\pi_*$ is an equivalence between
$\Coh(\wfrakg^{(1)})$ and $\Coh(\calB^{(1)}, \, \pi_*
\calO_{\wfrakg^{(1)}})$ (\cite[1.4.3]{EGA2}). If $\calF$ is in $\calD^b
\Coh(\wfrakg^{(1)})$, by \cite[1.5.7.1]{EGA2} we have \[ (\pi \times
\pi)_*(p_1^* \calF) \cong \bigl( (\pi \times \pi)_*
\calO_{\wfrakg^{(1)} \times \wfrakg^{(1)}} \bigr) \otimes_{q_1^* \pi_*
  \calO_{\wfrakg^{(1)}}} q_1^* \pi_* \calF.\] Using
\cite[1.4.8.1]{EGA2}, it follows that if $\alpha \in \Phi$, \[(\pi
\times \pi)_* (p_1^*
\calF \, \lotimes_{\calO_{\wfrakg^{(1)} \times \wfrakg^{(1)}}} \,
\calO_{S_{\alpha}^{(1)}}) \cong \bigl( (\pi \times \pi)_*
\calO_{S_{\alpha}^{(1)}} \bigr) \, \lotimes_{q_1^* \pi_*
  \calO_{\wfrakg^{(1)}}} \, q_1^* \pi_* \calF.\] Hence, finally,
\begin{equation} \label{eq:formulaconvolution} \pi_*
\mathbf{J}_{T_{\alpha}} (\calF)
\ \cong \ R(q_2)_* ((\pi \times \pi)_*
\calO_{S_{\alpha}^{(1)}} \lotimes_{q_1^* \pi_* \calO_{\wfrakg^{(1)}}}
q_1^* \pi_* \calF). \end{equation} These isomorphisms are functorial. In
this formula, $(\pi \times \pi)_* \calO_{S_{\alpha}^{(1)}}$ is
considered as a right $q_1^* \pi_* \calO_{\wfrakg^{(1)}}$-module, and
a left $q_2^* \pi_* \calO_{\wfrakg^{(1)}}$-module.

We define the action of $B_{\aff}'$ using the
equivalences of Lemma \ref{lem:equDGCoh}. It is enough to define the
action of the generators $\theta_x$ and $T_{\alpha}$, and to prove that
they satisfy the relations of Definition \ref{def:defBaff'}. First, the action of $\theta_x$ is defined as the
tensor product with $\calO_{\calB^{(1)}}(x)$. Let
$\alpha \in \Phi$. Consider the functor $$\left\{ \begin{array}{ccc}
    \calC(\calB^{(1)}, \, \pi_* \calO_{\wfrakg^{(1)}} \otimes_{\bk}
\Lambda(\frakg^{(1)})) & \to & \calC(\calB^{(1)} \times \calB^{(1)}, \, q_2^*
(\pi_* \calO_{\wfrakg^{(1)}} \otimes_{\bk} \Lambda(\frakg^{(1)}))) \\
\calG & \mapsto & \bigl( (\pi \times \pi)_*
\calO_{S_{\alpha}^{(1)}} \bigr) \otimes_{q_1^* \pi_*
  \calO_{\wfrakg^{(1)}}} q_1^* \calG \end{array} \right.$$ where $(\pi
\times \pi)_* \calO_{S_{\alpha}^{(1)}}$ is considered as a bimodule,
as above. This functor has a left derived functor, denoted by $\calG \mapsto
(\pi \times \pi)_* \calO_{S_{\alpha}^{(1)}}
\lotimes_{q_1^* \pi_* \calO_{\wfrakg^{(1)}}} q_1^* \calG.$ Consider the natural morphism induced by $q_2$:
\[ \widetilde{q}_2 : \calC(\calB^{(1)} \times \calB^{(1)}, \, q_2^*
(\pi_* \calO_{\wfrakg^{(1)}} \otimes_{\bk} \Lambda(\frakg^{(1)}))) \to
\calC(\calB^{(1)}, \, \pi_* \calO_{\wfrakg^{(1)}} \otimes_{\bk}
\Lambda(\frakg^{(1)})).\] Then we define the action of $T_{\alpha}$ as the functor
\[ F_{\alpha} : \calG \mapsto R(\widetilde{q}_2)_* \bigl( (\pi \times \pi)_*
\calO_{S_{\alpha}^{(1)}} \lotimes_{q_1^* \pi_* \calO_{\wfrakg^{(1)}}}
q_1^* \calG \bigr).\] Easy arguments show that this functor indeed
restricts to the subcategories of dg-modules with quasi-coherent,
locally finitely generated cohomology. Moreover, the following diagram
commutes: \[
\xymatrix@R=16pt{\DGCoh((\wfrakg \, \rcap_{\frakg^* \times \calB} \, \calB)^{(1)})
  \ar[rr]^-{F_{\alpha}} \ar[d]_-{Rp_*} & & \DGCoh((\wfrakg
  \, \rcap_{\frakg^* \times \calB} \, \calB)^{(1)}) \ar[d]^{Rp_*} \\
  \calD^b \Coh(\wfrakg^{(1)})
  \ar[rr]^-{\mathbf{J}_{T_{\alpha}}} & & \calD^b \Coh(\wfrakg^{(1)}) } \] (see \eqref{eq:formulaconvolution}, and use the fact that a K-flat $\pi_*
\calO_{\wfrakg^{(1)}} \otimes_{\bk} \Lambda(\frakg^{(1)})$-dg-module
is also K-flat over $\pi_* \calO_{\wfrakg^{(1)}}$). 

With these definitions, it follows
easily from the results of \cite{RAct, RPhD} that the actions of
the $T_{\alpha}$'s and the $\theta_x$'s
satisfy the defining relations of $B_{\aff}'$. \end{proof}

For $b \in B_{\aff}'$, we denote the action of $b$ of Propositions
\ref{prop:actionBaff'Gm}, \ref{prop:actionBaff'DGCoh} by \begin{align*} \mathbf{J}_b^{\Gm} : \calD^{b}
\Coh^{\Gm}(\wfrakg^{(1)}) \ & \to \ \calD^{b} \Coh^{\Gm}(\wfrakg^{(1)}), \\
\mathbf{K}_b^{\Gm} : \calD^{b} \Coh^{\Gm}(\wcalN^{(1)}) \ & \to \ \calD^{b} \Coh^{\Gm}(\wcalN^{(1)}), \\ \mathbf{J}_b^{\dg} :
\DGCoh((\wfrakg \, \rcap_{\frakg^* \times \calB} \, \calB)^{(1)}) \ & \to \
\DGCoh((\wfrakg \, \rcap_{\frakg^* \times \calB} \, \calB)^{(1)}), \\
\mathbf{J}_b^{\dg,\gr} :
\DGCoh^{\gr}((\wfrakg \, \rcap_{\frakg^* \times \calB} \, \calB)^{(1)}) \ & \to \
\DGCoh^{\gr}((\wfrakg \, \rcap_{\frakg^* \times \calB} \, \calB)^{(1)}).
\end{align*} 

It follows in particular from Proposition \ref{prop:actionBaff'DGCoh} that the $B_{\aff}'$-action on $\calD^b \Mod^{\fg}_{(0,0)}(\calU \frakg)$ factorizes through an action on $\calD^b \Mod^{\fg}_{0}((\calU \frakg)_0)$, which corresponds to the action on the category $\DGCoh((\wfrakg \, \rcap_{\frakg^* \times \calB} \, \calB)^{(1)})$ via the equivalence $\widehat{\gamma}^{\calB}_0$ of Theorem \ref{thm:localizationfixedFr}. We denote the action of $b \in B_{\aff}'$ by \[ \mathbf{I}_b^{{\rm res}} : \calD^b \Mod^{\fg}_{0}((\calU \frakg)_0) \to \calD^b \Mod^{\fg}_{0}((\calU \frakg)_0). \] 

\subsection{Some exact sequences}\label{ss:exactsequenceN}

Geometrically, $S_{\alpha}'$ can be described as: \[ S_{\alpha}'=\{(X,
g_1 B, g_2 B) \in \frakg^* \times \calB \times_{\calP_{\alpha}} \calB
\mid X_{|g_1 \cdot \frakb + g_2 \cdot \frakb}=0\}. \] It has two
irreducible components:
$\Delta \wcalN$, the diagonal embedding of $\wcalN$, and $Y_{\alpha}:=\{(X,
g_1 B, g_2 B) \in \frakg^* \times (\calB
\times_{\calP_{\alpha}} \calB) \mid
X_{|g_1 \cdot \frakp_{\alpha}}=0\},$ a vector bundle on
$\calB \times_{\calP_{\alpha}} \calB$ of rank $\dim(\frakg / \frakb)
-1$.

Recall the morphism $\widetilde{\pi}_{\alpha}: \wfrakg \to
\wfrakg_{\alpha}$ (see equation \eqref{eq:defpiP}). There exist exact sequences of sheaves on $\wfrakg \times \wfrakg$, resp. $\wcalN
\times \wcalN$ (see \cite[5.3.2, 6.1.1]{RAct}): \begin{eqnarray}
  \label{eq:kernelswfrakg}
  & \calO_{\Delta \wfrakg} \hookrightarrow \calO_{\wfrakg
\times_{\wfrakg_{\alpha}} \wfrakg} \twoheadrightarrow \calO_{S_{\alpha}}, &
\\ \label{eq:kernelswfrakg2} & \calO_{S_{\alpha}}(\rho - \alpha, -\rho)
\hookrightarrow \calO_{\wfrakg
\times_{\wfrakg_{\alpha}} \wfrakg} \twoheadrightarrow \calO_{\Delta
\wfrakg}, & \\ \label{eq:kernelswcalN} & \calO_{\Delta \wcalN} \hookrightarrow
\calO_{S_{\alpha}'}(\rho - \alpha, -\rho) \twoheadrightarrow
\calO_{Y_{\alpha}}(\rho - \alpha, -\rho), & \\ \label{eq:kernelswcalN2} &
\calO_{Y_{\alpha}}(\rho - \alpha, -\rho) \hookrightarrow
\calO_{S_{\alpha}'} \twoheadrightarrow \calO_{\Delta \wcalN}. \end{eqnarray}

The exact sequences \eqref{eq:kernelswfrakg2} and \eqref{eq:kernelswcalN2}
are $\Gm$-equivariant. The exact sequences
\eqref{eq:kernelswfrakg} and \eqref{eq:kernelswcalN} admit the
$\Gm$-equivariant analogues \begin{eqnarray} \label{eq:kernelswfrakggr} &
  \calO_{\Delta \wfrakg} \langle 2 \rangle \hookrightarrow \calO_{\wfrakg
    \times_{\wfrakg_{\alpha}} \wfrakg} \twoheadrightarrow
  \calO_{S_{\alpha}}, & \\ \label{eq:kernelswcalNgr} & \calO_{\Delta \wcalN}
  \langle -2 \rangle \hookrightarrow
\calO_{S_{\alpha}'}(\rho - \alpha, -\rho) \twoheadrightarrow
\calO_{Y_{\alpha}}(\rho - \alpha, -\rho). & \end{eqnarray}

Recall that we have $\calO_{\calB
  \times_{\calP_{\alpha}} \calB}(\rho - \alpha, -\rho) \cong \calO_{\calB
  \times_{\calP_{\alpha}} \calB}(-\rho, \rho - \alpha)$ (\cite[\S 1.5]{RAct}). Hence one can
exchange $-\rho$ and $\rho - \alpha$ in these exact sequences.

\subsection{Geometric counterparts of the translation
  functors}\label{ss:geometrictranslation}

Let us recall the geometric interpretation of the translation functors (see \S \ref{ss:paragraphtranslationfunctors}). Let $P$ be a parabolic subgroup of $G$
containing $B$ and let $\calP=G/P$. Recall the morphism
$\widetilde{\pi}_{\calP}$ of \eqref{eq:defpiP}. By \cite[2.2.5]{BMR2} we have:

\begin{prop}\label{prop:translationBMR}

Let $\lambda \in \bbX$ be regular, and let $\mu \in \bbX$ be in the
closure of the facet of $\lambda$. Assume that ${\rm
  Stab}_{(W_{\aff},\bullet)}(\mu) = W_P$ (with the same notation as in Theorem {\rm \ref{thm:thmBMR}}$\rmii$). There exist isomorphisms of
functors \[ T_{\lambda}^{\mu} \circ \gamma^{\calB}_{\lambda} \cong
\gamma^{\calP}_{\mu} \circ R(\widetilde{\pi}_{\calP})_*
\quad \text{and} \quad T_{\mu}^{\lambda} \circ
\gamma^{\calP}_{\mu} \cong \gamma^{\calB}_{\lambda} \circ
L(\widetilde{\pi}_{\calP})^*.\]

\end{prop}

Let $\lambda$ and $\mu$ be as in Proposition \ref{prop:translationBMR}. The morphism
$\widetilde{\pi}_{\calP} : \wfrakg \to \wfrakg_{\calP}$ induces a morphism of dg-schemes
\begin{equation} \label{eq:defpiPhat} \widehat{\pi}_{\calP} : (\wfrakg
  \, \rcap_{\frakg^* \times \calB} \, \calB)^{(1)} \to (\wfrakg_{\calP}
  \, \rcap_{\frakg^* \times \calP} \, \calP)^{(1)}. \end{equation} It can be realized as the
morphism $(\wfrakg^{(1)}, \, \calO_{\wfrakg^{(1)}} \otimes_{\bk}
\Lambda(\frakg^{(1)})) \to (\wfrakg_{\calP}^{(1)}, \,
\calO_{\wfrakg_{\calP}^{(1)}} \otimes_{\bk} \Lambda(\frakg^{(1)}))$,
or as the morphism $(\calB^{(1)}, \,
\Lambda_{\calO_{\calB^{(1)}}}(\calT_{\calB^{(1)}}^{\vee})) \to
(\calP^{(1)}, \,
\Lambda_{\calO_{\calP^{(1)}}}(\calT_{\calP^{(1)}}^{\vee}))$.
Easy arguments show that $R(\widehat{\pi}_{\calP})_*$ and
$L(\widehat{\pi}_{\calP})^*$ restrict to functors between
the categories $\DGCoh((\wfrakg \, \rcap_{\frakg^* \times \calB} \,
\calB)^{(1)})$ and $\DGCoh((\wfrakg_{\calP} \, \rcap_{\frakg^*
  \times \calP} \, \calP)^{(1)})$, with usual
compatibility conditions. Recall the equivalences of Theorems
\ref{thm:localizationfixedFr} and \ref{thm:localizationfixedFrparabolic}. A
proof similar to that of \cite[2.2.5]{BMR2} gives:

\begin{prop} \label{prop:translationfixedFr}

Let $\lambda, \mu, P, \calP$ be as in Proposition {\rm
  \ref{prop:translationBMR}}. There exist isomorphisms of functors
\[ T_{\lambda}^{\mu} \circ \widehat{\gamma}^{\calB}_{\lambda} \cong
\widehat{\gamma}^{\calP}_{\mu} \circ R(\widehat{\pi}_{\calP})_*
\quad \text{and} \quad T_{\mu}^{\lambda} \circ
\widehat{\gamma}^{\calP}_{\mu} \cong \widehat{\gamma}^{\calB}_{\lambda} \circ
L(\widehat{\pi}_{\calP})^*. \]

\end{prop}

If $P=P_{\alpha}$ for $\alpha \in \Phi$,
we simplify the notation and set $\widehat{\pi}_{\alpha} :=
\widehat{\pi}_{\calP_{\alpha}}$.

\subsection{Some results from Representation Theory} \label{ss:paragraphRT}

One of our main tools will be the reflection functors, defined in the following way.

\begin{defin} \label{def:defreflection}

Let $\delta \in \Phi_{\aff}$. Choose a weight $\mu_{\delta} \in \bbX$
on the $\delta$-wall of $\overline{C}_0$, and not on
any other wall. The \emph{reflection functor} $R_{\delta}$ is
defined as \[ R_{\delta}:=T^0_{\mu_{\delta}} \circ
T_0^{\mu_{\delta}}.\] It does not depend on the choice of
$\mu_{\delta}$ by \cite[2.2.7]{BMR2}. It is an exact, auto-adjoint
endofunctor of $\Mod^{\fg}_{(0,0)}(\calU \frakg)$, which stabilizes
$\Mod^{\fg}_0((\calU \frakg)_0)$.

\end{defin}

In this subsection we recall classical results describing the
action of reflection functors on simple and projective
modules. Recall that it has been proved that Lusztig's conjecture on the
characters of simple $G$-modules (\cite{LUSPro}) is satisfied for $p \gg 0$ (see \S \ref{ss:introlusztig}). From now on we assume: \[ p
\text{ is large enough so that Lusztig's conjecture is satisfied.}
\leqno(\#) \] This restriction is needed only to apply Theorem
\ref{thm:conjectureAndersen}$\rmi$ below.

Let $\delta \in \Phi_{\aff}$. Consider a simple $(\calU
\frakg)_0$-module $L(w \bullet 0)$ ($w \in W^0$), where $w s_{\delta}
\bullet 0 > w \bullet 0$ (see \S \ref{ss:objectsLwPw}). There are natural adjunction morphisms $L(w \bullet 0) \xrightarrow{\phi_{\delta}^w} R_{\delta}
L(w \bullet 0) \xrightarrow{\psi_{\delta}^w} L(w \bullet 0).$ It is known
(\cite[II.7.20]{JANAlg}) that $\phi_{\delta}^w$ is injective, and
$\psi_{\delta}^w$ surjective. Consider \begin{equation} \label{eq:defQ} Q_{\delta}(w):={\rm
    Ker}(\psi_{\delta}^w) / {\rm Im}(\phi_{\delta}^w). \end{equation}

\begin{thm} \label{thm:conjectureAndersen}

$\rmi$ Let $\delta \in \Phi_{\aff}$, and $w \in W^0$ such that $w
\bullet 0 < w s_{\delta} \bullet 0$. Then $Q_{\delta}(w)$ is a
semi-simple $\calU \frakg$-module.

$\rmii$ If $w s_{\delta} \in W^0$, $L(w s_{\delta} \bullet
0)$ appears with multiplicity one in $Q_{\delta}(w)$, and all the other simple constituents of the form $L(x \bullet 0)$ for $x \in W^0$
satisfying $\ell(x)<\ell(ws_{\delta})$. If $w s_{\delta} \notin W^0$, all the simple constituents of $Q_{\delta}(w)$ are of the form $L(x \bullet 0)$ for $x \in W^0$
satisfying $\ell(x)<\ell(ws_{\delta})$.


\end{thm}

\begin{proof} $\rmi$ follows from a
conjecture by Andersen, which is equivalent
to Lusztig's conjecture (\cite{ANDInv}, \cite[II.C]{JANAlg}). Hence it holds under assumption
$(\#)$.

$\rmii$ By \cite[II.7.19-20]{JANAlg} and the strong
linkage principle (\cite[II.6.13]{JANAlg}), the simple
factors of $Q_{\delta}(w)$ as a $G$-module are $L(w s_{\delta} \bullet
0)$ with multiplicity one, and some $L(x \bullet 0)$ with $x \in
W_{\aff} - \{w s_{\delta} \}$, such that $x \bullet 0$ is dominant and $x \bullet 0 \uparrow w
s_{\delta} \bullet 0$ (notation of \cite[II.6.4]{JANAlg}). By
\cite[II.6.6]{JANAlg}, such an $x$ satifies $\ell(x) < \ell(w
s_{\delta})$. Some of these simple $G$-modules may not be simple as $\calU
\frakg$-modules if $x \bullet 0$ is not restricted. But if $\lambda=\lambda_1 + p \lambda_2$ for
$\lambda_1 \in \bbX$ restricted dominant and $\lambda_2 \in \bbX$ dominant, then by
Steinberg's theorem (\cite[II.3.17]{JANAlg}), as $\calU \frakg$-modules we
have $L(\lambda) \cong L(\lambda_1)^{\oplus \dim(L(\lambda_2))}$. To
conclude the proof, one observes that if $v \bullet 0$ and $\nu \neq
0$ are dominant, then $\ell(t_{\nu} v) > \ell(v)$. \end{proof}

The following proposition is ``dual'' to point $\rmii$ of Theorem \ref{thm:conjectureAndersen}. Recall the modules $P(w \bullet 0)$ ($w \in W^0$) defined in \S \ref{ss:objectsLwPw}.

\begin{prop} \label{prop:reflectionprojectivesdetails}

Let $w \in W^0$, and $\delta \in \Phi_{\aff}$ such that $w s_{\delta}
\in W^0$ and $w s_{\delta} \bullet 0 < w \bullet 0$. Then $R_{\delta}
P(w \bullet 0)$ is a direct sum of $P(w s_{\delta} \bullet 0)$ and
some $P(v \bullet 0)$ with $v \in W^0$, $\ell(v) > \ell(w
s_{\delta})$.

\end{prop}

\begin{proof} As $R_{\delta}$ is exact
and self-adjoint, $R_{\delta} P(w \bullet 0)$ is a projective
$(\calU \frakg)_0^{\hat{0}}$-module, hence a direct sum of $P(v \bullet 0)$ for $v \in
W^0$. The multiplicity of $P(v \bullet 0)$ is the dimension of
$\Hom_{\frakg}(R_{\delta} P(w \bullet 0), L(v
\bullet 0)) \cong \Hom_{\frakg}(P(w \bullet 0), R_{\delta} L(v \bullet
0)).$ By \eqref{eq:translationsimples}, this multiplicity is $0$ if $v s_{\delta} \bullet 0 < v \bullet 0$ (in
particular for $v=w$).

Assume now that $v s_{\delta} \bullet 0 > v \bullet 0$. The exact sequences
$ Q_{\delta}(v) \hookrightarrow (R_{\delta} L(v \bullet 0))/L(v \bullet 0) \twoheadrightarrow L(v \bullet 0), \ L(v \bullet 0) \hookrightarrow R_{\delta} L(v \bullet 0) \twoheadrightarrow (R_{\delta} L(v \bullet 0))/L(v \bullet 0)$ induce an isomorphism: \[\Hom_{\frakg}(P(w \bullet 0), R_{\delta}
L(v \bullet 0)) \cong \Hom_{\frakg}(P(w \bullet 0), Q_{\delta}(v)).\]
By Theorem \ref{thm:conjectureAndersen}, $Q_{\delta}(v)$ is
semi-simple, $L(vs_{\delta} \bullet 0)$ appears with
multiplicity $1$ in this module if $vs_{\delta} \bullet 0$ is
restricted, and the other simple components have their
highest weight of the form $x \bullet 0$ for $x \in W^0$ with $\ell(x)
< \ell(vs_{\delta})$. Hence if $\Hom_{\frakg}(P(w \bullet 0),
Q_{\delta}(v)) \neq 0$ and $v \neq ws_{\delta}$, then
$\ell(w) < \ell(vs_{\delta})=l(v)+1$, hence $\ell(v) >
\ell(ws_{\delta})$. For $v=ws_{\delta}$, $\Hom_{\frakg}(P(w
\bullet 0), Q_{\delta}(ws_{\delta}))=\bk$. \end{proof}

\subsection{Reminder on graded algebras} \label{ss:paragraphgradedrings}

Consider a $\mathbb{Z}$-graded, finite dimensional $\bk$-algebra $A$. Let $\Mod(A)$,
resp. $\Mod^{\gr}(A)$, be the category of $A$-modules,
resp. graded $A$-modules. Let also $\Mod^{\fg,\gr}(A)$,
$\Mod^{\fg}(A)$ be the categories of finitely generated modules. As
in \S \ref{ss:koszulinclusion}, we denote by $\langle j\rangle$ the shift in the grading
given by $(M \langle j \rangle)_n=M_{n-j}$. Let $\For : \Mod^{\gr}(A)
\to \Mod(A)$ be the forgetful functor. Following \cite{GGGr}, we call
\emph{gradable} the $A$-modules in the essential image of this functor. If $M \in \Mod(A)$, we denote by $\rad(M)$ the radical of $M$ (the intersection of all maximal submodules), and by $\soc(M)$ the
socle of $M$ (the sum of all simple submodules).

\begin{thm}\label{thm:thmGG}

$\rmi$ If $M \in \Mod^{\fg,\gr}(A)$, then $M$ is indecomposable in
$\Mod^{\fg,\gr}(A)$ iff $\, \For(M)$ is indecomposable in
$\Mod^{\fg}(A)$.

$\rmii$ Simple and projective modules in $\Mod^{\fg}(A)$ are
gradable.

$\rmiii$ If $M \in \Mod^{\fg,\gr}(A)$, then $\soc(\For(M))$ and
$\rad(\For(M))$ are homogeneous submodules.

$\rmiv$ If $M,N \in \Mod^{\fg,\gr}(A)$ are indecomposable and non-zero and if we have an isomorphism $\, \For(M) \cong \For(N)$, then there exists a unique $j \in \mathbb{Z}$ such
that $M \cong N\langle j\rangle$ in $\Mod^{\fg,\gr}(A)$.

$\rmv$ If $M \in \Mod^{\fg,\gr}(A)$, then $M$ is projective in
$\Mod^{\fg,\gr}(A)$ iff $\, \For(M)$ is projective in
$\Mod^{\fg}(A)$.

\end{thm}

\begin{proof} $\rmi$ to $\rmiv$ are proved in
\cite[3.2, 3.4, 3.5, 4.1]{GGGr}. $\rmv$ follows from the
iso\-morphism $\Hom_{A}(\For(M),\For(N)) \cong \bigoplus_{i \in
  \mathbb{Z}} \Hom_{\Mod^{\gr}(A)}(M,N \langle i \rangle)$.\end{proof}

The following results can be proved exactly as in the non-graded case
(see also \cite[E.6]{AJS} for a proof in a more general context):

\begin{prop} \label{prop:KrullSchmidt}

$\rmi$ If $M \in \Mod^{\fg,\gr}(A)$, then $M$ is indecomposable in
$\Mod^{\fg,\gr}(A)$ iff the algebra $\Hom_{\Mod^{\fg,\gr}(A)}(M,M)$ is local.

$\rmii$ The Krull-Schmidt theorem holds in $\Mod^{\fg,\gr}(A)$.

\end{prop}

These results can be used to deduce information on the structure of a
graded $A$-module $M$ when we know the structure of $\For(M)$. Let $M \in \Mod^{\fg,\gr}(A)$, with decomposition $M=M^1 \oplus \cdots \oplus M^n$ as a sum of indecomposable submodules in $\Mod^{\fg,\gr}(A)$. Then
\begin{equation} \label{eq:decompositionM} \For(M)=\For(M^1) \oplus
  \cdots \oplus \For(M^n) \end{equation} in $\Mod(A)$. By
Theorem \ref{thm:thmGG}$\rmi$, $\For(M^j)$ is
indecomposable for all $j$. Hence \eqref{eq:decompositionM} is the decomposition of
$\For(M)$ as a sum of indecomposable submodules (which is unique, up
to isomorphism and permutation, by the Krull-Schmidt theorem). So the
$M^j$'s are lifts of the indecomposable direct summands of $\For(M)$. For later reference, let us spell out the following easy consequence of these remarks and
Theorem \ref{thm:thmGG}, which is implicit in \cite{GGGr}.

\begin{cor} \label{cor:corGG}

$M \in \Mod^{\gr}(A)$ is  semi-simple (resp. simple) in $\Mod^{\gr}(A)$ iff
$\For(M)$ is a semi-simple (resp. simple) $A$-module.

\end{cor}

\section{Projective $(\calU \frakg)_0$-modules}
\label{sec:sectionreflectionfunctors}

In this section we study the RHS of
diagram $(*)$ in the introduction of section \ref{sec:statement}. More precisely, we introduce reflection functors (and their ``graded versions''), and study their action on projective modules.

From now on, for simplicity we assume that $G$ is quasi-simple.

\subsection{Geometric reflection functors}\label{ss:paragraphreflection}

Let $\alpha \in \Phi$. Recall the reflection functors (Definition \ref{def:defreflection}), and
the morphism $\widehat{\pi}_{\alpha}$ (see \eqref{eq:defpiPhat}). Consider the functor $\frakR_{\alpha}:=L(\widehat{\pi}_{\alpha})^*
\circ R(\widehat{\pi}_{\alpha})_*$. The reason for this notation is the commutativity of the following diagram, by
Proposition \ref{prop:translationfixedFr}: \begin{equation} \label{eq:diagramRalpha}
\xymatrix@R=16pt{ \DGCoh((\wfrakg \, \rcap_{\frakg^* \times \calB} \, \calB)^{(1)})
  \ar[d]_-{\widehat{\gamma}^{\calB}_{0}}^-{\wr}
  \ar[rr]^-{\frakR_{\alpha}} & & \DGCoh((\wfrakg \, \rcap_{\frakg^*
  \times \calB} \, \calB)^{(1)})
\ar[d]^-{\widehat{\gamma}^{\calB}_{0}}_-{\wr}
\\ \calD^b \Mod_{0}((\calU \frakg)_0)
\ar[rr]^-{R_{\alpha}} & & \calD^b
  \Mod_{0}((\calU \frakg)_0). } \end{equation}

Now we want to make such a construction for the affine simple root
$\alpha_0$. For simplicity, sometimes we write $s_0$ for the corresponding simple reflection, instead of $s_{\alpha_0}$. We will use the following lemma. Recall the lift $C: W_{\aff}' \to B_{\aff}'$ of the natural projection.

\begin{lem} \label{lem:conjugationaffinesimpleref}

There exists $\beta \in \Phi$ and $b_0 \in B_{\aff}'$ such that $C(s_0) = b_0 \cdot C(s_{\beta}) \cdot (b_0)^{-1}$.

\end{lem}

\begin{proof} First, assume $G$ is not of type $\mbfG_2$, $\mbfF_4$ or $\mbfE_8$. Then $\bbX / \bbY \neq 0$, hence there exists $\omega \in W_{\aff}'$ with $\ell(\omega)=0$, but $\omega \neq 1$. Then $\omega \cdot s_0 \cdot \omega^{-1}$ is a simple reflection $s_{\beta}$ for some $\beta \in \Phi$. As lengths add in this relation, we have also $C(s_0) = b_0 \cdot C(s_{\beta}) \cdot (b_0)^{-1}$ for $b_0 = C(\omega)$.

Now assume\footnote{More generally, this second argument works if $G$ is not of type $\mbfC_n$, $n \geq 2$.} $G$ is of type $\mbfG_2$, $\mbfF_4$ or $\mbfE_8$. There exists a simple root $\beta$ such that the braid relation between $s_0$ and $s_{\beta}$ is of length $3$. Then we have $C(s_0) = C(s_{\beta}) C(s_0) C(s_{\beta}) C(s_0)^{-1} C(s_{\beta})^{-1}$. \end{proof}

In the rest of this paper, we fix such a $\beta$ and such a $b_0$.

\begin{cor} \label{cor:affinereflectionfunctor}

Keep the notation of Lemma {\rm \ref{lem:conjugationaffinesimpleref}}. For any object $M \in \calD^b \Mod^{\fg}_{(0,0)}(\calU \frakg)$, resp. $M \in \calD^b \Mod^{\fg}_{0}((\calU \frakg)_0)$, there exists an isomorphism\footnote{It is not clear from our proof whether or not these isomorphisms are functorial. However, this can be checked easily if $G$ is not of type $\mbfG_2$, $\mbfF_4$ or $\mbfE_8$.} \[ R_{\alpha_0}(M) \cong  \mathbf{I}_{b_0} \circ R_{\beta} \circ
\mathbf{I}_{(b_0)^{-1}}(M), \quad \text{resp.} \quad R_{\alpha_0}(M) \cong  \mathbf{I}_{b_0}^{{\rm res}} \circ R_{\beta} \circ
\mathbf{I}^{{\rm res}}_{(b_0)^{-1}}(M). \]

\end{cor}

\begin{proof} We only prove the first isomorphism, the second one is similar. First, Lemma \ref{lem:conjugationaffinesimpleref} implies that $\mathbf{I}_{C(s_0)} \cong  \mathbf{I}_{b_0} \circ \mathbf{I}_{C(s_\beta)} \circ \mathbf{I}_{(b_0)^{-1}}$. By definition, for any $N \in \calD^b \Mod^{\fg}_{(0,0)}(\calU \frakg)$ there is an exact triangle $N \to R_{\beta} N \to \mathbf{I}_{C(s_{\beta})} N$. Hence, for $M \in \calD^b \Mod^{\fg}_{(0,0)}(\calU \frakg)$ there is an exact triangle \[ M \to \mathbf{I}_{b_0} \circ R_{\beta} \circ \mathbf{I}_{(b_0)^{-1}}(M) \to  \mathbf{I}_{b_0} \circ \mathbf{I}_{C(s_\beta)} \circ \mathbf{I}_{(b_0)^{-1}}(M) \cong \mathbf{I}_{C(s_0)}(M). \] On the other hand, by definition there is an exact triangle  $M \to R_{\alpha_0} M \to \mathbf{I}_{C(s_0)} M$. Identifying these triangles we deduce the isomorphism. \end{proof}

{F}or this reason we define the functor \[\frakR_{\alpha_0} : \DGCoh((\wfrakg
\, \rcap_{\frakg^* \times \calB} \, \calB)^{(1)}) \to \DGCoh((\wfrakg
\, \rcap_{\frakg^* \times \calB} \, \calB)^{(1)})\] as follows:
$\frakR_{\alpha_0} :=
\mathbf{J}^{{\rm dg}}_{b_0} \circ L(\widehat{\pi}_{\beta})^* \circ
R(\widehat{\pi}_{\beta})_* \circ \mathbf{J}^{{\rm
    dg}}_{(b_0)^{-1}}$ (see \S \ref{ss:paragraphgradedversionsaction} for notation). With this definition, by Corollary \ref{cor:affinereflectionfunctor}, the diagram
analogous to \eqref{eq:diagramRalpha} is commutative, at least on every object.

\subsection{Dg versions of the reflection functors}

Let $\alpha \in \Phi$. The dg-ringed spaces $(\calB^{(1)}, \,
\Lambda_{\calO_{\calB^{(1)}}}(\calT_{\calB^{(1)}}^{\vee}))$ and
$(\calP_{\alpha}^{(1)}, \,
\Lambda_{\calO_{\calP_{\alpha}^{(1)}}}(\calT_{\calP_{\alpha}^{(1)}}^{\vee}))$
are naturally $\Gm$-equivariant (see \S \ref{ss:gradedcase}), and
$\widehat{\pi}_{\alpha}$ is also $\Gm$-equivariant. Easy arguments
show that the functors $R(\widehat{\pi}_{\alpha,\Gm})_*$ and
$L(\widehat{\pi}_{\alpha,\Gm})^*$ restrict to functors between the
categories $\DGCoh^{\gr}((\wfrakg \, \rcap_{\frakg^* \times \calB} \,
\calB)^{(1)})$ and $\DGCoh^{\gr}((\wfrakg_{\alpha} \, \rcap_{\frakg^*
  \times \calP_{\alpha}} \, \calP_{\alpha})^{(1)})$, with usual
compatibility conditions. We
define \[ \frakR^{\gr}_{\alpha} := L(\widehat{\pi}_{\alpha,\Gm})^* \circ
R(\widehat{\pi}_{\alpha,\Gm})_*. \] This is an endofunctor of
$\DGCoh^{\gr}((\wfrakg \, \rcap_{\frakg^* \times \calB} \,
\calB)^{(1)})$. For the affine simple root $\alpha_0$ we define similarly, with the
notation of Lemma \ref{lem:conjugationaffinesimpleref},
\begin{equation}\label{eq:defreflectionaffineroot}
  \frakR^{\gr}_{\alpha_0} := \mathbf{J}_{b_0}^{\dg,\gr}
\circ L(\widehat{\pi}_{\beta,\Gm})^* \circ R(\widehat{\pi}_{\beta,\Gm})_*
\circ \mathbf{J}_{(b_0)^{-1}}^{\dg,\gr}. \end{equation} With these definitions, for any $\delta \in \Phi_{\aff}$ the following
diagram commutes: \begin{equation} \label{eq:diagramRdeltagr}
  \xymatrix@R=16pt{\DGCoh^{\gr}((\wfrakg
\, \rcap_{\frakg^* \times \calB} \, \calB)^{(1)}) \ar[d]_{\For}
\ar[r]^{\frakR_{\delta}^{\gr}} & \DGCoh^{\gr}((\wfrakg
\, \rcap_{\frakg^* \times \calB} \, \calB)^{(1)}) \ar[d]^{\For} \\ \DGCoh((\wfrakg
\, \rcap_{\frakg^* \times \calB} \, \calB)^{(1)}) \ar[r]^{\frakR_{\delta}} &
\DGCoh((\wfrakg \, \rcap_{\frakg^* \times \calB} \, \calB)^{(1)}). }
\end{equation}

To conclude this subsection, for later use we study the relation between the
functor $\frakR_{\alpha}$ ($\alpha \in
\Phi$) and the braid group action. Consider the following diagram of $\Gm$-equivariant dg-schemes: \[ \xymatrix@R=8pt{ &
\bigl( (\wfrakg \times_{\wfrakg_{\alpha}} \wfrakg) \, \rcap_{\frakg^* \times
(\calB \times \calB)} \, (\calB \times \calB) \bigr)^{(1)} \ar[ld]^{q_{1}}
\ar[rd]_{q_{2}} & \\ \bigl( \wfrakg \, \rcap_{\frakg^* \times \calB} \,
\calB \bigr)^{(1)} \ar[rd]^{\widehat{\pi}_{\alpha}} & & \bigl( \wfrakg
\, \rcap_{\frakg^*
\times \calB} \, \calB \bigr)^{(1)}. \ar[ld]_{\widehat{\pi}_{\alpha}} \\ &
\bigl( \wfrakg_{\alpha} \, \rcap_{\frakg^* \times \calP_{\alpha}} \,
\calP_{\alpha} \bigr)^{(1)} } \] Here we consider the realization of the
dg-schemes given by the first equivalence of Lemma \ref{lem:equDGCoh} (and
analogues for the other dg-schemes). We want to construct an
isomorphism of
endofunctors of $\DGCoh^{\gr}((\wfrakg \, \rcap_{\frakg^* \times \calB} \,
\calB)^{(1)})$: \begin{equation}\label{eq:isomactiongr}
  L(\widehat{\pi}_{\alpha,\Gm})^* \circ
R(\widehat{\pi}_{\alpha,\Gm})_* \cong R(q_{2,\Gm})_* \circ
L(q_{1,\Gm})^*.\end{equation} There is a natural adjunction morphism
$\Id \to R(q_{1,\Gm})_* \circ L(q_{1,\Gm})^*$. Applying
$R(\widehat{\pi}_{\alpha,\Gm})_*$ to this morphism, and using that $\widehat{\pi}_{\alpha} \circ q_{1} =
\widehat{\pi}_{\alpha} \circ q_{2}$, one obtains a morphism
$R(\widehat{\pi}_{\alpha,\Gm})_* \to R(\widehat{\pi}_{\alpha,\Gm})_* \circ
R(q_{2,\Gm})_* \circ L(q_{1,\Gm})^*$. Now, applying again
adjunction, one obtains the desired morphism
$$L(\widehat{\pi}_{\alpha,\Gm})^* \circ R(\widehat{\pi}_{\alpha,\Gm})_* \to
R(q_{2,\Gm})_* \circ L(q_{1,\Gm})^*.$$ Under the
functor $\DGCoh^{\gr}((\wfrakg \, \rcap_{\frakg^* \times \calB} \,
\calB)^{(1)}) \xrightarrow{R(p_{\Gm})_*} \calD^b
\Coh^{\Gm}(\wfrakg^{(1)}) \xrightarrow{\For} \calD^b
\Coh(\wfrakg^{(1)})$, this morphism corresponds to the isomorphism
considered in \cite[5.2.2]{RAct}. Hence it is also an isomorphism.

Recall the shift functor $\langle 1 \rangle$ defined in \S \ref{ss:koszulinclusion}. The following lemma follows
immediately from isomorphism
\eqref{eq:isomactiongr} and the exact sequence \eqref{eq:kernelswfrakggr}.

\begin{lem}\label{lem:trianglereflection-braidgroup}

For $\alpha \in \Phi$, there exists a distinguished triangle of functors
\begin{equation*} \Id\langle 1\rangle
\ \to \ \frakR^{\gr}_{\alpha}\langle -1\rangle \ \to \
\mathbf{J}^{\dg,\gr}_{T_{\alpha}}. \end{equation*}

\end{lem}

\subsection{Gradings} \label{ss:paragraphgradedmodules}

As in \S \ref{ss:paragraphfixedFrobenius}, for simplicity we denote the
variety $\wfrakg^{(1)} \times_{\frakh^*{}^{(1)}} \frakh^*$ by $X$ in
this subsection. Recall the algebra $\widetilde{U}:=\calU \frakg
\otimes_{\frakZ_{{\rm HC}}} S(\frakh)$, also considered in
\S \ref{ss:paragraphfixedFrobenius}. By \cite[3.4.1]{BMR} we have
$R\Gamma(X,\wcalD) \cong \widetilde{U}$. Let
$\widetilde{U}^{\hat{0}}_{\hat{0}}$ be the completion of
$\widetilde{U}$ with respect to the maximal ideal of its center
$\frakZ \otimes_{\frakZ_{{\rm HC}}} {\rm S}(\frakh)$ generated by
  $\frakh$ and $\frakg^{(1)}$; let also
  $(\calU\frakg)^{\hat{0}}_{\hat{0}}$ be the
completion of $\calU \frakg$ with respect to the maximal ideal of
$\frakZ$ corresponding to the character $(0,0)$. Then the projection $\frakh^* \to
  \frakh^* / (W,\bullet)$ induces an isomorphism
  $\widetilde{U}^{\hat{0}}_{\hat{0}} \cong (\calU
  \frakg)^{\hat{0}}_{\hat{0}}$.

As in \S \ref{ss:paragraphfixedFrobenius}, let
$\widehat{\calB^{(1)}}$ be the formal neighborhood of
$\calB^{(1)} \times \{ 0 \}$ in $\wfrakg^{(1)} \times_{\frakh^*{}^{(1)}}
\frakh^*$. Applying \cite[4.1.5]{EGA3_1} to the proper morphism
$\wfrakg^{(1)} \times_{\frakh^*{}^{(1)}} \frakh^* \to \frakg^*{}^{(1)}
\times_{\frakh^*{}^{(1)}/W} \frakh^*$, and using the fact that
$\frakg^*{}^{(1)} \times_{\frakh^*{}^{(1)}/W} \frakh^*$ is affine, we
obtain \begin{equation} \label{eq:isomorphismUtilde}
R^i\Gamma(\widehat{\calB^{(1)}}, \wcalD_{|\widehat{\calB^{(1)}}}) \
  \cong \left\{ \begin{array}{cl} (\calU \frakg)^{\hat{0}}_{\hat{0}} &
      \text{if } i=0, \\ 0 & \text{otherwise.} \end{array}
  \right. \end{equation} Recall also the isomorphism of sheaves of
algebras on $\widehat{\calB^{(1)}}$ (see \S \ref{ss:reviewlocalization}) \begin{equation} \label{eq:isomorphismUtilde2}
  \wcalD_{|\widehat{\calB^{(1)}}} \cong
  \sheafEnd_{\calO_{\widehat{\calB^{(1)}}}}(\calM^0).\end{equation} Let $\frakZ_{\Fr}^+ \subset \frakZ_{\Fr}$ denote the maximal ideal of associated to $0$. There is a surjection $(\calU \frakg)^{\hat{0}}_{\hat{0}} \twoheadrightarrow (\calU \frakg)^{\hat{0}}_{\hat{0}} / \langle \frakZ_{\Fr}^+ \rangle \cong (\calU \frakg)^{\hat{0}}_0.$ Hence the algebra $(\calU \frakg)_0^{\hat{0}}$
is a quotient of $(\calU \frakg)^{\hat{0}}_{\hat{0}} \cong
\Gamma(\widehat{\calB^{(1)}},
\sheafEnd_{\calO_{\widehat{\calB^{(1)}}}}(\calM^0))$. \smallskip

Let $Y$ be a noetherian scheme and $Z \subset Y$ a closed subscheme,
with corresponding ideal $\calI_Z \subset \calO_Y$. Let $\widehat{Z}$
be the formal neighborhood of $Z$ in $Y$ (a formal scheme). Assume
$\widehat{Z}$ is endowed with a $\Gm$-action. If $\calF$ is a coherent
sheaf on $\widehat{Z}$, a structure of $\Gm$-equivariant coherent
sheaf on $\calF$ is the data, for any $n$, of a structure of
$\Gm$-equivariant coherent sheaf on $\calF / (\calI_Z^n \cdot \calF)$
(as a coherent sheaf on the $n$-th infinitesimal neighborhood of $Z$
in $Y$), all these structures being compatible. The
following result is due to V. Vologodsky (see the second appendix in the
preprint version of \cite{BFG}):

\begin{lem}\label{lem:lemmaVologodsky}

Let $f : Y \to Z$ be a proper morphism of $\bk$-schemes. Let $z$
be a point in $Z$, and $Y_{\hat{z}}$ be the formal neighborhood of
$f^{-1}(z)$ in $Y$. Let $\calE$ be a vector bundle on $Y_{\hat{z}}$, such
that $\Ext^1(\calE,\calE)=0$. If $Y_{\hat{z}}$ is endowed with a (arbitrary) $\Gm$-action,
then there exists a
$\Gm$-equivariant structure on $\calE$.

\end{lem}

Now we consider $\widehat{\calB^{(1)}}$ as the formal neighborhood of the zero-section in $\wfrakg^{(1)}$. We have defined a $\Gm$-action on $\wfrakg^{(1)}$ in \eqref{eq:defactionGm}. This action stabilizes the zero-section, hence induces an action on $\widehat{\calB^{(1)}}$. We can apply Lemma \ref{lem:lemmaVologodsky} to the splitting bundle $\calM^0$, the vanishing hypothesis following from \eqref{eq:isomorphismUtilde} and
\eqref{eq:isomorphismUtilde2}. Hence we obtain a
$\Gm$-equivariant structure on $\calM^0$, and a structure of a
$\Gm$-equivariant sheaf of algebras on
$\wcalD_{|\widehat{\calB^{(1)}}}$.

Applying $\Gamma(\widehat{\calB^{(1)}},-)$, we obtain a
$\Gm$-equivariant algebra structure on
$(\calU \frakg)_{\hat{0}}^{\hat{0}}$,
compatible with the $\Gm$-structure on $\frakg^*{}^{(1)}$ induced by the action on $\wfrakg^{(1)}$. Taking the quotient (by a homogeneous ideal), we obtain a
grading on the algebra $(\calU \frakg)_0^{\hat{0}}$. Let
$\Mod^{\fg,\gr}_0((\calU\frakg)_0)$ denote the
category of finitely generated graded modules over this graded
algebra. The following theorem is a
``graded version'' of Theorem \ref{thm:localizationfixedFr}:

\begin{thm}\label{thm:equivUg_0gr}

There exists a fully faithful triangulated functor
\[ \widetilde{\gamma}^{\calB}_0 : \DGCoh^{\gr}((\wfrakg \, \rcap_{\frakg^*
  \times \calB} \, \calB)^{(1)}) \ \to \ \calD^b
\Mod^{\fg,\gr}_0((\calU\frakg)_0),\] commuting with the shifts $\langle 1 \rangle$, such that the following diagram
commutes: \[ \xymatrix@R=14pt{ \DGCoh^{\gr}((\wfrakg \, \rcap_{\frakg^* \times
    \calB} \, \calB)^{(1)}) \ar[rr]^-{\widetilde{\gamma}^{\calB}_0}
  \ar[d]_-{\For} & & \calD^b \Mod^{\fg,\gr}_0((\calU\frakg)_0)
  \ar[d]^-{\For} \\ \DGCoh((\wfrakg \, \rcap_{\frakg^* \times \calB} \,
  \calB)^{(1)}) \ar[rr]^-{\widehat{\gamma}^{\calB}_0} & & \calD^b
  \Mod^{\fg}_0((\calU\frakg)_0). } \]

\end{thm}

This theorem would be easy if we had a $\Gm$-structure on
the whole of $\wcalD$, $\calU \frakg$. Unfortunately we only have such a structure on
completions. As the details are not needed, we postpone the proof of the theorem to \S\S \ref{ss:rkqcGm}, \ref{ss:paragraphUg_0gr}.

\begin{remark} \label{rk:gammaequivalence} Arguing as in the proof of Proposition \ref{prop:simplesinimage} below, one can prove that the functor $\widetilde{\gamma}^{\calB}_0$ is essentially surjective, hence an equivalence of categories (see Remark \ref{rk:remarksurjective} for the ``dual'' statement).\end{remark}

\subsection{Complexes representing a projective module}
\label{ss:paragraphprojectives1}

The abelian category
$\Mod_{(0,0)}^{\fg}(\calU \frakg)$ does not have any projective object. Nevertheless, in the
category $\calD^b \Coh(\wfrakg^{(1)})$ one can define a substitute for this notion. For
$\calF,\calG \in \calD^b \Coh(\wfrakg^{(1)})$, we write simply
$\Hom_{\wfrakg^{(1)}}(\calF,\calG)$ for $\Hom_{\calD^b
  \Coh(\wfrakg^{(1)})}(\calF,\calG)$.

\begin{defin}

Let $\lambda \in \bbX$ be regular. An object $\calM$ of $\calD^b
\Coh(\wfrakg^{(1)})$ is said to
\emph{represent a projective module under} $\gamma^{\calB}_{\lambda}$ if
$\Hom_{\wfrakg^{(1)}}(\calM, (\gamma^{\calB}_{\lambda})^{-1} N[i])
=0$ for any $N \in \Mod^{\fg}_{(\lambda,0)}(\calU \frakg)$ and $i
\neq 0$.

Let $\mu \in W_{\aff}' \bullet \lambda$ be a restricted dominant weight. An object $\calM$ of $\calD^b
\Coh(\wfrakg^{(1)})$ is said to \emph{represent the projective cover of}
$L(\mu)$ \emph{under} $\gamma^{\calB}_{\lambda}$ if for
any $\nu \in W_{\aff}' \bullet \lambda$
restricted and dominant and $i \in \bbZ$, $$\Hom_{\wfrakg^{(1)}}(\calM,
(\gamma^{\calB}_{\lambda})^{-1} L(\nu)[i]) =
\left\{ \begin{array}{cl}
  \bk & \text{if} \ \nu=\mu \text{ and } i=0, \\
  0 & \text{otherwise}. \\
\end{array} \right.$$

\end{defin}

\begin{lem}\label{lem:lemmatranslationsimples}

Let $\lambda \in C_0$, and $v \in W^0$. Then $T_{\lambda}^{-\rho} L(v \bullet
\lambda) \neq 0$ iff $v=\tau_0$. Moreover, $T_{\lambda}^{-\rho}
L(\tau_0 \bullet \lambda) = L(\tau_0 \bullet (-\rho)) = L((p-1)\rho)$.

\end{lem}

\begin{proof} Using \eqref{eq:translationsimples}, we only have to
prove that $v \bullet (-\rho)$ is in the upper closure of $v \bullet
C_0$ iff $v=\tau_0$. Write $v=t_{\nu} \cdot w$ with $\nu
\in \bbX$, $w \in W$. Then $v \bullet
(-\rho)$ is in the upper closure of $v \bullet C_0$ iff
$w=w_0$. The result follows since, in this situation, $\nu$ is uniquely determined
by $w$ (see equation \eqref{eq:restricteddominantweights}). \end{proof}

\begin{prop} \label{prop:projectivecover}

$\rmi$ Let $\lambda \in C_0$, and $w \in W$. The object $\calO_{\wfrakg^{(1)}}$
represents the projective cover
of $L(\tau_0 \bullet \lambda)$ under $\gamma^{\calB}_{w \bullet \lambda}$.

$\rmii$ Let $\lambda = \omega \bullet 0 \in C_0$, with $\omega=w \cdot t_{\mu}$ 
($\mu \in \bbX$, $w \in W$). Then $\calO_{\wfrakg^{(1)}}(\mu)$ represents the projective cover of $L(\tau_0 \bullet \lambda)$ under $\gamma^{\calB}_0$.

\end{prop}

\begin{proof} $\rmi$ Consider the functor $T_{\lambda}^{-\rho}=T_{w \bullet
  \lambda}^{w \bullet (-\rho)}=T_{w \bullet \lambda}^{-\rho}$. By
Proposition \ref{prop:translationBMR} applied to the weights $w \bullet
\lambda$ and $-\rho$, with $\calP=G/G=\{{\rm pt}\}$, we
have \begin{equation} \label{eq:structuresheafproj} T_{w \bullet
    \lambda}^{-\rho} \circ \gamma^{\calB}_{w \bullet \lambda} \cong
  \gamma^{ \{ {\rm pt} \} }_{-\rho} \circ R\Gamma(\wfrakg^{(1)}, -). \end{equation}
Moreover, $\Hom_{\wfrakg^{(1)}}(\calO_{\wfrakg^{(1)}}, -) \cong H^0
(R\Gamma(\wfrakg^{(1)}, -))$. Now the result follows from \eqref{eq:structuresheafproj}
and Lemma \ref{lem:lemmatranslationsimples}, using the fact that $\gamma^{
  \{{\rm pt}\} }_{-\rho}(\bk) = L((p-1)\rho)$. (The latter fact follows from the definition of the splitting bundles, or from the fact that $L((p-1)\rho)$ is
the only simple module in $\Mod^{\fg}_{(-\rho,0)}(\calU \frakg)$.) 

$\rmii$ By hypothesis, $\lambda= w \bullet
(p\mu)$. Hence $w^{-1} \bullet \lambda=p\mu$. By $\rmi$, $\calO_{\wfrakg^{(1)}}$ represents the projective cover
of $L(\tau_0 \bullet \lambda)$ under $\gamma^{\calB}_{w^{-1} \bullet
  \lambda}=\gamma^{\calB}_{p\mu}$. But
$\gamma^{\calB}_{p\mu}(-)=\gamma^{\calB}_0(-
\otimes_{\calO_{\wfrakg^{(1)}}} \calO_{\wfrakg^{(1)}}(\mu))$ (see
\S \ref{ss:reviewlocalization}). The result follows. \end{proof}

Recall that we have defined the objects $P(w \bullet 0)$, $\calP_w$ in \S \ref{ss:objectsLwPw}. Consider the projection $p: (\wfrakg \, \rcap_{\frakg^*
  \times \calB} \, \calB)^{(1)} \to \wfrakg^{(1)}$. By adjunction, if $\calM \in \calD^b \Coh(\wfrakg^{(1)})$ represents a
projective module under $\gamma^{\calB}_0$, then $\widehat{\gamma}^{\calB}_0 (Lp^* \calM)$ is a
projective $(\calU \frakg)_0^{\hat{0}}$-module. In particular, with the notation of Proposition \ref{prop:projectivecover}$\rmii$,
\begin{equation} \label{eq:projective(p-2)rho} \calO_{(\wfrakg
    \, \rcap_{\frakg^* \times \calB} \, \calB)^{(1)}}(\mu) \cong
\calP_{\tau_0 \omega}. \end{equation}

\subsection{Graded projective $(\calU \frakg)_0$-modules} \label{ss:paragraphindecomposableprojectives}

Using Theorem \ref{thm:thmGG}$\rmii$,$\rmiv$, the projective modules $P(w \bullet 0)$ ($w \in W^0$) can be lifted to graded modules (uni\-quely, up to a shift). Fix an arbitrary choice of a lift $P^{\gr}(w \bullet 0)$ for each $P(w \bullet 0)$. Recall the fully faithful functor $\widetilde{\gamma}^{\calB}_0$ of Theorem \ref{thm:equivUg_0gr}.

\begin{prop} \label{prop:propprojectives}

The $P^{\gr}(w \bullet 0)$'s are in the essential image of $\widetilde{\gamma}^{\calB}_0$.

\end{prop}

\begin{proof} We prove the result by descending induction on $\ell(w)$. By Proposition \ref{prop:tau_0}, the $w \in W^0$ such that $\ell(w)$ is maximal are of the form $w=\tau_0 \omega$, for with $\ell(\omega)=0$. In this case, $\calP_{\tau_0 \omega}$ is given by \eqref{eq:projective(p-2)rho}. Clearly, this object can be considered as an object of $\DGCoh^{\gr}((\wfrakg \, \rcap_{\frakg^* \times \calB} \, \calB)^{(1)})$. By Theorems \ref{thm:thmGG}$\rmiv$ and \ref{thm:equivUg_0gr}, the image of this (graded) object under $\widetilde{\gamma}^{\calB}_0$ is isomorphic to $P^{\gr}(\tau_0 \omega \bullet 0)$, up to a shift. This proves the result when $\ell(w)=\ell(\tau_0)$.

Now let $n < \ell(\tau_0)$, and
assume the result is true for all $v \in W^0$ such that $\ell(v) >
n$. Let $w \in W^0$ such that $\ell(w)=n$. Let $\delta \in
\Phi_{\aff}$ be such that $ws_{\delta} \in W^0$ and $ws_{\delta}
\bullet 0 > w \bullet 0$, i.e.~$\ell(w s_{\delta}) > \ell(w)$. By
induction, there exists $\calP^{\gr}$ in $\DGCoh^{\gr}((\wfrakg
\, \rcap_{\frakg^* \times \calB} \, \calB)^{(1)})$ such that
$\widetilde{\gamma}^{\calB}_0(\calP^{\gr}) \cong P^{\gr}(w s_{\delta}
\bullet 0)$. Then, consider
$\widetilde{\gamma}^{\calB}_0(\frakR_{\delta}^{\gr} \calP^{\gr}).$
By diagrams \eqref{eq:diagramRalpha} and
\eqref{eq:diagramRdeltagr}, the image of this object under the
forgetful functor $$\For : \calD^b \Mod^{\fg,\gr}_0((\calU\frakg)_0)
\ \to \ \calD^b \Mod^{\fg}_0((\calU\frakg)_0)$$ is $R_{\delta} P(w
s_{\delta} \bullet 0)$. In particular
$\widetilde{\gamma}^{\calB}_0(\frakR_{\delta}^{\gr} \calP^{\gr})$ is
concentrated in degree $0$. By Proposition
\ref{prop:reflectionprojectivesdetails}, $R_{\delta} P(w s_{\delta}
\bullet 0)$ is a direct sum of $P(w \bullet 0)$ and some $P(v \bullet
0)$ with $v \in W^0$ such that $\ell(v) > \ell(w)$. Hence, using the
remarks before Corollary \ref{cor:corGG},
$\widetilde{\gamma}^{\calB}_0(\frakR_{\delta}^{\gr} \calP^{\gr}) \cong
P^{\gr}(w \bullet 0) \langle i \rangle \oplus Q^{\gr}$ for some $i \in
\bbZ$, where $Q^{\gr}$ is a direct sum of graded modules of the form
$P^{\gr}(v \bullet 0) \langle j \rangle$ with $j \in \bbZ$ and $v \in
W^0$ such that $\ell(v) > \ell(w)$. By induction, there
exists $\calQ^{\gr}$ in $\DGCoh^{\gr}((\wfrakg
\, \rcap_{\frakg^* \times \calB} \, \calB)^{(1)})$ such that $Q^{\gr} \cong
\widetilde{\gamma}^{\calB}_0 (\calQ^{\gr})$. Then we have
\[ \widetilde{\gamma}^{\calB}_0(\frakR_{\delta}^{\gr} \calP^{\gr})
\ \cong \ \widetilde{\gamma}^{\calB}_0(\calQ^{\gr}) \oplus P^{\gr}(w
\bullet 0) \langle i \rangle. \] As $\widetilde{\gamma}^{\calB}_0$ is
fully faithful, the injection
$\widetilde{\gamma}^{\calB}_0(\calQ^{\gr}) \hookrightarrow
\widetilde{\gamma}^{\calB}_0(\frakR_{\delta}^{\gr} \calP^{\gr})$ comes
from a morphism $\calQ^{\gr} \to \frakR_{\delta}^{\gr} \calP^{\gr}$ in
$\DGCoh^{\gr}((\wfrakg \, \rcap_{\frakg^* \times \calB} \,
\calB)^{(1)})$. Let $\calX^{\gr}$ be the cone of this morphism. Then,
by usual properties of triangulated categories, $\widetilde{\gamma}^{\calB}_0(\calX^{\gr} \langle -i
\rangle) \cong P^{\gr}(w \bullet 0).$ This concludes the proof. \end{proof}

\subsection{Generalities on $\Gm$-equivariant quasi-coherent
  dg-mo\-dules} \label{ss:rkqcGm}

In the next two subsections we prove Theorem \ref{thm:equivUg_0gr}.

Let us consider a noetherian scheme $A$, and a non-positively graded
$\Gm$-dg-algebra $\calA$ on $A$ (as in
\S \ref{ss:gradedcase}). Assume also that $\calA$ is locally finitely
generated as an $\calO_A$-algebra. Let $\calD^{\qc}_{\Gm}(A, \, \calA)$,
resp. $\calD^{\qc,\fg}_{\Gm}(A, \, \calA)$, be the full subcategory of
$\calD_{\Gm}(A, \, \calA)$ whose objects have their cohomology
$\calO_A$-quasi-coherent, resp. $\calO_A$-quasi-coherent and
locally finitely generated over $H(\calA)$. Let also
$\calC^{\qc}_{\Gm}(A, \, \calA)$ be the category of $\Gm$-dg-modules which are $\calO_A$-quasi-coherent, and let $\calD \bigl(
\calC^{\qc}_{\Gm}(A, \, \calA) \bigr)$ be the corresponding derived
category. Let $\calD^{\fg} \bigl(
\calC^{\qc}_{\Gm}(A, \, \calA) \bigr)$ be the full subcategory of $\calD \bigl( \calC^{\qc}_{\Gm}(A, \, \calA) \bigr)$ whose objects have their cohomology locally finitely generated over $H(\calA)$.

A proof similar to that of Lemma \ref{lem:Kinjresolutionsupport2} shows that if $\calF \in \calD^{\qc}_{\Gm}(A, \, \calA)$ and $H(\calF)$ is bounded, there exists
a $\Gm$-equivariant K-injective $\calA$-dg-module $\calI$ and a
quasi-isomorphism $\calF \to \calI$, where $\calI$ is $\calO_A$-quasi-coherent. We deduce:

\begin{lem} \label{lem:equivqcGm}

Assume $\calA$ is bounded for the cohomological grading. There exists
an equivalence of categories $\calD^{\fg} \bigl(
  \calC^{\qc}_{\Gm}(A, \, \calA) \bigr) \cong
  \calD^{\qc,\fg}_{\Gm}(A, \, \calA).$

\end{lem}

Let $Y$ be a noetherian scheme endowed with a
$\Gm$-action (possibly non trivial). Now we consider two
situations, denoted $(a)$ and $(b)$.

Situation $(a)$ is the following. Let $\calY$ be a non-positively
graded dg-algebra on $Y$. We have not defined $\Gm$-equivariant dg-algebras and
dg-modules in this case. But assume that
$\calY$ is \emph{coherent} as an $\calO_Y$-module, and that each $\calY^p$ is
equipped with a $\Gm$-equivariant structure (as a coherent
$\calO_Y$-module), such that the multiplication and the differential
are $\Gm$-equivariant. Then we can consider the notion of an $\calO_Y$-\emph{quasi-coherent},
$\Gm$-\emph{equivariant} dg-module
over $\calY$. We denote by $\calC^{\qc}_{\Gm}(Y, \, \calY)$ the
corresponding category, and by $\calC^{\qc,\fg}_{\Gm}(Y, \, \calY)$ the
full subcategory of dg-modules locally finitely generated
over $\calY$. We denote the corresponding derived categories by $\calD
\bigl( \calC^{\qc}_{\Gm}(Y, \, \calY) \bigr)$ and $\calD \bigl(
\calC^{\qc,\fg}_{\Gm}(Y, \, \calY) \bigr)$. We also denote by $\calD^{\fg}
\bigl( \calC^{\qc}_{\Gm}(Y, \, \calY) \bigr)$ the full subcategory of
$\calD \bigl( \calC^{\qc}_{\Gm}(Y, \, \calY) \bigr)$ whose objects have
locally finitely generated cohomology.

Consider a closed $\Gm$-subscheme $Z \subset Y$. Denote by
$\calD^{\fg}_Z \bigl( \calC^{\qc}_{\Gm}(Y, \, \calY) \bigr)$ the full
subcategory of $\calD^{\fg} \bigl( \calC^{\qc}_{\Gm}(Y, \, \calY) \bigr)$
whose objects have their cohomology supported on $Z$. We also consider
the category $\calC^{\qc,\Gm}_Z(Y, \, \calY)$ of $\Gm$-equivariant,
quasi-coherent $\calY$-dg-modules supported on $Z$, its subcategory
$\calC^{\qc,\fg,\Gm}_Z(Y, \, \calY)$, the derived categories $\calD \bigl(
\calC^{\qc,\Gm}_Z(Y, \, \calY) \bigr)$, $\calD \bigl(
\calC^{\qc,\fg,\Gm}_Z(Y, \, \calY) \bigr)$, and the full subcategory
$\calD^{\fg} \bigl( \calC^{\qc,\Gm}_Z(Y, \, \calY) \bigr)$ of $\calD
\bigl( \calC^{\qc,\Gm}_Z(Y, \, \calY) \bigr)$ of objects having locally
finitely generated cohomology.

Now we consider situation $(b)$. As above, let $Z \subset Y$ be a
closed $\Gm$-subscheme. Let $\widehat{\calY}$ be a coherent sheaf of
dg-algebras on the formal neighborhood $\widehat{Z}$ of $Z$ in $Y$,
endowed with a $\Gm$-equivariant structure. Hence,
if $\calI_Z$ is the defining ideal
of $Z$ in $Y$, we have a $\Gm$-equivariant structure on
$\widehat{\calY} / (\calI_Z^n \cdot \widehat{\calY})$ for any $n > 0$,
and all these structures are compatible. Then we define the
category $\calC^{\qc,\Gm}_Z(Y, \, \widehat{\calY})$ whose objects
are quasi-coherent, $\Gm$-$\calO_Y$-dg-modules supported
on $Z$, endowed with a compatible action of $\widehat{\calY}$. (By
definition such an object is a direct limit of dg-modules over some
quotients $\widehat{\calY} / (\calI_Z^n \cdot \widehat{\calY})$ for $n
\gg 0$.) We use the same notation as above for the
categories of locally finitely generated dg-modules, and for the
derived categories.  Observe that situation $(a)$ is a particular case of situation $(b)$. (Take $\widehat{\calY}$ to be the completion of $\calY$.)

Recall the construction of resolutions by injective $\Gm$-equivariant
quasi-coherent shea\-ves on $Y$ (see \cite{BEZPer}): if $\calF$ is an
injective object of
$\QCoh(Y)$, then $\Av(\calF):=a_* p_Y^* \calF$ is injective in
$\QCoh^{\Gm}(Y)$, where $a,p_Y: \Gm \times Y \to Y$ are the
action and the projection, respectively. It follows from this
construction, using the non-equivariant case (see \cite[3.1.7]{BMR2}),
that any $\Gm$-equivariant quasi-coherent
sheaf on $Y$ which is supported on $Z$ can be embedded in an
injective $\Gm$-equivariant quasi-coherent sheaf supported on
$Z$. Then, arguments similar to those of the proof of
Proposition \ref{prop:dgmodulessupport} give:

\begin{lem} \label{lem:equivqcGmsupport}

$\rmi$ Assume we are in situation $(b)$. Then there exists
an equivalence of categories $\calD \bigl(
  \calC^{\qc,\fg,\Gm}_Z(Y, \, \widehat{\calY}) \bigr) \cong \calD^{\fg} \bigl( \calC^{\qc,\Gm}_Z(Y, \, \widehat{\calY}) \bigr)$.

$\rmii$ Assume we are in situation $(a)$. Then there exists an
equivalence of categories $\calD \bigl(
  \calC^{\qc,\fg,\Gm}_Z(Y, \, \calY) \bigr) \cong \calD^{\fg}_Z \bigl(
  \calC^{\qc}_{\Gm}(Y, \, \calY) \bigr)$.

\end{lem}

As in \S \ref{ss:koszulinclusion}, we denote by $\langle 1 \rangle$ the shift in the internal grading.

\begin{lem} \label{lem:morphsimsGm-nonGm}

$\rmi$ Consider situation $(a)$. For $\calF, \calG \in \calD^{\fg} \bigl( \calC^{\qc}_{\Gm}(Y, \, \calY) \bigr)$, \begin{equation*} \bigoplus_{m \in \mathbb{Z}} \ \Hom_{\calD^{\fg} \bigl(
\calC^{\qc}_{\Gm}(Y, \, \calY) \bigr)}(\calF, \calG \langle m \rangle) \
\cong \ \Hom_{\calD^{\qc,\fg}(Y, \, \calY)}(\For \ \calF, \For \
\calG), \end{equation*} where $\For$ is the forgetful functor.

$\rmii$ Consider situation $(b)$. For $\calF, \calG \in \calD^{\fg} \bigl( \calC^{\qc,\Gm}_Z(Y, \,
\widehat{\calY}) \bigr)$,
\begin{equation*} \bigoplus_{m \in \mathbb{Z}} \ \Hom_{\calD^{\fg} \bigl(
\calC^{\qc,\Gm}_Z(Y, \, \widehat{\calY}) \bigr)}(\calF, \calG \langle m \rangle) \
\cong \ \Hom_{\calD^{\fg} \bigl( \calC^{\qc}_Z (Y, \, \widehat{\calY})
  \bigr) }(\For \ \calF, \For \ \calG), \end{equation*} where $\For$
is the forgetful functor.

\end{lem}

\begin{proof} $\rmi$ Using an open affine covering, we can assume $Y$ is affine, hence consider categories
of modules over a dg-algebra (see Proposition \ref{prop:dgmodulessupport}). By Lemma
\ref{lem:equivqcGmsupport}$\rmii$, we can assume $\calG$ is finitely
generated. Using a truncation functor, we can assume $\calF$ is
bounded above. Using the remarks after Lemma \ref{lem:Ug0centralcharacter} and the construction
of K-projective resolutions as in \cite[10.12]{BLEqu}, we can even assume that
$\calF^p$ is finitely generated over $\calY^0$ for any $p$, that
for all $m \in \mathbb{Z}$ we have $\Hom_{\calD^{\fg} \bigl(
\calC^{\qc}_{\Gm}(Y, \, \calY) \bigr)}(\calF, \calG \langle m \rangle)
\ \cong \ \Hom_{\calH^{\fg} \bigl(
\calC^{\qc}_{\Gm}(Y, \, \calY) \bigr)}(\calF, \calG \langle m \rangle)$
(where $\calH$ denotes the homotopy category),
and finally that $\Hom_{\calD^{\qc,\fg}(Y, \, \calY)}(\For \ \calF, \For \
\calG) \ \cong \ \Hom_{\calH^{\qc,\fg}(Y, \, \calY)}(\For \ \calF, \For \
\calG).$ The result follows, since it is clear that
\[ \Hom_{\calH^{\qc,\fg}(Y, \, \calY)}(\For \ \calF, \For \ \calG) \ \cong \
\bigoplus_m \Hom_{\calH^{\fg} \bigl(
\calC^{\qc}_{\Gm}(Y, \, \calY) \bigr)}(\calF, \calG \langle m \rangle).\] 

Now we deduce $\rmii$ from $\rmi$. By Lemma
\ref{lem:equivqcGmsupport}$\rmi$ we can assume $\calF$ and $\calG$ are
locally finitely generated. Let us prove that for any $m \in
\mathbb{Z}$ the morphism \begin{equation} \label{eq:morphismmorphisms}
  \Hom_{\calD^{\fg} \bigl( \calC^{\qc,\Gm}_Z(Y, \, \widehat{\calY})
    \bigr)}(\calF, \calG \langle m \rangle) \ \to \ \Hom_{\calD^{\fg}
    \bigl( \calC^{\qc}_Z(Y, \, \widehat{\calY}) \bigl)}(\For
  \ \calF, \For \ \calG)\end{equation} is injective. It is sufficient
to prove that if $f : \calF \to \calG$ is a morphism of
$\Gm$-dg-modules such that $\For(f)=0$, then
$f=0$ in $\calD^{\fg} \bigl(
\calC^{\qc,\Gm}_Z(Y, \, \widehat{\calY}) \bigr)$. Using a non-equivariant
analogue of Lemma \ref{lem:equivqcGmsupport}$\rmi$, there exists
$\calP$ in $\calC^{\qc,\fg}_Z(Y, \, \widehat{\calY})$ and a
quasi-isomorphism $\calG
\xrightarrow{\qis} \calP$ whose composition with $f$ is homotopic to
$0$. The dg-modules $\calF$, $\calG$ and $\calP$ live on a certain
infinitesimal neighborhood $Z^{[i]}$ of $Z$ in $Y$. Applying the
injectivity statement in $\rmi$ to the scheme $Z^{[i]}$,
we obtain that we can choose $\calP$ and the quasi-isomorphism $\calG
\to \calP$ to be $\Gm$-equivariant. This proves the injectivity of
\eqref{eq:morphismmorphisms}.

The injectivity of the morphism in the lemma follows
from the injectivity of \eqref{eq:morphismmorphisms}, using the fact
that $\Gm$ acts on 
$\Hom_{\calD^{\fg} \bigl( \calC^{\qc}_Z(Y, \widehat{\calY})
  \bigr)}(\For \calF, \For \calG)$, and
that the image of $\Hom_{\calD^{\fg} \bigl(
  \calC^{\qc,\Gm}_Z(Y, \widehat{\calY}) \bigr)}(\calF, \calG \langle m
\rangle)$ has weight $m$. The surjectivity can be proved similarly. \end{proof}

\subsection{Proof of Theorem \ref{thm:equivUg_0gr}}
\label{ss:paragraphUg_0gr}

By Lemma \ref{lem:equDGCoh}, we have \begin{equation} \label{eq:equivDGCohgr}
  \DGCoh^{\gr}((\wfrakg \, \rcap_{\frakg^* \times \calB} \, \calB)^{(1)})
  \ \cong \ \calD^{\qc,\fg}_{\Gm}(\calB^{(1)}, \, \pi_* \calO_{\wfrakg^{(1)}}
  \otimes_{\bk} \Lambda(\frakg^{(1)})). \end{equation} In this section we consider $\widehat{\calB^{(1)}}$ as the formal
neighborhood of $\calB^{(1)}$ in $\wfrakg^{(1)}$. We have seen in
\S \ref{ss:paragraphgradedmodules} that
$\wcalD_{|\widehat{\calB^{(1)}}}$, considered as a coherent sheaf of
rings on $\widehat{\calB^{(1)}} \subset \wfrakg^{(1)}$, is endowed
with a $\Gm$-structure. Hence we can
consider the category
$\calD^{\fg} \bigl( \calC^{\qc,\Gm}_{\calB^{(1)}}(\wfrakg^{(1)}, \,
\wcalD_{|\widehat{\calB^{(1)}}} \otimes_{\bk} \Lambda(\frakg^{(1)}))
\bigr)$ as in \S \ref{ss:rkqcGm}, situation $(b)$.

\begin{lem} \label{lem:FGgr}

There exists an equivalence of categories
\[ \calD^{\qc,\fg}_{\Gm}(\calB^{(1)}, \, \pi_* \calO_{\wfrakg^{(1)}}
\otimes_{\bk} \Lambda(\frakg^{(1)})) \ \cong \ \calD^{\fg} \bigl(
\calC^{\qc,\Gm}_{\calB^{(1)}}(\wfrakg^{(1)}, \,
\wcalD_{|\widehat{\calB^{(1)}}} \otimes_{\bk} \Lambda(\frakg^{(1)}))
\bigr). \]

\end{lem}

\begin{proof} By Lemma \ref{lem:equivqcGm}, there exists an equivalence $\calD^{\qc,\fg}_{\Gm}(\calB^{(1)}, \, \pi_*
  \calO_{\wfrakg^{(1)}} \otimes_{\bk} \Lambda(\frakg^{(1)})) \cong
  \calD^{\fg} \bigl( \calC^{\qc}_{\Gm}(\calB^{(1)}, \, \pi_*
  \calO_{\wfrakg^{(1)}} \otimes_{\bk} \Lambda(\frakg^{(1)}))
  \bigr)$. Now $\pi_*$
induces an equivalence $\calC^{\qc}_{\Gm}(\wfrakg^{(1)}, \, \calO_{\wfrakg^{(1)}} \otimes_{\bk}
  \Lambda(\frakg^{(1)})) \ \xrightarrow{\sim} \ \calC^{\qc}_{\Gm}(\calB^{(1)}, \, \pi_*
  \calO_{\wfrakg^{(1)}} \otimes_{\bk} \Lambda(\frakg^{(1)}))$. Hence \begin{equation*}
  \calD^{\qc,\fg}_{\Gm}(\calB^{(1)}, \, \pi_*
  \calO_{\wfrakg^{(1)}} \otimes_{\bk} \Lambda(\frakg^{(1)})) \ \cong \
  \calD^{\fg} \bigl( \calC^{\qc}_{\Gm}(\wfrakg^{(1)}, \,
  \calO_{\wfrakg^{(1)}} \otimes_{\bk} \Lambda(\frakg^{(1)})) \bigr)
\end{equation*}

Now, using the fact that any object of
$\calC^{\qc}_{\Gm}(\wfrakg^{(1)}, \, \calO_{\wfrakg^{(1)}}
\otimes_{\bk} \Lambda(\frakg^{(1)}))$ has its cohomology supported on
$\calB^{(1)}$, we obtain by Lemma \ref{lem:equivqcGmsupport}$\rmii$ an
equivalence $\calD^{\fg} \bigl( \calC^{\qc}_{\Gm}(\wfrakg^{(1)}, \,
  \calO_{\wfrakg^{(1)}} \otimes_{\bk} \Lambda(\frakg^{(1)})) \bigr)
  \ \cong \ \calD \bigl( \calC^{\qc,\fg,\Gm}_{\calB^{(1)}}(\wfrakg^{(1)}, \,
  \calO_{\wfrakg^{(1)}} \otimes_{\bk} \Lambda(\frakg^{(1)}))
  \bigr)$.  Then, using analogues of
the functors $F, G$ of the proof of Theorem
\ref{thm:localizationfixedFr}, we have
$\calC^{\qc, \fg, \Gm}_{\calB^{(1)}}(\wfrakg^{(1)}, \,
  \calO_{\wfrakg^{(1)}} \otimes_{\bk} \Lambda(\frakg^{(1)}))
  \ \cong \ \calC^{\qc,\fg,\Gm}_{\calB^{(1)}}(\wfrakg^{(1)}, \,
  \wcalD_{|\widehat{\calB^{(1)}}} \otimes_{\bk}
  \Lambda(\frakg^{(1)})).$ The lemma follows, using Lemma \ref{lem:equivqcGmsupport}$\rmi$. \end{proof}

We have seen in \S \ref{ss:paragraphgradedmodules} that $(\calU
\frakg)^{\hat{0}}_{\hat{0}}$, considered as a sheaf of algebras on the
formal neighborhood\footnote{This formal neighborhood is also
  isomorphic to the formal neiborhood of $\{(0,0)\}$ in
  ${\rm Spec}(\frakZ) \cong \frakg^* {}^{(1)} \times_{\frakh^* {}^{(1)} / W}
  \frakh^* / (W,\bullet)$. We do not distinguish these formal
  neighborhoods.} $\widehat{\{0\}}$ of $\{0\}$ in $\frakg^*
{}^{(1)}$, is endowed with a $\Gm$-structure. Hence we are again in situation
$(b)$ of \S \ref{ss:rkqcGm}. We simplify the notation for the
categories of $\calU \frakg$-modules, and denote e.g.~by
$\calC^{\fg,\Gm}_{(0,0)}(\calU \frakg \otimes_{\bk}
\Lambda(\frakg^{(1)}))$ the category $\calC^{\qc,\fg,\Gm}_{\{0\}}(
\frakg^* {}^{(1)}, \, \calU \frakg_{|\widehat{\{0\}}} \otimes_{\bk}
\Lambda(\frakg^{(1)}))$. By Lemma \ref{lem:equivqcGmsupport}$\rmi$ we
have \begin{equation} \label{eq:equiv1}
  \calD \bigl(
  \calC^{\fg,\Gm}_{(0,0)}(\calU \frakg \otimes_{\bk}
\Lambda(\frakg^{(1)})) \bigr) \ \cong \ \calD^{\fg} \bigl(
\calC^{\Gm}_{(0,0)}(\calU \frakg \otimes_{\bk}
\Lambda(\frakg^{(1)})) \bigr).\end{equation}

Recall the remarks before Lemma \ref{lem:Ug0centralcharacter}. Let us
consider the following forgetful functors (of the internal grading):
{\small \begin{align*} \For: \calD^{\fg} \bigl( \calC^{\Gm}_{(0,0)}(\calU
  \frakg \otimes_{\bk}
\Lambda(\frakg^{(1)})) \bigr) \ & \to \ \calD^{\fg}_{0}(\calU \frakg
\otimes_{\bk} \Lambda(\frakg^{(1)})), \\ \For: \calD^{\fg} \bigl(
\calC^{\qc,\Gm}_{\calB^{(1)}}(\wfrakg^{(1)}, \,
\wcalD_{|\widehat{\calB^{(1)}}} \otimes_{\bk} \Lambda(\frakg^{(1)}))
\bigr) \ & \to \
\calD^{\fg} \bigl( \calC^{\qc}_{\calB^{(1)}} (\wfrakg^{(1)}, \,
\wcalD_{|\widehat{\calB^{(1)}}} \otimes_{\bk}
\Lambda(\frakg^{(1)})) \bigr). \end{align*}}Clearly, the category
$\calD^{\fg} \bigl( \calC^{\qc}_{\calB^{(1)}} (\wfrakg^{(1)}, \,
\wcalD_{|\widehat{\calB^{(1)}}} \otimes_{\bk}
\Lambda(\frakg^{(1)})) \bigr)$ is equivalent to the category
$\calD^{\qc,\fg}_{\calB^{(1)} \times \{0\}}(X, \, \wcalD \otimes_{\bk}
\Lambda(\frakg^{(1)}))$ (see Proposition
\ref{prop:dgmodulessupport}). Here $X=\wfrakg^{(1)}
\times_{\frakh^* {}^{(1)}} \frakh^*$, as in \S \ref{ss:paragraphfixedFrobenius}. By Lemma \ref{lem:morphsimsGm-nonGm}$\rmii$ we have:

\begin{lem} \label{lem:lemmorphismsGm}

$\rmi$ For $M,N \in \calD^{\fg} \bigl( \calC^{\Gm}_{(0,0)}(\calU \frakg
\otimes_{\bk} \Lambda(\frakg^{(1)})) \bigr)$, \begin{multline*} \bigoplus_{m \in \mathbb{Z}} \Hom_{\calD^{\fg}
    \bigl( \calC^{\Gm}_{(0,0)}(\calU \frakg \otimes_{\bk}
\Lambda(\frakg^{(1)})) \bigr)}(M,N\langle m \rangle) \\ \cong \
\Hom_{\calD^{\fg}_{0}(\calU \frakg \otimes_{\bk}
  \Lambda(\frakg^{(1)}))}(\For \ M, \For \ N). \end{multline*}

$\rmii$ For $\calF,\calG \in \calD^{\fg} \bigl(
\calC^{\qc,\Gm}_{\calB^{(1)}}(\wfrakg^{(1)}, \, \wcalD_{|\widehat{\calB^{(1)}}}
\otimes_{\bk} \Lambda(\frakg^{(1)})) \bigr)$, \begin{multline*} \bigoplus_{m \in \mathbb{Z}} \Hom_{\calD^{\fg}
    \bigl(
\calC^{\qc,\Gm}_{\calB^{(1)}} (\wfrakg^{(1)}, \,
\wcalD_{|\widehat{\calB^{(1)}}} \otimes_{\bk} \Lambda(\frakg^{(1)}))
\bigr)}(\calF, \calG \langle m
\rangle) \\ \cong \ \Hom_{\calD^{\qc,\fg}_{\calB^{(1)} \times
    \{0\}}(X, \, \wcalD \otimes_{\bk} \Lambda(\frakg^{(1)}))}(\For \
\calF,\For \ \calG). \end{multline*}

\end{lem}

\begin{cor} \label{cor:RGammagr}

There exists a fully faithful functor \[ R\Gamma_{\Gm} : \calD^{\fg} \bigl(
\calC^{\qc,\Gm}_{\calB^{(1)}}(\wfrakg^{(1)}, \, \wcalD_{|\widehat{\calB^{(1)}}}
\otimes_{\bk} \Lambda(\frakg^{(1)})) \bigr) \ \to \ \calD^{\fg} \bigl(
\calC^{\Gm}_{(0,0)}(\calU \frakg \otimes_{\bk}
\Lambda(\frakg^{(1)})) \bigr). \]

\end{cor}

\begin{proof} Let us denote by $\calC^{+,\qc,\Gm}_{\calB^{(1)}}(\wfrakg^{(1)},
\, \wcalD_{|\widehat{\calB^{(1)}}} \otimes_{\bk} \Lambda(\frakg^{(1)}))$
the full subcategory
of $\calC^{\qc,\Gm}_{\calB^{(1)}}(\wfrakg^{(1)}, 
\wcalD_{|\widehat{\calB^{(1)}}} \otimes_{\bk} \Lambda(\frakg^{(1)}))$
consisting of bounded below
objects. Using truncation functors, with obvious notation, we have an equivalence
\[\calD^{\fg} \bigl( \calC^{+,\qc,\Gm}_{\calB^{(1)}}(\wfrakg^{(1)}, \, 
\wcalD_{|\widehat{\calB^{(1)}}} \otimes_{\bk} \Lambda(\frakg^{(1)}))
\bigr) \ \cong \ \calD^{\fg}
\bigl( \calC^{\qc,\Gm}_{\calB^{(1)}}(\wfrakg^{(1)}, \,
\wcalD_{|\widehat{\calB^{(1)}}}
\otimes_{\bk} \Lambda(\frakg^{(1)})) \bigr).\]

Consider the functor induced by
$\Gamma(\wfrakg^{(1)},-)$: \[ \Gamma^+ : \calC^{+,\qc,\Gm}_{\calB^{(1)}}
(\wfrakg^{(1)}, \, \wcalD _{|\widehat{\calB^{(1)}}}\otimes_{\bk}
\Lambda(\frakg^{(1)})) \ \to \ \calC^{\Gm}_{(0,0)}(\calU \frakg \otimes_{\bk}
\Lambda(\frakg^{(1)})). \] Let us first show that the derived functor
$R\Gamma^+$ is defined everywhere, i.e.~that
every object of $\calC^{+,\qc,\Gm}_{\calB^{(1)}}(\wfrakg^{(1)}, \,
\wcalD_{|\widehat{\calB^{(1)}}} \otimes_{\bk} \Lambda(\frakg^{(1)}))$
has a right resolution
which is split on the right. We claim that every object $\calF \in \calC^{+,\qc,\Gm}_{\calB^{(1)}}
(\wfrakg^{(1)}, \, \wcalD_{|\widehat{\calB^{(1)}}} \otimes_{\bk}
\Lambda(\frakg^{(1)}))$ has a resolution $\calF \xrightarrow{\qis}
\calI$ where $\calI \in \calC^{+,\qc,\Gm}_{\calB^{(1)}}(\wfrakg^{(1)}, \,
\wcalD_{|\widehat{\calB^{(1)}}} \otimes_{\bk} \Lambda(\frakg^{(1)}))$,
and each $\calI^p$ is acyclic for $\Gamma(\wfrakg^{(1)},-) : \QCoh(\wfrakg^{(1)}) \to {\rm
  Vect}(\bk)$. Indeed, let $\wfrakg^{(1)}=\bigcup_{\alpha=1}^n
X_{\alpha}$ be an affine $\Gm$-stable open covering (e.g.~the inverse image of an affine open covering of $\calB^{(1)}$). Let
$j_{\alpha}: X_{\alpha} \hookrightarrow X$ be the inclusion. Then
there is an inclusion $\calF \hookrightarrow \bigoplus_{\alpha=1}^n
\ (j_{\alpha})_* (j_{\alpha})^* \calF.$ Doing the same construction for
the cokernel, repeating, and taking a total
complex, one
obtains the resolution $\calI$. Such a resolution is clearly split on
the right.

By this construction, it is clear that the following diagram is
commutative, where the vertical arrows are the natural forgetful
functors, and the bottom horizontal arrow is the functor considered in
\eqref{eq:diagram-RGamma-For}: {\small \[ \xymatrix@R=14pt{ \calD \bigl(
  \calC^{+,\qc,\Gm}_{\calB^{(1)}}(\wfrakg^{(1)}, \,
  \wcalD_{|\widehat{\calB^{(1)}}}
  \otimes_{\bk} \Lambda(\frakg^{(1)})) \bigr) \ar[rr]^-{R\Gamma^+}
  \ar[d]_-{\For} & & \calD \bigl( \calC^{\Gm}_{(0,0)}(\calU \frakg
  \otimes_{\bk}
\Lambda(\frakg^{(1)})) \bigr) \ar[d]^-{\For} \\ \calD(X, \, \wcalD
\otimes_{\bk} \Lambda(\frakg^{(1)})) \ar[rr]^{R\Gamma} & & \calD({\rm
  Spec}(\bk), \, \calU \frakg \otimes_{\bk} \Lambda(\frakg^{(1)})). } \] }Using the results just below
\eqref{eq:diagram-RGamma-For}, the functor $R\Gamma^+$ restricts
to \[ R\Gamma_{\Gm}: \calD^{\fg} \bigl(
\calC^{+,\qc,\Gm}_{\calB^{(1)}}(\wfrakg^{(1)}, \,
\wcalD_{|\widehat{\calB^{(1)}}}
\otimes_{\bk} \Lambda(\frakg^{(1)})) \bigr) \ \to \ \calD^{\fg} \bigl(
\calC^{\Gm}_{(0,0)}(\calU \frakg \otimes_{\bk}
\Lambda(\frakg^{(1)})) \bigr),\] which corresponds to $R\Gamma: \calD^{\qc,\fg}_{\calB^{(1)} \times \{0\}}(X, \, \wcalD
\otimes_{\bk} \Lambda(\frakg^{(1)})) \ \to \ \calD^{\fg}_0(\calU \frakg
\otimes_{\bk} \Lambda(\frakg^{(1)}))$ of \eqref{eq:diagramRgamma} under
the natural forgetful functors. The latter functor is
fully faithful. Hence, using Lemma \ref{lem:lemmorphismsGm},
$R\Gamma_{\Gm}$ is also fully faithful. \end{proof}

Using equivalence \eqref{eq:equivDGCohgr}, Lemma \ref{lem:FGgr},
Corollary \ref{cor:RGammagr} and equivalence \eqref{eq:equiv1}, we
obtain a fully faithful functor \[\DGCoh^{\gr}((\wfrakg
\, \rcap_{\frakg^* \times \calB} \, \calB)^{(1)}) \ \to \ \calD \bigl(
\calC^{\fg,\Gm}_{(0,0)}(\calU \frakg \otimes_{\bk}
\Lambda(\frakg^{(1)})) \bigr).\] To finish the proof of Theorem
\ref{thm:equivUg_0gr} we only have to prove:

\begin{lem}\label{lem:lemrestrictedgraded}

There exists an equivalence $\calD \bigl( \calC^{\fg,\Gm}_{(0,0)}(\calU \frakg \otimes_{\bk}
\Lambda(\frakg^{(1)})) \bigr) \cong \calD^b \Mod^{\fg,\gr}_0((\calU
\frakg)_0).$

\end{lem}

\begin{proof} The natural morphism $\calU \frakg \otimes_{\bk}
\Lambda(\frakg^{(1)}) \twoheadrightarrow (\calU \frakg)_0
\twoheadrightarrow (\calU \frakg)_0^{\hat{0}}$ induces: \[\Psi: \calD^b \Mod^{\fg,\gr}_0((\calU
\frakg)_0) \ \to \ \calD \bigl( \calC^{\fg,\Gm}_{(0,0)}(\calU \frakg
\otimes_{\bk}
\Lambda(\frakg^{(1)})) \bigr).\] This functor corresponds to equivalence \eqref{eq:Ug0} under the forgetful
functors. We deduce, as in Corollary \ref{cor:RGammagr}, that
$\Psi$ is fully faithful. Then one checks that $\Psi$ is essentially surjective by induction on the amplitude of the cohomology of dg-modules, using arguments similar to those of Lemma \ref{lem:Kinjresolutionsupport2}. \end{proof}

\section{Simple restricted $(\calU \frakg)^0$-modules} \label{sec:sectionsimplemodules}

In this section we study the LHS of diagram $(*)$ of section \ref{sec:statement}. More precisely, we introduce functors $\frakS_{\delta}$, and study their action on simple modules.

\subsection{The ``semi-simple'' functors $\frakS_{\delta}$}

In this subsection, for each simple root $\delta$ we construct a functor $\frakS_{\delta}$ which ``morally'' represents, on the representation-theoretic side, the complex of functors $\Id \to R_{\delta} \to \Id$. This functor has a ``semi-simplicity'' property (see Proposition \ref{prop:propSalphasimple}).

Let $\alpha \in \Phi$. Recall the subvariety $Y_{\alpha} \subset \wcalN \times \wcalN$ (\S \ref{ss:exactsequenceN}). We denote by $\frakS_{\alpha}$ the convolution functor \[ F_{\wcalN^{(1)} \to
  \wcalN^{(1)}}^{\calO_{Y_{\alpha}^{(1)}}(-\rho,\rho-\alpha)}: \calD^b
\Coh(\wcalN^{(1)}) \ \to \ \calD^{b} \Coh(\wcalN^{(1)}).\]  Now let $\alpha_0 \in \Phi_{\aff} - \Phi$. Recall the notation $\beta, b_0$ of Lemma \ref{lem:conjugationaffinesimpleref}. We define \[ \frakS_{\alpha_0}:= \mathbf{K}_{b_0} \circ \frakS_{\beta} \circ \mathbf{K}_{(b_0)^{-1}}.\] These functors stabilize the subcategory $\calD^{b} \Coh_{\calB^{(1)}}(\wcalN^{(1)})$. They will be related in \S \ref{ss:paragraphdualreflection} to the reflection functors of \S \ref{ss:paragraphreflection}.

For all $\delta \in \Phi_{\aff}$ we have an exact triangle of
endofunctors of $\calD^b \Coh(\wcalN^{(1)})$:
\begin{equation}\label{eq:exacttrianglefunctorsN} \frakS_{\delta} \to
  \mathbf{K}_{C(s_{\delta})} \to \Id. \end{equation} For $\delta \in \Phi$, this follows from \eqref{eq:kernelswcalN2}; for $\delta=\alpha_0$, this is the conjugate of the corresponding triangle for $\beta$.

Now we give a representation-theoretic interpretation of these functors. Recall the equivalence $\epsilon^{\calB}_0$ of \eqref{eq:equivBMR2}, and the module $Q_{\delta}(w)$ defined in \eqref{eq:defQ}.

\begin{prop} \label{prop:propSalphasimple}

Let $w \in W^0$, $\delta \in \Phi_{\aff}$ such that $w s_{\delta} \bullet 0 > w \bullet 0$. Then \[\frakS_{\delta} \calL_w \ \cong \ (\epsilon^{\calB}_0)^{-1}(Q_{\delta}(w)).\]

\end{prop}

\begin{proof} The exact triangle \eqref{eq:exacttrianglefunctorsN} induces an exact triangle \begin{equation} \label{eq:triangle1} \frakS_{\delta}(\calL_w) \to \mathbf{K}_{C(s_{\delta})}(\calL_w) \to \calL_w \end{equation} in $\calD^b \Coh_{\calB^{(1)}}(\wcalN^{(1)})$. Let $i: \wcalN \hookrightarrow \wfrakg$ be the inclusion. Then $i_* \circ  \mathbf{K}_{C(s_{\delta})} \cong  \mathbf{J}_{C(s_{\delta})} \circ i_*$ (Theorem \ref{thm:actionBaff'}). Hence triangle \eqref{eq:triangle1} induces an exact triangle \begin{equation} \label{eq:triangle2} \gamma^{\calB}_0 \circ i_* \circ \frakS_{\delta}(\calL_w) \to \gamma^{\calB}_0 \circ \mathbf{J}_{C(s_{\delta})} \circ i_* (\calL_w) \to \gamma^{\calB}_0 \circ i_* (\calL_w). \end{equation} By construction we have an isomorphism of functors $\gamma^{\calB}_0 \circ i_* \cong {\rm Incl} \circ \epsilon^{\calB}_0$, where ${\rm Incl}$ is induced by the inclusion $\Mod^{\fg}_{0}((\calU \frakg)^0) \hookrightarrow \Mod^{\fg}_{(0,0)}(\calU \frakg)$. In particular, $L(w \bullet 0) \cong \gamma^{\calB}_0 \circ i_* (\calL_w)$. Using diagram \eqref{eq:diagramactionUg}, we deduce $\gamma^{\calB}_0 \circ \mathbf{J}_{C(s_{\delta})} \circ i_* (\calL_w) \cong \mathbf{I}_{C(s_{\delta})}(L(w \bullet 0)).$ Hence triangle \eqref{eq:triangle2} induces a triangle \begin{equation} \label{eq:triangle3} {\rm Incl} \circ \epsilon^{\calB}_0 \circ \frakS_{\delta}(\calL_w) \to \mathbf{I}_{C(s_{\delta})}(L(w \bullet 0)) \to L(w \bullet 0).
\end{equation}

Now by definition (see \cite[2.3]{BMR2}), $\mathbf{I}_{C(s_{\delta})}(L(w \bullet 0))$ is the cone of the natural morphism $L(w \bullet 0) \to R_{\delta} L(w \bullet 0)$. This morphism is the morphism $\phi_{\delta}^w$ of \S \ref{ss:paragraphRT}, hence $\mathbf{I}_{C(s_{\delta})}(L(w \bullet 0)) \cong {\rm Coker}(\phi_{\delta}^w)$. Moreover, under this identification, the second morphism in \eqref{eq:triangle3} is induced by $\psi_{\delta}^w$. Hence triangle \eqref{eq:triangle3} induces an isomorphism ${\rm Incl} \circ \epsilon^{\calB}_0 \circ \frakS_{\delta}(\calL_w) \cong Q_{\delta}(w)$. In particular, $\epsilon^{\calB}_0 \circ \frakS_{\delta}(\calL_w)$ has cohomology only in degree $0$; as the restriction of ${\rm Incl}$ to such objects is fully faithful, the result follows. \end{proof}

To finish this subsection, remark that for $\delta \in \Phi_{\aff}$ there is a functor $\frakS_{\delta}^{\Gm}$ such that the following diagram commutes: \begin{equation} \label{eq:diagramSalphaGm} \xymatrix@R=14pt{\calD^b \Coh^{\Gm}(\wcalN^{(1)}) \ar[rr]^-{\frakS_{\delta}^{\Gm}} \ar[d]^-{\For} & & \calD^b \Coh^{\Gm}(\wcalN^{(1)}) \ar[d]^-{\For} \\ \calD^b \Coh(\wcalN^{(1)}) \ar[rr]^-{\frakS_{\delta}} & & \calD^b \Coh(\wcalN^{(1)}), } \end{equation} namely the graded convolution with kernel $\calO_{Y_{\delta}^{(1)}}(-\rho,\rho-\delta)$ (with its natural $\Gm$-structure) if $\delta \in \Phi$, or the conjugate of the convolution with kernel $\calO_{Y_{\beta}^{(1)}}(-\rho,\rho-\beta)$ by $\mathbf{K}^{\gr}_{b_0}$ if $\delta=\alpha_0$.

\subsection{Graded $(\calU \frakg)^0$-modules} \label{ss:paragraphgradedUg^0modules}

As in \S \ref{ss:paragraphgradedmodules}, we have (see \cite[3.4.1]{BMR2}): \[ (\calU
\frakg)^0 \ \cong \ R\Gamma(\wcalN^{(1)}, \wcalD_{|\wcalN^{(1)} \times
  \{0\}}).\] Recall the $\Gm$-action on $\wcalN^{(1)}$ (\S \ref{ss:paragraphgradedmodules}). The same arguments as in \S \ref{ss:paragraphgradedmodules} show that there exists a $\Gm$-equivariant structure on the algebra $(\calU
\frakg)^0_{\hat{0}}$ (the completion of $(\calU \frakg)^0$ with
respect to the image of the maximal ideal of $\frakZ_{{\rm Fr}}$
associated to $0 \in \frakg^*{}^{(1)}$). We denote by $\Mod^{\fg,\gr}_0((\calU\frakg)^0)$ the
category of graded $(\calU \frakg)^0$-modules with trivial generalized
Frobenius central character\footnote{These modules are modules over the
quotient of $(\calU \frakg)^0$ by a power of the ideal generated by
$\frakg^{(1)}$; this quotient is a graded algebra, hence we can speak
of \emph{graded} modules.}. Arguments similar to those of \S \ref{ss:paragraphUg_0gr} prove the following theorem, which gives a ``graded version'' of equivalence \eqref{eq:equivBMR2}:

\begin{thm}\label{thm:equivUg^0gr}

There exists a fully faithful functor
\[ \widetilde{\epsilon}^{\calB}_0 : \calD^b
\Coh^{\Gm}_{\calB^{(1)}}(\wcalN^{(1)}) \ \to \ \calD^b
\Mod^{\fg,\gr}_0((\calU\frakg)^0),\] commuting with the shifts $\langle 1 \rangle$, such that the following diagram
commutes: \[ \xymatrix@R=12pt{ \calD^b \Coh^{\Gm}_{\calB^{(1)}}(\wcalN^{(1)})
  \ar[rr]^-{\widetilde{\epsilon}^{\calB}_0} \ar[d]^-{\For} & & \calD^b
  \Mod^{\fg,\gr}_0((\calU\frakg)^0) \ar[d]^-{\For} \\ \calD^b
  \Coh_{\calB^{(1)}}(\wcalN^{(1)}) \ar[rr]^-{\epsilon^{\calB}_0} & &
  \calD^b \Mod^{\fg}_0((\calU\frakg)^0). } \]

\end{thm}

Now, consider the category $\Mod^{\fg}_0((\calU \frakg)^0)$. By Theorem \ref{thm:thmGG}, each simple module $L(w \bullet 0)$ ($w \in W^0$) can be lifted to a graded module $L^{\gr}(w \bullet 0)$ in $\Mod^{\fg,\gr}_0((\calU \frakg)^0)$ (uniquely, up to isomorphism and shift). Here our algebra is not finite dimensional, but it acts on every module we consider through a finite dimensional quotient, hence we can apply Theorem \ref{thm:thmGG}. We fix an arbitrary choice for these lifts.

Let $j : \calB^{(1)} \hookrightarrow \wcalN^{(1)}$, $k : \calB^{(1)} \hookrightarrow \wfrakg^{(1)}$ be the inclusions. Let also $\Fr : \calB \to \calB^{(1)}$ be the Frobenius morphism. Note that if $\calG \in \Coh(\calB^{(1)})$, then $\Fr^* \calG \in \Coh(\calB)$ has a structure of $\calD^0$-module, induced by the action on $\calO_{\calB}$.

\begin{lem} \label{lem:zero-section}

For $\calF \in \Coh(\calB^{(1)})$ we have isomorphisms \[ \epsilon_{0}^{\calB}(j_* \calF) \cong R\Gamma \bigl( \calB, \Fr_{\calB}^* (\calF (\rho)) \bigr), \qquad \gamma_{0}^{\calB}(k_* \calF) \cong R\Gamma \bigl( \calB, \Fr_{\calB}^* (\calF (\rho)) \bigr). \]

\end{lem}

\begin{proof} It is well-known that $(\calU \frakg)_0^{-\rho} \cong {\rm End}_{\bk} \bigl( L((p-1)\rho) \bigr)$. It follows, by the choice of the splitting bundles, that \begin{equation} \label{eq:splittingbundle} k^* \calM^0 \ \cong \ \Fr_* (\calO_{\calB}(\rho)) \otimes_{\Fr_* \calO_{\calB}}  \bigl( L((p-1)\rho) \otimes_{\bk} \calO_{\calB^{(1)}} \bigr). \end{equation} Here the structure of $(\calU \frakg)^{-\rho}_0$-module on $L((p-1)\rho)$ gives an action of $\calD^{-\rho}$ on $L((p-1)\rho) \otimes_{\bk} \calO_{\calB^{(1)}}$, hence an action of $\calD^0$ on $\Fr_* (\calO_{\calB}(\rho)) \otimes_{\Fr_* \calO_{\calB}}  \bigl( L((p-1)\rho) \otimes_{\bk} \calO_{\calB^{(1)}} \bigr)$. By Andersen (\cite{ANDFro}) or Haboush (\cite{HABSho}) we have \begin{equation} \label{eq:Haboush-And} \bigl( \Fr_* (\calO_{\calB}(-\rho)) \bigr) \otimes_{\calO_{\calB^{(1)}}} \calO_{\calB^{(1)}} (\rho) \cong L((p-1)\rho) \otimes_{\bk} \calO_{\calB^{(1)}}. \end{equation} Here the LHS has a natural action of $\calD^{-\rho}$, and the isomorphism is $\calD^{-\rho}$-equivariant. From \eqref{eq:splittingbundle} and \eqref{eq:Haboush-And} we deduce \begin{equation} \label{eq:splittingbundle2} (k^* \calM^0) \otimes_{\calO_{\calB^{(1)}}} \calO_{\calB^{(1)}}(-\rho) \cong \Fr_* \calO_{\calB}, \end{equation} where the structure of $\calD^0$-module on the RHS comes from the action on $\calO_{\calB}$.

Using \eqref{eq:splittingbundle2} and the projection formula, we deduce \[ \gamma^{\calB}_{0}(k_* \calF) \cong R\Gamma \bigl( \wfrakg^{(1)}, \calM^0
\otimes_{\calO_{\wfrakg^{(1)}}} k_* \calF \bigr) \cong R\Gamma \bigl( \calB^{(1)}, ({\rm Fr}_* \calO_{\calB}) \otimes_{\calO_{\calB^{(1)}}} (\calF(\rho)) \bigr). \] We deduce the second isomorphism. The first one is similar. \end{proof}

We deduce the following corollary, which generalizes some of the computations of the appendix to \cite{BMR}.

\begin{cor} \label{cor:descriptionsimples}

Let $\omega \in W_{\aff}'$ such that $\ell(\omega)=0$. Write $\omega=w \cdot t_{\mu}$  ($\mu \in \bbX$, $w \in W$). Then we have $\calL_{\omega} \cong j_*
\calO_{\calB^{(1)}}(-\rho + \mu)[\ell(w)]$.

\end{cor}

\begin{proof} By Lemma \ref{lem:zero-section}, $\epsilon^{\calB}_{0}(j_* \calO_{\calB^{(1)}}(-\rho + \mu)) \cong R\Gamma(\calB, \calO_{\calB}(p \mu))$. By hypothesis, $\omega \bullet 0 = w \bullet
(p\mu)$. Hence $w^{-1} \bullet (\omega \bullet 0)=p\mu$. Using Borel-Weil-Bott theorem (\cite[II.5.5-6]{JANAlg}), we deduce \[ \epsilon^{\calB}_0(j_*
\calO_{\calB^{(1)}}(-\rho + \mu)[\ell(w)]) \ \cong \ \Ind_B^G(\omega \bullet 0) \ \cong \ L(\omega \bullet 0).\] This concludes the proof. \end{proof}

\begin{prop} \label{prop:simplesinimage}

For $w \in W^0$, $L^{\gr}(w \bullet 0)$ is in the essential image of $\widetilde{\epsilon}^{\calB}_0$.

\end{prop}

\begin{proof} This proof is similar to that of Proposition \ref{prop:propprojectives}. We use induction on $\ell(w)$. For $\ell(w)=0$, by Corollary \ref{cor:descriptionsimples} we have $\calL_w \cong j_* \calO_{\calB^{(1)}}(-\rho + \mu)[\ell(v)]$. Clearly, $j_* \calO_{\calB^{(1)}}(-\rho + \mu)$ can be lifted to $\calD^b \Coh^{\Gm}_{\calB^{(1)}}(\wcalN^{(1)})$. The result follows when $\ell(w)=0$.

Now assume the result is true when $\ell(w)<n$, and let $w \in W^0$ such that $\ell(w)=n$. Let $\delta \in \Phi_{\aff}$ such that $w s_{\delta} \in W^0$ and $\ell(w s_{\delta}) < \ell(w)$. By induction there exists $\calL^{\gr} \in \calD^b \Coh^{\Gm}_{\calB^{(1)}}(\wcalN^{(1)})$ such that $\widetilde{\epsilon}^{\calB}_0(\calL^{\gr}) \cong L^{\gr}(ws_{\delta} \bullet 0).$ By diagram \eqref{eq:diagramSalphaGm} and Proposition \ref{prop:propSalphasimple}, the image under the forgetful functor (of the grading) of $\widetilde{\epsilon}^{\calB}_0( \frakS_{\delta}^{\Gm} \calL^{\gr})$ is $Q_{\delta}(w s_{\delta})$. By Theorem \ref{thm:conjectureAndersen}, $Q_{\delta}(w s_{\delta}) \cong L(w \bullet 0) \oplus N$ where $N$ is a sum of modules of the form $L(v \bullet 0)$ with $\ell(v) < \ell(w)$. Hence, by the remarks before Corollary \ref{cor:corGG}, we have $\widetilde{\epsilon}^{\calB}_0( \frakS_{\delta}^{\Gm} \calL^{\gr}) \cong L^{\gr}(w \bullet 0) \langle i \rangle \oplus N^{\gr}$ for some $i \in \bbZ$, where $N^{\gr}$ is a sum of modules of the form $L(v \bullet 0) \langle j \rangle$ with $\ell(v) < \ell(w)$. By induction, $N^{\gr}$ is in the essential image of $\widetilde{\epsilon}^{\calB}_0$. We deduce that $L^{\gr}(w \bullet 0)$ is in this image. \end{proof}

\begin{remark} \label{rk:remarksurjective} It follows easily from Proposition \ref{prop:simplesinimage} that the functor $\widetilde{\epsilon}^{\calB}_0$ is essentially surjective. Hence it is an equivalence of categories. \end{remark}

\subsection{Dg versions of the functors $\frakS_{\delta}$} \label{ss:paragraphSalphagr}

Let $\alpha \in \Phi$, $P_{\alpha}$ the parabolic subgroup of $G$ containing $B$
associated to $\{\alpha\}$, $\frakp_{\alpha}$ its Lie algebra, and
$\calP_{\alpha}=G/P_{\alpha}$ the partial flag
variety. We define the variety \begin{equation} \label{eq:defNalpha} \wcalN_{\alpha}:=T^* \calP_{\alpha} =
\{(X,gP_{\alpha}) \in \frakg^* \times \calP_{\alpha} \mid X_{|g \cdot
  \frakp_{\alpha}}=0\}.\end{equation} There exists a natural injection $j_{\alpha} : (\wcalN_{\alpha} \times_{\calP_{\alpha}} \calB)^{(1)}
\hookrightarrow \wcalN^{(1)}.$ We also denote by $\rho_{\alpha} : (\wcalN_{\alpha} \times_{\calP_{\alpha}} \calB)^{(1)} \to \wcalN_{\alpha}^{(1)}$ the morphism defined by base change. Consider the following diagram: {\footnotesize \[ \xymatrix@R=14pt{ & & (\wcalN_{\alpha}
    \times_{\calP_{\alpha}} \calB)
\times_{\wcalN_{\alpha}} (\wcalN_{\alpha} \times_{\calP_{\alpha}}
\calB) \ar[ld]^-{p_1} \ar[rd]_-{p_2} & & \\ & \wcalN_{\alpha}
\times_{\calP_{\alpha}} \calB \ar[rd]^-{\rho_{\alpha}}
\ar@{^{(}->}[ld]_-{j_{\alpha}} & &
\wcalN_{\alpha} \times_{\calP_{\alpha}} \calB
\ar[ld]_{\rho_{\alpha}} \ar@{_{(}->}[rd]^-{j_{\alpha}} \\ \wcalN & &
\wcalN_{\alpha} & & \wcalN. \\ } \] }

\noindent Here to save space we have omitted Frobenius twists. Now $(\wcalN_{\alpha}
\times_{\calP_{\alpha}} \calB) \times_{\wcalN_{\alpha}} (\wcalN_{\alpha}
\times_{\calP_{\alpha}} \calB)$ is isomorphic to the variety
$Y_{\alpha}$. For $\lambda \in \bbX$, we denote by ${\rm Shift}_{\lambda}$ the tensor product with $\calO_{\wcalN^{(1)}}(\lambda)$. Then we have \begin{align*} \Shift_{-\rho} \circ \frakS_{\alpha} \circ \Shift_{\rho} \ & \cong \ \Shift_{-\alpha} \circ F_{\wcalN^{(1)} \to \wcalN^{(1)}}^{\calO_{Y_{\alpha}^{(1)}}} \\ & \cong \ \Shift_{-\alpha} \circ (R(j_{\alpha})_* \circ R(p_2)_* \circ L(p_1)^* \circ L(j_{\alpha})^*) \\ & \cong \ \Shift_{-\alpha} \circ (R(j_{\alpha})_* \circ L(\rho_{\alpha})^* \circ R(\rho_{\alpha})_* \circ L(j_{\alpha})^*). \end{align*} Here the third isomorphism is given by the flat base
change theorem (\cite[II.5.12]{HARAG}). In Corollary \ref{cor:functors+2} we have constructed functors associated to $j_{\alpha}$: \[ \xymatrix{ \DGCoh^{\gr}((\wcalN_{\alpha} \times_{\calP_{\alpha}} \calB)^{(1)}) \ar@<0.5ex>[rr]^-{R(\widetilde{j_{\alpha}}_{\Gm})_*} & & \DGCoh^{\gr}(\wcalN^{(1)}) \ar@<0.5ex>[ll]^-{L(\widetilde{j_{\alpha}}_{\Gm})^*}. } \] Similarly, in Corollary \ref{cor:functors+} we have constructed functors associated to $\rho_{\alpha}$: \[ \xymatrix{ \DGCoh^{\gr}((\wcalN_{\alpha} \times_{\calP_{\alpha}} \calB)^{(1)}) \ar@<0.5ex>[rr]^-{R(\widetilde{\rho_{\alpha}}_{\Gm})_*} & & \DGCoh^{\gr}(\wcalN_{\alpha}^{(1)}) \ar@<0.5ex>[ll]^-{L(\widetilde{\rho_{\alpha}}_{\Gm})^*}. } \] We define the functor $\frakS_{\alpha}^{\gr} : \DGCoh^{\gr}(\wcalN^{(1)}) \ \to \
\DGCoh^{\gr}(\wcalN^{(1)}),$ which sends the object $\calM$
to \[ \calO_{\calB^{(1)}}(\rho - \alpha)
\otimes \bigl( R(\widetilde{j_{\alpha}}_{\Gm})_*
L(\widetilde{\rho_{\alpha}}_{\Gm})^* 
R(\widetilde{\rho_{\alpha}}_{\Gm})_* 
L(\widetilde{j_{\alpha}}_{\Gm})^*( \calM \otimes
\calO_{\calB^{(1)}}(-\rho)) \bigr). \] Using the isomorphism above, the following diagram commutes: \begin{equation*} \xymatrix@R=14pt{ \DGCoh^{\gr}(\wcalN^{(1)}) \ar[d]_-{\For} \ar[rr]^-{\frakS_{\alpha}^{\gr}} & & \DGCoh^{\gr}(\wcalN^{(1)}) \ar[d]^-{\For} \\ \calD^{b} \Coh(\wcalN^{(1)}) \ar[rr]^-{\frakS_{\alpha}} & & \calD^{b} \Coh(\wcalN^{(1)}). } \end{equation*} The following diagram also commutes, where $\eta$ and $\zeta$ were defined in \S \ref{ss:paragraphsimplemodules}: \begin{equation} \label{eq:diagramSalphagr} \xymatrix@R=16pt{ \calD^{b} \Coh^{\Gm}_{\calB^{(1)}}(\wcalN^{(1)}) \ar[d]_-{\frakS_{\alpha}^{\Gm}} \ar[r]^-{\zeta} & \DGCoh^{\gr}(\wcalN^{(1)}) \ar[r]^-{\eta} \ar[d]^-{\frakS_{\alpha}^{\gr}} & \calD^{b} \Coh^{\Gm}(\wcalN^{(1)}) \ar[d]^-{\frakS_{\alpha}^{\Gm}} \\ \calD^{b} \Coh^{\Gm}_{\calB^{(1)}}(\wcalN^{(1)}) \ar[r]^-{\zeta} & \DGCoh^{\gr}(\wcalN^{(1)}) \ar[r]^-{\eta} & \calD^{b} \Coh^{\Gm}(\wcalN^{(1)}) } \end{equation} Indeed, the commutation of the diagram on the right is a consequence of Lemmas \ref{lem:lemmafunctors+1}, \ref{lem:lemmafunctors+}; the commutation of the one on the left follows.

Now we construct a $B_{\aff}'$-action on $\DGCoh^{\gr}(\wcalN^{(1)})$: for $b \in B_{\aff}'$ we define \[ \mathbf{K}_b^{\gr} :
\DGCoh^{\gr}(\wcalN^{(1)}) \ \to \ \DGCoh^{\gr}(\wcalN^{(1)}) \] by the
formula $\mathbf{K}^{\gr}_b := \Shift_{\rho} \circ \kappa_{\calB}^{-1} \circ
\mathbf{J}_b^{\gr} \circ \kappa_{\calB} \circ \Shift_{-\rho}$ (see \eqref{eq:defkappa} for the Koszul duality $\kappa_{\calB}$). Here, $\Shift_{\lambda}$ is the shift by $\calO_{\calB^{(1)}}(\lambda)$. 

Consider the affine simple root $\alpha_0 \in \Phi_{\aff} - \Phi$. Recall the notation $b_0$, $\beta$ from Lemma \ref{lem:conjugationaffinesimpleref}. Then we define the functor \begin{equation} \label{eq:defSalpha0} \frakS_{\alpha_0}^{\gr}:=\mathbf{K}_{b_0}^{\gr} \circ
\frakS_{\beta}^{\gr} \circ \mathbf{K}_{(b_0)^{-1}}^{\gr}.\end{equation} It is not clear from this definition that the diagram analogous to \eqref{eq:diagramSalphagr} is commutative. We will consider this issue in \S \ref{ss:paragraphactionKgr}.

\section{Proof of Theorem \ref{thm:mainthm}} \label{sec:sectionmainthm}

In this section we prove the key-result of our reasoning, Theorem \ref{thm:mainthm}. The main step is to relate the functors $\frakR_{\delta}$ and $\frakS_{\delta}$ via linear Koszul duality.

\subsection{Alternative statement of the theorem} \label{ss:paragraphmainthm}

First, let us state a version of Theorem \ref{thm:mainthm} in representation-theoretic terms, i.e.~involving lifts of $\calU \frakg$-modules instead of coherent sheaves. Recall the Koszul duality $\kappa_{\calB}$ of \eqref{eq:defkappa}. Recall also that the functor $\widetilde{\gamma}^{\calB}_0$ is fully faithful (Theorem \ref{thm:equivUg_0gr}), and that its essential image contains the lifts of projectives (Proposition \ref{prop:propprojectives}). Hence, if $v \in W^0$, for any choice of a lift $P^{\gr}(v \bullet 0)$ of $P(v \bullet 0)$ as a graded $(\calU \frakg)_0^{\hat{0}}$-module, there exists an object\footnote{As observed in \S \ref{ss:objectsLwPw}, this object does not depend on the choice $\lambda=0$.} $\calP^{\gr}_v$ of $\DGCoh^{\gr}((\wfrakg \, \rcap_{\frakg^* \times \calB} \, \calB)^{(1)})$, unique up to isomorphism, such that $P^{\gr}(v \bullet 0) \cong \widetilde{\gamma}^{\calB}_0(\calP^{\gr}_v).$ The same applies to the functor $\widetilde{\epsilon}^{\calB}_0$ of Theorem \ref{thm:equivUg^0gr}, replacing projective by simple.

Theorem \ref{thm:mainthm} is clearly equivalent to the following statement, which we will refer to as statement $(\ddag)$. It will be proved in \S \ref{ss:paragraphproof}. \medskip

\emph{Assume} $p>h$ \emph{is large enough so that Lusztig's conjecture is true.}

\emph{There is a unique choice of the lifts} $P^{\gr}(v \bullet 0)$, $L^{\gr}(v \bullet 0)$ ($v \in W^0$) \emph{such that, if} $\calP^{\gr}_v$, \emph{resp.} $\calL^{\gr}_v$ \emph{is the object of} $\DGCoh^{\gr}((\wfrakg \, \rcap_{\frakg^* \times \calB} \, \calB)^{(1)})$, \emph{resp.} $\calD^b \Coh^{\Gm}_{\calB^{(1)}}(\wcalN^{(1)})$, \emph{such that} $P^{\gr}(v \bullet 0) \cong \widetilde{\gamma}^{\calB}_0(\calP^{\gr}_v)$, \emph{resp.} $L^{\gr}(v \bullet 0) \cong \widetilde{\epsilon}^{\calB}_0(\calL^{\gr}_v)$, \emph{for all} $w \in W^0$ \emph{we have in the category} $\DGCoh^{\gr}(\wcalN^{(1)})$: \begin{equation} \label{eq:isommainthm} \kappa_{\calB}^{-1} \calP^{\gr}_{\tau_0 w} \ \cong \ \zeta(\calL^{\gr}_w)
\otimes_{\calO_{\calB^{(1)}}} \calO_{\calB^{(1)}}(-\rho). \end{equation}

Let us remark that the functors $\widetilde{\gamma}^{\calB}_0$, $\widetilde{\epsilon}^{\calB}_0$ and $\kappa_{\calB}$ commute with the shifts in both the cohomological and the internal grading. The functor $\zeta$ (of Lemma \ref{lem:lemmagradedmodules}) commutes with the cohomological shift, but not with the internal one. More precisely, for $\calF \in \calD^b \Coh^{\Gm}_{\calB^{(1)}}(\wcalN^{(1)})$ one has $\zeta(\calF \langle j \rangle)=\zeta(\calF)[j] \langle j \rangle$. The unicity in Theorem \ref{thm:mainthm} follows easily from these remarks, using the fact that each lift $P^{\gr}(v \bullet 0)$, $L^{\gr}(v \bullet 0)$ ($v \in W^0$) is unique up to a shift $\langle j \rangle$.

The proof of the existence statement occupies the rest of this section. 

\subsection{Koszul dual of the reflection functors} \label{ss:paragraphdualreflection}

Our proof of statement $(\ddag)$ is based on the following result, which shows that the reflection functor $\frakR_{\delta}^{\gr}$ is (almost) conjugate to the semi-simple functor $\frakS^{\gr}_{\delta}$ under Koszul duality.

\begin{thm} \label{thm:dualreflection}

For $\delta \in \Phi_{\aff}$ there is an isomorphism of functors: \[ (\kappa_{\calB})^{-1} \circ \frakR_{\delta}^{\gr} \circ
\kappa_{\calB} \ \cong \ \Shift_{-\rho} \circ \frakS_{\delta}^{\gr} \circ \Shift_{\rho} \ [1] \langle 2 \rangle. \]

\end{thm}

\begin{proof} By definition of $\frakR_{\alpha_0}^{\gr}$ (equation \eqref{eq:defreflectionaffineroot}) and $\frakS^{\gr}_{\alpha_0}$ (equation \eqref{eq:defSalpha0}), it is enough to prove the isomorphism for $\delta \in \Phi$. From now on we write $\alpha$ instead of $\delta$. We derive the theorem from the general results of \S\S \ref{ss:koszulbasechange}, \ref{ss:koszulinclusion}.

First, consider the inclusion of vector bundles
$j_{\alpha} : (\wcalN_{\alpha} \times_{\calP_{\alpha}} \calB)^{(1)}
\hookrightarrow \wcalN^{(1)}.$ We apply to this inclusion the
constructions of \S \ref{ss:koszulinclusion}, with $X=\calB^{(1)}$,
$E=(\frakg^* \times \calB)^{(1)} \cong E^*$, $F_1=(\wcalN_{\alpha}
\times_{\calP_{\alpha}} \calB)^{(1)}$, $F_2=\wcalN^{(1)}$. Then we
have \begin{align*} F_1^{\bot} = (\wfrakg_{\alpha}
  \times_{\calP_{\alpha}} \calB)^{(1)}, \quad F_2^{\bot} =
  \wfrakg^{(1)}, & \quad n_1={\rm dim}(\frakg/\frakb)-1, \quad
n_2={\rm dim}(\frakg/\frakb), \\
\calL_1=\Lambda^{n_1}(\calF_1)=\calO_{\calB^{(1)}}(-2\rho + \alpha), & \quad
\calL_2=\Lambda^{n_2}(\calF_2)=\calO_{\calB^{(1)}}(-2\rho). \end{align*}

We denote by $\widehat{\pi_{\alpha,1}} : \bigl( \wfrakg
\, \rcap_{\frakg^* \times \calB} \, \calB \bigr)^{(1)} \to \bigl( (\wfrakg_{\alpha}
\times_{\calP_{\alpha}} \calB) \, \rcap_{\frakg^* \times \calB} \,
\calB \bigr)^{(1)}$ the morphism of dg-schemes induced by the
inclusion $\wfrakg^{(1)} \hookrightarrow
(\wfrakg_{\alpha} \times_{\calP_{\alpha}} \calB)^{(1)}$, and by $\kappa^{\alpha}$ the ``new''
Koszul duality, as shown in the diagram: \[ \xymatrix@R=18pt{\DGCoh^{\gr}(\wcalN^{(1)})
\ar[d]_-{\kappa_{\calB}}^-{\wr}
\ar@<0.5ex>[rr]^-{L(\widetilde{j_{\alpha}}_{\Gm})^*}
& & \DGCoh^{\gr}((\wcalN_{\alpha} \times_{\calP_{\alpha}}
\calB)^{(1)})
\ar[d]^-{\kappa^{\alpha}}_-{\wr}
\ar@<0.5ex>[ll]^-{R(\widetilde{j_{\alpha}}_{\Gm})_*} \\
\DGCoh^{\gr}((\wfrakg \, \rcap_{\frakg^* \times \calB} \, \calB)^{(1)})
\ar@<0.5ex>[rr]^-{R(\widehat{\pi_{\alpha,1}}_{\Gm})_*} & &
\DGCoh^{\gr}(((\wfrakg_{\alpha} \times_{\calP_{\alpha}} \calB)
\, \rcap_{\frakg^* \times \calB} \, \calB)^{(1)})
\ar@<0.5ex>[ll]^-{L(\widehat{\pi_{\alpha,1}}_{\Gm})^*},}
\] where the functors are defined as in
\S \ref{ss:koszulinclusion}. Applying Proposition \ref{prop:koszulinclusionprop},
one obtains isomorphisms of functors
{\small \begin{equation}\label{eq:dualreflection1} \left\{ \begin{array}{ccl}
\kappa^{\alpha} \circ L(\widetilde{j_{\alpha}}_{\Gm})^* & \cong &
R(\widehat{\pi_{\alpha,1}}_{\Gm})_* \circ \kappa_{\calB}, \\[2pt]
\kappa_{\calB} \circ R(\widetilde{j_{\alpha}}_{\Gm})_* & \cong &
\bigl(L(\widehat{\pi_{\alpha,1}}_{\Gm})^* \circ \kappa^{\alpha} \bigr)
\otimes_{\calO_{\calB^{(1)}}} \calO_{\calB^{(1)}}(\alpha) [-1]
\langle -2\rangle. \end{array} \right. \end{equation}}

Now, consider the base change $\rho_{\alpha} : (\wcalN_{\alpha}
\times_{\calP_{\alpha}} \calB)^{(1)} \to \wcalN_{\alpha}^{(1)}.$ We
apply the constructions of \S \ref{ss:koszulbasechange} to this base
change, with $X=\calB^{(1)}$, $Y=(\calP_{\alpha})^{(1)}$, $E=(\frakg^*
\times \calP_{\alpha})^{(1)}$, $F=\wcalN_{\alpha}^{(1)}$. We
denote by $\widehat{\pi_{\alpha,2}} : \bigl( (\wfrakg_{\alpha}
\times_{\calP_{\alpha}} \calB) \, \rcap_{\frakg^* \times \calB} \, \calB \bigr)^{(1)} \to \bigl( \wfrakg_{\alpha} \, \rcap_{\frakg^* \times \calP_{\alpha}} \, \calP_{\alpha} \bigr)^{(1)}$ the morphism of dg-schemes induced by the base change $\wfrakg_{\alpha} \times_{\calP_{\alpha}} \calB \to
\wfrakg_{\alpha}$. Let $\kappa_{\alpha}$ be the ``new'' Koszul duality equivalence, as shown in the
diagram {\small \[ \xymatrix@R=18pt{\DGCoh^{\gr}((\wcalN_{\alpha}
  \times_{\calP_{\alpha}}
\calB)^{(1)}) \ar[d]^-{\wr}_-{\kappa^{\alpha}}
\ar@<0.5ex>[rr]^-{R(\widetilde{\rho_{\alpha}}_{\Gm})_*} & &
\DGCoh^{\gr}(\wcalN_{\alpha}^{(1)})
\ar[d]_-{\wr}^-{\kappa_{\alpha}}
\ar@<0.5ex>[ll]^-{L(\widetilde{\rho_{\alpha}}_{\Gm})^*} \\
\DGCoh^{\gr}(((\wfrakg_{\alpha} \times_{\calP_{\alpha}} \calB)
\, \rcap_{\frakg^* \times \calB} \, \calB)^{(1)})
\ar@<0.5ex>[rr]^-{R(\widehat{\pi_{\alpha,2}}_{\Gm})_*} & &
\DGCoh^{\gr}((\wfrakg_{\alpha} \, \rcap_{\frakg^* \times
\calP_{\alpha}} \, \calP_{\alpha})^{(1)}),
\ar@<0.5ex>[ll]^-{L(\widehat{\pi_{\alpha,2}}_{\Gm})^*} }
\]}where the functors are defined as in
\S \ref{ss:koszulbasechange}. Applying Proposition
\ref{prop:koszulbasechangeprop}, one obtains
isomorphisms of functors \begin{equation}\label{eq:dualreflection2}
  \left\{ \begin{array}{ccc}
R(\widehat{\pi_{\alpha,2}}_{\Gm})_* \circ \kappa^{\alpha} & \cong &
\kappa_{\alpha} \circ R(\widetilde{\rho_{\alpha}}_{\Gm})_*, \\[3pt]
\kappa^{\alpha} \circ
L(\widetilde{\rho_{\alpha}}_{\Gm})^* & \cong &
L(\widehat{\pi_{\alpha,2}}_{\Gm})^* \circ
\kappa_{\alpha}. \end{array} \right. \end{equation}

Consider the morphism $\widehat{\pi}_{\alpha}$. The composition $\wfrakg \hookrightarrow \wfrakg_{\alpha} \times_{\calP_{\alpha}} \calB \twoheadrightarrow \wfrakg_{\alpha}$ coincides with $\widetilde{\pi}_{\alpha}$. Hence 
$\widehat{\pi}_{\alpha}=\widehat{\pi_{\alpha,2}} \circ
\widehat{\pi_{\alpha,1}}$. It follows that $R(\widehat{\pi}_{\alpha,\Gm})_* \cong R(\widehat{\pi_{\alpha,2}}_{\Gm})_* \circ R(\widehat{\pi_{\alpha,1}}_{\Gm})_*$
and $L(\widehat{\pi}_{\alpha,\Gm})^* \cong L(\widehat{\pi_{\alpha,1}}_{\Gm})^*
\circ L(\widehat{\pi_{\alpha,2}}_{\Gm})^*$ (see \eqref{eq:compositiondirectimageGm}). Hence \eqref{eq:dualreflection1} and \eqref{eq:dualreflection2} allow to
compute $(\kappa_{\calB})^{-1} \circ \frakR^{\gr}_{\alpha} \circ \kappa_{\calB} = (\kappa_{\calB})^{-1} \circ
L(\widehat{\pi}_{\alpha,\Gm})^* \circ R(\widehat{\pi}_{\alpha,\Gm})_*
\circ \kappa_{\calB}$. Namely, we obtain isomorphisms
\begin{align*} R(\widehat{\pi}_{\alpha,\Gm})_* \circ \kappa_{\calB} \ & \cong \
\kappa_{\alpha} \circ R(\widetilde{\rho_{\alpha}}_{\Gm})_* \circ
L(\widetilde{j_{\alpha}}_{\Gm})^* \\
  (\kappa_{\calB})^{-1} \circ L(\widehat{\pi}_{\alpha,\Gm})^* \ & \cong \ \bigl( R(\widetilde{j_{\alpha}}_{\Gm})_* \circ
L(\widetilde{\rho_{\alpha}}_{\Gm})^* \circ
(\kappa_{\alpha})^{-1} \bigr) \otimes
\calO_{\calB^{(1)}}(-\alpha)[1]\langle 2\rangle.\end{align*}
Comparing this with the definition of $\frakS_{\alpha}^{\gr}$ in \S \ref{ss:paragraphSalphagr}, one obtains the isomorphism of the theorem. \end{proof}

\subsection{Action of the braid group on $\DGCoh^{\gr}(\wcalN^{(1)})$} \label{ss:paragraphactionKgr}

Recall that we have defined in \S\S \ref{ss:paragraphgradedversionsaction}, \ref{ss:paragraphSalphagr}, actions of $B_{\aff}'$ on $\calD^b \Coh^{\Gm}(\wcalN^{(1)})$ and $\DGCoh^{\gr}(\wcalN^{(1)})$. Consider the following diagram, where $\eta$ is the
functor of \S \ref{ss:paragraphsimplemodules}: \[ \xymatrix@R=14pt{\DGCoh^{\gr}(\wcalN^{(1)})
  \ar[d]_-{\eta} \ar[rr]^-{\mathbf{K}_b^{\gr}} & &
  \DGCoh^{\gr}(\wcalN^{(1)}) \ar[d]^-{\eta} \\ \calD^b \Coh^{\Gm}(\wcalN^{(1)}) \ar[rr]^-{\mathbf{K}_b^{\Gm}} & &
  \calD^b \Coh^{\Gm}(\wcalN^{(1)}). } \]

\begin{lem} \label{lem:lemcompatibilityactionsK}

For any $\calM \in \DGCoh^{\gr}(\wcalN^{(1)})$, there exists an
isomorphism\footnote{It is not clear from our proof
  whether or not these isomorphisms yield an isomorphism of
  \emph{functors}. This is not important for our arguments, hence we will not consider this issue.} $\eta
\circ \mathbf{K}^{\gr}_b(\calM) \cong
\mathbf{K}^{\Gm}_b \circ \eta(\calM).$

\end{lem} 

\begin{proof} It is
sufficient to prove the isomorphism on a set of generators of
$B_{\aff}'$. For $b=\theta_x$ ($x \in \bbX$), it is obvious. Hence we only have to prove it for $b=T_{\alpha}$. Fix $\alpha \in \Phi$. By Lemma \ref{lem:trianglereflection-braidgroup}, there is a distinguished triangle \[\Id\langle
1\rangle \to \Shift_{\rho} \circ
(\kappa_{\calB})^{-1} \circ \frakR_{\alpha}^{\gr} \circ \kappa_{\calB}
\circ \Shift_{-\rho} \langle -1\rangle \to \mathbf{K}^{\gr}_{T_{\alpha}}.\]
Using Theorem \ref{thm:dualreflection}, for $\calM$ in $\DGCoh^{\gr}(\wcalN^{(1)})$ we obtain a
distinguished \begin{equation} \label{eq:triangleS}
  \eta(\calM)[-1]\langle 1\rangle \to \eta \circ
\frakS_{\alpha}^{\gr}(\calM)\langle 1\rangle \to \eta \circ
\mathbf{K}^{\gr}_{T_{\alpha}}(\calM)\end{equation} (observe that $\eta(\calF \langle j \rangle)=\eta(\calF)[-j] \langle j \rangle$). By \eqref{eq:diagramSalphagr} we have $\eta \circ
\frakS_{\alpha}^{\gr} = \frakS_{\alpha}^{\Gm} \circ \eta$.  Now the exact sequence \eqref{eq:kernelswcalN2} induces a distinguished triangle of functors \begin{equation} \label{eq:triangleS2} \frakS_{\alpha}^{\Gm}\langle 1 \rangle \to
\mathbf{K}^{\Gm}_{T_{\alpha}} \to \Id \langle
1\rangle. \end{equation} Identifying triangle \eqref{eq:triangleS} with
triangle \eqref{eq:triangleS2} applied to $\eta(\calM)$, one obtains the
isomorphisms for $b=T_{\alpha}$. \end{proof}

\begin{remark} \label{rk:rkSalphagr} It follows in particular that diagram \eqref{eq:diagramSalphagr}, with $\alpha$ replaced by $\alpha_0$, is commutative on objects. In particular, for any $\calM \in \DGCoh^{\gr}(\wcalN^{(1)})$ there is an isomorphism $\eta \circ \frakS_{\alpha_0}^{\gr}(\calM) \cong \frakS_{\alpha_0}^{\Gm} \circ \eta(\calM)$. \end{remark}

\subsection{End of the proof of Theorem \ref{thm:mainthm}} \label{ss:paragraphproof}

In this subsection we finally give a proof of the existence statement in $(\ddag)$, by induction on $\ell(w)$.

To begin induction, let us consider some $w \in W^0$ with $\ell(w)=0$. Write $w=v \cdot t_{\mu}$. By Corollary \ref{cor:descriptionsimples}, $\calL_w \cong j_* \calO_{\calB^{(1)}}(-\rho + \mu)[\ell(v)]$. Let us set \[ \calL^{\gr}_w := j_* \calO_{\calB^{(1)}}(-\rho + \mu)[\ell(v)] \langle N - \ell(v) \rangle,\] where $N=\# R^+$, and $j_* \calO_{\calB^{(1)}}$ is endowed with its natural (trivial) $\Gm$-equivariant structure. Then $L^{\gr}(w \bullet 0) := \widetilde{\epsilon}^{\calB}_0(\calL^{\gr}_w)$ is a lift of $L(w \bullet 0)$ as a graded module. By definition of Koszul duality (see \eqref{eq:formulakappa}) we have \[ \kappa_{\calB} \bigl( \zeta(\calL^{\gr}_w) \otimes \calO_{\calB^{(1)}}(-\rho) \bigr) \ \cong \ \Lambda(\calT_{\calB^{(1)}}^{\vee}) \otimes_{\calO_{\calB^{(1)}}} \calO_{\calB^{(1)}}(\mu) \langle - N - \ell(v) \rangle. \] We set \[ \calP^{\gr}_{\tau_0 w} := \Lambda(\calT_{\calB^{(1)}}^{\vee}) \otimes_{\calO_{\calB^{(1)}}} \calO_{\calB^{(1)}}(\mu) \langle - N - \ell(v) \rangle.\] By \eqref{eq:projective(p-2)rho}, $P^{\gr}(\tau_0 w \bullet 0) := \widetilde{\gamma}^{\calB}_0(\calP^{\gr}_{\tau_0 w})$ is a lift of $P(\tau_0 w \bullet 0)$ as a graded module. Moreover, \eqref{eq:isommainthm} holds. This concludes the proof if $\ell(w)=0$.

Now, consider some $w \in W^0$, and assume the result is known for all
$v \in W^0$ with $\ell(v) < \ell(w)$. For all such $v$, we fix the objects $L^{\gr}(v \bullet 0)$, $P^{\gr}(\tau_0 v \bullet 0)$, $\calL^{\gr}_v$, $\calP^{\gr}_{\tau_0 v}$ such that \eqref{eq:isommainthm} is satisfied. Choose some $\delta \in \Phi_{\aff}$ such that, for $s=s_{\delta}$, one has $ws \in W^0$ and $ws \bullet 0 < w \bullet 0$, i.e.~$\ell(ws) < \ell(w)$. In particular we have $\kappa_{\calB} \bigl( \zeta (\calL^{\gr}_{w s}) \otimes \calO_{\calB^{(1)}}(-\rho) \bigr) \cong \calP^{\gr}_{\tau_0 w s}.$  Applying $\frakR^{\gr}_{\delta}$ and using Theorem \ref{thm:dualreflection}, it follows that
\begin{equation} \label{eq:lab} \kappa_{\calB}(\frakS_{\delta}^{\gr} \circ \zeta (\calL^{\gr}_{w s}) \otimes \calO_{\calB^{(1)}}(-\rho))[1] \langle 1 \rangle \ \cong \ \frakR^{\gr}_{\delta} \calP^{\gr}_{\tau_0 w s} \langle -1 \rangle. \end{equation} 

As in the proof of Proposition \ref{prop:propprojectives}, the image under the forgetful functor of $\widetilde{\gamma}^{\calB}_0 (\frakR^{\gr}_{\delta} \calP^{\gr}_{\tau_0 w s})$ is $R_{\delta} P(\tau_0 w s \bullet 0)$; hence there exists a lift $P^{\gr}(\tau_0 w \bullet 0)$ of $P(\tau_0 w \bullet 0)$, and graded finite dimensional vector spaces $V_{\tau_0 v}$ such that \begin{equation} \label{eq:decomp} \widetilde{\gamma}^{\calB}_0 (\frakR^{\gr}_{\delta} \calP^{\gr}_{\tau_0 w s}) \langle -1 \rangle \ \cong \ P^{\gr}(\tau_0 w \bullet 0) \oplus \bigl( \bigoplus_{\genfrac{}{}{0pt}{}{v \in W^0}{\ell(v) < \ell(w)}} P^{\gr}(\tau_0 v \bullet 0) \otimes_{\bk} V_{\tau_0 v} \bigr). \end{equation}

Now let us consider the LHS of equation \eqref{eq:lab}. By \eqref{eq:diagramSalphagr} and Remark \ref{rk:rkSalphagr} we have $\frakS_{\delta}^{\gr} \circ \zeta (\calL^{\gr}_{w s}) \cong \zeta \circ \frakS_{\delta}^{\Gm}  (\calL^{\gr}_{w s})$. As in the proof of Proposition \ref{prop:simplesinimage}, the image of $\widetilde{\epsilon}^{\calB}_0(\frakS_{\delta}^{\Gm} \calL^{\gr}_{w s})$ under the forgetful functor is $Q_{\delta}(ws)$. Hence there is a lift $L^{\gr}(w \bullet 0)$ of $L(w \bullet 0)$ as a graded module, an object $\calQ^{\gr}$ of $\calD^b \Coh^{\Gm}_{\calB^{(1)}}(\wcalN^{(1)})$, and an isomorphism \[ \widetilde{\epsilon}^{\calB}_0(\frakS_{\delta}^{\Gm} \calL^{\gr}_{w s}) \ \cong \ L^{\gr}(w \bullet 0) \langle -1 \rangle \oplus \widetilde{\epsilon}^{\calB}_0(\calQ^{\gr}). \] Let $\calL^{\gr}_w$ be the object of $\calD^b \Coh_{\calB^{(1)}}^{\Gm}(\wcalN^{(1)})$ such that $\widetilde{\epsilon}^{\calB}_0(\calL^{\gr}_w)= L^{\gr}(w \bullet 0)$. Then $\calL^{\gr}_w$ is a direct summand of $\frakS_{\delta}^{\Gm} \calL^{\gr}_{w s} \langle 1 \rangle$, hence $\kappa_{\calB}(\zeta(\calL^{\gr}_w) \otimes_{\calO_{\calB^{(1)}}}
\calO_{\calB^{(1)}}(-\rho))$ is a direct summand of the LHS of \eqref{eq:lab}, thus also of its RHS.

Let us define $\calP^{\gr}_{\tau_0 w} := \kappa_{\calB}(\zeta(\calL^{\gr}_w) \otimes_{\calO_{\calB^{(1)}}} \calO_{\calB^{(1)}}(-\rho)).$ To conclude the induction step, it is enough to prove: \begin{equation} \label{eq:defP2} \widetilde{\gamma}^{\calB}_0(\calP^{\gr}_{\tau_0 w}) \ \cong \ P^{\gr}(\tau_0 w \bullet 0). \end{equation} By definition, $\calP^{\gr}_{\tau_0 w}$ is a direct summand of $\frakR^{\gr}_{\delta} \calP^{\gr}_{\tau_0 w s} \langle -1 \rangle$. Hence $\widetilde{\gamma}^{\calB}_0(\calP^{\gr}_{\tau_0 w})$ is a direct summand of \eqref{eq:decomp}. In particular, it is a graded $(\calU \frakg)_0^{\hat{0}}$-module. Let us show that it is indecomposable. By Proposition \ref{prop:KrullSchmidt}$\rmi$, it is enough to show that its endomorphism algebra is local. This algebra is \begin{multline*} {\rm End}_{\calD^b \Mod^{\fg,\gr}_0((\calU \frakg)_0)} (\widetilde{\gamma}^{\calB}_0(\calP^{\gr}_{\tau_0 w})) \ \cong \ {\rm End}_{\DGCoh^{\gr}((\wfrakg \, \rcap_{\frakg^* \times \calB} \, \calB)^{(1)})}(\calP^{\gr}_{\tau_0 w}) \\ \cong \ {\rm End}_{\calD^b \Coh^{\Gm}_{\calB^{(1)}}(\wcalN^{(1)})}(\calL^{\gr}_w) \ \cong \ {\rm End}_{\calD^b \Mod^{\fg,\gr}_0((\calU \frakg)^0)}(L^{\gr}(w \bullet 0)) \ \cong \ \bk. \end{multline*} Here the first isomorphism follows from the fact that $\widetilde{\gamma}^{\calB}_0$ is fully faithful; the second one from the fact that $\kappa_{\calB}$ and $\zeta$ are fully faithful; the third one from the fact that $\widetilde{\epsilon}^{\calB}_0$ is fully faithful. Hence $\widetilde{\gamma}^{\calB}_0(\calP^{\gr}_{\tau_0 w})$ is indecomposable.

By the Krull-Schmidt theorem (see Proposition \ref{prop:KrullSchmidt}$\rmii$), $\widetilde{\gamma}^{\calB}_0(\calP^{\gr}_{\tau_0 w})$ is one of the indecomposable summands appearing in the RHS of \eqref{eq:decomp}. Hence, to conclude the proof of \eqref{eq:defP2} it is enough to prove that there cannot exist $i \in \bbZ$ and $v \in W^0$ with $\ell(v) < \ell(w)$ such that $\widetilde{\gamma}^{\calB}_0(\calP^{\gr}_{\tau_0 w}) \cong P^{\gr}(\tau_0 v \bullet 0) \langle i \rangle.$

Let us assume that there exist such an $i$ and such a $v$. By induction we have $P^{\gr}(\tau_0 v \bullet 0) \langle i \rangle \cong \widetilde{\gamma}^{\calB}_0 (\calP^{\gr}_{\tau_0 v} \langle i \rangle)$, and $\calP^{\gr}_{\tau_0 v} \langle i \rangle \cong \kappa_{\calB}(\zeta(\calL^{\gr}_v) \otimes \calO_{\calB^{(1)}}(-\rho)) \langle i \rangle.$ Hence, as $\widetilde{\gamma}^{\calB}_0$, $\kappa_{\calB}$ and $\zeta$ are fully faithful, we have $\calL^{\gr}_w \cong \calL^{\gr}_v[-i] \langle i \rangle.$ Applying $\widetilde{\epsilon}^{\calB}_0$ one obtains $L^{\gr}(w \bullet 0) \ \cong \ L^{\gr}(v \bullet 0)[-i] \langle i \rangle,$ which is a contradiction as $v \neq w$.  This concludes the proof of $(\ddag)$, hence also of Theorem \ref{thm:mainthm}.

\subsection{Remark on other alcoves} \label{ss:otheralcoves}

In Theorem \ref{thm:mainthm}, the objects $\calL_w$ and $\calP_w$ correspond to simple and projective modules for any choice of $\lambda \in C_0$, i.e.~they are the simple, resp. projective, objects for the $t$-structure on $\calD^b \Coh_{\calB^{(1)}}(\wcalN^{(1)})$, resp. $\DGCoh((\wfrakg \, \rcap_{\frakg^* \times \calB} \, \calB)^{(1)})$, assigned to the fundamental alcove (see \cite[2.1.5]{BEZICM} for this point of view). We could also consider the simples and projectives for the $t$-structure assigned to another alcove $C_1$, i.e.~the objects which are sent by the equivalence $\epsilon^{\calB}_{\lambda}$, resp. $\widehat{\gamma}^{\calB}_{\lambda}$, to the simple, resp. projective, modules, for any $\lambda \in C_1 \cap \bbX$. The different $t$-structures are related by the braid group action, which commutes with $\kappa_{\calB}$ (see Lemma \ref{lem:lemcompatibilityactionsK}). Hence a statement similar to Theorem \ref{thm:mainthm} is true for any alcove. We will need this extension later to prove Koszulity of singular blocks, but it is not needed in section \ref{sec:sectionapplications}.

More precisely, let $C$ be the intersection of an alcove with $\bbX$. Let $y \in W_{\aff}$ be the unique element such that $C = y \bullet C_0$. Then there exist unique objects $\calL_w^y \in \calD^b \Coh_{\calB^{(1)}}(\wcalN^{(1)})$, $\calP_w^y \in \DGCoh((\wfrakg \, \rcap_{\frakg^* \times \calB} \, \calB)^{(1)})$ ($w \in W^0$) such that for any $\lambda \in C$ and $w \in W^0$ we have \begin{equation}\label{eq:objectsy} \epsilon^{\calB}_{\lambda}(\calL_w^y) \ \cong \ L(w \bullet (y^{-1} \bullet \lambda)), \qquad \widehat{\gamma}^{\calB}_{\lambda}(\calP^y_w) \ \cong \ P(w \bullet (y^{-1} \bullet \lambda)). \end{equation} (In this formula, $y^{-1} \bullet \lambda \in C_0$.) Indeed, there is an element $\overline{y} \in B_{\aff}'$ such that $\gamma^{\calB}_{\lambda} \cong \gamma^{\calB}_{y^{-1} \bullet \lambda} \circ \mathbf{J}_{\overline{y}}$ for any $\lambda \in C$ (see \cite{BEZICM} and \cite[section 2]{BMR2}). Here $\overline{y}$ is not unique, but the functor $\mathbf{J}_{\overline{y}}$ is clearly unique. Then, if we set $\calL_w^y:=\mathbf{K}_{\overline{y}}^{-1}(\calL_w)$ and $\calP_w^y:=(\mathbf{J}_{\overline{y}}^{\dg})^{-1}(\calP_w)$, isomorphisms \eqref{eq:objectsy} are satisfied. Then the following theorem holds true:

\begin{thm} \label{thm:mainthmalcove}

Assume $p>h$ is such that Lusztig's conjecture is true. There is a unique choice of lifts $\calP^{y,\gr}_v$ of $\calP_v^y$, $\calL^{y,\gr}_v$ of $\calL_v^y$ ($v \in W^0$), such that for $w \in W^0$, $\kappa_{\calB}^{-1} \calP^{y,\gr}_{\tau_0 w} \cong \zeta(\calL^{y,\gr}_w)
\otimes_{\calO_{\calB^{(1)}}} \calO_{\calB^{(1)}}(-\rho)$ in $\DGCoh^{\gr}(\wcalN^{(1)})$.

\end{thm}

Indeed, put $\calL_w^{y,\gr}:=(\mathbf{K}^{\Gm}_{\overline{y}})^{-1}(\calL_w^{\gr})$ and $\calP_w^{y,\gr}:=(\mathbf{J}_{\overline{y}}^{\dg,\gr})^{-1}(\calP_w^{\gr})$. Then the isomorphism of the theorem follows from the fact that $\kappa_{\calB}$ and $\zeta$ commute with the braid group action. 

Similarly, for $\lambda \in C$ there are ``graded versions'' of $\epsilon^{\calB}_{\lambda}$, $\widehat{\gamma}^{\calB}_{\lambda}$, with properties similar to those of $\widetilde{\epsilon}^{\calB}_0$, $\widetilde{\gamma}^{\calB}_0$.

\section{Application to Koszulity of the regular blocks of $(\calU \frakg)_0$} \label{sec:sectionapplications}

In this section we derive from Theorem
\ref{thm:mainthm} (or rather from the equivalent statement $(\ddag)$
of \S \ref{ss:paragraphmainthm}) that, for $\lambda \in
C_0$, the category $\Mod^{\fg}_0((\calU \frakg)^{\lambda})$ is
``controlled'' by a Koszul ring, whose Koszul dual controls the
category $\Mod^{\fg}_{\lambda}((\calU \frakg)_0)$. These results are counterparts in positive characteristic of the
results in \cite{SOEKat, BGS}; they also extend some results
of \cite[\S 18]{AJS}. We deduce this property from a general criterion for a graded ring to
be Morita equivalent to a Koszul ring, proved in
\S \ref{ss:paragraphkoszulitycriterion}.

\subsection{More on graded algebras}\label{ss:paragraphA}

Let $A$ be a $\bbZ$-graded ring. Recall the notation of \S \ref{ss:paragraphgradedrings}. Following \cite{NOGra}, if $M \in \Mod^{\gr}(A)$, the \emph{graded radical}
$\rad^{\gr}(M)$ of $M$ is the intersection of all maximal \emph{graded}
submodules of $M$. With this definition, $\rad^{\gr}$ has all the
usual properties of the radical (see \cite[A.I.7.4]{NOGra}). In
particular, if $A$ is considered as an $A$-module via left
multiplication, $\rad^{\gr}(A)$ is a graded two-sided ideal of $A$,
and \begin{equation} \label{eq:radicalgr} \rad^{\gr}(A)=\bigcap_{X \in \Mod^{\gr}(A) \ \text{simple}} {\rm Ann}(X). \end{equation}

{F}rom now on in this section we restrict to the following situation. Let $V$ be a graded finite
dimensional $\bk$-vector space, concentrated in positive degrees. Let
${\rm S}(V)$ be the symmetric algebra of $V$; it is naturally a graded
$\bk$-algebra, concentrated in non-negative degrees. We assume that
$A$ is a graded ${\rm S}(V)$-algebra, which is finitely generated as an
${\rm S}(V)$-module. Note in particular that the grading of $A$ is bounded below.

Consider the finite dimensional graded $\bk$-algebra
$\overline{A}:=A/(V \cdot A)$. By Theorem \ref{thm:thmGG}$\rmii$ and Corollary \ref{cor:corGG}, the simple $\overline{A}$-modules are
exactly the images of the simple \emph{graded} $\overline{A}$-modules under
the forgetful functor. Comparing \eqref{eq:radicalgr} with
\cite[5.5]{CRMet}, we deduce that \begin{equation}
  \label{eq:equalityradical} \rad(\overline{A}) =
  \rad^{\gr}(\overline{A}). \end{equation} A proof similar to
that of \cite[5.22]{CRMet} yields the following result.

\begin{prop} \label{prop:proptopA}

$\rmi$ The morphism $A \to \overline{A}$ induces an isomorphism of
graded rings $A/\rad^{\gr}(A) \cong
\overline{A}/\rad^{\gr}(\overline{A})$.

$\rmii$ For $k \gg 0$, $(\rad^{\gr}(A))^k \subseteq V \cdot A$.

\end{prop}

We denote by $\Hom_{A,\bbZ}(M,N)$ the morphisms in
$\Mod^{\gr}(A)$, and the corresponding
extension groups by $\Ext^i_{A,\bbZ}(M,N)$. By \cite[E.6]{AJS} we have:

\begin{lem} \label{lem:lemkrullschmidt}

$\rmi$ Let $M \in \Mod^{\fg,\gr}(A)$. If $M$ is indecomposable in the category $\Mod^{\fg,\gr}(A)$, then ${\rm End}_{A,\bbZ}(M)$ is a local algebra.

$\rmii$ The Krull-Schmidt theorem holds in $\Mod^{\fg,\gr}(A)$.

\end{lem}

If $L$ is a simple graded $A$-module, then $V \cdot L = 0$ (because $L$ is bounded below). Hence the simple graded $A$-modules are the simple graded $\overline{A}$-modules. Let $L_1, \ldots, L_r$ be representatives of
the simple non-graded $\overline{A}$-modules, and, for $i=1 \ldots r$,
let $L_i^{\gr}$ be a lift of $L_i$ as a graded
$\overline{A}$-module (which exists by Theorem \ref{thm:thmGG}$\rmii$). Using Corollary \ref{cor:corGG} and Theorem \ref{thm:thmGG}$\rmiv$, the $L_i \langle j \rangle$ are representatives of the simple graded
$\overline{A}$-modules, hence also of the simple graded
$A$-modules. As the ring $\overline{A}/\rad(\overline{A})$ is semi-simple
(\cite[5.19]{CRMet}), using \eqref{eq:equalityradical}, Proposition \ref{prop:proptopA}$\rmi$ and Corollary \ref{cor:corGG}, every graded $A/\rad^{\gr}(A)$-module is semi-simple in $\Mod^{\fg,\gr}(A/\rad^{\gr}(A))$. Using also Lemma \ref{lem:lemkrullschmidt}, every object of $\Mod^{\fg,\gr}(A)$ has a projective cover. For $i=1 \ldots r$, let $P_i^{\gr}$ be a projective cover of $L_i^{\gr}$. We
have \begin{equation} \label{eq:simplestop} L_i^{\gr} = P_i^{\gr} / \rad^{\gr}(P_i^{\gr}). \end{equation}

For $M$ in $\Mod^{\gr}(A)$
and $i \geq 0$, we define $\rad^{\gr,i}(M)$ by induction, setting
$\rad^{\gr,0}(M)=M$, and $\rad^{\gr,i}(M)=\rad^{\gr}(\rad^{\gr,i-1}(M))$ if $i \geq 1$.

\begin{lem} \label{lem:lemradicalfiltration}

Let $M$ be an object of $\Mod^{\fg,\gr}(A)$.

$\rmi$ $\rad^{\gr}(M)=\rad^{\gr}(A) \cdot M$.

$\rmii$ $\bigcap_{i \geq 0} \rad^{\gr,i}(M)=\{0\}$.

\end{lem}

\begin{proof} The proof of $\rmi$ is similar to that of
\cite[5.29]{CRMet}. As $A$ is noetherian we deduce that $\rad^{\gr,i}(M) = (\rad^{\gr}(A))^i \cdot M$ for $i \geq 0$. Then $\rmii$ follows from Proposition \ref{prop:proptopA}$\rmii$ and the fact that $M$ is bounded below.\end{proof}

\subsection{A Koszulity criterion} \label{ss:paragraphkoszulitycriterion}

Recall that a \emph{Koszul ring} $A=\bigoplus_{n \geq 0} A_n$ is a
non-negatively graded ring such that $A_0$ is a semi-simple ring and
the graded left $A$-module $A_0 \cong A/A_{>0}$ admits a graded projective
resolution \[ \cdots \to P^2 \to P^1 \to P^0 \to A_0 \to 0 \] such that
$P^i$ is generated by its degree $i$ part, for all $i$ (see
\cite{BGS}). If $A$ is a Koszul ring,
then its \emph{dual Koszul ring} is the graded ring\footnote{A Koszul
  ring is in particular a quadratic ring, and the dual Koszul ring is
  also the dual quadratic ring. The definition chosen here in easier
  to state, though less concrete.} $$A^! := \bigl(
\oplus_{n \geq 0} \ \Ext^n_A(A_0, A_0) \bigr)^{\op}$$ (here the
$\Ext$-groups are taken in the category of non-graded $A$-modules). If $A_1$ is an $A_0$-module of finite type, then $A^!$ is also a Koszul ring. If $A$ is a ring, one says that $A$ \emph{admits a Koszul
  grading} if it can be endowed with a grading which makes it a
Koszul ring. If $A$ is artinian, this grading is unique up to
automorphism if it exists (\cite[2.5.2]{BGS}).

The main result of this subsection is the following.

\begin{thm} \label{thm:thmkoszulity}

Let $A$, $L_i$, $L_i^{\gr}$ be as in \S {\rm \ref{ss:paragraphA}}. Assume
one can choose the lifts $L_i^{\gr}$ such that for $i,j=1,
\ldots, r$, \begin{equation} \label{eq:extvanishing}
  \Ext^n_{A,\bbZ}(L_i^{\gr}, L_j^{\gr} \langle m \rangle) = 0 \quad
  \text{unless } n=m.\end{equation} Then there exists a Koszul ring
$B$ which is (graded) Morita equivalent to $A$. If $L=\bigoplus_{i=1}^n L_i$, then $B^!$ is isomorphic to $
\bigl( \bigoplus_{n \geq 0} \ \Ext_A^n(L,L) \bigr)^{\op}$.

\end{thm}

The proof will occupy the rest of this subsection. Assume that \eqref{eq:extvanishing} is
satisfied, and let $P_i^{\gr}$ be the projective cover of
$L_i^{\gr}$.

\begin{lem} \label{lem:lemradPgr}

For $n \geq 0$ and $i=1 \ldots r$, $\rad^{\gr,n}(P_i^{\gr}) /
\rad^{\gr,n+1}(P_i^{\gr})$ is a direct sum of simple modules of the
form $L_j^{\gr} \langle n \rangle$ $(j \in \{1, \ldots, r\})$.

\end{lem}

\begin{proof} We prove the result by induction on $n \geq 0$. It is
clear for $n=0$, by \eqref{eq:simplestop}. Let $n \geq 1$, and assume it is true
for $n-1$. The graded $A$-module $\rad^{\gr,n}(P_i^{\gr}) / \rad^{\gr,n+1}(P_i^{\gr})$ factorizes through an $A/\rad^{\gr}(A)$-module. Hen\-ce it is semi-simple,
hence a direct sum of modules $L_j^{\gr} \langle m \rangle$ ($j \in
\{1, \ldots, r\}$, $m \in \bbZ$). The multiplicity of $L_j^{\gr}
\langle m \rangle$ is the dimension of the vector space {\small \[ \Hom_{A,\bbZ}(\rad^{\gr,n}(P_i^{\gr}) /
  \rad^{\gr,n+1}(P_i^{\gr}),L_j^{\gr} \langle m \rangle) \ \cong \
  \Hom_{A,\bbZ}(\rad^{\gr,n}(P_i^{\gr}), L_j^{\gr} \langle m
  \rangle). \]}Hence we have to prove that $\Hom_{A,\bbZ}(\rad^{\gr,n}(P_i^{\gr}),L_j^{\gr} \langle m \rangle)=0$
for $m \neq n$.

Consider the exact sequence \[ \rad^{\gr,n}(P_i^{\gr}) \hookrightarrow
\rad^{\gr,n-1}(P_i^{\gr}) \twoheadrightarrow \rad^{\gr,n-1}(P_i^{\gr}) /
\rad^{\gr,n}(P_i^{\gr}). \] For $j \in \{1, \ldots, r\}$ and $m
\in \bbZ$, it induces an exact sequence \begin{multline*} 0 \to
  \Hom_{A,\bbZ}(\rad^{\gr,n-1}(P_i^{\gr}) / \rad^{\gr,n}(P_i^{\gr}),
  L_j^{\gr} \langle m \rangle) \\ \xrightarrow{f}
  \Hom_{A,\bbZ}(\rad^{\gr,n-1}(P_i^{\gr}), L_j^{\gr} \langle m
  \rangle) \to \Hom_{A,\bbZ}(\rad^{\gr,n}(P_i^{\gr}), L_j^{\gr}
  \langle m \rangle) \\ \xrightarrow{g}
  \Ext^1_{A,\bbZ}(\rad^{\gr,n-1}(P_i^{\gr}) / \rad^{\gr,n}(P_i^{\gr}),
  L_j^{\gr} \langle m \rangle). \end{multline*} By usual properties of
$\rad^{\gr}$, $f$ is an isomorphism. Hence $g$ is
injective. Moreover, using induction and
\eqref{eq:extvanishing}, the last term is $0$ unless $m=n$. \end{proof}

We define $P^{\gr}:=\bigoplus_{i=1}^r P_i^{\gr}$, and $B:=\Hom_A(P^{\gr}, P^{\gr})^{\op}.$ As $P^{\gr}$ is
finitely generated, $B$ is naturally graded, with $n$-th
component \[ B_n:=\Hom_{A,\bbZ}(P^{\gr} \langle n \rangle, P^{\gr})
\ \cong \ \Hom_{A,\bbZ}(P^{\gr}, P^{\gr} \langle -n \rangle). \] Now we
prove, as a corollary of Lemma \ref{lem:lemradPgr}:

\begin{cor} \label{cor:corBpositive}

The algebra $B$ is non-negatively graded.

\end{cor}

\begin{proof} We have to prove that $\Hom_{A,\bbZ}(P^{\gr}, P^{\gr}
\langle n \rangle) =0$ for $n > 0$. Let $n \in \bbZ$, and $f: P^{\gr} \to P^{\gr} \langle n \rangle$ a non-zero
morphism. By Lemma \ref{lem:lemradicalfiltration}$\rmii$, the set $I=\{ i \geq
0 \mid f(P^{\gr}) \subseteq \rad^{\gr,i}(P^{\gr} \langle n \rangle)
\}$ is bounded above; let $i=\max(I)$. Then $f$ induces a non-zero
morphism $g: P^{\gr} \to \bigl(\rad^{\gr,i}(P^{\gr}) /
\rad^{\gr,i+1}(P^{\gr}) \bigr) \langle n \rangle.$ By Lemma
\ref{lem:lemradPgr}, we must have $n=-i$. \end{proof}

The algebra $B$ is finitely generated as an ${\rm S}(V)$-module, hence
noetherian (even as a non-graded ring). If $M \in
\Mod^{\fg,\gr}(A)$, then $\Hom_{A}(P^{\gr},M)$ is
naturally a graded $B$-module (see
\cite[E.3]{AJS}). By \cite[E.4]{AJS} we have:

\begin{prop} \label{prop:propequivAB}

There is an equivalence of abelian categories: \[\left\{ \begin{array}{ccc} \Mod^{\fg,\gr}(A) & \to &
    \Mod^{\fg,\gr}(B) \\ M & \mapsto & \Hom_A(P^{\gr}, M) \end{array}
\right. .\]

\end{prop}

We denote by $S_i^{\gr}$ the image of $L_i^{\gr}$ under this
equivalence; it is a simple graded $B$-module, one-dimensional, concentrated in
degree $0$. By
\eqref{eq:extvanishing}, \begin{equation}
  \label{eq:extvanishingB} \Ext^n_{B,\bbZ}(S_i^{\gr}, S_j^{\gr} \langle m
  \rangle) = 0 \quad \text{unless } n=m. \end{equation}

\begin{lem} \label{lem:lemB_0semisimple}

The (non-graded) ring $B_0$ is semi-simple.

\end{lem}

\begin{proof} Let $S_i$ be the image of $S_i^{\gr}$ under $\For:
  \Mod^{\gr}(B) \to \Mod(B)$. By Corollary \ref{cor:corBpositive},
  the $S_i$ are representatives of the simple $B_0$-modules. Now if $S_j \hookrightarrow M \twoheadrightarrow S_i$ is a non-split $B_0$-extension, we can consider
$M$ as a graded $B$-module concentrated in degree $0$, where $B$ acts
via the quotient $B/B_{>0} \cong B_0$. Then the exact sequence yields a
non-split graded $B$-extension of $S_i^{\gr}$ by $S_j^{\gr}$, contradicting
\eqref{eq:extvanishingB}. \end{proof}

\begin{prop} \label{prop:propBkoszul}

$B$ is a Koszul ring.

\end{prop}

\begin{proof} Apply \cite[2.1.3]{BGS}, using Corollary
\ref{cor:corBpositive}, Lemma \ref{lem:lemB_0semisimple},
\eqref{eq:extvanishingB}. \end{proof}

To conclude the proof of Theorem \ref{thm:thmkoszulity}, we only have
to compute $B^!$. The graded $B$-module $B_0$ is semi-simple, and $S_i^{\gr}$ occurs in this module with multiplicity
$\dim_{\bk}(\Hom_{B,\mathbb{Z}}(B_0,S_i^{\gr})) =
\dim_{\bk}(S_i^{\gr}) = 1$. Hence \[(B^!)^{\op} \ = \ \bigoplus_n
\Ext_B^n(B_0, B_0) \ \cong \ \bigoplus_{n,m}
\Ext^n_{B,\mathbb{Z}}(\bigoplus_i S_i^{\gr}, \bigoplus_i S_i^{\gr} \langle m
\rangle ). \] Then the result follows from Proposition
\ref{prop:propequivAB}.

\subsection{First consequences of Theorem \ref{thm:mainthm}}

We return to the setting of statement $(\ddag)$ (see \S \ref{ss:paragraphmainthm}), and choose the lifts
$\calP^{\gr}_w$, $P^{\gr}(w \bullet 0)$ and $\calL^{\gr}_v$, $L^{\gr}(v \bullet 0)$ as in the statement. Let $v,w \in
W^0$, and $i,j \in
\mathbb{Z}$. As the functors $\widetilde{\epsilon}^{\calB}_0$, $\zeta$, $\kappa_{\calB}$ and $\widetilde{\gamma}^{\calB}_0$ are fully faithful, and using \eqref{eq:isommainthm}, we have: \[ \begin{array}{cl} &
  \Hom_{\calD^b \Mod^{\fg,\gr}_0((\calU \frakg)^0)} \bigl( L^{\gr}(v \bullet 0), L^{\gr}(w \bullet 0)[i] \langle j \rangle \bigr) \\[4pt] \cong &
  \Hom_{\DGCoh^{\gr}((\wfrakg \, \rcap \, \calB)^{(1)})} \bigl( \kappa_{\calB}( \zeta(\calL^{\gr}_v) \otimes \calO_{\calB^{(1)}}(-\rho)), \\[3pt] & \qquad \qquad \qquad \qquad
  \qquad \qquad \kappa_{\calB}( \zeta(\calL^{\gr}_w) \otimes
  \calO_{\calB^{(1)}}(-\rho)) [i+j] \langle j \rangle \bigr) \\[4pt] \cong & \Hom_{\calD^b
    \Mod^{\fg,\gr}_0((\calU \frakg)_0)} \bigl( P^{\gr}(\tau_0 v \bullet 0),
  P^{\gr} (\tau_0 w \bullet 0)[i+j] \langle j \rangle \bigr). \end{array} \] As the objects $P^{\gr}(-)$ are projective, from these isomorphisms we
deduce: 

\begin{prop} \label{prop:propextvanishing}

Keep the assumptions of Theorem {\rm \ref{thm:mainthm}}. Let $v,w \in W^0$, and $i,j \in \mathbb{Z}$. We have $$\Hom_{\calD^b
    \Mod^{\fg,\gr}_0((\calU \frakg)^0)}(L^{\gr}(v \bullet 0), L^{\gr}
  (w \bullet 0)[i] \langle j \rangle) =0 \quad \text{unless } i=-j.$$

\end{prop}

Using the isomorphisms \[ \begin{array}{l} \bigoplus_{i \geq 0} \,
  \Ext^i_{(\calU \frakg)^0}(L(v
  \bullet 0), L(w \bullet 0)) \ \cong \\[4pt] \qquad \qquad \qquad \bigoplus_{i,j \in \mathbb{Z}} \,
  \Hom_{\calD^b \Mod^{\fg,\gr}_0((\calU \frakg)^0)}(L^{\gr}(v \bullet 0), L^{\gr}(w \bullet 0)[i]\langle j \rangle), \\[6pt] \Hom_{(\calU \frakg)_0}(P(v \bullet 0), P(w
  \bullet 0)) \ \cong \\[4pt] \qquad \qquad \qquad \bigoplus_{j \in \mathbb{Z}} \,
  \Hom_{\calD^b \Mod^{\fg,\gr}_0((\calU \frakg)_0}(P^{\gr}(v
  \bullet 0), P^{\gr} (w \bullet 0) \langle j \rangle) \end{array} \] (where we use \cite[3.1.7]{BMR} to identify the Ext groups in $\Mod^{\fg}((\calU \frakg)^0)$ and in $\Mod^{\fg}_0((\calU \frakg)^0)$),
we also deduce the following:

\begin{prop} \label{prop:propExtsimples}

Keep the assumptions of Theorem {\rm \ref{thm:mainthm}}.

$\rmi$ Let $v,w \in W^0$. There exists an isomorphism
\[ \bigoplus_{i \geq 0} \, \Ext^i_{(\calU \frakg)^0}(L(v
  \bullet 0), L(w \bullet 0)) \ \cong \ \Hom_{(\calU \frakg)_0}(P(\tau_0 v
  \bullet 0), P(\tau_0 w \bullet 0)). \]

$\rmii$ Let $L:=\bigoplus_{w \in W^0} L(w \bullet 0)$ and $P:=\bigoplus_{w \in
  W^0} P(w \bullet 0)$. There exists an algebra isomorphism
\[ \bigoplus_{i \geq 0} \ \Ext^i_{(\calU \frakg)^0}
  (L,L) \cong {\rm End}_{(\calU \frakg)_0}(P). \]

\end{prop}

Note that $\rmii$ is a modular counterpart of \cite[Theorem 18]{SOEKat}.

\subsection{The ring $A_{\wcalN}$} \label{ss:ring}

The arguments for this subsection are taken from \cite[\S 5.3.1]{BM}. Recall the vector bundle $\calM^0$ on the formal neighborhood of
$\calB^{(1)}$ in $\wfrakg^{(1)}$ defined in \S \ref{ss:reviewlocalization}. Let $\calM^0_0$ be the
restriction of $\calM^0$ to the formal neighborhood of $\calB^{(1)}$
in $\wcalN^{(1)}$. (This is the splitting bundle involved
in the definition of $\epsilon^{\calB}_0$.) In \S \ref{ss:paragraphgradedmodules} we have endowed $\calM^0$, hence also $\calM^0_0$, with a $\Gm$-equivariant structure. As the action of $\Gm$ contracts $\wcalN^{(1)}$ to the projective variety $\calB^{(1)}$, there exists a unique $\Gm$-equivariant vector bundle $\calM_{\wcalN}$ on $\wcalN^{(1)}$, whose
restriction to the formal neighborhood of $\calB^{(1)}$ is $\calM^0_0$.

Consider the algebra $A_{\wcalN} := \Gamma(\wcalN^{(1)},
\sheafEnd_{\calO_{\wcalN^{(1)}}}(\calM_{\wcalN})).$ This is a ${\rm
  S}(\frakg^{(1)})$-algebra, and it is finitely generated as a ${\rm
  S}(\frakg^{(1)})$-module (because the morphism $\wcalN^{(1)} \to
\frakg^*{}^{(1)}$ is proper). For any ${\rm S}(\frakg^{(1)})$-algebra
$A$, we denote by $\Mod^{\fg}_0(A)$ the
category of finitely generated $A$-modules, on which the
image of $\frakg^{(1)}$ acts nilpotently. As the algebras $A_{\wcalN}$ and $(\calU \frakg)^0$ have the same completion at the central character $0 \in \frakg^*{}^{(1)}$, we have
an equivalence \begin{equation} \label{eq:equivdef1}
  \Mod^{\fg}_0(A_{\wcalN}) \ \cong \ \Mod^{\fg}_0((\calU
  \frakg)^0). \end{equation} Observe also that $A_{\wcalN}$ has a natural grading, induced by the $\Gm$-equivariant structure on $\calM_{\wcalN}$.

\subsection{Koszulity of regular blocks of $(\calU \frakg)_0$}

One of our main results is the following. It is a modular counterpart of \cite[Theorem 3.9.1]{BGS}.

\begin{thm} \label{thm:thmkoszul}

Assume $p>h$ is large enough so that Lusztig's conjecture is true,
and let $\lambda \in \bbX$ be regular.

There exists a Koszul ring $B_{\calB}$, which is a ${\rm
  S}(\frakg^{(1)})$-algebra, and equivalences
\[ \Mod^{\fg}_0(B_{\calB}) \ \cong \
  \Mod^{\fg}_0((\calU \frakg)^{\lambda}), \qquad
  \Mod^{\fg}((B_{\calB})^!) \ \cong \ \Mod^{\fg}_{\lambda}((\calU
  \frakg)_0). \] In particular, the ring $(\calU \frakg)_0^{\hat{\lambda}}$ can be endowed
with a Koszul grading.

\end{thm}

\begin{remark} The fact that the category $\Mod^{\fg}_{\lambda}((\calU
  \frakg)_0)$ is equivalent to the category of modules
  over a Koszul ring was proved in \cite[18.21]{AJS}. Their proof
  relies on the computation of the Poincar{\'e} polynomial of
  $(\calU \frakg)_0^{\hat{\lambda}}$. The fact that
  the dual Koszul ring ``controls'' the category $\Mod^{\fg}_0((\calU
  \frakg)^{\lambda})$ is new. \end{remark}

\begin{proof}[Proof of Theorem \ref{thm:thmkoszul}] As $C_0$ is
  a fundamental domain for the action of $W_{\aff}$ on
  regular integral weights, we can assume $\lambda \in C_0$. Then, as
  the categories $\Mod^{\fg}_{\lambda}((\calU \frakg)_0)$ and $\Mod^{\fg}_0((\calU \frakg)^{\lambda})$) do not
  depend, up to equivalence, on the choice of $\lambda \in C_0$ (use
  translation functors), we can assume $\lambda=0$.

By definition (see \S \ref{ss:ring}), the algebra $A_{\wcalN}$ can be endowed with a
  grading; let $A_{\wcalN}^+$ be $A_{\wcalN}$ with this grading. We define the category
  $\Mod^{\fg,\gr}_0(A_{\wcalN}^+)$ as above. Then, as in \eqref{eq:equivdef1}, we have an equivalence
  \begin{equation} \label{eq:equivdef2} \Mod^{\fg,\gr}_0(A_{\wcalN}^+)
   \ \cong \ \Mod^{\fg,\gr}_0((\calU \frakg)^0). \end{equation}

Now, let $A_{\wcalN}^-$ be $A_{\wcalN}$ with the opposite grading. This is a
finite ${\rm S}(\frakg^{(1)})$-algebra, with $\frakg^{(1)}$ in
degree $2$. There is an equivalence $\Mod^{\gr}(A_{\wcalN}^+) \cong \Mod^{\gr}(A_{\wcalN}^-)$ inverting the
grading. Hence, using equivalence \eqref{eq:equivdef2} and
Proposition \ref{prop:propextvanishing}, the assumptions of Theorem
\ref{thm:thmkoszulity} are satisfied by
$A_{\wcalN}^-$. It follows that there exists a Koszul ring
$B_{\calB}$, Morita equivalent to $A_{\wcalN}^-$. By \eqref{eq:equivdef1}, the first equivalence of the
theorem is satisfied.

Again by Theorem \ref{thm:thmkoszulity} and equivalence
\eqref{eq:equivdef1}, with the notation of Proposition
\ref{prop:propExtsimples}, we have $(B_{\calB})^! \cong 
\bigl( \bigoplus_n \ \Ext^n_{(\calU \frakg)^0}(L,L) \bigr)^{\op}.$ By Proposition
\ref{prop:propExtsimples}$\rmii$, this ring is isomorphic to
$({\rm End}_{(\calU \frakg)_0}(P))^{\op}$, which is Morita equivalent to
$(\calU \frakg)_0^{\hat{0}}$ (\cite{Ba}). This gives the second equivalence.

Finally, the second assertion of the theorem follows from the second
equivalence (and the fact that $B_{\calB}^!$ is Koszul), using \cite[F.3]{AJS}. \end{proof}

\section[Parabolic analogues: Koszulity of singular blocks]{Parabolic analogues: Koszulity of singular blocks of $(\calU \frakg)_0$} \label{sec:sectionparabolicanalogues}

In this section we extend the results of sections \ref{sec:sectionmainthm}, \ref{sec:sectionapplications} to singular weights.

\subsection{Review of some results of \cite{BMR2}} \label{ss:paragraphreviewBMR2}

Let $P \subset G$ be a standard parabolic subgroup, $\calP:=G/P$
the associated flag variety, $\frakp$ be the Lie algebra of $P$, $\rho_{P}$
the half sum of the positive roots of the Levi of $P$, and $N_{\calP}:=\dim(\calP)$. Recall the variety $\wfrakg_{\calP}$ introduced in \S \ref{ss:reviewlocalization}. Let us also
consider \[ \wcalN_{\calP} \, := \, T^*
  \calP \, = \, \{(X,gP) \in \frakg^* \times \calP \mid X_{|g \cdot
    \frakp}=0\}. \] We have already considered this variety in \eqref{eq:defNalpha} in the special case $P=P_{\alpha}$. Under the isomorphism $\frakg \cong \frakg^*$, $\wfrakg_{\calP}$
identifies with the orthogonal of $\wcalN_{\calP}$ in $\frakg^* \times
\calP$. Hence we have a Koszul duality (see Theorem \ref{thm:thmlkd}): \[\kappa_{\calP} : \DGCoh^{\gr}(\wcalN_{\calP}^{(1)}) \ \xrightarrow{\sim} \
\DGCoh^{\gr}((\wfrakg_{\calP} \, \rcap_{\frakg^* \times \calP} \,
\calP)^{(1)}).\]

Choose $\mu \in \bbX$, on the reflection hyperplanes corresponding to the parabolic $P$, and not on any other reflection hyperlane (for $W_{\aff}$). By Theorem \ref{thm:localizationfixedFrparabolic} we have an equivalence of categories \[ \widehat{\gamma}^{\calP}_{\mu}: \DGCoh((\wfrakg_{\calP} \, \rcap_{\frakg^* \times \calP} \,
\calP)^{(1)}) \ \xrightarrow{\sim} \ \calD^b \Mod^{\fg}_{\mu}((\calU \frakg)_0). \]

Now let $\bbX_{\calP}$ be the sublattice of $\bbX$
consisting of the $\lambda \in \bbX$ such that $\langle
\lambda, \alpha^{\vee} \rangle = 0$ for any root $\alpha$ of the Levi
of $P$. For $\lambda \in \bbX_{\calP}$, let
$\mathfrak{D}^{\lambda}_{\calP}:=\calO_{\calP}(\lambda)
\otimes_{\calO_{\calP}} \calD_{\calP} \otimes_{\calO_{\calP}}
\calO_{\calP}(-\lambda)$ be the sheaf of twisted differential
operators on $\calP$ (as in \cite[1.10]{BMR2}). Let $\lambda \in \bbX_{\calP}$ be \emph{regular}. We
will assume\footnote{This condition is satisfied in particular if ${\rm
    char}(\bk)$ is greater than an explicitly computable
  bound depending on $G$ and $\lambda$ (see
  \cite[1.10.9$\rmii$]{BMR2}).} that \begin{equation}\label{eq:highervanishing} R^i \Gamma(\mathfrak{D}_{\calP}^{\lambda})=0 \quad \text{for} \ i>0. \end{equation} Then we define
$U^{\lambda}_{\calP}:=\Gamma(\mathfrak{D}^{\lambda}_{\calP}).$ 

We denote by $\Mod^{\fg}_0(U^{\lambda}_{\calP})$ the category of finitely generated $U^{\lambda}_{\calP}$-modules on which the central subalgebra $\Gamma(\wcalN_{\calP}^{(1)},\calO_{\wcalN_{\calP}^{(1)}})$ (the image of the center of $\mathfrak{D}^{\lambda}_{\calP}$) acts with trivial generalized character. By \cite[1.10.4]{BMR2} we have:

\begin{thm} \label{thm:thmBMR2}

Assume {\rm \eqref{eq:highervanishing}} is satisfied. There exists an equivalence \[\calD^b
\Coh_{\calP^{(1)}}(\wcalN_{\calP}^{(1)}) \ \xrightarrow{\sim} \ \calD^b
\Mod^{\fg}_0(U^{\lambda}_{\calP}).\]

\end{thm}

This theorem gives a representation-theoretic interpretation for the category
$\DGCoh^{\gr}(\wcalN_{\calP}^{(1)})$. As in Theorem \ref{thm:thmBMR}, the equivalence of Theorem \ref{thm:thmBMR2} depends on the choice of a splitting bundle. We choose it as in \cite[1.10.3]{BMR2}, and denote by $\Upsilon^{\calP}_{\lambda}$ the equivalence associated to $\lambda$. Let us remark that for $\calP=\calB$ we have $U^{\lambda}_{\calB} = (\calU \frakg)^{\lambda}$, but $\Upsilon^{\calB}_{\lambda}=\epsilon^{\calB}_{\lambda - p\rho}$ (see \cite[1.10.5]{BMR2}, and compare with the proof of Lemma \ref{lem:zero-section}). We deduce (see \S \ref{ss:reviewlocalization}):
\begin{equation} \label{eq:Upsilonepsilon} \Upsilon^{\calB}_{\lambda}(\calF)=\epsilon^{\calB}_{\lambda}(\calF
\otimes_{\calO_{\wfrakg^{(1)}}} \calO_{\wfrakg^{(1)}}(-\rho)). \end{equation} 

There is a natural morphism of algebras $\phi^{\lambda}_{\calP} : (\calU
\frakg)^{\lambda} \to U^{\lambda}_{\calP}$, induced by the action of $G$ on $\calP$ (see \cite[1.10.7]{BMR2}). We denote by $(\phi^{\lambda}_{\calP})^* : \calD^b \Mod^{\fg}_0(U^{\lambda}_{\calP}) \to \calD^b
\Mod^{\fg}_0((\calU \frakg)^{\lambda})$ the corresponding ``restriction'' functor. Consider
the diagram \[ \xymatrix{ \wcalN & \wcalN_{\calP} \times_{\calP} \calB \ar@{_{(}->}[l]_-{j_{\calP}} \ar@{->>}[r]^-{\rho_{\calP}} &
\wcalN_{\calP},} \] where $j_{\calP}$ is the natural embedding, and $\rho_{\calP}$ is induced by the projection $\pi_{\calP}: \calB \to \calP$. Then by \cite[1.10.7]{BMR2} the following holds:

\begin{prop} \label{prop:propBMR2}

The following diagram is commutative: \[ \xymatrix@R=15pt{ \calD^b
\Coh_{\calP^{(1)}}(\wcalN_{\calP}^{(1)}) \ar[rr]^-{\Upsilon^{\calP}_{\lambda}}_-{\sim} \ar[d]_-{(j_{\calP})_* (\rho_{\calP})^*} & & \calD^b
\Mod^{\fg}_0(U^{\lambda}_{\calP}) \ar[d]^-{(\phi^{\lambda}_{\calP})^*} \\ \calD^b
\Coh_{\calB^{(1)}}(\wcalN^{(1)}) \ar[rr]^-{\Upsilon^{\calB}_{\lambda}}_-{\sim} & & \calD^b
\Mod^{\fg}_0((\calU \frakg)^{\lambda}). } \]

\end{prop}

\subsection{Koszul duality for singular blocks} \label{ss:Kdualitysingular}

Choose $\lambda$ and $\mu$ as in \S \ref{ss:paragraphreviewBMR2}, and assume moreover that $\mu$ is in the closure of the alcove of $\lambda$. Let $y \in W_{\aff}$ be such that $\lambda_0:=y^{-1} \bullet \lambda \in C_0$. Then $\mu_0:=y^{-1} \bullet \mu \in \overline{C_0}.$

In what follows we make the following
assumption\footnote{By \cite[1.10.9]{BMR2}, this
  assumption is satisfied if ${\rm char}(\bk)$ is greater than an
  explicit bound depending on $G$ and $\lambda$ and, moreover, a sufficient
  condition is given for this to be satisfied in arbitrary
  characteristic. The latter condition is satisfied if $G={\rm
    SL}(n,\bk)$ (\cite[5.5]{HUMConj} and \cite{DONNor} or
  \cite{MKFro}) or if $P=P_{\alpha}$ for a short simple root
  $\alpha$ (\cite[5.3]{BKFro}).}: \begin{equation}
  \label{eq:phisurjective} \phi^{\lambda}_{\calP} \ \text{is
    surjective.} \end{equation} It follows that if $L$
is a simple $U^{\lambda}_{\calP}$-module then
$(\phi^{\lambda}_{\calP})^* L$ is a simple $(\calU
\frakg)^{\lambda}$-module. If $L$ has trivial central
character, then $(\phi^{\lambda}_{\calP})^* L \cong L(w \bullet
\lambda_0)$ for a unique $w \in W^0$ (see \S \ref{ss:objectsLwPw}). In this case, by definition we set $L=L_{\calP}(w \bullet \lambda_0)$. We denote
by $I_{\lambda}$ the set of $w \in W^0$ such that $L_{\calP}(w \bullet
\lambda_0)$ is defined.

Let $W^0_{\mu} \subset W^0$ by the subset of elements $w$ such that $w \bullet \mu_0$ is in the upper closure of $w \bullet C_0$. As in \S \ref{ss:objectsLwPw}, $\Mod^{\fg}_{\mu}((\calU
\frakg)_0)$ is the category of finitely generated modules over the
algebra $(\calU \frakg)_0^{\hat{\mu}}$ (the block of $(\calU
\frakg)_0$ associated to $\mu$). The simple objects
are the image of the simple $G$-modules $L(w \bullet \mu_0)$ for $w \in
W^0_{\mu}$. We denote by $P(w \bullet \mu_0)$ the projective cover of
$L(w \bullet \mu_0)$.

It is not clear a priori how to determine $I_{\lambda}$ in general; this will be part of Theorem \ref{thm:mainthmprojective} below. However, let us remark already that \begin{equation} \label{eq:cardinalIW} \# I_{\lambda} \ = \ \# W^0_{\mu}. \end{equation} Indeed, the left hand side of this equation is the rank of the Grothendieck group $K^0(\Mod^{\fg}_0(U^{\lambda}_{\calP}))$, which is isomorphic, by Theorem \ref{thm:thmBMR2}, to the Grothendieck group $K^0(\Coh_{\calP^{(1)}}(\wcalN_{\calP}^{(1)})) \cong K(\calP)$, while the right hand side is the rank of $K^0(\Mod^{\fg}_{\mu}((\calU
\frakg)_0))$, which is isomorphic to $K^0(\Mod^{\fg}_{(0,\mu)}(\calU
\frakg))$, hence, by Theorem \ref{thm:thmBMR}, to $K^0(\Coh_{\calP^{(1)}}(\wfrakg^{(1)}_{\calP})) \cong K(\calP)$.

As in \S \ref{ss:paragraphgradedmodules}, the algebra $(\calU
\frakg)_0^{\hat{\mu}}$ can be endowed with a grading, and there exists
a fully faithful triangulated functor $\widetilde{\gamma}^{\calP}_{\mu}$ commuting with internal shifts such that the following diagram
commutes: \[ \xymatrix@R=14pt{ \DGCoh^{\gr}((\wfrakg_{\calP} \, \rcap_{\frakg^* \times
    \calP} \, \calP)^{(1)}) \ar[rr]^-{\widetilde{\gamma}^{\calP}_{\mu}}
  \ar[d]_-{\For} & & \calD^b \Mod^{\fg,\gr}_{\mu}((\calU\frakg)_0)
  \ar[d]^-{\For} \\ \DGCoh((\wfrakg_{\calP} \, \rcap_{\frakg^* \times \calP} \, \calP)^{(1)}) \ar[rr]^-{\widehat{\gamma}^{\calP}_{\mu}}_-{\sim} & & \calD^b
  \Mod^{\fg}_{\mu}((\calU\frakg)_0). } \] One can lift the projective
modules $P(w \bullet \mu_0)$ to graded $(\calU
\frakg)_0^{\hat{\mu}}$-modules (uniquely, up to a shift; see Theorem
\ref{thm:thmGG}). Moreover, we have:

\begin{lem} \label{lem:projectivesparabolic}

The functor $\widetilde{\gamma}^{\calP}_{\mu}$ is an equivalence. In particular, the lifts of the projective modules $P(w \bullet \mu_0)$ are in the essential image of
$\widetilde{\gamma}^{\calP}_{\mu}$.

\end{lem}

\begin{proof} It is enough to prove that the lifts of the simple $(\calU \frakg)_0^{\hat{\mu}}$-modules are in the essential image of $\widetilde{\gamma}^{\calP}_{\mu}$. Let $\nu \in y \bullet C_0$, and let $\nu_0=y^{-1} \bullet \nu$. Consider the translation functor $T_{\nu}^{\mu} : \Mod^{\fg}_{\nu}((\calU \frakg)_0) \to \Mod^{\fg}_{\mu}((\calU \frakg)_0)$. For $w \in W^0_{\mu}$ we have $L(w \bullet \mu_0)=T_{\nu}^{\mu} L(w \bullet \nu_0)$. Moreover, by Proposition
\ref{prop:translationfixedFr}, we have an isomorphism $\widehat{\gamma}^{\calP}_{\mu} \circ R(\widehat{\pi}_{\calP})_* \cong T_{\nu}^{\mu} \circ \widehat{\gamma}^{\calB}_{\nu}$. Now $R(\widehat{\pi}_{\calP})_*$ has a graded
version \[ R(\widehat{\pi}_{\calP,\Gm})_*:
\DGCoh^{\gr}((\wfrakg \, \rcap_{\frakg^* \times \calB} \,
\calB)^{(1)}) \ \to \ \DGCoh^{\gr}((\wfrakg_{\calP} \, \rcap_{\frakg^* \times \calP} \,
\calP)^{(1)}).\] The functor $\widehat{\gamma}_{\nu}^{\calB}$ has a ``graded version'' $\widetilde{\gamma}_{\nu}^{\calB}$ (see \S \ref{ss:otheralcoves}) which, by Remark \ref{rk:gammaequivalence}, is an equivalence of categories. If, for $w \in W^0_{\mu}$, $\calM_w$ is the inverse image under $\widetilde{\gamma}_{\nu}^{\calB}$ of a lift of $L(w \bullet \nu_0)$, then $R(\widehat{\pi}_{\calP,\Gm})_* \calM_w$ is sent by $\widetilde{\gamma}^{\calP}_{\mu}$ to a lift of the simple module $L(w \bullet \mu_0) \in \Mod^{\fg}_{\mu}((\calU \frakg)_0)$. \end{proof}

Similarly, as in \S \ref{ss:paragraphgradedUg^0modules}, the completion
of $U^{\lambda}_{\calP}$ with respect to the trivial
central character can be endowed with a $\Gm$-equivariant structure,
and there exists a fully faithful functor $\widetilde{\Upsilon}^{\calP}_{\lambda}$ commuting with internal shifts such that the following diagram
commutes: \[ \xymatrix@R=14pt{ \calD^b \Coh^{\Gm}_{\calP^{(1)}}(\wcalN_{\calP}^{(1)})
  \ar[rr]^-{\widetilde{\Upsilon}^{\calP}_{\lambda}} \ar[d]_-{\For} & & \calD^b
  \Mod^{\fg,\gr}_0(U^{\lambda}_{\calP}) \ar[d]^-{\For} \\ \calD^b
  \Coh_{\calP^{(1)}}(\wcalN_{\calP}^{(1)}) \ar[rr]^-{\Upsilon^{\calP}_{\lambda}}_-{\sim} & &
  \calD^b \Mod^{\fg}_0(U^{\lambda}_{\calP}). } \] The simple objects in $\Mod^{\fg}_0(U^{\lambda}_{\calP})$ are the $L_{\calP}(w \bullet \lambda_0)$, $w \in I_{\lambda}$. They can be lifted to graded modules. We will prove below that the lifts of the simple modules are in the essential image of $\widetilde{\Upsilon}^{\calP}_{\lambda}$; in particular, this is an equivalence. 

Finally, as in \S \ref{ss:paragraphsimplemodules}, there exists a fully faithful functor $$\zeta_{\calP} : \calD^b \Coh^{\Gm}_{\calP^{(1)}}(\wcalN_{\calP}^{(1)}) \ \to \ \DGCoh^{\gr}(\wcalN_{\calP}^{(1)})$$ with the same properties as $\zeta$.

The following theorem is a ``parabolic analogue'' of Theorem \ref{thm:mainthm}.

\begin{thm} \label{thm:mainthmprojective}

Assume $p>h$ is large enough so that Lusztig's conjecture is true\footnote{See \S \ref{ss:introlusztig}.}. Assume
moreover that {\rm \eqref{eq:highervanishing}} and {\rm
  \eqref{eq:phisurjective}} are satisfied.

$\rmi$ We have $I_{\lambda} = \tau_0 W^0_{\mu}$, and the lifts of the simple modules are in the essential image of $\widetilde{\Upsilon}^{\calP}_{\lambda}$.

$\rmii$ There is a unique choice of the lifts\footnote{A priori, these lifts depend on the choice of $\lambda,\mu$, i.e.~on $y$.} $P^{\gr}(v \bullet \mu_0)$
($v \in W^0_{\mu}$), $L_{\calP}^{\gr}(u \bullet \lambda_0)$ ($u \in
I_{\lambda}$) such that, if $\calQ^{y,\gr}_{\calP,v}$,
resp. $\calL^{y,\gr}_{\calP,u}$ is the object of the category
$\DGCoh^{\gr}((\wfrakg_{\calP} \, \rcap_{\frakg^* \times \calP} \,
\calP)^{(1)})$, resp. $\calD^b
\Coh^{\Gm}_{\calP^{(1)}}(\wcalN_{\calP}^{(1)})$, such that $P^{\gr}(v
\bullet \mu_0) \cong
\widetilde{\gamma}^{\calP}_{\mu}(\calQ^{y,\gr}_{\calP,v})$,
resp. $L^{\gr}_{\calP}(u \bullet \lambda_0) \cong
\widetilde{\Upsilon}^{\calP}_{\lambda}(\calL^{y,\gr}_{\calP,u})$, for
all $w \in W^0_{\mu}$ we have: \begin{equation}
  \label{eq:isomparabolic} \kappa_{\calP}^{-1} \calQ^{y,\gr}_{\calP,w}
  \cong \zeta_{\calP}(\calL^{y,\gr}_{\calP,\tau_0 w})
  \otimes_{\calO_{\calP^{(1)}}}
\calO_{\calP^{(1)}}(2\rho_{P}-2\rho) \ \text{in }
\DGCoh^{\gr}(\wcalN_{\calP}^{(1)}). \end{equation}

\end{thm}

\begin{proof} We prove $\rmi$ and $\rmii$ simultaneously. Choose the objects $\calP^{y,\gr}_w$, $\calL^{y,\gr}_w$
  ($w \in W^0$) as in Theorem \ref{thm:mainthmalcove} (i.e.~as in Theorem \ref{thm:mainthm} if $y=1$). Here, to avoid
  confusion, we change the notation $\calP^{y,\gr}_w$ in
  $\calQ^{y,\gr}_w$. As for Theorem \ref{thm:mainthm}, the unicity
  statement is easy to prove; we concentrate on the existence.

By Proposition
\ref{prop:translationfixedFr}, we have an isomorphism of functors
\begin{equation} \label{eq:translationin} T_{\mu}^{\lambda} \circ
  \widehat{\gamma}^{\calP}_{\mu} \ \cong \
  \widehat{\gamma}^{\calB}_{\lambda} \circ
  L(\widehat{\pi}_{\calP})^*.\end{equation} By adjunction, and using equation
\eqref{eq:translationsimples}, we have for $w \in W^0_{\mu}$:
\begin{equation} \label{eq:translationprojectives} T_{\mu}^{\lambda}
  P(w \bullet \mu_0) \ \cong \ P(w \bullet \lambda_0). \end{equation} The
functor $L(\widehat{\pi}_{\calP})^*$ has a natural graded
version $L(\widehat{\pi}_{\calP,\Gm})^*$. For $w \in W^0_{\mu}$, we define $P^{\gr}(w \bullet
\mu_0)$ as the unique lift of $P(w \bullet \mu_0)$ such that, if
$\calQ^{y,\gr}_{\calP,w}$ is the object of
$\DGCoh^{\gr}((\wfrakg_{\calP} \, \rcap_{\frakg^* \times \calP} \,
\calP)^{(1)})$ such that $P^{\gr}(w \bullet \mu_0) \cong
\widetilde{\gamma}^{\calP}_{\mu}(\calQ^{y,\gr}_{\calP,w})$,
\begin{equation} \label{eq:projectiveparabolic} \calQ^{y,\gr}_w \langle
  N - N_{\calP} \rangle \ \cong \
  L(\widehat{\pi}_{\calP,\Gm})^*
  \calQ^{y,\gr}_{\calP,w}. \end{equation} Such a lift exists thanks to \eqref{eq:translationin}, \eqref{eq:translationprojectives} and Lemma \ref{lem:projectivesparabolic}.

The morphisms $j_{\calP}$ and $\rho_{\calP}$ induce functors {\small \[ \calD^b \Coh^{\Gm}_{\calP^{(1)}}(\wcalN_{\calP}^{(1)}) \xrightarrow{(\rho_{\calP,\Gm})^*}
  \calD^b \Coh^{\Gm}_{\calB^{(1)}}((\wcalN_{\calP} \times_{\calP}
  \calB)^{(1)}) \xrightarrow{(j_{\calP,\Gm})_*} \calD^b
  \Coh^{\Gm}_{\calB^{(1)}}(\wcalN^{(1)}).\]}
Consider the following factorization of $\widetilde{\pi}_{\calP}$: $ \wfrakg \xrightarrow{\widetilde{\pi}_{\calP,1}}
\wfrakg_{\calP} \times_{\calP} \calB
\xrightarrow{\widetilde{\pi}_{\calP,2}} \wfrakg_{\calP}, $
where $\widetilde{\pi}_{\calP,2}$ is induced by the projection
$\pi_{\calP}$. These morphisms induce \[
  (\wfrakg \, \rcap_{\frakg^* \times \calB} \,
  \calB)^{(1)} \xrightarrow{\widehat{\pi}_{\calP,1}} ((\wfrakg_{\calP} \times_{\calP} \calB)
  \, \rcap_{\frakg^* \times \calB} \, \calB)^{(1)} \xrightarrow{\widehat{\pi}_{\calP,2}}
  (\wfrakg_{\calP} \, \rcap_{\frakg^* \times \calP} \,
  \calP)^{(1)}. \] Then we have
$L(\widehat{\pi}_{\calP,\Gm})^* \ \cong \
L(\widehat{\pi}_{\calP,1,\Gm})^* \circ
L(\widehat{\pi}_{\calP,2,\Gm})^*$. Using this and the results of \S\S \ref{ss:koszulbasechange} and
\ref{ss:koszulinclusion}, one can identify the Koszul dual (with
respect to $\kappa_{\calB}$, $\kappa_{\calP}$) of
$L(\widehat{\pi}_{\calP,\Gm})^*$. Namely, by a proof similar to that
of Theorem \ref{thm:dualreflection} gives an isomorphism
\begin{multline} \label{eq:dualtranslationprojective}
  (\kappa_{\calB})^{-1} \circ L(\widehat{\pi}_{\calP,\Gm})^* \circ
  \kappa_{\calP} \ \cong \\ \bigl( R(\widetilde{j_{\calP}}_{\Gm})_*
  \circ L(\widetilde{\rho_{\calP}}_{\Gm})^* \bigr)
  \otimes_{\calB^{(1)}}
\calO_{\calB^{(1)}}(-2\rho_{P}) [N - N_{\calP}] \langle 2(N -
N_{\calP}) \rangle, \end{multline} where the functors
$R(\widetilde{j_{\calP}}_{\Gm})_*$ and
$L(\widetilde{\rho_{\calP}}_{\Gm})^*$ are defined as in \S\S \ref{ss:koszulbasechange} and \ref{ss:koszulinclusion}.

Now let $w \in W^0_{\mu}$. Consider $\calF_w := \bigl( \kappa_{\calP}^{-1} \calQ^{y,\gr}_{\calP,w}
\bigr) \otimes_{\calO_{\calP^{(1)}}}
\calO_{\calP^{(1)}}(2\rho-2\rho_{P}) \in \DGCoh^{\gr}(\wcalN_{\calP}^{(1)})$. By equation
\eqref{eq:dualtranslationprojective} we have \begin{multline*} \bigl(
  R(\widetilde{j_{\calP}}_{\Gm})_* \circ
  L(\widetilde{\rho_{\calP}}_{\Gm})^* \bigr)(\calF_w) \ \cong \
  \bigl( (\kappa_{\calB})^{-1} \circ
  L(\widehat{\pi}_{\calP,\Gm})^*(\calQ^{y,\gr}_{\calP,w} \\
  \otimes_{\calP^{(1)}} \calO_{\calP^{(1)}}(2\rho-2\rho_{P})) \bigr)
  \otimes_{\calB^{(1)}} \calO_{\calB^{(1)}}(2\rho_{P}) [N_{\calP}
  - N] \langle 2(N_{\calP} - N) \rangle. \end{multline*} Using
\eqref{eq:projectiveparabolic} and
\eqref{eq:isommainthm} (or its analogue in \S \ref{ss:otheralcoves} if $y \neq 1$) we deduce \[\bigl(
R(\widetilde{j_{\calP}}_{\Gm})_* \circ
L(\widetilde{\rho_{\calP}}_{\Gm})^* \bigr)(\calF_w) \ \cong \
\zeta(\calL^{y,\gr}_{\tau_0 w} \langle N_{\calP} - N
\rangle) \otimes_{\calB^{(1)}} \calO_{\calB^{(1)}}(\rho).\] Hence there exists $\calG_w \in \calD^b
\Coh_{\calP^{(1)}}^{\Gm}(\wcalN_{\calP}^{(1)})$ such that $\calF_w
\cong \zeta_{\calP}(\calG_w)$, and
\begin{equation} \label{eq:forsimple} (j_{\calP,\Gm})_*
  (\rho_{\calP,\Gm})^* \calG_w \ \cong \ \calL^{y,\gr}_{\tau_0 w}
  \otimes_{\wcalN^{(1)}} \calO_{\wcalN^{(1)}}(\rho) \langle
  N_{\calP} - N \rangle.\end{equation} By \eqref{eq:forsimple}, Proposition \ref{prop:propBMR2} and
\eqref{eq:Upsilonepsilon}, the image of $\widetilde{\Upsilon}^{\calP}_{\lambda}(\calG_w)$ under the composition
$\calD^b \Mod^{\fg,\gr}_0(U_{\calP}^{\lambda}) \xrightarrow{\For}
\calD^b \Mod^{\fg}_0(U_{\calP}^{\lambda}) \xrightarrow{(\phi^{\lambda}_{\calP})^*} \calD^b \Mod^{\fg}_0((\calU
\frakg)^{\lambda})$ is the simple module $L(\tau_0 w \bullet
\lambda_0)$. Hence $\tau_0 w \in I_{\lambda}$, and a lift of $L_{\calP}(\tau_0 w \bullet \lambda_0)$ is in the essential image of $\widetilde{\Upsilon}^{\calP}_{\lambda}$. If we set $L^{\gr}_{\calP}(\tau_0 w \bullet \lambda_0):=\widetilde{\Upsilon}^{\calP}_{\lambda}(\calG_w)$ and $\calL^{y,\gr}_{\calP,\tau_0 w}:=\calG_w$, then isomorphism \eqref{eq:isomparabolic} is true in this case.

In particular, we have proved that $\tau_0 W^0_{\mu} \subseteq I_{\lambda}$. As these sets have the same cardinality (see \eqref{eq:cardinalIW}), they must coincide. This finishes the proof. \end{proof}



\subsection{Koszulity of singular blocks of $(\calU \frakg)_0$}

The following theorem follows from Theorem \ref{thm:mainthmprojective},
as Theorem \ref{thm:thmkoszul} follows from Theorem
\ref{thm:mainthm}. It is a modular counterpart of \cite[Theorem 3.10.2]{BGS}.

\begin{thm} \label{thm:thmkoszulsingular}

Let $\lambda,\mu$ be as in \S {\rm \ref{ss:Kdualitysingular}}, and keep the assumptions\footnote{Assumptions \eqref{eq:highervanishing} and
  \eqref{eq:phisurjective} should not be essential for this result.} of Theorem {\rm
  \ref{thm:mainthmprojective}}. There exists a Koszul ring
$B_{\calP}$, which is a $\Gamma(\wcalN_{\calP}^{(1)},
\calO_{\wcalN_{\calP}^{(1)}})$-algebra, and equivalences of categories
\[ \Mod^{\fg}_0(B_{\calP}) \ \cong \
  \Mod^{\fg}_0(U^{\lambda}_{\calP}), \qquad \Mod^{\fg}((B_{\calP})^!) \
  \cong \ \Mod^{\fg}_{\mu}((\calU \frakg)_0). \] In particular, the ring $(\calU \frakg)_0^{\hat{\mu}}$ can be endowed
with a Koszul grading. 

\end{thm}

For any $\nu \in \bbX$, there exists a standard parabolic subgroup $P$, a weight $\mu \in W_{\aff}' \bullet \nu$, and a weight $\lambda$ which satisfy the hypotheses of Theorem
\ref{thm:thmkoszulsingular} (see \cite[1.5.2]{BMR2}). Hence the
ring $(\calU \frakg)_0^{\hat{\nu}} = (\calU \frakg)_0^{\hat{\mu}}$ can be endowed with a Koszul
grading for $p \gg 0$. As there are finitely many blocks, all the
blocks of $(\calU \frakg)_0$ can be endowed with a Koszul grading if $p \gg 0$. Finally, by \cite[F.4]{AJS} (the
implication we use is trivial) we deduce:

\begin{cor} \label{cor:KoszulUg0}

For $p \gg 0$, $(\calU \frakg)_0$ can be endowed with a
Koszul grading.

\end{cor}

\subsection{Remark on the choice of $\lambda$}

Let $p>h$. Fix a parabolic subgroup $P \supset B$, and let $I \subset \Phi$ be the corresponding simple roots. In \S \ref{ss:Kdualitysingular}, we have chosen $\lambda$ such that the closure of its alcove contains a weight $\mu$ of singularity $P$, i.e.~an integral weight in a facet which is open in $H_{P} := \{ \nu \in \bbX \otimes_{\mathbb{Z}} \mathbb{R} \mid \forall \alpha \in I, \ \langle \nu + \rho, \alpha^{\vee} \rangle = 0 \} $. It is not clear a priori that any regular $\lambda \in \bbX_{\calP}$ satisfies this assumption\footnote{This condition clearly does not hold in general if one does not require $\lambda$ to be integral.}. We claim that it is the case, however.

We can assume that $G$ is quasi simple. Let $A_0$ be the fundamental alcove, $A$ the alcove of $\lambda$, and let $w \in W_{\aff}'$ such that $A=w \bullet A_0$. What we have to check is that $\overline{A} \cap H_{P}$ contains an integral weight in an open facet of $H_{P}$, or that $\overline{A_0} \cap (w^{-1} \bullet H_{P})$ contains an integral weight in an open facet of $w^{-1} \bullet H_P$. Write $w=t_{\nu} v$, with $\nu \in \bbX$ and $v \in W$. Let $\lambda_0:=w^{-1} \bullet \lambda \in C_0$. If $\alpha \in I$ we have $0 = \langle \lambda, \alpha^{\vee} \rangle = \langle \lambda_0 + \rho, v^{-1} \alpha^{\vee} \rangle -1 + p \langle \nu, \alpha^{\vee} \rangle.$ By definition of $C_0$ we have $| \langle \lambda_0 + \rho, v^{-1} \alpha^{\vee} \rangle | < p$. Hence either $\rmi$ $\langle \nu, \alpha^{\vee} \rangle = 0$ and $\langle \lambda_0 + \rho, v^{-1} \alpha^{\vee} \rangle = 1$ (in this case $v^{-1} \alpha$ has to be a simple root), or $\rmii$ $\langle \nu, \alpha^{\vee} \rangle = 1$ and $\langle \lambda_0 + \rho, v^{-1} \alpha^{\vee} \rangle = 1-p$ (in this case $v^{-1} \alpha$ has to be the opposite of the highest short root). It follows that $\overline{A_0} \cap w^{-1} \bullet H_P$ is the closure of the facet of $\overline{A_0}$ defined by the simple roots appearing in $\rmi$ (if there are any) and the affine simple root (if case $\rmii$ occurs). This facet contains integral weights because it is the image under $w^{-1}$ of an open facet in $H_P$. This concludes the proof of the claim.

Hence Theorem \ref{thm:thmkoszulsingular} gives a Koszul duality for \emph{all} algebras $U^{\lambda}_{\calP}$.

\end{document}